\documentclass[11pt]{amsart}
\usepackage{amsmath,amsthm, amscd, amssymb, amsfonts}
\usepackage[all]{xy}
\usepackage{mathrsfs}
\usepackage{enumitem}
\usepackage[T2A,T1]{fontenc}
\newcommand{\ztt}{\mbox{\usefont{T2A}{\rmdefault}{m}{n}\cyrc}}
\newcommand{\sh}{\mbox{\usefont{T2A}{\rmdefault}{m}{n}\cyrsh}}

\newcommand{\ch}{\mbox{\usefont{T2A}{\rmdefault}{m}{n}\cyrch}}
\newcommand{\ya}{\mbox{\usefont{T2A}{\rmdefault}{m}{n}\cyrya}}
\newcommand{\ur}{\mbox{\usefont{T2A}{\rmdefault}{m}{n}\cyryu}}
\newcommand{\sch}{\mbox{\usefont{T2A}{\rmdefault}{m}{n}\cyrshch}}
\newcommand{\Sch}{\mbox{\usefont{T2A}{\rmdefault}{m}{n}\CYRSHCH}}
\newcommand{\zh}{\mbox{\usefont{T2A}{\rmdefault}{m}{n}\cyrzh}}
\usepackage{caption}

\usepackage[ansinew]{inputenc}
\usepackage{graphicx,fancyhdr}

\usepackage[dvips, dvipsnames, usenames]{color}

\newcommand{\comment}[1]{}

\numberwithin{equation}{section}
\newtheorem{theorem}{Theorem}[section]
\newtheorem{lemma}[theorem]{Lemma}
\newtheorem{coro}[theorem]{Corollary}
\newtheorem{hypothesis}[theorem]{Hypothesis}
\newtheorem{conjecture}[theorem]{Conjecture}
\newtheorem{prop}[theorem]{Proposition}

\theoremstyle{definition}
\newtheorem{definition}[theorem]{Definition}
\newtheorem{example}[theorem]{Example}
\newtheorem{question}[theorem]{Question}

\theoremstyle{remark}
\newtheorem{notation}[theorem]{Notation}
\newtheorem{remark}[theorem]{Remark}
\newtheorem{paso}{Step}

\newtheorem{case}{Case}
\newtheorem{kase}{Case}

\newtheorem{casso}{Step}

\newtheorem{caso2}{Step}

\def\pf{\begin{proof}}
\def\epf{\end{proof}}

\newcommand{\fidos}{\hspace{-1pt}\frac{2i+1}{2}}
\newcommand{\fhdos}{\hspace{-1pt}\frac{2h+1}{2}}

\newcommand{\fkdos}{\hspace{-1pt}k+\frac{1}{2}}
\newcommand{\fudos}{\hspace{-1pt}\frac{3}{2}}

\newcommand{\futres}{\hspace{-1pt}\frac{5}{2}}

\newcommand{\vi}{\textbf{(i)} }
\newcommand{\vii}{\textbf{(ii)} }

\newcommand{\ba}{ \mathbf{a}}
\newcommand{\bk}{ \mathbf{k}}
\newcommand{\bm}{ \mathbf{m}}
\newcommand{\bn}{ \mathbf{n}}
\newcommand{\ku}{ \Bbbk}
\newcommand{\kut}{ \mathbf{k}^{\times}}
\newcommand{\G}{\mathbb G}
\newcommand{\gb}{\mathbf g}
\newcommand{\ghost}{\mathscr{G}}
\newcommand{\inc}{\mathscr{I}}

\newcommand{\as}{a}
\newcommand{\qmb}{\mathtt{q}}
\newcommand{\I}{\mathbb I}
\newcommand{\Iw}{\mathbb I^{\dagger}}
\newcommand{\Idd}{\mathbb I^{\ddagger}}

\newcommand{\N}{\mathbb N}

\newcommand{\bq}{\mathbf{q}}

\newcommand{\Z}{\mathbb Z}
\newcommand{\zt}{\Z^{\theta}}

\newcommand{\cA}{\mathcal{A}}
\newcommand{\cB}{\mathcal{B}}
\newcommand{\cBt}{\widetilde{\mathcal{B}}}

\newcommand{\D}{\mathcal{D}}

\newcommand{\cJ}{\mathcal{J}}
\newcommand{\cK}{\mathcal{K}}

\newcommand{\cR}{\mathcal{R}}

\newcommand{\cV}{\mathcal{V}}
\newcommand{\cW}{\mathcal{W}}
\newcommand{\X}{\mathcal{X}}

\newcommand{\g}{\mathfrak g}

\newcommand{\Bg}{\mathfrak B}

\newcommand{\lstr}{\mathfrak L}
\newcommand{\cyc}{\mathfrak C}
\newcommand{\pos}{\mathfrak P}
\newcommand{\eny}{\mathfrak E}

\newcommand\ad{\operatorname{ad}}

\newcommand{\diag}{\operatorname{diag}}

\newcommand{\id}{\operatorname{id}}
\newcommand{\gr}{\operatorname{gr}}
\newcommand{\GK}{\operatorname{GKdim}}

\newcommand{\ord}{\operatorname{ord}}

\newcommand{\soc}{\operatorname{soc}}

\newcommand{\gkv}{\mathfrak d}
\def\ydh{{}^{H}_{H}\mathcal{YD}}
\def\ydk{{}^{K}_{K}\mathcal{YD}}

\newcommand{\Bdiag}{\mathcal{B}^\mathrm{diag}}
\newcommand{\Vdiag}{\mathcal{V}^\mathrm{diag}}
\newcommand{\NA}{\mathcal{B}}
\newcommand{\toba}{\mathcal{B}}
\newcommand{\ot}{\otimes}

\newcommand{\roots }{\boldsymbol{\Delta }}

\newcommand{\ydG}{{}^{\ku \Gamma }_{\ku \Gamma }\mathcal{YD}}

\newcommand{\ydz}{{}^{\ku \Z}_{\ku \Z}\mathcal{YD}}

\allowdisplaybreaks

\newcounter{tabla}\stepcounter{tabla}

\begin{document}

\title[On finite GK-dimensional Nichols algebras]{On finite GK-dimensional Nichols algebras over abelian groups}

\author[Andruskiewitsch; Angiono; Heckenberger]
{Nicol\'as Andruskiewitsch, Iv\'an Angiono, Istv\'an Heckenberger}

\address{FaMAF-CIEM (CONICET), Universidad Nacional de C\'ordoba,
Medina A\-llen\-de s/n, Ciudad Universitaria, 5000 C\' ordoba, Rep\'
ublica Argentina.} \email{(andrus|angiono)@famaf.unc.edu.ar}

\address{Philipps-Universität Marburg,
	Fachbereich Mathematik und Informatik,
	Hans-Meerwein-Straße,
	D-35032 Marburg, Germany.} \email{heckenberger@mathematik.uni-marburg.de}

\thanks{The work of N. A. and I. A.  was partially supported by CONICET,
Secyt (UNC), the MathAmSud project GR2HOPF. The work of I. A. was partially supported by ANPCyT (Foncyt).
The work of N. A., respectively I. A., was partially done during a visit to the University of Marburg, respectively the MPI (Bonn), supported by the Alexander von Humboldt Foundation.}

\begin{abstract}
We  contribute to the classification of Hopf algebras with finite Gelfand-Kirillov dimension, $\GK$ for short, through the study of Nichols algebras over abelian groups.
We deal first with braided vector spaces over $\mathbb Z$ with the generator acting as a single Jordan block and show that the corresponding Nichols algebra
has finite $\GK$ if and only if the size of the block is 2 and the eigenvalue is $\pm 1$; when this is 1, we recover the quantum Jordan plane.
We consider next a class of braided vector spaces that are direct sums of blocks and points that contains those of diagonal type. We conjecture that a
Nichols algebra of diagonal type has finite $\GK$ if and only if the corresponding generalized root system is finite. Assuming the validity of this conjecture,
we classify all braided vector spaces in the mentioned class whose
Nichols algebra has finite $\GK$. Consequently we present several new examples of  Nichols algebras with finite $\GK$,
including two not in the class alluded to above. We determine which among these Nichols algebras are domains.
\end{abstract}

\maketitle

\begin{quote}{\textit{As you set out for Ithaka \newline
			hope the voyage is a long one, \newline
			full of adventure, full of discovery.\newline
			Laistrygonians and Cyclops,\newline
			angry Poseidon—don’t be afraid of them: \newline
			you’ll never find things like that on your way \newline
			as long as you keep your thoughts raised high, \newline
			as long as a rare excitement \newline
			stirs your spirit and your body. \newline
			Laistrygonians and Cyclops, \newline
			wild Poseidon—you won’t encounter them \newline
			unless you bring them along inside your soul, \newline
			unless your soul sets them up in front of you.
		} \newline \textsc{C. P. Cavafy}}\end{quote}

\setcounter{tocdepth}{2}
\newpage
\tableofcontents

\newpage
\section{Introduction}\label{section:introduction}

\subsection{Antecedents}\label{subsec:intro-indecomposable}
Since the inception of quantum groups in the early 80's of the last century, the natural question of their characterization in structural terms is present. A first illuminating contribution was the intrinsic description by Lusztig  of the positive part $U^+_q(\g)$ of the quantized enveloping algebra of a simple  finite-dimensional Lie algebra $\g$  as what is called today a Nichols algebra \cite{L-libro}, see also \cite{R cras,Sbg}.
Then the pioneering paper \cite{R quantum groups} explained that finite Gelfand-Kirillov dimension is a crucial requirement to single out the $U^+_q(\g)$ (and their multiparametric versions) among the Nichols algebras of diagonal type, assuming $q >0$. The  classification of pointed Hopf algebras with finite $\GK$ that are domains, with abelian group of grouplikes and infinitesimal braiding of diagonal type was achieved in \cite{AS-crelle} starting from this result; the extension to $q$ generic was obtained in \cite{AA} by means of the theory in \cite{H-inv}. A finer characterization of quantum groups was offered further in \cite{ARS} in terms of representation theory. In all the mentioned articles, the starting point is a braided vector space of diagonal type. However there are very natural braided vector spaces not of diagonal type but related to natural examples of Hopf algebras, e.g. the quantum Jordan plane, see \S \ref{subsection:jordanian}. 
This paper starts a systematic treatment of Nichols algebras of braided vector spaces over abelian groups that are not of diagonal type.

We point out that the
classification of finite-dimensional Nichols algebras of diagonal type was
performed using the Weyl groupoid \cite{H-inv,H-classif}. However, this theory does not apply directly
to the braided vector spaces under our consideration, since these are not direct sums of simple Yetter-Drinfeld modules.

The last years witnessed a development of the interest on the classification and understanding of Hopf algebras with controlled growth, see \cite{BGLZ}. For instance, the study of noetherian Hopf algebras was raised and discussed in \cite{B-seattle}; updates of the state-of-the-art were given in \cite{B-turkish,BG,G-survey}. The first results on classification of Hopf algebras with low $\GK$
appeared in \cite{LWZ,BZ,Liu,GZ,WZZ,Z}; see an account in \cite{G-survey}.
Connected Hopf algebras with finite Gelfand-Kirillov dimension are quantum deformations of algebraic unipotent groups \cite{EG}.

The results of the present paper fit into the approach of the classification of pointed Hopf algebras through the lifting method \cite{AS Pointed HA}, as was the viewpoint  in \cite{AS-crelle, AA, ARS}; as said we focus on Nichols algebras with finite $\GK$ over abelian groups, one of the main steps of the lifting method. The explanation of how the remaining steps
can be performed is contained, in a special situation, in the article \cite{aah-jordan}.

\subsection{Points and blocks}\label{subsec:intro-history}
If $\theta \in\N$, then let $\I_\theta=\{1,2,\dots,\theta\}$. When $\theta$ is clear from the context we simply say $\I$.
If $\theta > \ell \in \N$, then we set $\I_{\ell, \theta}=\{\ell, \ell + 1,\dots,\theta\}$.
Let $\ku$ be an algebraically closed field of characteristic zero and $\Gamma$
an abelian group. We refer to \cite{AS Pointed HA} for the definitions of braided vector space, of
the category $\ydh$ of Yetter-Drinfeld modules over a Hopf algebra $H$ (always assumed with bijective antipode), and of the Nichols algebra $\NA(V)$ of $V \in \ydh$. Here we deal with the following question:

\begin{question}\label{question:nichols-finite-gkd} If  $V \in \ydG$, $\dim V < \infty$, when $\GK \NA(V) < \infty$?
\end{question}

As customary, we go back and forth between Yetter-Drinfeld modules and braided vector spaces. Note that the braided vector space $V$ is triangular as in \cite{Ufer}
(but we do not rely on the results of \cite{Ufer}). Also, not every triangular braiding arises from $\ydG$ for some abelian group $\Gamma$.

Braided vector spaces arising as Yetter-Drinfeld modules over abelian groups are direct sums of \emph{points} and \emph{blocks}:

\begin{itemize}[leftmargin=*]\renewcommand{\labelitemi}{$\diamond$}
	
	\item A \emph{point} of label $q\in \kut$ is a braided vector space $(V,c)$ of dimension 1 with $c = q\id$.
	Then $\NA(V)$ is isomorphic to a truncated polynomial ring $\ku[T]/\langle T^m\rangle$
	if $q$ is a root of 1 of order $m > 1$; or to a polynomial ring $\ku[T]$, otherwise.
	All irreducible objects in $\ydG$ are like this.
	
\medbreak
\item A \emph{block}  $\cV(\epsilon,\ell)$, where  $\epsilon\in \kut$ and $\ell \in \N_{\ge 2}$, is a braided vector space
	with a basis $(x_i)_{i\in\I_\ell}$ such that for $i, j \in \I_\ell  = \{1,2,\dots,\ell\}$, $1 < j$:
	\begin{align}\label{equation:basis-block}
		c(x_i \ot  x_1) &= \epsilon x_1 \ot  x_i,& c(x_i \ot  x_j) &=(\epsilon x_j+x_{j-1}) \ot  x_i.
	\end{align}
\end{itemize}

Here is our first significant result:

\begin{theorem}\label{theorem:blocks}  $\GK \NA(\cV(\epsilon,\ell)) < \infty$ if and only if
 $\ell = 2$ and $\epsilon \in \{\pm 1\}$.

	If this happens, then $\GK \NA(\cV(\epsilon,\ell)) = 2$.
\end{theorem}

See Propositions  \ref{pr:1block}, \ref{pr:-1block} and
Theorems  \ref{th:infGK}, \ref{thm:eigenvalue 1, size geq 3}.
The Nichols algebra $\NA(\cV(1,2))$ is a quadratic domain and has a PBW-basis with 2 generators;
it appeared in the literature under the name of \emph{Jordan plane}. See \S \ref{subsection:jordanian}.
The Nichols algebra $\NA(\cV(-1,2))$ is not a domain and has a PBW-basis with 3 generators, one of them
of height 2; it is presented by  one quadratic and one cubic relations. We call it a \emph{super Jordan plane}. See \S \ref{subsection:super-jordanian}.
If $\epsilon \in \{\pm 1\}$, then we refer to $\cV(\epsilon,2)$ as an \emph{$\epsilon$-block}.

\subsection{The main result}\label{subsec:intro-sb}
\subsubsection{The class of braided vector spaces}\label{subsubsec:intro-class}
We start by fixing the class of braided vector spaces $(V, c)$  considered in this paper. First, we suppose that
\begin{align}\label{eq:bradinig-generalform}
V &=  V_{1} \oplus \dots \oplus V_t \oplus  V_{t + 1} \oplus \dots \oplus V_\theta,
\\ \label{eq:bradinig-generalform1}
c(V_i \otimes V_j) &=  V_j \otimes V_i,\, i,j\in \I_{\theta},
\end{align}
where  $V_h$ is a $\epsilon_h$-block, with $\epsilon_h^2 = 1$, for $h\in \I_t$; and $V_{i}$ is a $q_{ii}$-point, with $q_{ii} \in \ku^{\times}$, $i\in \I_{t+1, \theta}$.
For simplicity, we assume that $V_h$ is a $1$-block,  for $1 \le h \le t_+$,  and a $-1$-block,  for $t_+ + 1 \le h \le t_+ + t_- = t$, where $t_+ \in \{0, \dots, t\}$ is the number of $1$-blocks.
If $i \neq j \in \I_{\theta}$, then we set
$c_{ij} = c_{\vert V_i \otimes V_j}: V_i \otimes V_j \to V_j \otimes V_i$.

\smallbreak
If $i,j \in \I_{t+1, \theta}$, i.e. they are points, then there exists $q_{ij} \in \ku^{\times}$ such that
\begin{align*}
c_{ij} &= q_{ij} \tau,
\end{align*}
where $\tau$ is the usual flip. We fix $x_j\in V_j$, non-zero, for $j \in \I_{t+1, \theta}$.

\smallbreak
Let now $i\in \I_t$ and $j \in \I_{t+1, \theta}$ (a block and a point); set $B_{ij} = \{i, \fidos, j\}$.
Here is our second hypothesis: we  assume that there exists a basis $x_i, x_{\fidos}$ of $V_i$, $q_{ij}, q_{ji} \in \kut$ and $a_{ji} \in \ku$ such that the braiding is given  by

\begin{align}\label{eq:intro-braiding-block-point}
(c(x_r \otimes x_s))_{r, s\in B_{ij}} &=
\begin{pmatrix}
\epsilon_i x_i  \otimes x_i &  (\epsilon_i x_{\fidos} + x_i ) \otimes x_i & q_{ij} x_j  \otimes x_i
\\
\epsilon_i x_i  \otimes x_{\fidos} & (\epsilon_i x_{\fidos} + x_i ) \otimes x_{\fidos}& q_{ij} x_j  \otimes x_{\fidos}
\\
q_{ji} x_i  \otimes x_j &  q_{ji}(x_{\fidos} + a_{ji} x_i ) \otimes x_j& q_{jj} x_j  \otimes x_j
\end{pmatrix}.
\end{align}
Notice that
\begin{align}
c_{ij}c_{ji} = \id \iff q_{ij}q_{ji} = 1 \text{ and } a_{ji}=0.
\end{align}
The \emph{interaction} between the block $i$ and the point $j$ is $\inc_{ij} = q_{ij}q_{ji}$; it is
\begin{align*}
\text{weak if } q_{ij}q_{ji}&= 1, &\text{mild if } q_{ij}q_{ji}&= -1,&\text{strong if } q_{ij}q_{ji}&\notin \{\pm 1\}.
\end{align*}

We define the \emph{ghost} between $i$ and $j$ as $
\ghost_{ij} = \begin{cases} -2a_{ij}, &\epsilon_j = 1, \\
a_{ij}, &\epsilon_j = -1.
\end{cases}$.
So $c_{ij}c_{ji}$ is determined by the interaction and the ghost.
If $\ghost_{ij} \in \N$, then we say that the ghost is \emph{discrete}.

\medbreak
We next impose the form of the braidings between two different blocks.
For every  $i \neq h \in \I_t$, there exist
	$q_{ih}, q_{hi} \in \kut$ and $a_{ih}, a_{hi} \in \ku$ such that the braiding between $V_i$ and $V_h$ with respect to the basis $B_i$ and $B_h$ as above
	is given  by
	\begin{align*} 
	&\text{the braidings of $V_i\oplus \ku x_h$ and $V_h\oplus \ku x_i$ are given as in \eqref{eq:intro-braiding-block-point};}
	\\
	&c \left(x_{\fidos} \otimes x_{\fhdos} \right) =   q_{ih} \left(x_{\fhdos} + a_{ih} x_h \right) \otimes x_{\fidos};
	\\
	&c \left(x_{\fhdos} \otimes x_{\fidos} \right) =   q_{hi} \left(x_{\fidos} + a_{hi} x_i \right) \otimes x_{\fhdos}.
	\end{align*}

\medbreak
Finally, set $i \sim j$ when $c_{ij}c_{ji} \neq \id_{V_j	 \otimes V_i}$, $i \neq j \in \I_\theta$. Let
$\approx$ be the equivalence relation on $\I_\theta$ generated by $\sim$.
Here is our last assumption:  $V$ is \emph{connected}, i.e  $i \approx j$ for all $i, j \in \I_\theta$.

\begin{remark} If $V \in \ydG$, $\dim V < \infty$, then $V$ does not have necessarily the form \eqref{eq:bradinig-generalform},
	\eqref{eq:bradinig-generalform1} with $c_{ij}$ given by \eqref{eq:intro-braiding-block-point} for $i\in \I_t$
and $j \in \I_{t+1, \theta}$. 
In our opinion, the best way to understand the class of braided vector spaces defined above is to 
define, following \cite[Definition 2.1]{Grana}, a decomposition of a braided vector space $\cW$ as a family of subspaces
$(\cW_i)_{i \in I}$ such that $\cW =  \oplus \cW_i$, $c(\cW_i \otimes\cW_j) =  \cW_j \otimes \cW_i$, $i,j\in I$.
Thus, a braided vector space is of diagonal type when it has a decomposition with all subspaces of dimension 1, 
whereas for the class we consider here, all subspaces are either points (i.e. of dimension 1) or blocks (specifically Jordan or super Jordan).
	
However there are examples of decompositions $\cW =  \cW_1\oplus \cW_2$ where $\cW_2$ is a point and $\cW_1$ is of diagonal type,
but the braiding of $\cW$ is not of diagonal type, in rough terms because `the action of the group-like corresponding to $\cW_2$ on $\cW_1$ is not diagonal'. We call these examples \emph{a pale block plus a point}.
The classification of the Nichols algebras of such braidings
with finite $\GK$ when $\dim \cW_1 =2$, hence $\dim \cW =3$, is given in \S \ref{section:bvs-abgps};
see Table \ref{tab:finiteGK-intro-5}.
The classification of all Nichols algebras over abelian groups
with finite $\GK$ requires an inductive approach, parallel the one in the present paper;
the first step would be to deal with Question \ref{question:paleblock}.

Also, if $V$ is a Yetter-Drinfeld module oven an abelian
group and the decomposition in \eqref{eq:bradinig-generalform} is  of Yetter-Drinfeld
modules (where $V_1,\dots ,V_t$ are blocks and $V_{t+1}, \dots V_{\theta}$ are points as assumed) then \eqref{eq:bradinig-generalform1} and \eqref{eq:intro-braiding-block-point}  hold.
\end{remark}

\begin{remark} Let $U, W \in \ydh$ both finite-dimensional
	and assume that $c_{U, W}c_{W, U} = \id_{W \otimes U}$. Then $\cB(U \oplus W) \simeq  \cB(U) \otimes \cB(W)$ by \cite{Grana} (at least as graded vector spaces) and
	$\GK\cB(U \oplus W) = \GK\cB(U) + \GK \cB(W)$.
Hence it is enough to reduce to  connected $V$.
\end{remark}

\subsubsection{Diagonal type}\label{subsubsec:intro-dt}
We discuss here the case $t = 0$, i.e. $V$ is of diagonal type. A complete answer is not known presently, but we dare to propose:

\begin{conjecture}\label{conjecture:nichols-diagonal-finite-gkd} If $V$ is of diagonal type and $\GK \NA(V) < \infty$, then the corresponding root system is finite.
\end{conjecture}

This would imply the classification of the Nichols algebras of diagonal type with finite $\GK$ from \cite{H-classif}.
See \cite{A-jems, Ang-crelle} for the defining relations. There is  evidence on the validity of this Conjecture.

\begin{theorem}\label{thm:nichols-diagonal-finite-gkd} \cite{diag}
	Assume that $V$ is of diagonal type.
	If either its generalized root system is infinite and $\dim V = 2$, or else $V$ is of affine Cartan type, then $\GK \NA(V) = \infty$.
\end{theorem}

The converse implication (finite root system implies finite $\GK$)
is known, see \cite{H-inv}.
Braided vector spaces of  diagonal type play a crucial role in our approach;  in this text we  assume that

\begin{hypothesis}\label{hyp:nichols-diagonal-finite-gkd} Conjecture \ref{conjecture:nichols-diagonal-finite-gkd} is true.
\end{hypothesis}

This Hypothesis is used in the proof of our main result Theorem \ref{theorem:main} not as a restriction
on the sub-diagram of diagonal type but rather as a tool to discard possibilities. 
Namely, in several instances we 
obtain an auxiliary braided vector space of diagonal type that is discarded by Hypothesis \ref{hyp:nichols-diagonal-finite-gkd}. See \S \ref{subsubsection:about-the-proofs} for an exposition;
the Hypothesis is used in the proofs of
Lemmas \ref{lemma:mild-odd-order-2}, \ref{lemma:KJ-card2}, \ref{lemma:KJ-card-geq3}, \ref{lemma:superJordan-2-mild},
\ref{lemma:block-points-eps-1-aux1}, \ref{lemma:several-blocks-connected-components} and \ref{lemma:several-blocks-connected-components-2}
and in the proof of Theorem \ref{th:paleblock-point-resumen}.

\subsubsection{Flourished graphs}\label{subsubsec:intro-flourished-graph}
Towards our main classification result, we need to introduce a class of decorated graphs, extending the generalized Dynkin diagrams of \cite{H-classif},
designed to describe (some) braided vector spaces of the form 
\eqref{eq:bradinig-generalform}, \eqref{eq:bradinig-generalform1}.

\begin{definition}\label{def:flourished-graph}
	A flourished graph is a graph $\D$ with $\theta$  vertices and  the following decorations:
	\begin{itemize} [leftmargin=*]\renewcommand{\labelitemi}{$\circ$}
		\item The vertices have three kind of decorations $+$, $-$ and $q \in \ku^{\times}$; they are depicted respectively as $\boxplus$, $\boxminus$ and $\overset{q} {\circ}$. There are $t_+$ vertices of the first kind, and $t_-$ of the second.
		They are numbered respectively $1, \dots, t_+$; $t_+ + 1, \dots, t_+ + t_- =: t$; $t+1, \dots, \theta$ (with possibly different $q$'s). The vertices in $\I_t$ are called blocks, the remaining are called points.

		\smallbreak
		\item If $i \neq j$ are points, and there is an edge between them, then it is decorated by some $\widetilde{q}_{ij} \in \ku^{\times} - 1$: $\xymatrix{\overset{q_i}  {\circ} \ar  @{-}[r]^{\widetilde{q}_{ij}}  & \overset{q_j}  {\circ} }$.

		\smallbreak
		\item  If $h$ is a block and $j$ is a point, then an edge between $h$ and $j$ is decorated either by $\ghost_{hj}$ or by $(-, \ghost_{hj})$ for some $\ghost_{hj}\in \ku^{\times}$; or not decorated at all.
	\end{itemize}
	The full subgraph with vertices $\I_{t+1, \theta}$ is denoted $\D_{\diag}$; it is a generalized Dynkin diagram \cite{H-classif}.
	The set of its connected components is denoted $\X$.
\end{definition}

Let $V$ be as in \S \ref{subsubsec:intro-class}. We attach a flourished graph $\D$ to $V$ by the following rules. The set of vertices of $\D$ is $\I_\theta$. The edges obey the following rules:

\begin{itemize} [leftmargin=*]\renewcommand{\labelitemi}{$\bullet$}
	\item  If $h\in \I_{t}$, then   corresponding   vertex is depicted as $\boxplus$ when $\epsilon_h = 1$, respectively $\boxminus$ when $\epsilon_h = -1$.
	
	\smallbreak
	\item  If $j \in \I_{t+1, \theta}$,  then   the corresponding   vertex is depicted as $\overset{q_{jj}}  {\circ}$.
	
	\smallbreak
	\item There is an edge between $i$ and $j \in \I_{\theta}$ iff $i \sim j$.
	
	\smallbreak
	\item  If $i \in \I_{t}$, $j \in \I_{t+1, \theta}$, $\inc_{ij}$ is weak and $\ghost_{ij}\neq 0$,
	respectively $\inc_{ij}$ is mild, then the  edge between $i$ and $j$ is labelled by $\ghost_{ij}$, respectively
	by $(-, \ghost_{ij})$.
	
	\smallbreak
	\item If $ i\neq j \in \I_{t+1, \theta}$ and $q_{ij}q_{ji} \neq 1$, then the corresponding edge is decorated by $\widetilde{q}_{ij} = q_{ij}q_{ji}$.
\end{itemize}

\bigbreak
Here are the flourished graphs parametrizing Nichols algebras with finite $\GK$:

\begin{definition}\label{def:flourished-graph-admissible}
	A flourished graph is \emph{admissible} when it is not of diagonal type (i.e. $t>0$) and
	the following conditions hold.
	
	\begin{enumerate}[leftmargin=*,label=\rm{(\alph*)}]
		\item\label{item:noedges} There are no edges between blocks.
		
			\smallbreak
			\item\label{item:local}
			The only possible connections between a connected component $J \in \X$
			and one block are described in Table  \ref{tab:conn-comp} (the point connected with the block is black for emphasis).
\begin{table}[ht]
	\caption{Connecting components and blocks; $\ghost_{ij} \in \N$, $\omega \in \G_3'$.} \label{tab:conn-comp}
	\begin{center}
		\begin{tabular}{|c c|}
			
			\hline
			$\xymatrix{\boxplus \ar  @{-}[r]^{\ghost_{ij}}  & \overset{1}{\bullet}}$ \qquad $\xymatrix{\boxminus \ar  @{-}[r]^{\ghost_{ij}}  & \overset{1}{\bullet}}$	
			& $\xymatrix{\boxplus \ar  @{-}[r]^{1} &\overset{-1}{\bullet} \ar  @{-}[r]^{\omega^2 }  & \overset{\omega}{\circ}}$
			\\ \hline	
			$\xymatrix{\boxplus \ar  @{-}[r]^{\ghost_{ij}}  & \overset{-1}{\bullet}}$ \qquad $\xymatrix{\boxminus \ar  @{-}[r]^{\ghost_{ij}}  &\overset{-1}{\bullet} }$	&
			$\xymatrix{\overset{-1}{\circ} \ar  @{-}[r]^{\omega^2 }  & \overset{\omega}{\bullet} \ar  @{-}[r]^{1} & \boxplus }$
			\\ \hline
			$\xymatrix{\boxplus \ar  @{-}[r]^{1}  & \overset{\omega}{\bullet}}$	\qquad $\xymatrix{\boxminus \ar  @{-}[r]^{(-1, 1)}  &\overset{-1}{\bullet} }$  & $\xymatrix{\boxplus \ar  @{-}[r]^{1} &\overset{-1}{\bullet} \ar  @{-}[r]^{\omega}  & \overset{\omega^2}{\circ} \ar  @{-}[r]^{\omega}  & \overset{\omega^2}{\circ} }$
			\\ \hline
			$\xymatrix{\boxplus \ar  @{-}[r]^{1}  &\overset{-1}{\bullet} \ar  @{-}[r]^{-1}  & \overset{-1}{\circ}} \dots \xymatrix{
				\overset{-1}{\circ} \ar  @{-}[r]^{-1}  & \overset{-1}{\circ} }$
			& 	$\xymatrix{\boxplus \ar  @{-}[r]^{1} & \overset{-1}{\bullet} \ar  @{-}[r]^{\omega}  & \overset{\omega^2}{\circ} \ar  @{-}[r]^{\omega^2}  & \overset{\omega}{\circ} }$
			\\ \hline
			$\xymatrix{\boxplus \ar  @{-}[r]^{2}  &\overset{-1}{\bullet} \ar  @{-}[r]^{-1}  & \overset{-1}{\circ}}$ 	 &  $\xymatrix{\boxplus \ar  @{-}[r]^{1} &\overset{-1}{\bullet} \ar  @{-}[r]^{r^{-1}}  & \overset{r}{\circ}}$, $r \neq \pm 1$
			\\ \hline
			$\xymatrix{\boxplus \ar  @{-}[r]^{1}  &\overset{-1}{\bullet} \ar  @{-}[r]^{\omega}  & \overset{-1}{\circ}}$
			& $\xymatrix{\boxminus \ar  @{-}[r]^{(-1, 1)}  &\overset{-1}{\bullet} \ar  @{-}[r]^{-1}  & \overset{-1}{\circ}}$
			
			\\ \hline
		\end{tabular}
	\end{center}
	
\end{table}

		\smallbreak
		\item\label{item:point} Let $J \in \X$ (a connected component of $\D_{\diag}$). Then there is a unique $j \in J$ connected to a block.

		\smallbreak
		\item\label{item:concom} If $J \in \X$ has $\vert J \vert > 1$, then it is connected to a unique block $i$.
		
			\smallbreak
			\item\label{item:concom-1} If $J =\{j\} \in \X$ and $q_{jj} \in \G'_3$, then it is connected to a unique block $i$.
		
		\smallbreak
		\item\label{item:block} If a block $h$ is connected to a point $j$ by an edge labelled $(-, \ghost_{hj})$ for some $\ghost_{hj}\in \ku^{\times}$,
		then there is no other edge connecting $h$ with a point and there is no other edge connecting $j$ with a block.
	\end{enumerate}
\end{definition}

The next is the main result of this monograph.

\begin{theorem}\label{theorem:main} Let $V$ be a braided vector space as in \S \ref{subsubsec:intro-class}, not of diagonal type. Then $\GK \NA(V) < \infty$ if and only if its flourished graph $\D$ is admissible.
\end{theorem}

\begin{table}[ht]
	\caption{Nichols algebras of a block and a point with finite $\GK$}\label{tab:finiteGK-intro-1}
	{\small
		\begin{gather}\notag
		\begin{tabular}{|c|c|c|}
		\hline $V$ & $\toba(V)$;  $\GK$  & generators and relations     \\
		\hline
		$\lstr(1, \ghost)$ 
		& \S \ref{subsubsection:lstr-11disc}; $\ghost + 3$   &  $\ku\langle x_1, x_2, x_3
		\vert x_2x_1-x_1x_2+\frac{1}{2}x_1^2,
		x_1x_3 - q_{12} \, x_3x_1,$
		\\
		$\xymatrix{\boxplus \ar  @{-}[r]^{\ghost}  & \overset{1}{\bullet}}$
		& & $z_{1+\ghost},  z_tz_{t+1} - q_{21}q_{22} \, z_{t+1}z_t,  0\le t < \ghost\rangle$
		\\ \hline
		$\lstr(-1, \ghost)$ & \S \ref{subsubsection:lstr-1-1disc}; $2$   &  $\ku\langle x_1, x_2, x_3
		\vert x_2x_1-x_1x_2+\frac{1}{2}x_1^2, $
		\\
		$\xymatrix{\boxplus \ar  @{-}[r]^{\ghost}  & \overset{-1}{\bullet}}$ & & $x_1x_3 - q_{12} \, x_3x_1, z_{1+\ghost},  z_t^2, \,  0\le t \le \ghost\rangle$
		\\ \hline
		$\lstr_{-}(1, \ghost)$ & \S \ref{subsubsection:lstr--11disc}; $\ghost + 3$
		&  $\ku\langle x_1, x_2, x_3\vert x_1^2, x_2x_{21}- x_{21}x_2 - x_1x_{21}$,
		\\
		$\xymatrix{\boxminus \ar  @{-}[r]^{\ghost}  & \overset{1}{\bullet}}$
		& & $x_1x_3 - q_{12} \, x_3x_1,  x_{21}x_3 - q_{12}^2 x_3x_{21}, z_{1+2\ghost},$
		\\
		& & $z_{2k+1}^2, z_{2k}z_{2k+1} - q_{21}q_{22} \, z_{2k+1}z_{2k},  0\le k < \ghost\rangle$
		\\ \hline
		$\lstr_{-}( -1, \ghost)$ & \S \ref{subsubsection:lstr--1-1disc}; $\ghost + 2$
		&  $\ku\langle x_1, x_2, x_3\vert x_1^2, x_2x_{21}- x_{21}x_2 - x_1x_{21}, x_3^2$,
		\\
		$\xymatrix{\boxminus \ar  @{-}[r]^{\ghost}  & \overset{-1}{\bullet}}$
		& & $x_1x_3 - q_{12} \, x_3x_1, x_{21}x_3 - q_{12}^2 x_3x_{21},z_{1+2\ghost},  $
		\\
		& & $z_{2k}^2, z_{2k-1}z_{2k} - q_{21}q_{22} \, z_{2k}z_{2k-1},  0 < k \le \ghost\rangle$
		\\ \hline
		$\lstr(\omega, 1)$ & \S \ref{subsubsection:lstr-1omega1}; $2$   &  $\ku\langle x_1, x_2, x_3
		\vert x_2x_1-x_1x_2+\frac{1}{2}x_1^2, $
		\\
		$\xymatrix{\boxplus \ar  @{-}[r]^{1}  & \overset{\omega}{\bullet}}$ 
		& & $x_1x_3 - q_{12} \, x_3x_1,\, z_{2}, \, x_3^3, \,  z_1^3,\, z_{1,0}^3\rangle$
		\\ \hline
		$\cyc_1$ & \S \ref{subsubsection:nichols-mild}; 2 &  $\ku\langle x_1, x_2, x_3
		\vert  x_1^2, x_2x_{21}- x_{21}x_2 - x_1x_{21}, $
		\\
		$\xymatrix{\boxminus \ar  @{-}[r]^{(-1, 1)}  &\overset{-1}{\bullet} }$
		& & $ x_3^2, f_0^2, \, f_1^2, \,  z_1^2, \, x_{21}x_3 - q_{12}^2 x_3x_{21},$
		\\
		& & $
		x_{2}z_1 + q_{12} z_1x_{2} - q_{12} \, f_0 x_2 - \frac{1}{2}f_1\rangle$
		\\ \hline
		\end{tabular}
		\end{gather} }
\end{table}

\subsubsection{Organization of the paper and scheme of the proof}
Section \ref{section:Preliminaries} is devoted to preliminaries. The next Sections contain several partial classification results.
In Section \ref{sec:yd-dim2} we discuss Nichols algebras of blocks and prove a large part of Theorem \ref{theorem:blocks}. In Section \ref{sec:yd-dim2} we classify Nichols algebras of a direct sum of a block and a point  with finite $\GK$, cf. Theorem \ref{thm:point-block}. Along the way, the techniques in this Section allow us to  finish the proof of Theorem \ref{theorem:blocks}. In Section \ref{sec:yd-dim>3} we classify Nichols algebras of a direct sum of a block and several points  with finite $\GK$, cf. Theorems \ref{thm:points-block-eps1} and \ref{thm:points-block-eps-1}.
In Section \ref{section:YD>3-2blocks} we show that the Nichols algebra of a direct sum of two blocks has finite $\GK$ if and only if the blocks commute in the braided sense. Section \ref{sec:yd-dim>3-2blocks} contains the discussion of the general case, while in Section \ref{section:bvs-abgps} we discuss Examples of admissible flourished graphs and  the braided vector spaces over abelian groups not in the class discussed in this monograph.

\medbreak Let us discuss now the general scheme of the proof of Theorem \ref{theorem:main},
assuming the results listed above and other technical Propositions.
We proceed inductively on $t$ and $\theta$. First, we observe that if $\GK \NA(V) < \infty$ and $i, j \in \I_t$ are two blocks, then $c_{ji}c_{ij} = \id$ by  \S \ref{section:YD>3-2blocks}, thus $\D$ satisfies \ref{item:noedges}.

\smallbreak
(i). \emph{There are no points, i.e. $t = \theta$}.
Hence $\theta = 1$ by the preceding observation (since we assume that $V$ is connected), i.e. it is a block.
Then  Theorem \ref{theorem:blocks} says that $V$ is a Jordan plane, or a super Jordan plane.

\smallbreak
(ii).  $\theta = 2$.
It is enough to consider $t = 1$ by hypothesis and the preceding observation, i.e. it is a block and a point.
Theorem \ref{thm:point-block} says that  $\GK \NA(V) < \infty$ iff  $V$  appears in Table \ref{tab:finiteGK-intro-1}.
Clearly, the corresponding flourished diagrams are those in Table  \ref{tab:conn-comp} with $\theta = 2$.

\smallbreak
(iii).  $\theta > 2$, $t =1$.
Theorems \ref{thm:points-block-eps1}, \ref{thm:points-block-eps-1} say that  $\GK \NA(V) < \infty$ if and only if  the connected components of $V_{\diag} := V_{t+1} \oplus \dots \oplus V_{\theta}$  appear in Tables \ref{tab:finiteGK-intro-1} and  \ref{tab:finiteGK-intro-2}, and  the existence of a component with mild interaction forces the diagram of $V_{\diag}$ to be connected.
This implies  \ref{item:local} and the first claim in  \ref{item:block}.

\smallbreak
(iv). Lemma \ref{lemma:several-blocks-connected-components} provides \ref{item:point} and Lemma \ref{lemma:several-blocks-connected-components-2} provides \ref{item:concom};  Theorem \ref{thm:variosblocks-unpto} implies \ref{item:concom-1}; the second claim in \ref{item:block} follows from Lemma \ref{lemma;several-blocks-1pt-weak}.

\smallbreak
Thus, we see that the flourished diagram of a $V$ with $\GK \NA(V) < \infty$ is admissible.
Conversely, if the flourished diagram of $V$ is admissible,  then $\GK \NA(V) < \infty$ by Theorem \ref{theorem:final}.

For illustration, we describe the Nichols algebras of one point and several blocks  in \S \ref{tab:finiteGK-intro-4}.

\begin{table}[ht]
	\caption{{\small Nichols algebras of one block and several points with finite $\GK$: connected components of $V_{\diag}$ not in Table \ref{tab:finiteGK-intro-1}}}\label{tab:finiteGK-intro-2}
	{\small
		\begin{gather}\notag
		\begin{tabular}{|c|c|c|}
		\hline $V$ & $\toba(V)$;  $\GK$  & generators and relations
		\\ \hline
		$\lstr(A(1\vert 0)_1; r)$, $r\in \G'_N$, $N>2$ 
		& \S \ref{subsubsection:lstr-a(10)1}; 2 &
		$\ku\langle x_i, i\in \Iw_{3} \vert
		x_{\fudos}x_1-x_1x_{\fudos}+\frac{1}{2}x_1^2, x_{21}, x_{\fudos\fudos 2}, $
		\\
		$\xymatrix{\boxplus \ar  @{-}[r]^{1}  &\overset{-1}{\bullet} \ar  @{-}[r]^{r^{-1}}  & \overset{r}{\circ}}$
		& &
		$x_2^2,x_{\fudos 2}^2, x_{13}, x_{\fudos 3},x_{332}, x_3^N, [x_{\fudos 2}, x_{23}]_c^N \rangle$ 	
		\\ \hline
		$\lstr(A(1\vert 0)_1; r)$, $r\notin \G_{\infty}$
		& \S \ref{subsubsection:lstr-a(10)1-bis}; 4 &
		$\ku\langle x_i, i\in \Iw_{3} \vert
		x_{\fudos}x_1-x_1x_{\fudos}+\frac{1}{2}x_1^2, x_{21}, x_{\fudos\fudos 2},$
		\\
		$\xymatrix{\boxplus \ar  @{-}[r]^{1}  &\overset{-1}{\bullet} \ar  @{-}[r]^{r^{-1}}  & \overset{r}{\circ}}$ & &
		$x_2^2, x_{\fudos 2}^2, x_{13}, x_{\fudos 3}, x_{332} \rangle$
		\\ \hline
		$\lstr(A(1\vert 0)_2; \omega)$, $\omega \in \G'_3$
		& \S \ref{subsubsection:lstr-a(10)2}; 2 &
		$\ku\langle x_i, i\in \Iw_{3} \vert
		x_{\fudos}x_1-x_1x_{\fudos}+\frac{1}{2}x_1^2, x_{21}, x_{\fudos\fudos 2}, x_2^2,$
		\\
		$\xymatrix{\boxplus \ar  @{-}[r]^{1}  &\overset{-1}{\bullet} \ar  @{-}[r]^{\omega}  & \overset{-1}{\circ}}$
		& & $x_{\fudos 2}^2, x_{13}, x_{\fudos 3},x_3^2, x_{23}^3, [x_{\fudos 2}, x_{3}]_c^3,$
		\\
		& & $[x_{\fudos 2}, x_{23}]_c^6, [[x_{\fudos 2}, x_{23}]_c, x_3]_c^3 \rangle$
		\\ \hline
		$\lstr(A(1\vert 0)_3; \omega)$, $\omega \in \G'_3$
		& \S \ref{subsubsection:lstr-a(10)3}; 2 &
		$\ku\langle x_i, i\in \Iw_{3} \vert
		x_{\fudos}x_1-x_1x_{\fudos}+\frac{1}{2}x_1^2, x_{21}, x_{\fudos\fudos 2},$
		\\
		$\xymatrix{\boxplus \ar  @{-}[r]^{1}  &\overset{\omega}{\bullet} \ar  @{-}[r]^{\omega^2 }  &\overset{-1}{\circ}} $
		& & $x_2^3, x_{\fudos 2}x_{\fudos 23}+q_{13}q_{23}x_{\fudos 23}x_{\fudos 2}, x_{\fudos 2}^3, [x_{\fudos 2}, x_2]_c^3,$
		\\
		& & $x_{13}, x_{\fudos 3},x_3^2, x_{223}, [x_{\fudos 2}, x_{23}]_c^6\rangle$
		\\ \hline
		$\lstr(A(2 \vert 0)_1; \omega)$, $\omega \in \G'_3$
		& \S \ref{subsubsection:lstr-a(20)1}; 2 &
		$\ku\langle x_i, i\in \Iw_{4} \vert
		x_{\fudos}x_1-x_1x_{\fudos}+\frac{1}{2}x_1^2, x_{21},x_{\fudos\fudos 2}, x_2^2,$
		\\
		$\xymatrix{\boxplus \ar  @{-}[r]^{1}  &\overset{-1}{\bullet} \ar  @{-}[r]^{\omega}  & \overset{\omega^2}{\circ} \ar  @{-}[r]^{\omega}   &\overset{\omega^2}{\circ} }$
		& & $x_{\fudos 2}^2, x_{13}, x_{\fudos 3}, x_{14}, x_{\fudos 4}, x_{24}, x_{332}, x_{334}, x_{443},x_3^3,$
		\\
		& & $ x_{34}^3, x_4^3, [x_{\fudos 23},x_2]_c^3, [[x_{\fudos 23},x_2]_c, [x_{\fudos 23},x_{24}]_c]_c^3,$
		\\
		& & $[x_{\fudos 23},x_{24}]_c^3, [ [x_{\fudos 23},x_2]_c, [[x_{\fudos 23},x_{24}]_c,x_2]_c]_c^3,$
		\\
		& & $[[x_{\fudos 23},x_{24}]_c,x_2]_c^3, [ [x_{\fudos 23},x_{24}]_c, [[x_{\fudos 23},x_{24}]_c,x_2]_c]_c^3, \rangle $
		\\ \hline
\end{tabular}
\end{gather} }
\end{table}

\begin{table*}[ht]
	\caption*{{\small \sc{Table \ref{tab:finiteGK-intro-2}} (continuation)}}
	{\small
		\begin{gather}\notag
		\begin{tabular}{|c|c|c|}
		\hline		
		$\lstr(D(2 \vert 1); \omega)$, $\omega \in \G'_3$
		& \S \ref{subsubsection:lstr-D(21)}; 2 &
		$\ku\langle x_i, i\in \Iw_{4} \vert
		x_{\fudos}x_1-x_1x_{\fudos}+\frac{1}{2}x_1^2, x_{21}, x_{\fudos\fudos 2}, x_2^2,$
		\\
		$\xymatrix{\boxplus \ar  @{-}[r]^{1}  &\overset{-1}{\bullet} \ar  @{-}[r]^{\omega}  & \overset{\omega^2}{\circ} \ar  @{-}[r]^{\omega^2}  & \overset{\omega}{\circ} }$
		& & $x_{\fudos 2}^2, x_{13}, x_{\fudos 3}, x_{14}, x_{\fudos 4}, x_{24}, [[x_{234},x_3]_c,x_3]_c, $
		\\
		& & $x_{332}, x_{443}, x_3^3, x_{34}^3, x_{334}^3, x_4^3, [x_{\fudos 234}, x_2]_c^3, x_{\fudos 234}^3,$
		\\
		& & $[[[x_{\fudos 234}, x_2]_c, x_3]_c, x_3]_c^3, [[x_{\fudos 234}, x_2]_c, x_3]_c^3,$
		\\
		& & $[[[x_{\fudos 234}, x_2]_c, x_3]_c, x_3]_c^3 \rangle$
		\\ \hline
		$\lstr(A_{2}, 2)$
		& \S \ref{subsubsection:lstr-a-22}; 2 &
		$\ku\langle x_i, i\in \Iw_{3} \vert
		x_{\fudos}x_1-x_1x_{\fudos}+\frac{1}{2}x_1^2, x_{21}, x_{\fudos\fudos\fudos 2}, x_2^2,$
		\\
		$\xymatrix{\boxplus \ar  @{-}[r]^{2}  &\overset{-1}{\bullet} \ar  @{-}[r]^{-1} &  \overset{-1}{\circ}}$
		& & $x_{\fudos 2}^2, \, x_{\fudos\fudos 2}^2, x_{13}, x_{\fudos 3}, x_3^2, x_{23}^2, x_{\fudos\fudos 23}^2, [x_{\fudos\fudos 23}, x_{\fudos 2}]_c^2,$
		\\
		& & $[x_{\fudos\fudos 23}, x_{\fudos 2}]_c, x_2]_c^2, [[x_{\fudos\fudos 23}, x_{\fudos 2}]_c, x_2]_c, x_3]_c^2,$
		\\
		& & $x_{\fudos 23}^2, [x_{\fudos 23},x_2]_c^2, [x_{\fudos\fudos 23},x_2]_c^2 \rangle$
		\\ \hline
		$\lstr(A_{\theta -1})$, $\theta > 2$	& \S \ref{subsubsection:lstr-a-n}; 2 &
		$\ku\langle x_i, i\in \Iw_{\theta} \vert
		x_{\fudos}x_1-x_1x_{\fudos}+\frac{1}{2}x_1^2, x_{21}, x_{\fudos\fudos 2}, x_2^2,$
		\\
		$\xymatrix{\boxplus \ar  @{-}[r]^{1}  &\overset{-1}{\bullet} \ar  @{-}[r]^{-1}  & \overset{-1}{\circ}} \dots \xymatrix{
			\overset{-1}{\circ} \ar  @{-}[r]^{-1}  & \overset{-1}{\circ} }$
		& & $x_{\fudos 2}^2, x_{1j}, x_{\fudos j} \, (j>2), x_{ij}^2 \, (2\le i\le j\le\theta),$
		\\
		& & $[x_{k-1 \, k \, k+1}, x_k] \,  (3\le k <\theta), \ya_{k\ell}^2 (k<\ell\in\I_\theta) \rangle$
		\\ \hline
		$\cyc_2$ 	& \S \ref{subsubsection:nichols-mild-cyc2}; 3 &
		$x_{\fudos}x_{\fudos 1}- x_{\fudos 1}x_{\fudos} - x_1x_{\fudos 1}, x_{\fudos}x_{12}+q_{12}x_{12}x_{\fudos}+x_{1 \fudos 2},$
		\\ 
		$\xymatrix{\boxminus \ar  @{-}[r]^{(-1, 1)}  &\overset{-1}{\bullet} \ar  @{-}[r]^{-1}  & \overset{-1}{\circ}}$ 
		& &
		$x_1^2, x_{\fudos}x_{\fudos 2}+q_{12}x_{\fudos 2}x_{\fudos}-\frac 1 2 x_{1 \fudos 2} - q_{12}x_{12}x_{\fudos},$
		\\ & &
		$x_{\fudos}x_{1 \fudos 2} -q_{12}x_{1 \fudos 2}x_{\fudos}, x_2^2, x_{12}^2, x_{\fudos 2}^2, x_{1 \fudos 2}^2,[x_{123}, x_2]_c, $
		\\ & &
		$[x_{\fudos 2}, x_{123}]_c^2, [x_1, [x_{\fudos 23}, x_2]_c]_c -q_{12}q_{13} x_{123} x_{12},$
		\\ & &
		$x_{123}x_{\fudos 23} + x_{\fudos 23}x_{123},x_3^2, x_{123}^2, x_{23}^2,x_{\fudos 23}^2, x_{1\fudos 23}^2,$
		\\ & &
		$ x_{13}, x_{\fudos 3}, [x_1, [x_{\fudos 2}, x_{123}]_c]_c -2q_{12}x_{12}x_{1\fudos 23},$
		\\ & &
		$[x_{\fudos}, [x_{\fudos 2}, x_{123}]_c]_c -2q_{12}x_{12}x_{1\fudos 23}+2q_{12}x_{\fudos 2}x_{1\fudos 23}.$
		\\ \hline
		\end{tabular}
		\end{gather} }
\end{table*}

\subsubsection{About the proofs}\label{subsubsection:about-the-proofs}
The proofs of the Theorems that support our main result involve two kind of arguments: first,
discarding Nichols algebras with infinite $\GK$; second, proving that the remaining have finite $\GK$, computing it and describing the relations and a PBW-basis.
For the first we either use  Lemmas \ref{lemma:rosso-lemma19-gral} and \ref{lemma:rosso-lemma19}  that are generalizations of results in \cite{R quantum groups}; or else reduce to the diagonal case by a variety of techniques, mainly:

\begin{itemize}[leftmargin=*]\renewcommand{\labelitemi}{$\heartsuit$}
		\item \emph{Filtrations of braided vector spaces}. Given a Yetter-Drinfeld module $V$ with a  flag of Yetter-Drinfeld submodules, 
		we consider the associated Yetter-Drinfeld module $\gr V$. Then $\GK \NA(\gr V) \leq \GK \NA(V)$. 
		See \S \ref{subsection:filtr-nichols}. If the flag is complete (as a flag of vector spaces), then
		$\gr V$ is of diagonal type and we apply Theorem \ref{thm:nichols-diagonal-finite-gkd}, or invoke  Hypothesis \ref{hyp:nichols-diagonal-finite-gkd}.   However there are instances where a more elaborated argument is needed: 
		we look at the filtration of $\cB = \NA(V)$ induced by the flag and correspondingly at $\gr \cB$; then we identify a braided subspace $U$
		of the space of primitive elements in $\gr \NA$ such that $\GK \NA(U) = \infty$; this does the job since clearly $\GK \NA(U) \leq \GK \NA(V)$.
		The space $U$ is not contained in $\gr V$ and derivations are used to verify that the extra elements are not 0.
		
\smallbreak		
		\item \emph{Splitting of Nichols algebras}.  We deal with decomposable Yetter-Drinfeld modules $V = U \oplus W$. Then 
		the natural projection $\NA(V) \twoheadrightarrow \NA(U)$ admits a section, namely the natural inclusion $\NA(U) \hookrightarrow \NA(V)$.
As in the usual context of Hopf algebras, the algebra $K=\NA (V)^{\mathrm{co}\,\NA (U)}$ is a braided Hopf algebra in a suitable category 
and $\cB(V)$ is reconstructed as $\NA(V) \simeq K \# \NA (U)$ by a sort of  braided bosonization \`a la Majid-Radford \cite[Lemma 3.2]{AHS}. In fact, $\GK \NA(V) = \GK K + \GK \NA (U)$ under suitable finiteness assumptions.
Furthermore, it turns out that $K$ is the Nichols algebra of $K^1 = \ad_c\NA (U) (W)$ \cite[Proposition 8.6]{HS}. 
That is, we are reduced to compute the braiding of $K^1$ (which is not the original braiding). In many cases, it happens that this braiding is of diagonal type, so we  apply  Theorem \ref{thm:nichols-diagonal-finite-gkd}  or invoke Hypothesis \ref{hyp:nichols-diagonal-finite-gkd}. 
However, the full strength of Hypothesis \ref{hyp:nichols-diagonal-finite-gkd} is not required: for instance, Theorem \ref{thm:point-block} (giving Table \ref{tab:finiteGK-intro-1})
requires only that the Nichols algebra of a rank 3 braided vector space of hyperbolic diagonal type has infinite $\GK$, and only at one step in the proof of Lemma \ref{lemma:mild-odd-order-2}.	 	
\end{itemize}

\smallbreak Notice that $K^1$ is not always of diagonal type;  for example $\cyc_1$, $\cyc_2$ and $\eny_{\star}$ are not, cf. \S \ref{subsubsection:nichols-mild}, \ref{subsubsection:nichols-mild-cyc2}, \ref{subsubsec:pale-block-case-1}; their Nichols algebras have to be computed in an ad-hoc manner. In other cases we reduce to previously investigated Nichols algebras. 

\medbreak
For the second kind or argument, we present by generators and relations and exhibit a PBW-basis of the Nichols algebras considered, so that the computation of the $\GK$ flows naturally (often this confirms a previous calculation by braided bosonization and reduction to the diagonal case). In the splitting context, if $K^1$ is of diagonal type, then the relations
of $\toba(V)$ can described in a straightforward way from the relations of $\NA(U)$, those of $K$ known by \cite{Ang-crelle},
and the action of the former on the latter. Otherwise, we guess the relations by looking at the associated graded  algebra and then prove that the guess was right; see for instance $\NA(\eny_{\star})$, \ref{subsubsec:pale-block-case-1}.

\medbreak
The notation for the braided vector spaces in the leftmost column in Table \ref{tab:finiteGK-intro-1}, respectively Table \ref{tab:finiteGK-intro-2}, 
respectively Table \ref{tab:finiteGK-intro-5}, is explained in \S \ref{subsection:point-block-presentation}, respectively Table \ref{tab:finiteGK-block-points-names} and \S \ref{subsection:YD>3-setting}, respectively 
\S \ref{subsubsec:pale-block-case-1}.
Notice that $\lstr$ stands for Laistrygonian,  $\cyc$ for Cyclop, $\pos$ for Poseidon and $\eny$ for Endymion.

\subsubsection{The Poseidon Nichols algebras}\label{tab:finiteGK-intro-4}	
Here we present $\toba(V)$, where  $V = \pos(\bq,\ghost)$ is the Poseidon braided vector space 
with parameters $\ghost = (\ghost_k) \in \N^t$ and $\bq = (q_{ij})_{i,j \in \I_{\theta}}$, such that 
$q_{ij}q_{ji} = 1$, if $i \neq j$, and $\epsilon_{i} := q_{ii} \in \G_2$ for all $i \in \I_{\theta}$.
Here $t\in \N_{\ge 2}$ and $\theta = t+1$. The braiding of $V$ is described by the diagram
$$ 
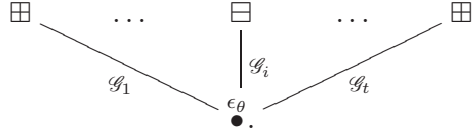
\xymatrix{ \boxplus \ar  @{-}[rrd]_{\ghost_1} & \dots  & \boxminus \ar  @{-}[d]^{\ghost_i} & \dots &
	\boxplus \ar  @{-}[lld]^{\ghost_t} 
	\\ & & \overset{\epsilon_{\theta}}{\bullet}. & & } $$
In the first line, one has $\boxplus$ or $\boxminus$ according to $\epsilon_{i} = 1$ or $-1$.
See \S \ref{subsection:YD>3-severalblocks-1pt-poseidon} for unexplained notation.
The presentation of $\toba \left(V\right)$ by generators and relations is

\begin{align}\label{eq:finiteGK-intro-4}	    
\begin{aligned}
\ku\langle x_i,&\, i\in \Idd
\vert  x_{i+\frac{1}{2} }x_i -x_ix_{i+\frac{1}{2} }+\frac{1}{2}x_i^2, \ i\in\I_t, \, \epsilon_{i}=1,
\\
&x_i^2, \quad x_{i+\frac{1}{2} }x_{i+\frac{1}{2} \, i}- x_{i+\frac{1}{2} \, i}x_{i+\frac{1}{2} } - x_ix_{i+\frac{1}{2} \, i}, \ i\in\I_t, \, \epsilon_{i}=-1;
\\		 
&x_ix_j - q_{ij} \, x_jx_i, \quad \lfloor i\rfloor \neq \lfloor j\rfloor  \in\I_t;
\\
&x_ix_{\theta} - q_{i \theta} \, x_{\theta} x_i, \quad	 
(\ad_c x_{i+\frac{1}{2}})^{1+|2a_i|}(x_{\theta}),\quad i \in\I_t;
\\ 	
&\sch_{\bm}\sch_{\bn} - p_{\bm,\bn} \, \sch_{\bn}\sch_{\bm}, \quad\bm  \neq \bn \in \cA;
\quad 	
\sch_{\bn}^2, \bn \in \cA, \, \epsilon_{\bn} = -1\rangle.
\end{aligned}	
	\end{align}
	
It can be shown that $\GK (\toba(V)) = 2t+ \vert \{\bm \in \cA:  p_{\bm,\bm}=1\}\vert$.

\begin{table}[ht]
	\caption{Nichols algebras of one pale block of dimension 2 plus one point with finite $\GK$, not covered by the classification; see \S \ref{subsubsec:pale-block-case-1}}\label{tab:finiteGK-intro-5}
	{\small
		\begin{gather}\notag
		\begin{tabular}{|c|c|c|}
		\hline $V$ & $\GK$  & generators and relations     
		\\\hline 
		$\eny_+(q)$ &   1  &		
		$\ku\langle x_1, x_2, x_3
		\vert x_1^2, \, x_2^2, \, x_1x_2 + x_2x_1$, $(x_2x_3 - q x_3x_2)^2$,
		\\ 	& &
		$x_3(x_2x_3 - q x_3x_2) - q^{-1} (x_2x_3 - q x_3x_2)x_3,$ $x_1x_3 - q x_3x_1\rangle$
		\\
		\hline
		$\eny_-(q)$ & 1  &		
		$\ku\langle x_1, x_2, x_3
		\vert x_1^2, \, x_2^2, \, x_1x_2 + x_2x_1$, $x_3 ^2$,
		\\ 	& &
		$x_3(x_2x_3 - q x_3x_2) + q^{-1} (x_2x_3 - q x_3x_2)x_3,$ $x_1x_3 - q x_3x_1\rangle$
		\\
		\hline
		$\eny_{\star}(q)$ & 2  &	$\ku\langle x_1, x_2, x_3
		\vert x_1^2, \, x_2^2, \, x_1x_2 + x_2x_1, x_3^2, x_{31}^2,x_{213}^2$,
        \\ 	& & $x_2 [x_{23},x_{13}]_c- q^2 [x_{23},x_{13}]_cx_2-q \, x_{13}x_{213}\rangle$
		\\ \hline
		\end{tabular}
		\end{gather} }
\end{table}

\medbreak
\subsection{Applications}\label{subsubsec:nichols-corollaries}
\subsubsection{Examples of Hopf algebras}\label{rem:examples-new}
Every Nichols algebra $\toba(V)$ for $V$ in Theorem \ref{theorem:main}, for instance those in
Tables \ref{tab:finiteGK-intro-1} or \ref{tab:finiteGK-intro-2}, can be realized in $\ydG$ for some finitely generated abelian group $\Gamma$. Same for those in Table \ref{tab:finiteGK-intro-5}.
Thus $\toba(V)\# \ku \Gamma$ is a new example of a pointed Hopf algebra with finite $\GK$. More examples would be available computing the liftings. See \cite{aah-jordan} for the case of the Jordan and super Jordan planes.
These examples could be worked out as test cases to verify conjectures on the homological behaviour of Hopf algebras with finite $\GK$, see e.g. \cite{BGLZ,B-seattle,B-turkish,BG,G-survey}.

\subsubsection{Domains}\label{rem:examples-domains}
Among the Nichols algebras either of blocks or else appearing in Tables \ref{tab:finiteGK-intro-1} or \ref{tab:finiteGK-intro-2}, or else in \S \ref{tab:finiteGK-intro-4},
the only domains are
\begin{itemize}[leftmargin=*]\renewcommand{\labelitemi}{$\diamond$}
	
	\item $\toba(\cV(1,2))$, see \cite{aah-jordan};

	\item $\toba(\lstr(1,\ghost))$, see Proposition \ref{pr:lstr-11disc-domain};

	\item $\toba(\pos(\bq,\ghost))$ where all diagonal entries in $\bq$ are 1, see Proposition \ref{pr:poseidon-domain}.	
\end{itemize}

\begin{theorem}
Let $V\in \ydG$,  $\dim V < \infty$, such that $\GK \NA(V) < \infty$. Then the following are equivalent:
\begin{enumerate}[leftmargin=*,label=\rm{(\alph*)}]
	\item\label{item:domain-a} $\NA(V)$ is a domain.
	
	\item\label{item:domain-b} $V$ is as in \eqref{eq:bradinig-generalform}, \eqref{eq:bradinig-generalform1},
	its flourished diagram is admissible,  all blocks are Jordan and all connected components $J\in \X$ are points with label 1.
\end{enumerate}
\end{theorem}

\pf Theorem \ref{theorem:final-domain} says  that if $V$ is as in \eqref{eq:bradinig-generalform}, \eqref{eq:bradinig-generalform1}
and its flourished diagram is admissible, then $\NA(V)$ is a domain if and only if all blocks are Jordan, i.e. $\boxplus$,
and all $J\in \X$ are points with label 1, i.e. $\overset{1}{\bullet}$.
Now if $V$ is not as in \eqref{eq:bradinig-generalform}, \eqref{eq:bradinig-generalform1}, then
it contains a braided subspace $W$ of the form \emph{a pale block and a point}, see the discussion at the end of \S \ref{subsect:pale-context}.
But then $W$ is one of $\eny_{\pm}$ or $\eny_{\star}$ that are not domains, so $\NA(V)$ could not be a domain.
\epf

This is a fundamental piece of information for the classification of
pointed Hopf algebras with finite $\GK$ that are domains.

\subsubsection{Co-Frobenius Hopf algebras}\label{re:cofrob}
A Hopf algebra is co-Frobenius if it admits a non-zero integral. See
\cite{ACE} for a list of equivalent characterizations-- and references for
them.  Here is an application of Theorem \ref{theorem:blocks}:

\begin{theorem}\label{theorem:cofrob} Let $H$ be a pointed Hopf algebra with $G(H)=: \Gamma$ abelian and let $V \in \ydG$ be its infinitesimal
	braiding. The following are equivalent:
	\begin{enumerate}
		\item\label{item:cofrob1} $H$ is co-Frobenius.
		\item\label{item:cofrob2} $\gr H \simeq \NA(V)$, where  $V$ is of diagonal type and $\dim \NA(V) < \infty$.
	\end{enumerate}
\end{theorem}

\pf
A Hopf algebra is co-Frobenius if and only if its coradical filtration is
finite \cite{ACE, AD}. Since $H$ is pointed,
this is also equivalent to $\dim R < \infty$, where $R$ is the diagram of $H$
\cite{AD}. Thus \eqref{item:cofrob2} implies \eqref{item:cofrob1}.
Conversely, if $H$ is co-Frobenius then $\dim \NA(V) < \infty$ as $\NA(V)\hookrightarrow R$, hence $V$ is of diagonal type by
Theorems \ref{theorem:blocks} and \ref{th:paleblock-point-resumen}; cf. \S \ref{subsection:ydz}. Then $\NA(V) =  R$ by \cite[Theorem 2]{Ang-crelle}.
\epf

\section{Preliminaries}\label{section:Preliminaries}

\subsection{Conventions}\label{subsection:conventions}
The $\qmb$-numbers are the polynomials
\begin{align*}
(n)_\qmb &=\sum_{j=0}^{n-1}\qmb^{j}, & (n)_\qmb^!&=\prod_{j=1}^{n} (j)_\qmb, &
\binom{n}{i}_\qmb & =\frac{(n)_\qmb^!}{(n-i)_\qmb^!(i)_\qmb^!} \in \Z[\qmb],
\end{align*}
$n\in \N$, $0 \leq i \leq n$. If  $q\in\ku$, then $(n)_q$, $(n)_q^!$, $\binom{n}{i}_q$ 
denote the  evaluations of $(n)_\qmb$, $(n)_\qmb^!$, $\binom{n}{i}_\qmb$ at $\qmb = q$.

Let $\G_N$ be the group of $N$-th roots of unity, and $\G_N'$ the subset of primitive roots of order $N$;
$\G_{\infty} = \bigcup_{N\in \N} \G_N$.
All the vector spaces, algebras and tensor products  are over $\ku$.

If $V, W \in \ydh$,  then $c_{V, W}: V\otimes W \to W \otimes V$ denotes the corresponding braiding.
If $R$ is a Hopf algebra  in $\ydh$, then $R\# H$ is the bosonization of $R$ by $H$.
Let $\ad$ be the adjoint action of $H$ and $\ad_c$ the braided adjoint  action of $R$.
Then $\ad_c x\otimes \id = \ad (x\# 1)$ for $x\in R$. If $x\in R$ is primitive, then
$\ad_c x (y) = xy - \text{multiplication} \circ c (x\otimes y)$ for all $y\in R$.

\smallbreak Let $\Gamma$ be an abelian group. We denote by $\widehat \Gamma$ the group of  characters  of $\Gamma$.
The objects in $\ydG$ are the same as $\Gamma$-graded $\Gamma$-modules,
the $\Gamma$-grading is denoted $V = \oplus_{g\in \Gamma} V_g$.
If $g\in \Gamma$ and $\chi \in \widehat\Gamma$, then  the one-dimensional vector space $\ku_g^{\chi}$,
with action and coaction given by $g$ and $\chi$, is in $\ydh$.
Let  $W \in \ydG$ and $(w_i)_{i\in I}$ a basis of $W$ consisting of homogeneous elements of degree $g_i$, $i\in I$,
respectively. Then there are skew derivations $\partial_i$, $i\in I$, of $T(W)$ such that
for all $x,y\in T(W)$, $i, j \in I$
\begin{align}\label{eq:skewderivations}
\partial _i(w_j) & =\delta _{ij},&
\partial _i(xy) &= \partial _i(x)(g_i\cdot y)+x\partial _i(y).
\end{align}

Nichols algebras are graded Hopf algebras in $\ydh$, or also braided graded Hopf algebras, coradically graded and generated
in degree one. See \cite{AS Pointed HA} for alternative characterizations.

Given a braided vector space   with a basis $(x_i)$, we denote in $T(V)$, or $\NA(V)$, or any intermediate Hopf algebra,
	\begin{align}\label{eq:xijk}
		x_{ij} &= (\ad_c x_i)\, x_j,& x_{i_1i_2 \dots i_M} &= (\ad_c x_{i_1})\, x_{i_2 \dots i_M}.
	\end{align}

\subsection{Nichols algebras of diagonal type}\label{subsection:diagonal-type}
Let $V$ be a braided vector space of diagonal type with $\theta :=\dim V<\infty$.
Let $(x_i)_{i\in \I_\theta }$ be a basis of $V$ and $\bq = (q_{ij})_{i,j\in \I_\theta }\in \ku^{\theta \times \theta }$
such that $q_{ij}\ne 0$ and $c(x_i\otimes x_j)=q_{ij}x_j\otimes x_i$ for all $i,j\in \I = \I_\theta $.
Let $(\alpha_i)_{i\in \I}$ be the canonical basis of $\Z^\theta$ and $\chi$ the bilinear form  on $\Z^\theta$ such that
$\chi (\alpha _i,\alpha _j)=q_{ij}$ for all $i,j\in \I $; set $q_{\alpha \beta }=\chi (\alpha ,\beta)$
for $\alpha, \beta \in \Z^\theta $. We refer to \cite{H-inv} for proofs or references of the following facts:

\subsubsection{}\label{subsubsection:diagonal-type-ztgraded} $\NA (V)$ is $\Z^\theta $-graded with $\deg x_i=\alpha _i$ for all $i\in \I$.

\subsubsection{}\label{subsubsection:diagonal-type-pbw}  There is a totally ordered subset $L \subset \NA(V)$ consisting of $\Z^\theta$-homoge\-neous elements such that
\begin{align}
	\{ l_1^{m_1}\cdots l_k^{m_k} \,|\,k\in \N _0,l_1>\cdots >l_k\in L,0 < m_i< N_{l_i}\,\text{for all $i \in \I_k$}\}
\end{align}
is a linear basis of $\NA(V)$ (a so called restricted  PBW basis); here
$$ N_l= \min\{n\in \N_: (n)_{q_{\deg l,\deg l}}=0\} \in \N \cup \infty. $$

\subsubsection{}\label{subsubsection:diagonal-type-rootsystem} Let $\roots (V) = \{\deg l, l\in L\} \subset \zt$,  or just $\roots$ if $V$ is clear
from the context, be the set of positive roots of $\NA (V)$. If $L$ is finite, then $\roots$ and the multiplicities
of appearance of each root are uniquely determined.

\subsubsection{}\label{subsubsection:diagonal-type-reflections}
Let $i\in\I$ such that for all $j\neq i$, there exists $n\in\N_0$ such that
$(n+1)_{q_{ii}}(1-q_{ii}^n q_{ij}q_{ji})=0$.
Let $c_{ii} =2$ and
\begin{align}\label{eq:defcij}
	c_{ij}&:=-\min\{n\in\N_0:(n+1)_{q_{ii}}(1-q_{ii}^nq_{ij}q_{ji})=0\},& & &j\neq i.
\end{align}
Let $s_i\in GL(\zt)$ be given by $s_i (\alpha_j) = \alpha_j - c_{ij}\alpha_i$, $j\in \I$.
The reflection at the vertex $i$ of $\bq$ is the matrix $\cR^i(\bq) = (t_{jk})_{j,k\in\I}$, where
\begin{align}\label{eq:def-rho-ij}
	t_{jk}&:= \chi(s_i(\alpha_j), s_i(\alpha_k)) =  q_{jk}q_{ik}^{-c_{ij}}q_{ji}^{-c_{ik}}q_{ii}^{c_{ij} c_{ik}}, & j,k&\in\I.
\end{align}

\begin{theorem} \cite{H-inv, AA}
	If $\cR^i(V)$ is the braided vector space corresponding to $\cR^i(\bq)$, then $\GK \toba(\cR^i(V)) = \GK \toba(V)$. \qed
\end{theorem}

\subsection{On the Gelfand-Kirillov dimension}

\subsubsection{Basic facts}\label{subsection:basic-defs}
Our main reference for this topic is \cite{KL}. Let $A$ be a finitely
generated $\ku$-algebra.
If $V$ is a finite-dimensional generating subspace of $A$ and
$A_n = \sum_{0\le j \le n} V^j$, then
\begin{align}\label{eq:def-GKdim}
 \GK A :=  \overline\lim_{n \to \infty} \log_n \dim A_n.
\end{align}
Then  $\GK A$ does not depend on the choice of $V$ \cite[1.1]{KL}. In general,
if $A$ is not finitely generated,
then
\begin{align}\label{eq:def-GKdim-gral}
 \GK A :=  \sup\{\GK B\vert B\subseteq A, \ B \text{ finitely generated}\}.
\end{align}

Suppose that $\GK A<\infty $.
We say that a finite-dimensional subspace $V\subseteq A$ is
\textit{GK-deterministic} if
$\GK A=\lim_{n \to \infty} \log_n \dim \sum_{0\le j\le n}V^j$.
Clearly, if $V$ is a GK-deterministic subspace of $A$, then any finite-dimensional
subspace of $A$ containing $V$ is GK-deterministic.

Let $A$ and $B$ be two algebras.
Then
\begin{align*}
\GK (A\otimes B) \leq \GK A + \GK B,
\end{align*}
but the equality does not hold in general. For instance,
it does hold when $A$ or $B$ has a GK-deterministic subspace, see \cite[Proposition 3.11]{KL}.
We need the smash product version of this result.

\begin{lemma}\label{lemma:GKdim-smashproduct}  Let $K$ be a Hopf algebra, $R$ a
  Hopf algebra in $\ydk$, $A$ a $K$-module algebra and
$B$ an $R$-module algebra in $\ydk$. Assume that the actions of $K$ on $A$, of
$K$ on $B$, of $K$ on $R$, and of $R$ on $B$ are locally finite.

\begin{enumerate}[leftmargin=*,label=\rm{(\alph*)}]
  \item\label{item:GKdim-smashproduct}
 $\GK A\#K \le \GK A+\GK K$. If either $K$ or $A$ has a GK-deterministic subspace, then $\GK A\# K = \GK A + \GK K$.

\smallbreak
  \item\label{item:GKdim-smashproduct-braided}
  $\GK B\# R \le \GK B + \GK R$.
 If either $K$ or $B$ has a GK-deterministic subspace, then
$\GK B\# R = \GK B + \GK R$.
  \end{enumerate}
\end{lemma}

\pf
(a) Let $X\subseteq A\#K$ be a finite-dimensional subspace
and let $V\subseteq A$, $W\subseteq K$ be finite-dimensional subspaces such
that $X\subseteq V\# W$. Since the action of $K$ on $A$ is locally
finite, we may assume that $V$ is a $K$-submodule of $A$. We may further
assume that $W$ is a subcoalgebra of $K$. Then
\begin{align*}
\sum_{i=0}^nX^i\subseteq \sum_{i,j=0}^nV^i\#W^j
\end{align*}
for any $n\in \N$ by the definition of the smash product $A\#K$.
Therefore $\GK A\#K\le \GK A+\GK K$. Assume now that $A$ has a
GK-deterministic subspace. One can argue similarly if $K$ has a
GK-deterministic subspace. Then we may choose $V$ in such a way that it is
additionally GK-deterministic. In that case the above formula also shows that
\begin{align*}
\GK A\#K=\GK A+\GK K.
\end{align*}

(b) Analogous to (a), but there is a slight point since
the comultiplication $R\to R\otimes R$ is not an
algebra map to the tensor product algebra. As the appropriate
algebra structure of $R\otimes R$ is defined via the Yetter-Drinfeld
structure of $R$, cf. \cite[Remark 1.8 (i)]{AHS}, we may apply the argument in (a) only when
 the actions of $K$ on $R$ and $B$ are locally finite.
\epf

\begin{remark}\label{rem:examples-GKdterministic} The algebra $A$
	has a  GK-deterministic subspace when:
	\begin{itemize} [leftmargin=*]\renewcommand{\labelitemi}{$\circ$}
		\item $\GK A \leq 2$, cf. the proof of \cite[Proposition 3.12]{KL}.
		
		\item $A = U(\g)$, where $\g$ is a finite-dimensional Lie algebra.
		
		\item $A = \ku G$, where $G$ is a finitely generated group.
	\end{itemize}
\end{remark}

\begin{example}\label{example:locally-finite} Let $\Gamma = \Z^2$, $g=(1,0), h=(0,1)$.
	Let $W\in \ydG$
	with basis $(w_i)_{i\in \Z }$ such that $W=W_g$,
	$g\cdot w_i = -w_i$ and $h\cdot w_i=w_{i+1}$ for all $i\in \Z$.
	Then the action of $\Gamma $ on $W$ is not locally finite, but
	$\NA (W)=\Lambda W$ and $\GK \NA (W)=0$. Furthermore
	\begin{align*}
	\GK (\NA (W) \# \ku \Gamma) \overset{\ast} = \infty > 2 = \GK \NA (W) + \GK \ku \Gamma.
	\end{align*}
$\ast$: Let $V = \langle g^{\pm 1}, h^{\pm 1}, w_0\rangle$;  then $\Lambda(\langle w_0, w_1 \dots w_{n-1}\rangle) \subset V^{n^2}$, hence
$2^{n} \leq \dim V^{n^2}$ and the claim follows.	
In particular,  the assumption on local finiteness in Lemma \ref{lemma:GKdim-smashproduct} \ref{item:GKdim-smashproduct} can not be omitted.
\end{example}

\subsubsection{A  criterium for infinite $\GK$}\label{subsection:infinite-criterium}
The next statement is inspired by \cite[Lemma 19]{R quantum groups}; in
many cases treated here, the infiniteness of the $\GK$ relies on it.

\begin{lemma}\label{lemma:rosso-lemma19-gral}  Let $\NA = \oplus_{n \ge 0} \NA^n$
be a finitely generated graded algebra with $\NA ^0 = \ku$
and $(y_k)_{k\ge 0}$  a family of homogeneous elements such that
\begin{align} \label{equation:S-LI}
\{y_{i_1}\dots y_{i_l}: i_j \in\N, i_1 <  \dots < i_l  \}
\end{align}
is linearly independent and there exist $m, p\in \N$ such that
\begin{align} \label{eq:bound-degree}
\deg y_i \leq mi + p, \, \forall i\in \N.
\end{align}
Then $\GK\cB = \infty$.
\end{lemma}

\pf Let  $1\le i_1<i_2<\cdots <i_l$; by \eqref{equation:S-LI}, we have
$$\deg y_{i_1}\cdots y_{i_l} \le  m(i_1+\cdots +i_l)+ lp \le m\frac{i_l(i_l+1)}{2} + lp \le (m+p) i_l^2.$$
Thus for each $M\in\N$, the monomials $y_{i_1}y_{i_2}\cdots y_{i_l}$ with
$$1\le i_1<\cdots <i_l\leq \lfloor\sqrt{M/(m+p)}\rfloor$$
have degree $\leq M$. Let $G$ be a  finite set of homogeneous generators of $\cB$
and let $\cB_M$ be the subspace generated by products of $j$ elements in $G$, $0\le j \le M$.
Then $\cB_M \supseteq \bigoplus _{k=0}^M\cB^k$
since $\cB^0=\ku 1$.
Therefore
\begin{align*}
\dim \cB_M \ge \sum _{k=0}^M\dim \cB^k\ge
2^{\lfloor\sqrt{M/(m+p)}\rfloor}
\end{align*}
by \eqref{equation:S-LI}. Hence
$\GK\cB = \infty$. \epf

Let $(W,c)$ be a braided vector space  and $x_1, x_2 \in W$ linearly independent.
Let $\cB$ be a pre-Nichols algebra of $W$.
Set $y_i = (\ad_cx_1)^i x_2 \in \cB$; clearly $\deg y_i = i +1$, so Lemma \ref{lemma:rosso-lemma19-gral} implies

\begin{lemma}\label{lemma:rosso-lemma19} If the  set  \eqref{equation:S-LI}
is linearly independent, then  $\GK \NA  = \infty$. \qed
\end{lemma}

Now assume that $\langle x_1, x_2 \rangle$ is a braided subspace of $W$ of diagonal type,
that is $c(x_i\otimes x_j)=q_{ij}x_j\otimes x_i$, where  $q_{ij} \in \ku^{\times}$ for all $i,j\in \I_2$. Set
\begin{align}\label{eq:muk}
\mu_k = \prod _{i=0}^{k-1}(1-q_{11}^iq_{12}q_{21}).
\end{align}

\begin{lemma}\label{lemmata:rosso} \cite[Lemma 14,\,Corollary 18]{R quantum groups}
\begin{enumerate}[leftmargin=*,label=\rm{(\alph*)}]
  \item\label{item:diagonal-relations}  If $k\in \N$, then
  $y_k=0$ iff $\mu_k x_1^k =0$,  iff $(k)_{q_{11}}^! \mu_k =0$.
  \item\label{lemma:rosso-cor18}  If $y_k\neq 0$ for all $k\in \N$, then the  set
\eqref{equation:S-LI} is linearly independent. \qed
\end{enumerate}
\end{lemma}

\begin{lemma}\label{lemma:points-trivial-braiding}
 If  $\langle x_1, x_2 \rangle$ is a braided subspace of $W$ of diagonal type as above,
  $q_{11} = 1$ and $q_{12}q_{21} \neq 1$, then $\GK \NA(W) = \infty$.
\end{lemma}

\pf By assumption, $(k)_{q_{11}}^!\mu_k = (1-q_{12}q_{21})^k \neq 0$ for all $k$,
hence the  set  \eqref{equation:S-LI} is linearly independent by Lemma \ref{lemmata:rosso}; then Lemma \ref{lemma:rosso-lemma19} applies.
\epf

\section{Yetter-Drinfeld modules of dimension 2}\label{sec:yd-dim2}

\subsection{Indecomposable modules and blocks}\label{subsection:indec-2}

Let $\Gamma$ be an abelian group,
$g\in \Gamma$,  $\chi\in \widehat\Gamma$ and $\eta: \Gamma \to \ku$ a $(\chi, \chi)$-derivation, i.~e.
\begin{align*}
\eta(ht) &= \chi(h)\eta(t) + \eta(h)\chi(t),& h,t &\in \Gamma.
\end{align*}
Let $\cV_g(\chi, \eta) \in \ydG$ be a vector space of dimension 2,
homogeneous of degree $g$ and with action of $\Gamma$ given in a basis $(x_i)_{i\in\I_2}$ by
\begin{align}\label{equation:basis-triangular-gral}
h\cdot x_1 &= \chi(h) x_1,& h\cdot x_2&=\chi(h) x_2 + \eta(h)x_{1},
\end{align}
for all $h \in\Gamma$.
Then $\cV_g(\chi, \eta)$ is indecomposable in $\ydG \iff \eta\neq 0$.
As a braided vector space, $\cV_g(\chi, \eta)$ is either of diagonal type, when $\eta(g) =0$, or else isomorphic to $\cV(\epsilon,2)$,
$\epsilon = \chi(g)$ (indecomposability as Yetter-Drinfeld module is not the same as indecomposability as braided vector space).

\begin{lemma}\label{lema:indec2-abgroup}
Let $V\in \ydG$, $\dim V =2$. Then either $V$ is of diagonal type or else
$V \simeq \cV_g(\chi, \eta)$ for unique $g$, $\chi$ and $\eta$ (up to non-zero scalars) with $\eta(g) \neq 0$.
\end{lemma}

\pf Assume that $V$ is not of diagonal type; then $V$ is indecomposable. Since
$\ku \Gamma$ is cosemisimple and $\Gamma$ is abelian, there exists $g \in \Gamma$ such that $V$
is homogeneous of degree $g$. As $V$ is not simple,
$\soc V \simeq \ku_{g}^{\chi_1}$ and $V/\soc V \simeq \ku_{g}^{\chi_2}$
for some $\chi_1, \chi_2 \in \widehat\Gamma$. Pick $x_1\in \soc V -0$ and $x_2\in V_{g_2} - \soc V$;  then
$h\cdot x_2 = \chi_2(h) x_2 + \eta(h)x_{1}$ for all $h \in\Gamma$, where $\eta$ is a $(\chi_1, \chi_2)$-derivation.
Since  $V$ is not of diagonal type,  $ \chi_1(g) =  \chi_2(g)$ and $\eta(g) \neq 0$. Now
\begin{align*}
\chi_1(h)\eta(g) + \eta(h)\chi_2(g) =  \eta(hg) = \chi_1(g)\eta(h) + \eta(g)\chi_2(h)
\end{align*}
(since $\Gamma$ is abelian), hence $\chi_1(h) = \chi_2(h)$ for all $h\in \Gamma$.
\epf

Here is our first result on Nichols algebras of blocks.

\begin{theorem} \label{th:infGK}
Let $\epsilon\in\ku^{\times}$, $\ell \ge2$. If $\GK \cB(\cV(\epsilon,\ell)) < \infty$, then $\epsilon = \pm 1$.
\end{theorem}

The proof of Theorem \ref{th:infGK} is given in \S \ref{subsection:blockneqpm1-1stpart}.

\medbreak
The Nichols algebras $\NA(\cV(1,2))$ and $\NA(\cV(-1,2))$ have both $\GK = 2$, see \S \ref{subsection:jordanian} and \ref{subsection:super-jordanian}.
We conclude:

\begin{coro} \label{cor:indecomp-2} Let $g$, $\chi$, $\eta$ be as above and assume that $\eta(g) \neq 0$.
Then $\GK \cB(\cV_g(\chi,\eta)) < \infty \iff \chi(g) = \pm 1$.
\end{coro}

Assume that $\epsilon \neq \pm 1$. To prove the Theorem, it suffices to consider the case $\ell =2$.
Following \cite[Remark 4.5]{GH}, we attach a braided Hopf algebra of diagonal type to $\cB(\cV(\epsilon,\ell))$
in \S \ref{subsection:filtr-nichols}.
Then we prove the Theorem for $\epsilon \notin \G_{\infty}$ or $\epsilon \in \G_{N}'$, $N >3$ in \S \ref{subsection:blockneqpm1-1stpart}.
A separate proof is required when $N=3$, see Step~\ref{le:infGK3}.
Throughout this Section, $\cV = \cV(\epsilon,2)$ is the braided vector space with basis $x_1$, $x_2$ such that
\begin{align}\label{eq:bloque2}
c(x_i \otimes x_1) &= \epsilon x_1 \otimes x_i,& c(x_i \otimes x_2) &= (\epsilon x_2 + x_1) \otimes x_i,& i&\in \I_2.
\end{align}
We realize it as $\cV_1(\epsilon,2) \in \ydz$ when needed.  As in \eqref{eq:xijk}, we denote
\begin{align}\label{eq:x12}
x_{21} &= (\ad_c x_2)\, x_1 = x_2x_1 - \epsilon x_1 x_2.
\end{align}

\subsection{The Jordan plane}\label{subsection:jordanian}

Here we consider  $\cV=\cV(1,2)$. The Nichols algebra $\cB(\cV)$ turns out to be a quadratic algebra well-studied in the literature,
the so-called Jordan plane.
Together with the quantum planes, the Jordan plane exhausts the classification of the AS-regular algebras of dimension 2,  see \cite{ArS}.
From a Hopf-theoretical viewpoint, it is associated to the quantum Jordan $SL(2)$ defined in \cite{G} and independently  but later in \cite{DMMZ, zak};
and in dual form independently in \cite{LM, ohn}. The appellation \emph{Jordan} comes
from \cite{DMMZ}\footnote{We thank Christian Ohn for clarifications on this story.}.

\begin{prop} \label{pr:1block}
 $\cB(\cV(1,2))$ is presented by generators $x_1,x_2$ and relation
\begin{align}\label{eq:rels B(V(1,2))}
&x_2x_1-x_1x_2+\frac{1}{2}x_1^2.
\end{align}
$\GK \cB(\cV(1,2)) = 2$ and $B=\{ x_1^a x_2^b: a,b\in\N_0\}$ is a basis of $\cB(\cV(1,2))$.
\end{prop}

\pf Let $\cB$ be the algebra presented by \eqref{eq:rels B(V(1,2))}. The monomials $x_1^a x_2^b$ generate the vector space $\cB$.
Using skew derivations, cf. \eqref{eq:skewderivations}, we see that
\eqref{eq:rels B(V(1,2))} holds in $\cB(\cV(1,2))$ and that $B \subset \cB(\cV(1,2))$ is linearly independent.
(The latter conclusion holds because $\Bbbk $ has characteristic $0$).
Thus there is an algebra isomorphism $\cB \to \cB(\cV(1,2))$. Since $\cB(\cV(1,2))$ has the same Hilbert series as a polynomial ring
in 2 variables,  $\GK \cB(\cV(1,2)) = 2$.
\epf

A recursive argument starting from \eqref{eq:rels B(V(1,2))} shows that in $\cB(\cV(1,2))$
\begin{align}\label{eq:relations B(W) - case 1}
x_2x_1^n &=x_1^n x_2-\frac{n}{2}x_1^{n+1}, &n \in \N.
\end{align}
Since $x_{21} = -\frac{1}{2}x_1^2$ by \eqref{eq:rels B(V(1,2))}, the case $n =2$ implies
\begin{align}\label{eq:relations B(W) - x12}
x_{21}x_2 &= (x_1 + x_2) x_ {21}.
\end{align}

\subsection{The  super Jordan plane}\label{subsection:super-jordanian}
We are tempted to call the next example a \emph{super Jordan plane} albeit it
does not coincide apparently
with some objects with similar names in the literature.
After appearance of this paper, the cohomology of the super Jordan plane was determined in 
\cite{RS}. In particular, it has infinite global dimension what corroborates its absence from the list in \cite{ArS}.

\begin{prop} \label{pr:-1block} The defining ideal
$\cJ(\cV(-1,2))$ is generated by
\begin{align}\label{eq:rels-B(V(-1,2))-1}
&x_1^2, \\
\label{eq:rels-B(V(-1,2))-2}
&x_2x_{21}- x_{21}x_2 - x_1x_{21}.
\end{align}
$\GK \cB(\cV(-1,2)) = 2$ and $B=\{x_1^a x_{21}^bx_2^c: a\in\{0,1\},
b,c\in\N_0\}$ is a basis of $\cB(\cV(-1,2))$.
\end{prop}

In presence of \eqref{eq:rels-B(V(-1,2))-1}, \eqref{eq:rels-B(V(-1,2))-2} is
equivalent to the quantum Serre relation
\begin{align}\label{eq:rels-B(V(-1,2))-serre}
(\ad_cx_2)^2 x_1-x_1 (\ad_cx_2) (x_1).
\end{align}

\pf Both \eqref{eq:rels-B(V(-1,2))-1} and \eqref{eq:rels-B(V(-1,2))-2} are 0 in $\cB(\cV(-1,2))$
being annihilated by the skew derivations $\partial_1$ and $\partial_2$, cf. \eqref{eq:skewderivations}.
Notice that $\partial_1(x_{21})=x_1$, $\partial_2(x_{21})=0$.
Hence the quotient $\cBt$ of $T(V)$ by \eqref{eq:rels-B(V(-1,2))-1} and \eqref{eq:rels-B(V(-1,2))-2} projects onto $\cB(\cV(-1,2))$. Now  by \eqref{eq:rels-B(V(-1,2))-1} we have
\begin{align}\label{eq:rels-B(V(-1,2))-dos}
x_{21}x_1 = x_1x_2x_1 = x_1x_{21}.
\end{align}
Hence the subspace $I$ spanned by $B$ is a left ideal of $\cBt$, by  \eqref{eq:rels-B(V(-1,2))-2}, \eqref{eq:rels-B(V(-1,2))-dos}.
Since $1\in I$, $\cBt$ is spanned by $B$. To prove that $\cBt \simeq \cB(\cV(-1,2))$, it remains to show that
$B$ is linearly independent in $\cB(\cV(-1,2))$. We claim that this is equivalent to prove that $B'=\{x_2^c x_{21}^b x_1^a: a\in\{0,1\}, b,c\in\N_0\}$ is linearly independent.
Indeed, $\cBt$ is spanned by $B'$ since the subspace spanned $B'$ is also a left ideal; if $B'$ is linearly independent, then
the Hilbert series of $\cBt$ is $(1-t)^{-1}(1-t^2)^{-1}(1+t)=(1-t)^{-2}$, so $B$ should be linearly independent and vice versa.
Now suppose that there is a non-trivial linear combination
of elements of $B'$ in $\cB(\cV(-1,2))$, say of minimal degree. As
\begin{align} \label{eq:-1del1}
\partial_1(x_2^c x_{21}^b)&= b \,  x_2^c x_{21}^{b-1} x_1, & \partial_1(x_2^c x_{21}^b x_1)&= x_2^c x_{21}^b,
\end{align}
such linear combination does not have terms with $a$ or $b$ greater than 0. We claim that the elements $x_2^c$, $c\in\N_0$,
are linearly independent, yielding a contradiction.
Indeed, it is enough to show that $\partial_1^{n-1}\partial_2(x_2^n)\neq 0$ for all $n\in\N$.
We argue by induction. For $n=1$, $\partial_2(x_2)=1$. Now assume that $\partial_1^{n-1}\partial_2(x_2^n)=\mu_n\in\ku^\times$.
Then
\begin{align*}
\partial_1^{n}\partial_2(x_2^{n+1}) \overset \clubsuit = \partial_1^{n-1}\partial_2(x_2^n)\sum_{k=1}^n \partial_1(g^k\cdot x_2)
\overset\star=\mu_n\left(\sum_{k=1}^n (-1)^{k+1}k \right).
\end{align*}
To establish $\clubsuit$, we prove by induction on $j$ that
\begin{align*}
\partial_1^{n}\partial_2(x_2^{n+1})
= \partial_1^{n-j} &\left(\partial_1^{j} \partial_2(x_2^n) \, g^{j+1}\cdot x_2 \right)
+  \partial_1^{n-1}\partial_2(x_2^n) \sum_{k=1}^{j} \partial_1(g^k\cdot x_2).
\end{align*}
Indeed, $\partial_1^{n}\partial_2(x_2^{n+1})
= \partial_1^{n} \left(\partial_2(x_2^{n})\, g\cdot x_2 + x_2^n \right) =  \partial_1^{n} (\partial_2(x_2^{n})\,g\cdot x_2 )$
is the step $j =0$;  and the recursive step follows from
$\partial_1^{n-j} \left(\partial_1^{j} \partial_2(x_2^n) \, g^{j+1}\cdot x_2 \right) =$
\begin{align*}
&= \partial_1^{n-j - 1}
\left(\partial_1^{j +1} \partial_2(x_2^n) \, g^{j+2}\cdot x_2
+ \partial_1^{j} \partial_2(x_2^n) \partial_1( g^{j+1}\cdot x_2)
\right) \\
&= \partial_1^{n-j - 1}
\left(\partial_1^{j +1} \partial_2(x_2^n) \, g^{j+2}\cdot x_2\right) +  \partial_1^{n-1}\partial_2(x_2^n) \partial_1(g^{j+1}\cdot x_2).
\end{align*}
Taking $j = n$, we get $\clubsuit$. Now from $\star$ we compute $\mu_{2k}=(-1)^k (k!)^2$, $\mu_{2k+1}=(-1)^k k!(k+1)!$.
This proves the claim.
Then $B$ is a basis of $\cB(\cV(-1,2))$ and $\cBt=\cB(\cV(-1,2))$. The computation of $\GK$ follows from the Hilbert series at once.
\epf

\subsection{Filtrations of Nichols algebras}\label{subsection:filtr-nichols}

Let $H$ be a Hopf algebra. The category of (increasingly) filtered objects in $\ydh$ is tensor. For,
if $V, W \in \ydh$ and $(V_{i})_{i\in \N_0}$, $(W_{j})_{j\in \N_0}$
are ascending filtrations of Yetter-Drinfeld submodules of $V$ and $W$, respectively, then $V \otimes W$ is filtered by
\begin{align*}
 (V\otimes W)_k &= \sum_{j\le k} V_j\otimes W_{k-j},& k&\in \N_{0}.
\end{align*}
So we may consider filtered Hopf algebras in $\ydh$ by requiring the structure maps to be  of filtered objects.
The following facts are standard.

\begin{remark}\label{obs:filtered} Let $\cB$ be filtered Hopf algebra in $\ydh$, with  filtration $(\cB_{i})_{i\in \N_0}$.

\begin{enumerate}\renewcommand{\theenumi}{\alph{enumi}}
\renewcommand{\labelenumi}{(\theenumi)}
  \item\label{obs:filtered-1} The associated graded object $\gr \cB = \oplus_{i\in \N_0} \gr^i \cB$, where $\gr^i \cB = \cB_{i} / \cB_{i-1}$,
  is a graded Hopf algebra in $\ydh$.

  \item\label{obs:filtered-2} If $\pi: \cB \to \mathcal P$ is an epimorphism of Hopf algebras in $\ydh$, then $\mathcal P$ is a filtered Hopf algebra
  with $\mathcal P_i = \pi(\cB_i)$, $i\in \N_0$.

  \item\label{obs:filtered-3}  If $\cB = \oplus_{n\in \N_0} \cB^n$ is also a graded Hopf algebra in $\ydh$, with each term of the filtration
  a graded subobject:  $\cB_{i} = \oplus_{n\in \N_0} \cB_i^n$ for all $i$, then $\gr \cB = \oplus_{i, n\in \N_0} \gr^i \cB^n$ is a
  $\N_0\times \N_0$-graded Hopf algebra  in $\ydh$. (In this situation we say that $\cB$ is a graded filtered Hopf algebra).
  \end{enumerate}
\end{remark}

Let now $0 =V_0 \subsetneq V_1 \dots \subsetneq V_d = V\in \ydh$ be a flag of Yetter-Drinfeld submodules, $\dim V = \theta$;
it induces a filtration of   $T = T(V)$ that becomes a graded filtered algebra in $\ydh$; let $T^n_{m}$ be the $m$-th term of the filtration of the homogeneous component $T^n$.

\begin{lemma}\label{lemma:filtered}
\begin{enumerate}\renewcommand{\theenumi}{\alph{enumi}}
\renewcommand{\labelenumi}{(\theenumi)}
  \item\label{obs:filtered-4} $T(V)$  is a graded filtered Hopf algebra in $\ydh$.

  \item\label{obs:filtered-5} Every pre-Nichols algebra $\cB$ of $V$ is a graded filtered Hopf algebra in $\ydh$ and $\gr \cB$ is a
pre-Nichols algebra of $\gr V$.

 \item\label{obs:filtered-6} $\GK \NA(\gr V) \leq \GK \NA(V)$.
\end{enumerate}
\end{lemma}

\pf As $\Delta(V_i)\subseteq V_i\ot \ku + \ku \ot V_i$,
$T^n_{m} \subseteq \sum_{\substack{i + j = n\\ k+l = m}} T^i_{k} T^{j}_{l}$ and
$c(T^i_{k} \ot T^j_{l}) = T^j_{l} \ot T^i_{k}$, we have
\begin{align*}
\Delta(T^n_{m})\subseteq \sum_{\substack{i + j = n \\ k+l = m}} T^i_{k} \ot T^{j}_{l},
\end{align*}
by induction on $n$; \eqref{obs:filtered-4} follows. Now the first claim of \eqref{obs:filtered-5}
is immediate, see Remark \ref{obs:filtered} \eqref{obs:filtered-2}, \eqref{obs:filtered-3}.
As $\cB$ is generated by $ \cB^1 = V$, the algebra $\gr \cB$
is generated by $\gr^1 \cB = \gr V$; \eqref{obs:filtered-5} is proved. Now \eqref{obs:filtered-6} follows by \cite[6.5]{KL}.
\epf

The flag $0 =V_0 \subsetneq V_1 \dots \subsetneq V_d = V\in \ydh$ of Yetter-Drinfeld submodules is \emph{complete} if it has no proper
refinement. If $\dim V_i=\dim V_{i-1}+1$ for all $i$, i.e. it is flag of vector
spaces, then  it is \emph{absolutely complete}.
If the flag is absolutely complete, then $V^{\textrm{diag}} := \gr V$ is of diagonal type. If $\cB$ is a pre-Nichols algebra of $V$, then
 $\Bdiag := \gr \cB$ is a pre-Nichols algebra of  $V^{\textrm{diag}}$.

\begin{lemma}\label{prop:gr br Hopf algebra} Let $\epsilon\in\ku^{\times}$, $\ell \ge2$.
If $\cB$ is a pre-Nichols algebra of $\cV(\epsilon,\ell)$, then
$\GK \NA(\Vdiag) \leq \GK \NA$, where $\Vdiag$ is the braided vector space of diagonal type
with matrix $(q_{ij})_{i,j\in\I}$, $q_{ij}=\epsilon$ for all $i,j\in\I$.
\end{lemma}
\pf  $\cV = \cV(\epsilon,\ell)$ has a complete flag
given by $V_k = \langle x_i: i\in \I_k \rangle$, $k\in \I_\ell$. Since
$c(x_i \ot x_j) = \epsilon x_{j} \ot x_{i} + V_{i-1} \ot x_{j}$, for  $i,j\in \I_\ell$, Lemma \ref{lemma:filtered} applies.
\epf

\subsection{Proof of Theorem \ref{th:infGK}}\label{subsection:blockneqpm1-1stpart}

\begin{paso} \label{prop:V-lambda-generic}
If $\epsilon\notin \G_{\infty}$, then $\cB(\cV(\epsilon,\ell)) = T(\cV(\epsilon,\ell))$.
\end{paso}

This was claimed without proof in \cite[3.20 (a)]{AS Pointed HA}.

\pf Let $\cB = \cB(\cV(\epsilon,\ell))$; then $\Bdiag$ is a pre-Nichols algebra of $\Vdiag$ as in Lemma  \ref{prop:gr br Hopf algebra}.
Now $\cB(\Vdiag) \simeq T(\Vdiag)$ by \cite[5.2 (b)]{FG}, hence  $\Bdiag \simeq T(\Vdiag)$ and the claim follows,
say by looking at the Hilbert series.
\epf

\begin{paso} \label{prop:V-lambda-N>3}
If  $\epsilon\in\G'_N$, $N\geq 4$, then $\GK \cB(\cV(\epsilon,\ell)) = \infty$ for all $\ell\ge2$.
\end{paso}

\pf We may assume that $\ell = 2$.
  By Lemma~\ref{prop:gr br Hopf algebra}, it suffices to show that $\GK \cB (\Vdiag) = \infty$.
Now $\Vdiag$ is of Cartan type with Cartan matrix
$$ A = \begin{pmatrix} 2 & 2-N \\ 2-N & 2\end{pmatrix}. $$
 Thus Theorem \ref{thm:nichols-diagonal-finite-gkd} applies.
\epf

\begin{paso} \label{le:infGK3}
Let $\epsilon\in\G'_3$. Then $\GK \cB(\cV(\epsilon,2)) = \infty$.
\end{paso}

\begin{proof}
Let $V=\cV(\epsilon, 2)$.
We determine first $\dim \NA^n (V)$ for $n\le 4$. To do so, we compute inductively for $n\le 4$
\begin{align*}
& &J(n) &:=\ker \partial _2:T^n(V)\to \NA^{n-1}(V),\\ &\text{and}
&I(n) &:=J(n)\cap \big(\ker \partial _1:T^n(V)\to \NA^{n-1}(V)\big),
\end{align*}
so that $\NA^n (V)=T^n(V)/I(n)$, $n\le 4$.
Now $\ker \partial _2$ is a subalgebra of $T(V)$, since
$\partial _2$ is a skew-derivation, and it is invariant
under the action of $g$ since
$$g^{-1}\partial _2g=\epsilon \partial _2.$$

Clearly $J(1)=\ku x_1$ and $I(1)=0$. So let $n=2$. Note that
$x_1^2\in J(2)$. Further, $\epsilon ^2+\epsilon +1=0$ by assumption.
  Let $x_{21}=x_2x_1-\epsilon x_1x_2$. Since
\begin{align*}
\partial _2(x_1x_2)=&\;x_1,&
\partial _2(x_2x_1)=&\;\epsilon x_1,&
\partial _2(x_2^2)=&\;-\epsilon ^2x_2+x_1,
\end{align*}
  we conclude that $J(2)=\ku x_1^2+\ku x_{21}$. We record that
\begin{align} \label{eq:gx21}
g\cdot x_{21}=\epsilon ^2x_{21}+(\epsilon -\epsilon ^2)x_1^2.
\end{align}
Further,
$$ \partial _1(x_1^2)=-\epsilon ^2x_1,\quad \partial _1(x_{21})=(1-\epsilon ^2)x_2-\epsilon x_1$$
and hence $I(2)=0$.

Now let $n=3$. Then $x_1^3,x_1x_{21},x_{21}x_1\in J(3)$. Moreover,
\begin{align*}
\partial _2(x_1^2x_2)=&\;x_1^2,\\
\partial _2(x_2^2x_1)=&\;-x_2x_1+\epsilon x_1^2,\\
\partial _2(x_2x_1x_2)=&\;x_2x_1+\epsilon ^2x_1x_2+\epsilon x_1^2,\\
\partial _2(x_1x_2^2)=&\;-\epsilon ^2x_1x_2+x_1^2,\\
\partial _2(x_2^3)=&\;-\epsilon ^2x_2x_1+\epsilon x_1x_2+x_1^2.
\end{align*}
Let
\begin{align*}
z_1=&\;x_2^2x_1+x_2x_1x_2+x_1x_2^2+(\epsilon ^2-\epsilon )x_1^2x_2,\\
z_2=&\;x_2^3-\epsilon ^2x_2^2x_1+\epsilon ^2x_1x_2^2-\epsilon ^2x_1^2x_2.
\end{align*}
Then $J(3)=\ku x_1^3+\ku x_1x_{21}+\ku x_{21}x_1+\ku z_1+\ku z_2$.
We record that
\begin{align*}
  g\cdot z_1=&\;z_1+2\epsilon ^2(x_2x_1^2+x_1x_2x_1+x_1^2x_2) + (4\epsilon -1)x_1^3,\\
  g\cdot z_2=&\;z_2+\epsilon ^2z_1+\epsilon x_1x_{21} + (1-\epsilon)x_1^3.
\end{align*}
Further,
\begin{align*}
\partial_1(x_1^3)=&\;0,\\
\partial_1(x_1x_{21})=&\;\epsilon ^2x_2x_1-\epsilon ^2x_1x_2-\epsilon ^2x_1^2,\\
\partial_1(x_{21}x_1)=&\;\epsilon x_2x_1-\epsilon x_1x_2-\epsilon ^2x_1^2,\\
\partial_1(z_1)=&\;-\epsilon ^2x_2x_1+(2\epsilon -\epsilon ^2)x_1x_2+(2-\epsilon
)x_1^2,\\
\partial_1(z_2)=&\;(\epsilon -\epsilon ^2)x_2^2+x_2x_1-\epsilon x_1x_2-x_1^2
\end{align*}
and hence $I(3)=\ku x_1^3$. At this place we used that $\mathrm{char}\,\ku\ne
2$.

Finally, let us assume that $n=4$. Since $\ker \partial _2$ is an algebra, we
conclude first that
$$x_1^4,x_1^2x_{21},x_1x_{21}x_1,x_{21}x_1^2,x_{21}^2,x_1z_1,x_1z_2,z_1x_1,z_2x_1\in
J(4).$$
Moreover, for $\partial_2:T^4(V)\to \NA^3 (V)$ we also obtain that
\begin{align*}
\partial _2(x_1^3x_2)=&\;0,\\
\partial _2(x_1^2x_2^2)=&\;-\epsilon ^2x_1^2x_2,\\
\partial _2(x_1x_2x_1x_2)=&\;x_1x_2x_1+\epsilon ^2x_1^2x_2,\\
\partial _2(x_2x_1x_1x_2)=&\;x_2x_1^2+x_1^2x_2,\\
\partial _2(x_2x_1x_2x_2)=&\;-\epsilon ^2 x_2x_1x_2+x_1x_2^2
+x_2x_1^2+\epsilon^2x_1x_2x_1+\epsilon^2x_1^2x_2,\\
\partial _2(x_2x_2x_1x_2)=&\;x_2^2x_1
  -\epsilon x_2x_1x_2-x_2x_1^2+\epsilon^2x_1^2x_2,\\
\partial _2(x_2x_2x_2x_2)=&\;x_2^3-x_2x_1x_2+\epsilon ^2x_1x_2^2
-\epsilon ^2x_2x_1^2+\epsilon x_1x_2x_1+\epsilon x_1^2x_2.
\end{align*}
  These formulas immediately imply that the elements
$$x_1^4,x_1^2x_{21},x_1x_{21}x_1,x_{21}x_1^2,x_{21}^2,x_1z_1,x_1z_2,z_1x_1,z_2x_1,x_1^3x_2$$
form a basis of $J(4)$.
Since $I(V)=\oplus _{n\ge 2}I(n)$ is a two-sided ideal and since $x_1^3\in
I(3)$, it is clear that
$x_1^4,x_1^3x_2,x_2x_1^3\in I(4)$. Moreover,
\begin{align*}
\partial_1(x_1^2x_{21})=&\;-\epsilon x_1x_2x_1+x_1^2x_2,\\
\partial_1(x_1x_{21}x_1)=&\;x_2x_1^2-\epsilon x_1^2x_2,\\
\partial_1(x_{21}^2)=&\;(\epsilon ^2-\epsilon )x_2^2x_1
  +(1-\epsilon )x_1x_2^2\\
  &\;+2\epsilon x_2x_1^2 +(\epsilon ^2-1)x_1x_2x_1+\epsilon x_1^2x_2,\\
\partial_1(x_1z_1)=&\;x_2^2x_1+x_2x_1x_2+x_1x_2^2\\
  &\;+2\epsilon ^2x_2x_1^2+\epsilon ^2x_1x_2x_1+(\epsilon ^2-1)x_1^2x_2,\\
\partial_1(z_1x_1)=&\;x_2^2x_1+x_2x_1x_2+x_1x_2^2\\
  &\;-x_2x_1^2+(2\epsilon ^2-1)x_1x_2x_1+(\epsilon ^2-\epsilon )x_1^2x_2,\\
\partial_1(x_1z_2)=&\;x_2^3+\epsilon ^2x_2x_1x_2-x_1x_2^2
  -\epsilon ^2x_1x_2x_1+(\epsilon -\epsilon ^2)x_1^2x_2,\\
\partial_1(z_2x_1)=&\;x_2^3-x_2^2x_1+\epsilon ^2x_1x_2^2
      +\epsilon x_2x_1^2-\epsilon ^2x_1x_2x_1-\epsilon ^2x_1^2x_2.
\end{align*}
Let
$$ z_3=z_1x_1-x_1z_1+(\epsilon ^2-\epsilon )x_1x_{21}x_1
  +(\epsilon -\epsilon ^2)x_1^2x_{21} \in T^4(V).$$
The above formulas imply that $x_1^4,x_1^3x_2,x_2x_1^3,z_3$ form a basis of
$I(4)$.

Consider now the $\N _0^2$-graded algebra $\Bdiag $. The elements $x_2^3$ and
$$w:=x_1x_2^3-x_2^3x_1$$
are primitive and non-zero in $\Bdiag $, since $I(3)$ has
no non-zero element of $\N_0^2$-degree $3\alpha_2$ and $I(4)$ has no non-zero element
of $\N _0^2$-degree $\alpha_1+3\alpha_2$. Let $\tilde{\cB}_1$ be the
braided Hopf subalgebra of $\Bdiag $ generated by $x_1,x_2,w$. Let $\tilde{\cB}_2$ be the
graded braided Hopf algebra associated to the natural Hopf algebra filtration of
$\tilde{\cB }_1$, where the generators $x_1,x_2,w$ have degree one. Let
$\tilde{\cB}_3$ be the Nichols algebra quotient of $\tilde{\cB}_2$. Then
$$ \GK \tilde{\cB}_3\le \GK \tilde{\cB}_2\le \GK \tilde{\cB}_1\le \GK\Bdiag
=\GK \cB (V).$$
  On the other hand, the Dynkin diagram of the degree one part of
$\tilde{\cB}_3$ is a triangle with vertices labeled by $\epsilon $ and edges
labeled by $\epsilon ^{-1}$. Thus $\tilde{\cB}_3$ is of Cartan type with Cartan
matrix of (affine) type $\tilde{A}_2$.
Thus
$$\GK \cB (V)=\GK\tilde{\cB}_3=\infty $$
by Theorem \ref{thm:nichols-diagonal-finite-gkd}.
\end{proof}

\section{Yetter-Drinfeld modules of dimension 3}\label{sec:yd-dim3}
	\subsection{The setting}\label{subsection:YD3-setting}
Let $\Gamma $ be an abelian group. In this Section we consider $V \in \ydG$, $\dim V =3$, such that
the corresponding braided vector space is not of diagonal type. So, $V$ is not semisimple and we have three possibilities
that we discuss in \S \ref{subsubsection:YD3-decomp}, \ref{subsubsection:YD3-3points-notdiag} and \ref{subsubsection:YD3-indecomp}.

\subsubsection{A block and a point}\label{subsubsection:YD3-decomp} Here
$V = \cV_{g_1}(\chi_1, \eta) \oplus \ku_{g_2}^{\chi_2}$, where
$g_1, g_2\in \Gamma$,  $\chi_1, \chi_2\in \widehat\Gamma$ and $\eta: \Gamma \to \ku$ is a $(\chi_1, \chi_1)$-derivation.
Here $\cV_{g_1}(\chi_1, \eta) \in \ydG$ is  indecomposable with basis $(x_i)_{i\in\I_2}$ and action given by
\eqref{equation:basis-triangular-gral}; while $\ku_{g_2}^{\chi_2}\in \ydG$ is irreducible with base $(x_3)$.
Also $\eta(g_1) \neq 0$, otherwise this is discussed in \S \ref{subsubsection:YD3-3points-notdiag}, and then we may suppose that  $\eta(g_1) = 1$ by normalizing $x_1$.
Let
\begin{align*}
q_{ij}&= \chi_j(g_i),& i,j&\in \I_2;& \epsilon &=q_{11};& a&= q_{21}^{-1}\eta(g_2).
\end{align*}
Then the braiding is given in the  basis $(x_i)_{i\in\I_3}$ by
\begin{align}\label{eq:braiding-block-point}
(c(x_i \otimes x_j))_{i,j\in \I_3} &=
\begin{pmatrix}
\epsilon x_1 \otimes x_1&  (\epsilon x_2 + x_1) \otimes x_1& q_{12} x_3  \otimes x_1
\\
\epsilon x_1 \otimes x_2 & (\epsilon x_2 + x_1) \otimes x_2& q_{12} x_3  \otimes x_2
\\
q_{21} x_1 \otimes x_3 &  q_{21}(x_2 + a x_1) \otimes x_3& q_{22} x_3  \otimes x_3
\end{pmatrix}.
\end{align}
Let $V_1 = \cV_{g_1}(\chi_1, \eta)$, $V_2 =\ku_{g_2}^{\chi_2}$. If $\epsilon^2 = 1$, then
\begin{align}
c^2_{\vert V_1 \otimes V_2} = \id \iff q_{12}q_{21} = 1 \text{ and } a=0.
\end{align}
The scalar $q_{12}q_{21}$ will be called the \emph{interaction} between the block and the point.
The interaction is
\begin{align*}
\text{weak if } q_{12}q_{21}&= 1, &\text{mild if } q_{12}q_{21}&= -1,&\text{strong if } q_{12}q_{21}&\notin \{\pm 1\}.
\end{align*}
So $c^2_{\vert V_1 \otimes V_2}$ is determined by the interaction and the (somewhat hidden) parameter $a$.
We introduce a normalized version of $a$, called the \emph{ghost}:
\begin{equation}\label{eq:discrete-ghost}
\ghost = \begin{cases} -2a, &\epsilon = 1, \\
a, &\epsilon = -1.
    \end{cases}
\end{equation}
If $\ghost \in \N$, then we say that the ghost is \emph{discrete}.

\begin{theorem}\label{thm:point-block} Let $V$ be a braided vector space with braiding \eqref{eq:braiding-block-point}.
Then the following are equivalent:
\begin{enumerate}
	\item $\GK \NA (V) < \infty$.
	
	\item $V$ is as in Table \ref{tab:finiteGK-block-point}.
\end{enumerate}	
\end{theorem}

\begin{table}[ht]
\caption{Nichols algebras of a block and a point with finite $\GK$}\label{tab:finiteGK-block-point}
\begin{center}
\begin{tabular}{|c|c|c|c|c|c|}
\hline \scriptsize{interaction} & $\epsilon$  & $q_{22}$  & $\ghost$ & $\NA(V)$, \S  & $\GK$  \\
\hline
weak & $\pm 1$ & $1$ or $\notin \G_{\infty}$ & $0$ & $\NA(\cV(\epsilon, 1)) \underline{\otimes} \NA(\ku x_3)$  &   3
\\ \cline{3-3}\cline{6-6}
& & $\in \G_{\infty} - \{1\}$ &  &  &  2
\\ \cline{2-6}
& $1$  & $1$  & \small{discrete}   & $\cB(\lstr( 1, \ghost))$, \ref{subsubsection:lstr-11disc} &    $\ghost + 3$
\\\cline{3-6}
& & $-1$  & \small{discrete}  &  $\cB(\lstr( -1, \ghost))$, \ref{subsubsection:lstr-1-1disc}  &   $2$
\\ \cline{3-6}
& & $\in \G'_3$  &  1 & $\cB(\lstr( \omega, \ghost))$, \ref{subsubsection:lstr-1omega1} &   $2$
\\
\cline{2-6}
& $-1$ & $1$  & \small{discrete} &  $\cB(\lstr_{-}( 1, \ghost))$, \ref{subsubsection:lstr--11disc} &    $\ghost + 3$
\\\cline{3-6}
& & $-1$  &  \small{discrete} &  $\cB(\lstr_{-}(- 1, \ghost))$, \ref{subsubsection:lstr--1-1disc} &   $\ghost + 2$
\\ \hline
mild & $-1$  & $-1$  &   1 & $\cB(\cyc_1)$, \ref{subsubsection:nichols-mild} &    $2$

\\\hline
\end{tabular}
\end{center}
\end{table}

\noindent \emph{Scheme of the proof.} First,
$\epsilon^2 = 1$ and the interaction is not strong, by Corollary \ref{cor:indecomp-2} and Lemma \ref{le:strong}.
If $a = 0$ and $q_{12}q_{21} = 1$, then $\NA(V) \simeq \NA(\cV(\epsilon, 1)) \underline{\otimes} \NA(\ku x_3)$ has either $\GK =2$
if $q_{22} \in \G'_N$, $N > 1$, or 3 otherwise.
For the rest, Theorems \ref{thm:pm1bp} and \ref{thm:pm1bp-mild} apply.
\qed

\subsubsection{A pale block and a point}\label{subsubsection:YD3-3points-notdiag} 
Here is the second possibility in dimension 3: again
$V = \cV_{g_1}(\chi_1, \eta) \oplus \ku_{g_2}^{\chi_2}$ is a Yetter-Drinfeld module over $\ku \Gamma$, where
$g_1, g_2\in \Gamma$,  $\chi_1, \chi_2\in \widehat\Gamma$ and $\eta: \Gamma \to \ku$ is a $(\chi_1, \chi_1)$-derivation, as in \S \ref{subsubsection:YD3-decomp}.
To be out of \S \ref{subsubsection:YD3-decomp} we need $\eta(g_1) = 0$; thus the braiding will \emph{not} 
be of the form \eqref{eq:braiding-block-point}.
Then $\eta(g_2) \neq 0$ because $V$ is assumed not of diagonal type.
Let
$q_{ij}= \chi_j(g_i)$, $i,j\in \I_2$; $\epsilon =q_{11}$;
by normalizing $x_1$, we assume that $q_{21}^{-1}\eta(g_2) = 1$.
Then the braiding is given in the  basis $(x_i)_{i\in\I_3}$ by
\begin{align}\label{eq:braiding-paleblock-point}
(c(x_i \otimes x_j))_{i,j\in \I_3} &=
\begin{pmatrix}
\epsilon x_1 \otimes x_1&  \epsilon x_2  \otimes x_1& q_{12} x_3  \otimes x_1
\\
\epsilon x_1 \otimes x_2 & \epsilon x_2  \otimes x_2& q_{12} x_3  \otimes x_2
\\
q_{21} x_1 \otimes x_3 &  q_{21}(x_2 +  x_1) \otimes x_3& q_{22} x_3  \otimes x_3
\end{pmatrix}.
\end{align}

Since $\langle x_1, x_2\rangle$ is of diagonal Cartan type, we may assume that $\epsilon \in \G_2 \cup \G_3$.
This case is solved in \S \ref{subsection:ydz}, see Theorem \ref{th:paleblock-point-resumen}.

\begin{notation}
A braided vector space of dimension 3 whose  braiding is given in a  basis $(x_i)_{i\in\I_3}$ by
\eqref{eq:braiding-paleblock-point} is called \emph{a pale block and a point}. Indeed, 
$(x_i)_{i\in\I_2}$ form a braided subspace 
that in itself is of diagonal type, but becomes a block only together with $x_3$.
\end{notation}

\subsubsection{Indecomposable of dimension 3}\label{subsubsection:YD3-indecomp}
Assume that $V$ is indecomposable but not of diagonal type. Then
$V = V_g$ is homogeneous of degree $g \in \Gamma$ where the action of $g$ in some basis is given  by
$\begin{pmatrix}
\epsilon_1 &  1 & 0
\\
0 & \epsilon_1 & \eta
\\
0 &  0 & \epsilon_2
\end{pmatrix}$, where $\epsilon_1, \epsilon_2\in \ku^{\times}$ and  $\eta \in \ku$.
If $\eta = 0$, then the braided vector space $V$ is as in \ref{subsubsection:YD3-decomp}.
If $\eta \neq 0$, then $\epsilon_1= \epsilon_2 =:\epsilon$ so that $V$ is isomorphic to $\cV(\epsilon, 3)$ as braided vector space. We may assume that $\epsilon \in \{\pm 1\}$, by Theorem \ref{th:infGK}.

\begin{theorem}\label{thm:eigenvalue 1, size geq 3} If $\epsilon \in \{\pm 1\}$, then
$\GK \cB(\cV(\epsilon,\ell))=\infty$ for all $\ell \ge3$.
\end{theorem}

The proof is given in \S \ref{pf:eigenvalue 1, size geq 3}. This completes the proof of Theorem \ref{theorem:blocks}.
We next summarize  Theorems \ref{thm:point-block} and \ref{thm:eigenvalue 1, size geq 3}, and the discussion above.

\begin{coro}\label{cor:indec-3} Let $\Gamma $ be an abelian group, $V \in \ydG$, $\dim V =3$, not of diagonal type.
If $\GK \cB(V) < \infty$, then $V$ is as listed in Table \ref{tab:finiteGK-block-point}.
\end{coro}

\subsubsection{}\label{subsubsection:YD3-notation} We fix the notation for the rest of this Section.
Let $(V, c)$ be a braided vector space of dimension 3, with a basis $(x_i)_{i\in\I_3}$ and braiding given  by \eqref{eq:braiding-block-point},
for some $q_{ij} \in \kut$, $i,j\in \I_2$, $a \in \ku$; and $\epsilon \in \{\pm 1\}$. The braided subspace spanned by $x_1, x_2$ is $\simeq \cV(\epsilon, 2)$
and we may use the material from \S \ref{subsection:indec-2}, \ref{subsection:jordanian}, \ref{subsection:super-jordanian},
e. g. $x_{21} = x_2x_1 - \epsilon x_1 x_2$.
We realize $(V, c)$ as a Yetter-Drinfeld module $\cV_{g_1}(\chi_1, \eta) \oplus \ku_{g_2}^{\chi_2}$
over some abelian group $\Gamma$ with suitable $g_1, g_2$,  $\chi_1, \chi_2$ and $\eta$; for instance $\Gamma = \Z^2$ would do.
Here $V_1 = \cV_{g_1}(\chi_1, \eta)$ is spanned by $x_1$ and $x_2$, while $V_2 = \ku_{g_2}^{\chi_2}$ is spanned by $x_3$.

In order to analyze the structure of $\NA (V)$, we consider
$K=\NA (V)^{\mathrm{co}\,\NA (V_1)}$, see \S \ref{subsubsection:weak-classification}, \ref{subsection:mild}.
By the general theory, see \cite[Proposition 8.6]{HS}, and also \cite[Lemma 3.2]{AHS}, 
$K = \oplus_{n\ge 0} K^n$ inherits the grading of $\NA (V)$;
$\NA(V) \simeq K \# \NA (V_1)$ and $K$ is the Nichols algebra of
\begin{align} \label{eq:1bpK^1}
  K^1= \ad_c\NA (V_1) (V_2).
\end{align}
Now $K^1\in {}^{\NA (V_1)\# \ku \Gamma}_{\NA (V_1)\# \ku \Gamma}\mathcal{YD}$ with the adjoint action
and the coaction given by
\begin{align} \label{eq:coaction-K^1}
\delta &=(\pi _{\NA (V_1)\#  \ku \Gamma}\otimes \id)\Delta _{\NA (V)\#  \ku \Gamma}.
\end{align}

For further use, we introduce
\begin{align}\label{eq:zn}
z_n &:= (\ad_c x_2)^n x_3,& n&\in\N_0.
\end{align}
\begin{remark}\label{remark:GK-block-point-K}
Assume that $\GK \toba(V_1) < \infty$ (hence, it is 2) and that the action of
$\toba(V_1)$ on $K$ is locally finite; then by Lemma \ref{lemma:GKdim-smashproduct}
and Remark \ref{rem:examples-GKdterministic},
\begin{align*}
\GK\NA(V) = \GK (K \# \NA (V_1))  = \GK K + 2.
\end{align*}
\end{remark}

\subsubsection{Strong interaction}\label{subsubsection:strong}
The following lemma simplifies significantly our problem.

\begin{lemma} \label{le:strong}
If the interaction is strong, then $\GK \NA (V)=\infty $.
\end{lemma}

\begin{proof} By assumption, $q_{12}^2q_{21}^2\ne 1$.
Let $\Bdiag$ be the braided Hopf algebra of diagonal type
associated to the natural flag of Yetter-Drinfeld submodules of $V$, see \S \ref{subsection:filtr-nichols}.
The class of $x_{21}\in \Bdiag $ is non-zero and primitive since
$$ \Delta (x_{21})=x_{21}\otimes 1+1\otimes x_{21} -\epsilon x_1\otimes x_1 $$
in $\NA (V)$. Consider the filtration of $\Bdiag$, where $x_1$, $x_2$, $x_3$ and
$x_{21}$ have degree one. Let $\widetilde{\cB}$ be the associated graded braided Hopf algebra;
$\widetilde V :=\widetilde{\cB}^1$, generated by $x_1$, $x_2$, $x_3$, and $x_{21}$, is of diagonal type, with
Dynkin diagram  consisting of four vertices. Since $g_1^2\cdot
x_{21}=x_{21}$, the vertex
of $x_{21}$ has label $1$. Moreover, this vertex is connected by an edge with the
vertex of $x_3$, which is labeled by $q_{12}^2q_{21}^2$, since
\begin{align*}
 g_1^2\cdot x_3 &= q_{12}^2x_3,& g_2\cdot x_{21} &= q_{21}^2x_{21}.
\end{align*}
Since $q_{12}^2q_{21}^2\ne 1$,  $\GK \NA(\widetilde{V}) = \infty$ by
Lemma \ref{lemma:points-trivial-braiding}. The Lemma follows because
$\GK \NA(\widetilde{V}) \le \GK \widetilde{\NA}\le \GK \NA (V)$.
\end{proof}

\subsection{Weak interaction}\label{subsection:weak}
\emph{We assume that $q_{12}q_{21} = 1$}.

\subsubsection{}\label{subsubsection:weak-preliminaries}
We establish first a series of useful formulae.

\begin{lemma} \label{le:-1bpz}The following hold in $\NA(V)$ for all $n\in\N_0$:
\begin{align}
\label{eq:-1block+point}
g_1\cdot z_n &= \epsilon^nq_{12}z_n,& x_1z_n &= \epsilon^n q_{12}z_nx_1,& x_{21}z_n &= q_{12}^2z_nx_{21},
\\\label{eq:-1block+point-bis}
x_{21}^n x_2 &= (n\epsilon x_1 + x_2) x^n_{21},& g_2\cdot z_n &= q_{21}^nq_{22}z_n.&
\end{align}
\end{lemma}
\pf \eqref{eq:-1block+point}: the case $n =0$ is $g_1\cdot x_3 = q_{12}x_3$, by hypothesis,
$x_1x_3 = q_{12}x_3x_1$, because the interaction is weak, and $ x_{21}x_3 = q_{12}^2x_3x_{21}$; this last
equality can be checked with derivations, for instance $\partial_3(x_{21}x_3 - q_{12}^2x_3x_{21}) = a(\epsilon - 1)x_1^2$.
Now suppose that \eqref{eq:-1block+point} holds for $n$. Then
\begin{align}
\label{eq:-1block+point-z2}
x_2z_n &= \epsilon^nq_{12}z_nx_2 + z_{n+1}.
\end{align}
We compute
\begin{align*}
g_1\cdot z_{n+1} & \overset{\star}= g_1\cdot (x_2 z_n - \epsilon^{n}q_{12}z_n x_2)
\overset{\star}= (\epsilon x_2+x_1)\epsilon^nq_{12}z_n - \epsilon^{2n}q_{12}^2z_n(\epsilon x_2+x_1)\\
&= \epsilon^{n+1}q_{12}\left[(x_2z_n - \epsilon^{n}q_{12}z_nx_2) + (x_1z_n - \epsilon^{n}q_{12}z_nx_1) \right]
\overset{\star}= \epsilon^{n+1}q_{12}z_{n+1}
\end{align*}
where we applied in $\star$ either of the equalities in \eqref{eq:-1block+point} for $n$. Similarly,
\begin{align*}
x_1 z_{n+1} &= x_1( x_2 z_n - \epsilon^{n}q_{12}z_n x_2)
= \epsilon (-x_{21}+ x_2x_1)z_n - \epsilon^{2n}q_{12}^2 z_n x_1 x_2 \\
&= - \epsilon q_{12}^2z_n x_{21} + \epsilon^{n+1}q_{12}x_2z_nx_1 -q_{12}^2 z_n x_1 x_2 \\
&= \epsilon^{n+1}q_{12}x_2z_nx_1 - q_{12}^2z_n (\epsilon  x_{21} + x_1 x_2) \\
& = \epsilon^{n+1}q_{12}x_2z_nx_1 - \epsilon  q_{12}^2z_n  x_2 x_1 \\
&=  \epsilon^{n+1}q_{12}(x_2z_nx_1 - \epsilon^n  q_{12} z_n  x_2) x_1 =\epsilon^{n+1}q_{12} z_{n+1}x_1.\end{align*}
Also $x_{21}x_2 = (\epsilon x_1 + x_2) x_ {12}$ by \eqref{eq:relations B(W) - x12} and \eqref{eq:rels-B(V(-1,2))-2}. Hence
the first equation in \eqref{eq:-1block+point-bis} follows by induction; and also
\begin{align*}
x_{21}z_{n+1}&= x_{21}( x_2 z_n - \epsilon^{n} q_{12}z_n x_2)
=(\epsilon x_1 + x_2)x_{21}z_n - \epsilon^{n} q_{12}^3 z_n x_{21} x_2  \\
&=q_{12}^2(\epsilon x_1 + x_2)z_nx_{21} - \epsilon^{n} q_{12}^3z_n (\epsilon x_1 + x_2)x_{21}\\
&=q_{12}^2(x_2z_n - \epsilon^{n} q_{12}z_nx_2)x_{21} = q_{12}^2z_{n+1} x_{21}.
\end{align*}
For the second equation in  \eqref{eq:-1block+point-bis}, we compute using \eqref{eq:-1block+point}
\begin{align*}
g_2\cdot z_{n+1} &= g_2\cdot (x_2 z_n  - \epsilon^{n} q_{12}z_n x_2) \\
&=q_{21}(x_2+ax_1)q_{21}^nq_{22}z_n  - \epsilon^{n}  q_{12}q_{21}^{n+1}q_{22}z_n(x_2+ax_1)\\
&=q_{21}^{n+1}q_{22}(x_2z_n - \epsilon^{n} q_{12}z_nx_2) = q_{21}^{n+1}q_{22}z_{n+1}.
\end{align*}
which completes the proof.
\epf

We define recursively a family $(\mu_n)_{n \in\N_0}$ by
\begin{align}\label{eq:def-mu-n}
\mu_0 &= 1, & \mu_{2k + 1} &=- (a + k\epsilon) \mu_{2k},& \mu_{2k} &= (a + k +  \epsilon (a + k - 1))\mu_{2k - 1}.
\end{align}
Thus
\begin{align*}
\mu_{n + 1} &= \begin{cases} (2a + n)\mu_n &\text{if $n$ is odd,} \\
    -\frac{2a + n}{2}  \mu_n &\text{if $n$ is even,}
   \end{cases} & &\text{when}& \epsilon &= 1;
   \\
\mu_{n + 1} &= \begin{cases} \mu_n &\text{if $n$ is odd,} \\
     -(a - \frac n{2}) \mu_n &\text{if $n$ is even,}
   \end{cases} & &\text{when}& \epsilon &= -1.
\end{align*}
Therefore, if $\ghost$ is discrete, then $\mu_n =0 \iff n > \vert 2a\vert$, while
\begin{align}\label{eq:a-condition}
 & \ghost \notin \N_0& &\implies  \mu_n \neq 0& \text{for all } n &\in \N_0.
\end{align}

\begin{lemma}\label{lemma:derivations-zn}
For all $k\in\N_0$, $\partial_1(z_k) =\partial_2(z_k) =  0$,
\begin{align}\label{eq:derivations-zn}
\partial_3(z_{2k}) &= \mu_{2k} x_{21}^k,&  \partial_3(z_{2k+1}) &= \mu_{2k + 1} x_1x_{21}^k.
\end{align}
Therefore, if $\ghost$ is discrete, then $z_n =0 \iff n > \vert 2a\vert$.
\end{lemma}
\pf
The claims for $k=0$ follow at once because $\partial_i(x_j)= \delta_{ij}$. Arguing recursively,
\begin{align*}
&\partial_1(z_{k+1}) \overset{\eqref{eq:-1block+point-z2}}= \partial_1(x_2 z_k - \epsilon^{k}q_{12}z_k x_2) = 0,
\\
&\partial_2(z_{k+1}) \overset{\eqref{eq:-1block+point-z2}}= \partial_2(x_2 z_k - \epsilon^{k}q_{12}z_k x_2)
= g_2 \cdot z_k - \epsilon^{k}q_{12}z_k \overset{\eqref{eq:-1block+point}}= 0
\end{align*}
Also
\begin{align*}
\partial_3(z_{2k+1})&= \partial_3(x_2 z_{2k} - q_{12} z_{2k} x_2)
= \mu_{2k} \left(x_2 x_{21}^k - q_{12}q_{21} \,x_{21}^k (x_2+ ax_1) \right) \\
\overset{\eqref{eq:-1block+point-bis}} = & \mu_{2k} \left(x_2 x_{21}^k -  (k\epsilon x_1 + x_2) x_{21}^k -  ax_{21}^k x_1 \right)
= - (a + k\epsilon) \mu_{2k}\, x_1x_{21}^k;
\end{align*}
\begin{align*}
\partial_3(z_{2k})&= \partial_3(x_2 z_{2k-1} - q_{12}\epsilon\,  z_{2k-1} x_2)
\\ &= \mu_{2k-1} \left(x_2 x_1x_{21}^{k-1} - \epsilon \,x_1x_{21}^{k-1} (x_2+ ax_1) \right)
\\ &= \mu_{2k-1} \left(x_2 x_1x_{21}^{k-1} - \epsilon \,x_1  ((k-1)\epsilon x_1 + x_2) x_{21}^{k-1} - \epsilon a \,x_1^2 x_{21}^{k-1}  \right)
\\ &= \mu_{2k-1} \left(x_{21}^{k} - (k-1 +  \epsilon a) x_1^2 x_{21}^{k-1} \right)
\\ &= \mu_{2k-1} \left(1  + (k-1 +  \epsilon a) (1 + \epsilon) \right) x_{21}^{k}
\\ &=  \left(a + k +  \epsilon (a + k - 1) \right) \mu_{2k-1} x_{21}^{k}
\end{align*}
where we used $x_{21}x_1 =  x_1 x_ {12}$,  by \eqref{eq:relations B(W) - x12} and \eqref{eq:rels-B(V(-1,2))-dos},
and $x_1^2 =  -(1 + \epsilon) x_ {12}$.
\epf

If $\epsilon = 1$, then \eqref{eq:derivations-zn} says that
\begin{align}\label{eq:derivations-zn-eps1}
\partial_3(z_{2k}) &= \frac{(-1)^k}{2^k} \mu_{2k} x_{1}^{2k}, &
\partial_3(z_{2k+1}) &= \frac{(-1)^k}{2^k}\mu_{2k + 1} x_{1}^{2k + 1}.
\end{align}

\begin{lemma}\label{lemma:weak-not-discrete}
If the ghost is not discrete, then $\GK \cB(V)=\infty$.
\end{lemma}

\pf  By hypothesis and \eqref{eq:a-condition}, $\deg z_n = n +1$. We claim that the set
$$ S= \{z_{2n_1}\cdots z_{2n_k}: k,n_i \in \N, n_1<\dots<n_k\} $$
is linearly independent.
Otherwise, there exists a non-trivial linear combination of elements of $S$, that we assume  it is of minimal degree.
By Lemmas \ref{le:-1bpz} and \ref{lemma:derivations-zn}, and \eqref{eq:-1del1} when $\epsilon =-1$,
\begin{align*}
\partial_1^{n_k}\partial_3(z_{2n_1}\dots z_{2n_k}) &= n_k!\, \mu_{2n_k}z_{2n_1}\dots z_{2n_{k-1}},& \epsilon &= -1,
\\
\partial_1^{2n_k}\partial_3(z_{2n_1}\dots z_{2n_k}) &= \frac{(-1)^{n_k}}{2^{n_k}} (2n_k)! \,
\mu_{2n_k}z_{2n_1}\dots z_{2n_{k-1}},& \epsilon &= 1.
\end{align*}
Hence the coefficient of $z_{2n_k}$, where $n_k$ is maximal, of the linear
combination is again a linear dependence by \eqref{eq:a-condition}, a contradiction; the claim follows.
Since $\cB(V)$ is finitely generated,  $\GK \cB(V)=\infty$ by  Lemma \ref{lemma:rosso-lemma19-gral}.
\epf

\subsubsection{Proof of Theorem \ref{thm:eigenvalue 1, size geq 3}}\label{pf:eigenvalue 1, size geq 3}

It is enough to assume that $\ell = 3$.
The ideal $ I $ generated by $x_1$ in $\cB(\cV(\epsilon,3))$ is a
coideal since $x_1$ is primitive, and a Yetter-Drinfeld submodule since $\ku x_1$ is stable by the $\Z$-action.
Then $\cBt=\cB(\cV(\epsilon,3))/ I $ is a $\N_0$-graded Hopf algebra in $\ydz$, generated by (the images of) $x_2,x_3$.
Let
\begin{align*}
w &= \begin{cases} x_3x_2-x_2x_3+\frac{1}{2}x_2^2 \in \cBt_{\gb^2}, & \epsilon = 1, \\
(\ad_c x_3)^2 (x_2) - x_2 (\ad_c x_3)(x_2) \in \cBt_{\gb^3}, & \epsilon = -1.   \end{cases}
\end{align*}
We claim that $w\neq 0$. Indeed,
$(\ad_c x_3)^2 x_2-x_2 (\ad_c x_3)(x_2) \notin  I $ by direct computation of of the component of degree $3$ of $\cB(\cV(-1,3))$,
and similarly for $\epsilon = 1$.
Now $w$ is a primitive element in $\cBt$, hence the Nichols algebra of the span $V$ of
the linearly independent primitive elements $x_2$, $x_3$ and $w$ is a quotient of $\cBt$. But  $V$ is as
in \eqref{eq:braiding-block-point} for
\begin{align*}
q_{12}=q_{21}=q_{22} &=1, & a &=2, & \text{ if }\epsilon &= 1, \\
q_{12}=q_{21}=q_{22} &= -1,&  a &=-1, &\text{ if } \epsilon &= -1.
\end{align*}

Since $\GK \cB(V) \leq \GK \cBt \leq \GK \cB(\cV(\epsilon,3))$, Lemma \ref{lemma:weak-not-discrete}
implies that $\GK \cB(\cV(\epsilon,3)) = \infty$.
\qed

\subsubsection{}\label{subsubsection:weak-classification}
Recall
$K=\NA (V)^{\mathrm{co}\,\NA (V_1)} \simeq \cB(K^1)$, $K^1= \ad\NA (V_1)(V_2)$, cf. \S \ref{subsubsection:YD3-notation}.

\begin{remark} \label{lemma:K-basis} If $\ghost$ is discrete, then the family $(z_n)_{0\le n\le \vert 2a\vert}$ is a basis of $K^1$.
\end{remark}

Indeed,  $(z_n)_{0\le n\le \vert 2a\vert}$ is linearly independent,
because the $z_n$'s are homogeneous of distinct degrees, and  are $\neq 0$ by \eqref{eq:derivations-zn}.
We have for all $n\in\N_0$
\begin{align}\label{eq:adx1-zn}
\ad_c x_1 (z_n) =  x_1z_n - g_1\cdot z_n x_1 \overset{\eqref{eq:-1block+point}}= \epsilon^n q_{12}z_nx_1 -  \epsilon^nq_{12}z_nx_1 = 0,
\\
\label{eq:adx12-zn}
\ad_c x_{21} (z_n) = \ad_c  x_2 \ad_c x_1 (z_n) - \epsilon \ad_c x_1 \ad_c x_2 (z_n) = 0.
\end{align}
Then $K^1$ is generated by $z_n$, $n\in \N_0$, and the Remark follows.

\smallbreak
If $\epsilon = 1$, then we define    recursively
$\nu_{k,n}$  as follows:  $\nu_{n,n}=1$,
\begin{align*}
\nu_{0, n+1} &= -  \big(\frac{n}{2}+ a \big)\nu_{0,n}, &
\nu_{k,n+1}&=\nu_{k-1,n}-\left(\frac{n+k}{2}+ a\right)\nu_{k,n}, & &  1\le k\le n.
\end{align*}

\begin{lemma} \label{le:zcoact}
The coaction \eqref{eq:coaction-K^1} on $z_n$, $n\in \N _0$, is given by \eqref{eq:coact-zn}, when $\epsilon = 1$,
and by \eqref{eq:coact-zn-even}, \eqref{eq:coact-zn-odd}, when $\epsilon = -1$:

\begin{align} \label{eq:coact-zn}
\delta (z_{n}) &= \sum _{k=0}^n \nu_{k,n}\, x_1^{n-k}g_1^{k} g_2 \otimes z_k.
\end{align} \begin{align}\label{eq:coact-zn-even}
\delta (z_{2n}) &=
\sum _{k=1}^n k\binom n k \mu_{k,n} \,
x_1x_{21}^{n-k}g_1^{2k-1}g_2\otimes z_{2k-1}
\\ \notag&\qquad  + \sum _{k=0}^n \binom n k\mu_{k, n} x_{21}^{n-k}g_1^{2k}g_2\otimes z_{2k},
\\ \label{eq:coact-zn-odd}
\delta (z_{2n+1}) &=
\sum _{k=0}^n \binom nk \mu_{k, n+1} \,   x_1x_{21}^{n-k}g_1^{2k}g_2\otimes z_{2k}
\\ \notag& \qquad + \sum _{k=0}^n \binom nk \mu_{k+1, n+1} x_{21}^{n-k}g_1^{2k+1}g_2\otimes z_{2k+1}.
\end{align}
\end{lemma}

\begin{proof}
We proceed by induction on $n$. Since $\delta (z_0)= g_2\otimes z_0$, the claim is clear for $n=0$.
Assume that $\epsilon = 1$ and that \eqref{eq:coact-zn} holds for $n$. Then
\begin{multline*}\allowdisplaybreaks
\delta (z_{n+1}) \overset{\eqref{eq:-1block+point-z2}} = (x_2\otimes 1+g_1\otimes x_2)\delta (z_{n})
-q_{12}\delta (z_{n})(x_2\otimes 1+g_1\otimes x_2)
\displaybreak[0]\\
= \sum _{k=0}^n \nu_{k,n}
\big(x_2 x_1^{n-k}g_1^{k} g_2 \otimes z_k +  x_1^{n-k}g_1^{k + 1} g_2 \otimes x_2 z_k
\displaybreak[0]\\
\qquad - q_{12} x_1^{n-k}g_1^{k} g_2 x_2 \otimes z_k - q_{12} x_1^{n-k}g_1^{k + 1} g_2 \otimes z_k x_2 \big)
\displaybreak[0]\\
\overset{\star} = -\sum _{k=0}^n \big(\frac{n+k}{2}+a \big) \nu_{k,n} x_1^{n +1 -k} g_1^kg_2\otimes z_k
+\sum _{k=1}^{n + 1} \nu_{k - 1,n} \, x_1^{n + 1 -k} g_1^kg_2\otimes z_{k}
\displaybreak[0]\\
= -  \big(\frac{n}{2}+ a \big)\nu_{0,n} x_1^{n+1} \otimes z_0 +
\sum _{k=1}^n \left[\nu_{k-1,n}- \big( \frac{n+k}{2}+ a \big)\nu_{k,n} \right] x_1^{n + 1 -k} g_1^kg_2\otimes z_k
\displaybreak[0]\\
+ \nu_{n,n} g_1^{n+1}g_2\otimes z_{n+1}.
\end{multline*}
Here, we use  \eqref{eq:relations B(W) - case 1} and the identity $x_2 g_1^k = g_1^kx_2 - kx_1g_1^k$ in $\star$.
But this is \eqref{eq:coact-zn} for $n+1$.
Assume next that $\epsilon =-1$  and that \eqref{eq:coact-zn-even} holds for $n$.
Then
\begin{multline*}\allowdisplaybreaks
\delta (z_{2n+1}) \overset{\eqref{eq:-1block+point-z2}} = (x_2\otimes 1+g_1\otimes x_2)\delta (z_{2n})
- q_{12}\delta (z_{2n})(x_2\otimes 1+g_1\otimes x_2)
\displaybreak[0]\\
= \underbrace{\sum _{k=1}^n k\binom n k \mu_{k,n} \, x_2 x_1x_{21}^{n-k}g_1^{2k-1}g_2\otimes z_{2k-1}}_{A}
+ \underbrace{\sum _{k=0}^n \binom n k\mu_{k, n}  x_2 x_{21}^{n-k}g_1^{2k}g_2\otimes z_{2k}}_{B}
\displaybreak[0]\\
\underbrace{\sum_{k=1}^n k\binom n k \mu_{k,n}  g_1 x_1x_{21}^{n-k}g_1^{2k-1}g_2\otimes x_2 z_{2k-1}}_{C}
 + \underbrace{\sum _{k=0}^n \binom n k\mu_{k, n} g_1 x_{21}^{n-k}g_1^{2k}g_2\otimes x_2 z_{2k}}_{D}
\displaybreak[0] \\
-q_{12} \Big( \underbrace{\sum _{k=1}^n k\binom n k \mu_{k,n} \,
x_1x_{21}^{n-k}g_1^{2k-1}g_2 x_2\otimes z_{2k-1}}_{E}
\displaybreak[0]\\
+ \underbrace{\sum _{k=0}^n \binom n k\mu_{k, n} x_{21}^{n-k}g_1^{2k}g_2 x_2 \otimes z_{2k}}_{F}
\displaybreak[0]\\
+ \underbrace{\sum _{k=1}^n k\binom n k \mu_{k,n}
x_1x_{21}^{n-k}g_1^{2k-1}g_2g_1\otimes z_{2k-1} x_2}_{G}
\displaybreak[0]\\
+ \underbrace{\sum _{k=0}^n \binom n k\mu_{k, n} x_{21}^{n-k}g_1^{2k}g_2g_1\otimes z_{2k} x_2}_{H}\Big).
\end{multline*}

Now $A -q_{12} E = \sum _{k=1}^n k\binom n k \mu_{k,n}\, I \otimes z_{2k-1}$ where

\begin{multline*}\allowdisplaybreaks
I \overset{\eqref{eq:x12}, \eqref{eq:braiding-block-point}}= x_{21}^{n + 1 -k}g_1^{2k-1}g_2 -
x_1 x_2 x_{21}^{n-k}g_1^{2k-1}g_2 - x_1x_{21}^{n-k} g_1^{2k-1} (x_2 + a x_1) g_2
\\
 \overset{\star} = \left(x_{21}^{n + 1 -k} -  x_1 x_2 x_{21}^{n-k} - x_1x_{21}^{n-k} \left(-x_2 + (2k - 1 - a) x_1 \right)  \right) g_1^{2k-1}g_2
\\
 \overset{\eqref{eq:rels-B(V(-1,2))-dos}, \eqref{eq:-1block+point-bis}} = \left(x_{21}^{n + 1 -k} -  x_1 x_2 x_{21}^{n-k} + x_1 (x_2 + (k - n) x_1) x_{21}^{n-k} \right) g_1^{2k-1}g_2
\\
 \overset{\eqref{eq:rels-B(V(-1,2))-1}}= x_{21}^{n + 1 -k} g_1^{2k-1}g_2.
\end{multline*}

Here and below, we use in $\star$ the identity $g_1^j x_2 = (-1)^j(x_2  - jx_1)g_1^j$.

Now  $B -q_{12} F =  \displaystyle\sum _{k=0}^n \binom n k\mu_{k, n} J \otimes z_{2k}$, where

\begin{align*}
J &=  x_2 x_{21}^{n-k}g_1^{2k}g_2 -   x_{21}^{n-k}g_1^{2k} (x_2 + a x_1) g_2
\\&\overset{\star}= \left(x_2 x_{21}^{n-k} -   x_{21}^{n-k} \left(x_2 - (2k - a) x_1 \right)
\right) g_1^{2k}g_2
\\
& \overset{\eqref{eq:rels-B(V(-1,2))-dos}, \eqref{eq:-1block+point-bis}} =  (n + k - a) x_1 x_{21}^{n-k} g_1^{2k}g_2.
\end{align*}

Also $\displaystyle C -q_{12} G = \sum _{k=1}^n k\binom n k \mu_{k,n} K$ where
\begin{align*}
K &= g_1 x_1x_{21}^{n-k}g_1^{2k-1}g_2\otimes x_2 z_{2k-1}
-q_{12} x_1x_{21}^{n-k}g_1^{2k-1}g_2g_1\otimes z_{2k-1} x_2
\\
&= - x_1 g_1 x_{21}^{n-k}g_1^{2k-1}g_2\otimes (z_{2k} - q_{12} z_{2k-1}x_2 ) -q_{12} x_1x_{21}^{n-k}g_1^{2k}g_2\otimes z_{2k-1} x_2
\\
& \overset{\ast}= - x_1  x_{21}^{n-k}g_1^{2k}g_2\otimes z_{2k},
\end{align*}
where in $\ast$ we use  $g_1 x_{21} = x_{21}g_1$. Next $\displaystyle D -q_{12} H
= \sum _{k=0}^n \binom n k\mu_{k, n} L$, where
\begin{align*}
L &= x_{21}^{n-k}g_1^{2k + 1}g_2\otimes (z_{2k + 1} + q_{12} z_{2k}x_2 ) -q_{12} x_{21}^{n-k}g_1^{2k + 1}g_2 \otimes z_{2k} x_2
\\ &= x_{21}^{n-k}g_1^{2k + 1}g_2\otimes z_{2k + 1}.
\end{align*}

Putting together the previous computations,
\begin{align*}
(B + C) - q_{12} (F + G) & = \sum _{k=0}^n \binom n k\mu_{k, n} (n - a) x_1 x_{21}^{n-k} g_1^{2k}g_2 \otimes z_{2k};
\\
(A + D) - q_{12} (E + H) & = \sum _{k=0}^{n-1} (k + 1)\binom{n}{k + 1} \mu_{k+ 1,n}\, x_{21}^{n  -k} g_1^{2k+1}g_2 \otimes z_{2k+1}
\\
& \qquad + \sum _{k=0}^n \binom n k\mu_{k, n} x_{21}^{n-k}g_1^{2k + 1}g_2\otimes z_{2k + 1}
\\& = \sum _{k=0}^{n} (k + 1)\binom{n}{k + 1}\binom n k\mu_{k + 1, n + 1}
\end{align*}
because
\begin{align*}
(k + 1)\binom{n}{k + 1} \mu_{k+ 1,n} + \binom n k\mu_{k, n} = \tfrac{n! \mu_{k+ 1,n}}{k! (n-k-1)!}\left(1 + \tfrac{k-a}{n-k} \right)
= \binom n k\mu_{k + 1, n + 1}.
\end{align*}

This gives \eqref{eq:coact-zn-odd}.
The proof of \eqref{eq:coact-zn-even} for $\delta(z_{2n+2})$ is similar.
\end{proof}

\smallbreak
We are ready for the main result of this Subsection. Recall that the interaction is weak.
We also assume $\ghost \neq 0$, otherwise $ \NA (V) \simeq  \NA (V_1) \underline{\otimes}  \NA (V_2)$.

\smallbreak
\begin{theorem} \label{thm:pm1bp}
$\GK \NA (V)$ is finite if and only if $\epsilon$, $q_{22}$ and $\ghost$ are as in Table \ref{tab:weak-discrete};
in such case, $\GK \NA (V) = \GK K + 2$.
\end{theorem}

\begin{table}[h]
\caption{Nichols algebras with finite $\GK$, weak interaction}\label{tab:weak-discrete}
\begin{center}
\begin{tabular}{|c|c|c|c|}
\hline $\epsilon$  & $q_{22}$  & $\ghost$ & $\GK K$  \\
\hline
$1$  & $1$  & discrete   &   $\ghost + 1$
\\\cline{2-4}
&  $-1$  & discrete  &    $0$
\\ \cline{2-4}
&  $\in \G'_3$  &  1 &  $0$
\\
\hline
$-1$ & $1$  & discrete &   $\ghost + 1$
\\\cline{2-4}
&  $-1$  &  discrete &   $\ghost$
\\ \hline
\end{tabular}
\end{center}
\end{table}

\begin{proof} By Lemma \ref{lemma:weak-not-discrete}, we may assume that the ghost is discrete.
We claim that the braided vector space  $K^1$ is of diagonal type
with  braiding matrix
\begin{align*}
(p_{ij})_{0\le i,j\le 2\vert a \vert} &= (\epsilon^{ij}q_{12}^iq_{21}^jq_{22})_{0\le i,j\le 2\vert a \vert}.
\end{align*}
Hence, the corresponding generalized  Dynkin diagram  has labels
\begin{align*}
p_{ii}&= \epsilon^iq_{22},& p_{ij}p_{ji}&= q_{22}^2, & &i\neq j\in \{0\} \cup \I_{2\vert a \vert}.
\end{align*}

Indeed, by Remark \ref{lemma:K-basis} it is enough to compute
\begin{align*}
c(z_i \otimes z_j) &= g_1^ig_2 \cdot z_j \otimes z_i = \epsilon^{ij}q_{12}^iq_{21}^jq_{22} z_j \otimes z_i,
\end{align*}
by Lemmas \ref{le:zcoact} and \ref{le:-1bpz}, \eqref{eq:adx1-zn} and \eqref{eq:adx12-zn}. We proceed then case by case.

\begin{case} \label{case:1}
$q_{22}^2=1$.
\end{case}
Here the Dynkin diagram of $K^1$ is totally disconnected with
vertices $i\in \I_{2\vert a \vert}$ labeled with $\epsilon^iq_{22}$.
The vertices with label $1$, respectively $-1$, contribute with $1$, respectively $0$, to $\GK \NA (K^1)$.

\begin{case} \label{case:2}
$\epsilon = 1$, $q_{22}\in\G_3'$, $\ghost = 1$.
\end{case}
The Dynkin diagram is of Cartan type $A_2$, so $\NA (K^1)$ is finite-dimensional.

\begin{case} \label{case:3}
$\epsilon = 1$, $q_{22}^2\ne 1$, $q_{22}\notin\G_3'$.
\end{case}

Since $\ghost \ge 1$, the Dynkin diagram  has at
least $2$ vertices. The Dynkin subdiagram corresponding to the vertices $0$ and $1$
has labels $q_{22}$ on the vertices and $q_{22}^2$ on the edge.
If $q_{22} \notin \G_{\infty}$, then $K^1$ does not admit all reflections; hence $\GK \NA (K^1)=\infty $.
 If $q_{22}\in \G_M'$ with $M\ge 4$, then $K^1$ is of Cartan type with Cartan matrix
$ \begin{pmatrix} 2 & 2-M \\ 2-M & 2\end{pmatrix}$.
Then $\GK \NA (K^1)=\infty $ by Theorem \ref{thm:nichols-diagonal-finite-gkd}.

\begin{case} \label{case:4}
$\epsilon = 1$, $q_{22}\in\G_3'$, $\ghost > 1$.
\end{case}
The Dynkin subdiagram of $K^1$ corresponding
to the vertices $0$, $1$ and $2$ has labels $q_{22}$ on the vertices, respectively,
and $q_{22}^{-1}$ on the edge between them. This Dynkin diagram is of Cartan type with
affine Cartan matrix $A_2^{(1)}$, and $\GK \NA (K^1)=\infty $ by Theorem \ref{thm:nichols-diagonal-finite-gkd}.

\begin{case} \label{step:-1bp} $\epsilon = -1$,  $q_{22}^2\neq 1$.
\end{case}

Since $a\neq 0$, the Dynkin diagram of $K^1$
has at least $3$ vertices. The Dynkin subdiagram corresponding to the vertices $0$ and $2$
has labels $q_{22}$ on the vertices and $q_{22}^2$ on the edge. If $q_{22}$ is
not a root of $1$, then $K^1$ does not admit all reflections and then $\GK
\NA (K^1)=\infty $. If $q_{22}\in \G_N'$ with $N\ge 4$, then $K^1$ is of
Cartan type with Cartan matrix
$$ \begin{pmatrix} 2 & 2-N \\ 2-N & 2\end{pmatrix}. $$
Then $\GK \NA (K^1)=\infty $ by Theorem \ref{thm:nichols-diagonal-finite-gkd}.
Finally, if $q_{22}\in \G_3'$ then the Dynkin subdiagram of $K^1$
corresponding to the vertices $0$ and $1$ has labels $q_{22}$ and $-q_{22}$ on
the vertices, respectively, and $q_{22}^{-1}$ on the edge between them.
This Dynkin diagram is of Cartan type with affine Cartan matrix
$ \begin{pmatrix} 2 & -1 \\ -4 & 2\end{pmatrix}$.
Therefore $\GK \NA (K^1)=\infty $ by Theorem \ref{thm:nichols-diagonal-finite-gkd}.

The last claim follows from the decomposition
$ \NA (V)\simeq \NA (K^1) \#\NA (V_1)$, Lemma \ref{lemma:GKdim-smashproduct}
and Proposition~\ref{pr:1block}, see Remark \ref{remark:GK-block-point-K}.
\end{proof}

\subsection{The Nichols algebras with finite $\GK$}\label{subsection:point-block-presentation}
Here we describe a presentation by generators and relations and exhibit an explicit PBW basis of the Nichols algebras in
Theorem \ref{thm:pm1bp}. We denote the braided vector space with braiding \eqref{eq:braiding-block-point} by
\begin{align*}
	&\lstr(q_{22}, \ghost),& &\text{if the interaction is weak, }& &\epsilon = 1;
	\\
	&\lstr_{-}(q_{22}, \ghost),& &\text{if the interaction is weak, } & &\epsilon = -1;
	\\
	&\cyc_1,& &\text{if the interaction is mild,}& &\epsilon = q_{22} = -1, \quad \ghost = 1. \label{page:cyclope}
\end{align*}

Recall the relations of the Jordan and super Jordan planes:
\begin{align*}
	\tag{\ref{eq:rels B(V(1,2))}} &x_2x_1-x_1x_2+\frac{1}{2}x_1^2,
	\\
	\tag{\ref{eq:rels-B(V(-1,2))-1}} &x_1^2,
	\\
	\tag{\ref{eq:rels-B(V(-1,2))-2}} & x_2x_{21}- x_{21}x_2 - x_1x_{21}.
\end{align*}

\begin{remark}\label{rem:xk qcommutes with z_k}
	Assume that the interaction is weak. Let
	\begin{align*}
		y_{2k} &=x_{21}^k, &  y_{2k+1} &= x_1x_{21}^k, & k & \in \N_0
	\end{align*}
	By Lemma \ref{le:-1bpz}
	\begin{align}\label{eq:yn zt commute}
		\partial_3(z_t)&=\mu_t y_t, & z_t y_n &= \epsilon^{nt}q_{21}^n y_n z_t, & &t, n \in\N_0.
	\end{align}
\end{remark}

\begin{lemma}\label{lemma:relations L(pm1,pm1,G)}
	Assume that $\epsilon^2=q_{22}^2=1$. In $\cB(\lstr(q_{22}, \ghost))$, or correspondingly
	$\cB_{-}(\lstr(q_{22}, \ghost))$
	\begin{align}
		\label{eq:q-serre}
		z_{|2a|+1}&=0, \\
		z_t z_{t+1} &= q_{21}q_{22} z_{t+1} z_t & t\in\N_0,& \, t<|2a|, \label{eq:zt zt+1 qcommute}\\
		z_t^2&=0 & t\in\N_0&, \,  \epsilon^tq_{22}=-1. \label{eq:zt square is 0} \\
		\partial_3(z_t^{n+1})&= \mu_{t}q_{21}^{nt}q_{22}^n n \, y_t z_t^n, & n,t\in\N_0,& \,  \epsilon^tq_{22}=1. \label{eq:zt n partial 3}
	\end{align}
\end{lemma}

\begin{proof} Lemma \ref{lemma:derivations-zn} contains \eqref{eq:q-serre}.
	Using Lemmas \ref{le:-1bpz}, \ref{lemma:derivations-zn}, and Remark \ref{rem:xk qcommutes with z_k}, we compute
	\begin{align*}
		\partial_3 & (z_{k} z_{k+1}  -q_{21}q_{22} z_{k+1} z_{k})= \mu_{k} y_k(g_2\cdot z_{k+1}) + \mu_{k+1} z_k y_{k+1} \\
		&  -q_{21}q_{22} \big(  \mu_{k+1} y_{k+1} (g_2\cdot z_{k}) + \mu_{k} z_{k+1} y_{k} \big) = \mu_{k+1}q_{21}^{k+1}(1- q_{22}^2)  y_{k+1} z_{k}=0,
	\end{align*}
	so \eqref{eq:zt zt+1 qcommute} holds for all $t\in\N_0$ by the same result.
	
	Now assume that $t\in\N_0$ is such that $\epsilon^tq_{22}=-1$.
	\begin{align*}
		\partial_3(z_t^2)=\mu_t y_t (g_2\cdot z_t)+\mu_t z_t y_t= \mu_t q_{21}^t (q_{22}+\epsilon^t) y_tz_t=0,
	\end{align*}
	so \eqref{eq:zt square is 0} also follows.
	
	Assume that $t\in\N_0$ is such that $\epsilon^tq_{22}=1$. To prove \eqref{eq:zt n partial 3} we use induction on $n$. The case $n=0$ follows by Lemma \ref{lemma:derivations-zn}. Now assume it holds for $n$ and compute
	\begin{align*}
		\partial_3(z_t^{n+2}) &=\mu_{t}q_{21}^{nt}q_{22}^n n \, y_tz_t^n (g_2\cdot z_t)+\mu_t z_t^{n+1} y_t \\
		& = \mu_{t}q_{21}^{(n+1)t}q_{22}^{n+1} (n+\epsilon^{t(n+1)}q_{22}^{n+1}) \, y_tz_t^n.
	\end{align*}
	Then \eqref{eq:zt n partial 3} holds for all $n\in\N_0$.
\end{proof}

\begin{lemma}\label{lemma:x1, x12 commute with zt}
	Let $\cB$ be a quotient algebra of $T(V)$. Assume that $x_1x_3=q_{12}x_3x_1$, and either
	\begin{enumerate}
		\item[(a)] \eqref{eq:rels B(V(1,2))}, or else
		\item[(b)] \eqref{eq:rels-B(V(-1,2))-2}, $x_{21}x_3 =q_{12}^2 x_3x_{21}$
	\end{enumerate}
	hold in $\cB$. Then for all $n\in\N_0$,
	\begin{align}\label{eq:x1, x12 commute with zt}
	x_1z_n &= \epsilon^n q_{12}z_nx_1 \\
	x_{21}z_n &= q_{12}^2z_nx_{21}.
	\end{align}	
\end{lemma}

\begin{proof}
	By hypothesis, $x_1z_0 =q_{12}z_0x_1$ since $z_0=x_3$. Now $x_{21}z_0 = q_{12}^2z_0x_{21}$ holds in \emph{(a)}
	by the previous relation and \eqref{eq:rels B(V(1,2))}, while it holds in \emph{(b)} by hypothesis.
	The inductive step follows as in the proof of Lemma \ref{le:-1bpz}.
\end{proof}

\begin{lemma}\label{lemma:zt zk}
	Let $\cB$ be a quotient algebra of $T(V)$, $\epsilon^2=q_{22}^2=1$.
	
	\medbreak
	\noindent \emph{\vi} Assume that \eqref{eq:zt zt+1 qcommute} and \eqref{eq:zt square is 0} hold in $\cB$.
	Then for $0\le t<k\le 2|a|$,
	\begin{align}\label{eq:bracket ztzk}
		z_tz_k&=\epsilon^{tk} q_{21}^{k-t}q_{22} z_kz_t.
	\end{align}

	\medbreak
	\noindent \emph{\vii} Assume that $z_t^2=0$ in $\cB$ for $t\in\N_0$ such that $\epsilon^t q_{22}=-1$. Then
	$z_tz_{t+1}=q_{21}q_{22} z_{t+1}z_t$ in $\cB$.
\end{lemma}

In other words, \vii says that \eqref{eq:zt square is 0} for a specific $t$ implies \eqref{eq:zt zt+1 qcommute} for $t$.

\begin{proof}
	\vi
	We argue by induction on $n=t+k$. For $n=1$, we have $t=0$, $k=1$, and $z_0z_1=q_{21}q_{22}z_1z_0$ by hypothesis.
	Now assume \eqref{eq:bracket ztzk} holds for $n\geq 1$.
	If $n$ is even and $k-1=t=\frac{n}{2}$, then $z_tz_{t+1}=q_{21}q_{22} z_{t+1}z_t$ by hypothesis.
	If $n$ is odd and $t=k-2=\frac{n-1}{2}$, then
	\begin{align*}
		0 & = \ad_c x_2 \left( z_tz_{t+1} - q_{21}q_{22} z_{t+1}z_t \right) = \ad_c x_2 (z_tz_{t+1}-q_{21}q_{22} z_{t+1}z_t) \\
		& = z_{t+1}^2 +\epsilon^tq_{12} z_tz_{t+2} - q_{21}q_{22} (z_{t+2}z_t+\epsilon^{t+1}q_{12} z_{t+1}^2) \\
		& = (1-\epsilon^{t+1}q_{22})z_{t+1}^2+ \epsilon^{t+1}q_{12} (z_tz_{t+2}-\epsilon^t q_{21}^2q_{22} z_{t+2}z_t).
	\end{align*}
	That is, $z_tz_{t+2}-\epsilon^t q_{21}^2q_{22} z_{t+2}z_t = q_{21}(q_{22}-\epsilon^{t+1})z_{t+1}^2$. If $\epsilon^{t+1} q_{22}=-1$,
then $z_{t+1}^2$ by hypothesis; otherwise $\epsilon^{t+1}=q_{22}$. Hence $z_tz_{t+2}=\epsilon^t q_{21}^2q_{22} z_{t+2}z_t$.
	
Finally, if $k-t>2$,
\begin{multline*}
0 = \ad_c x_2 \left( z_tz_k-\epsilon^{tk} q_{21}^{k-t}q_{22} z_kz_t \right)  \\
= z_{t+1}z_k-\epsilon^{(t+1)k} q_{21}^{k-t-1}q_{22} z_kz_{t+1}
+ \epsilon^tq_{12}\left( z_tz_{k+1}-\epsilon^{t(k+1)} q_{21}^{k+1-t}q_{22} z_{k+1} z_t \right),
\end{multline*}
and the proof follows recursively.
	
	\medbreak
	\noindent\vii Using the definition of $z_{t+1}$,
	\begin{align*}
		z_t z_{t+1} -q_{21}q_{22} z_{t+1}z_t &= z_t(x_2z_t-\epsilon^tq_{12} z_tx_2)- q_{21}q_{22} (x_2z_t-\epsilon^tq_{12} z_tx_2) z_t \\
		&= (1+\epsilon^tq_{22}) z_t x_2 z_t =0.
	\end{align*}
\end{proof}

\subsubsection{The Nichols algebra $\cB(\lstr( 1, \ghost))$}\label{subsubsection:lstr-11disc}
Recall that $ z_n = (ad_c x_2)^n x_3$.

\begin{prop} \label{pr:lstr-11disc} Let $\ghost \in \N$. The algebra
	$\cB(\lstr( 1, \ghost))$ is presented by generators $x_1,x_2, x_3$ and relations \eqref{eq:rels B(V(1,2))},
	\begin{align}
		x_1x_3&=q_{12} \, x_3x_1,  \label{eq:lstr-rels&11disc-1} \\
		z_{1+\ghost}&=0,  \label{eq:lstr-rels&11disc-qserre} \\
		z_tz_{t+1}&=q_{12}^{-1} \, z_{t+1}z_t, & 0\le & t < \ghost. \label{eq:lstr-rels&11disc-2}
	\end{align}
	$\cB(\lstr( 1, \ghost))$ has a PBW-basis
	\begin{align*}
		B=\{ x_1^{m_1} x_2^{m_2} z_{\ghost}^{n_{\ghost}} \dots z_1^{n_1} z_0^{n_0}: m_i, n_j \in\N_0\};
	\end{align*}
	hence $\GK \cB(\lstr( 1, \ghost)) = 3+\ghost$.
\end{prop}

\pf Relations \eqref{eq:lstr-rels&11disc-1}, \eqref{eq:lstr-rels&11disc-qserre} are 0 in $\cB(\lstr( 1, \ghost))$
being annihilated by $\partial_i$, $i=1,2,3$, and \eqref{eq:lstr-rels&11disc-2} holds by Lemma \ref{lemma:relations L(pm1,pm1,G)}.
Hence the quotient $\cBt$ of $T(V)$
by \eqref{eq:rels B(V(1,2))}, \eqref{eq:lstr-rels&11disc-1}, \eqref{eq:lstr-rels&11disc-qserre} and \eqref{eq:lstr-rels&11disc-2}
projects onto $\cB(\lstr( 1, \ghost))$.
Then \eqref{eq:bracket ztzk} holds in $\cBt$.

We claim that the subspace $I$ spanned by $B$ is a right ideal of $\cBt$. Indeed,
\begin{itemize}
	\item  $Ix_1\subseteq I$ follows by Lemma \ref{lemma:x1, x12 commute with zt},
	\item  $Ix_2\subseteq I$ since $z_t x_2=\epsilon^tq_{21}(x_2z_t-z_{t+1})$, so we use \eqref{eq:lstr-rels&11disc-qserre}, \eqref{eq:bracket ztzk},
\end{itemize}
and $I x_3\subseteq I$ by definition. Since $1\in I$, $\cBt$ is spanned by $B$.

To prove that $\cBt \simeq \cB(\lstr( 1, \ghost))$, it remains to show that
$B$ is linearly independent in $\cB(\lstr( 1, \ghost))$. For, suppose that there is a non-trivial linear combination $\mathtt{S}$
of elements of $B$ in $\cB(\lstr( 1, \ghost))$, say of minimal degree. Now
\begin{align*}
	\partial_1(x_1^{m_1} x_2^{m_2} z_{\ghost}^{n_{\ghost}} \dots z_1^{n_1} z_0^{n_0})&= m_1 \, q_{12}^{\sum n_i} \, x_1^{m_1-1} x_2^{m_2}
	z_{\ghost}^{n_{\ghost}} \dots z_1^{n_1} z_0^{n_0}, \\
	\partial_2(x_1^{m_1} x_2^{m_2} z_{\ghost}^{n_{\ghost}} \dots z_1^{n_1} z_0^{n_0})&= m_2 \, q_{12}^{\sum n_i} \, x_1^{m_1} x_2^{m_2-1}
	z_{\ghost}^{n_{\ghost}} \dots z_1^{n_1} z_0^{n_0},
\end{align*}
since $\partial_1$, $\partial_2$ are skew derivations, so we apply Lemma \ref{le:-1bpz} and $\partial_2(z_t)=0$.
Then such linear combination does not have terms with $m_1$ or $m_2$ greater than 0. Let $k$ be maximal such that
$z_{k}^{n_{k}} \dots z_1^{n_1} z_0^{n_0}$ has non-zero coefficient in $\mathtt{S}$ for some $k\geq 1$, and for
such $k$ fix the maximal $n_k$. By \eqref{eq:zt n partial 3}, $y_k z_{k}^{n_{k}-1} \dots z_1^{n_1} z_0^{n_0}$
has non-zero coefficient in $\partial_3(\mathtt{S})$, and $\partial_3(\mathtt{S})$ is also a non-trivial linear
combination of elements of $B$, a contradiction.
Then $B$ is a basis of $\cB(\lstr( 1, \ghost))$ and $\cBt=\cB(\lstr( 1, \ghost))$.
The computation of $\GK$ follows from the Hilbert series at once.
\epf

\begin{prop} \label{pr:lstr-11disc-domain} Let $\ghost \in \N$. The algebra
	$\cB(\lstr( 1, \ghost))$ is a domain.
\end{prop}

\pf
Consider the algebra $\Bg$ generated by $(X_i)_{i\in \I_3}$, $(Z_t)_{t\in \I_{\ghost}}$ with the following relations (where $Z_0 = X_3$ for convenience):
\eqref{eq:rels B(V(1,2))},
\eqref{eq:lstr-rels&11disc-1}, \eqref{eq:lstr-rels&11disc-2}  (with $X_i$ in the place of $x_i$ and $Z_t$ in the place of $z_t$),
$X_2Z_t - q_{12}Z_t X_2 = Z_{t+1}$, $0 \le  t < \ghost$, $X_2 Z_\ghost - q_{12} Z_\ghost X_2 = 0$.
Clearly the assignments $X_i \leftrightarrow x_i$ and $Z_t \leftrightarrow z_t$ provide an algebra isomorphism
$\Bg \simeq \cB(\lstr( 1, \ghost))$; in particular \eqref{eq:x1, x12 commute with zt} and \eqref{eq:bracket ztzk} hold in $\Bg$.
Consider the filtration of $\Bg$ where $\deg X_1 =0$ and all the other defining generators having degree 1. We claim that
$\gr \Bg$ is presented by $(\overline{X}_i)_{i\in \I_3}$, $(\overline{Z}_t)_{t\in \I_{\ghost}}$ with the relations
\begin{align}
\label{eq:grBg-1}
\overline{X}_2\overline{X}_1 - \overline{X}_1 \overline{X}_2 &=0, \\
\label{eq:grBg-2}
\overline{X}_1\overline{Z}_t - q_{12} \overline{Z}_t \overline{X}_1 &=0, & 0\le & t \le  \ghost,
\\
\label{eq:grBg-5}
\overline{X}_2\overline{Z}_t - q_{12}\overline{Z}_t \overline{X}_2 &= 0, & 0\le & t \le  \ghost,
\\
\label{eq:grBg-6}
\overline{Z}_t\overline{Z}_k-  q_{21}^{k-t} \overline{Z}_k\overline{Z}_t &= 0, & 0\le & t<k\le \ghost,\end{align}
where $\overline{Z}_0 = \overline{X}_3$ for convenience.
Indeed, the algebra $\widetilde \Bg$ with the mentioned presentation admits a surjective algebra homomorphism onto $\gr \Bg$.
But $\widetilde \Bg$ is a quantum polynomial ring, hence it has a PBW-basis analogous to $B$ above and the claim follows.
Now $\widetilde \Bg$ is a domain, hence so is $\Bg \simeq \cB(\lstr( 1, \ghost))$.
\epf

\subsubsection{The Nichols algebra $\cB(\lstr( -1, \ghost))$}\label{subsubsection:lstr-1-1disc}

\begin{prop} \label{pr:lstr1-1disc} Let $\ghost \in \N$. The algebra
	$\cB(\lstr( -1, \ghost))$ is presented by generators $x_1,x_2, x_3$ and relations  \eqref{eq:rels B(V(1,2))},
	\eqref{eq:lstr-rels&11disc-1}, \eqref{eq:lstr-rels&11disc-qserre} and
	\begin{align}\label{eq:lstr-rels&1-1disc}
		z_t^2&=0, & 0\le& t\le \ghost.
	\end{align}
	The set
	\begin{align*}
		B=\{ x_1^{m_1} x_2^{m_2} z_{\ghost}^{n_{\ghost}} \dots z_1^{n_1} z_0^{n_0}: n_i \in\{0,1\}, m_j \in\N_0 \}
	\end{align*}
	is a basis of $\cB(\lstr( -1, \ghost))$ and $\GK \cB(\lstr( -1, \ghost)) = 2$.
\end{prop}

\pf
Relations \eqref{eq:lstr-rels&11disc-1}, \eqref{eq:lstr-rels&11disc-qserre} are 0 in $\cB(\lstr( -1, \ghost))$
being annihilated by $\partial_i$, $i=1,2,3$, and \eqref{eq:lstr-rels&1-1disc} holds by Lemma \ref{lemma:relations L(pm1,pm1,G)}.
Hence the quotient $\cBt$ of $T(V)$
by \eqref{eq:rels B(V(1,2))}, \eqref{eq:lstr-rels&11disc-1}, \eqref{eq:lstr-rels&11disc-qserre} and \eqref{eq:lstr-rels&1-1disc}
projects onto $\cB(\lstr( -1, \ghost))$.
Then Lemma \ref{lemma:zt zk} \vii holds, so also \eqref{eq:bracket ztzk} holds in $\cBt$ by Lemma \ref{lemma:zt zk}\vi.

We claim that the subspace $I$ spanned by $B$ is a right ideal of $\cBt$. Indeed,
$Ix_1\subseteq I$ follows by Lemma \ref{lemma:x1, x12 commute with zt}, $Ix_2\subseteq I$ by \eqref{eq:lstr-rels&11disc-qserre},
\eqref{eq:bracket ztzk} and \eqref{eq:lstr-rels&1-1disc}, and $I x_3\subseteq I$ by definition. Since $1\in I$, $\cBt$ is spanned by $B$.

To prove that $\cBt \simeq \cB(\lstr( -1, \ghost))$, it remains to show that
$B$ is linearly independent in $\cB(\lstr( -1, \ghost))$. For, suppose that there is a non-trivial linear combination $\mathtt{S}$
of elements of $B$ in $\cB(\lstr( -1, \ghost))$, say of minimal degree.
Then such linear combination does not have terms with $m_1$ or $m_2$ greater than 0 as in Proposition \ref{pr:lstr-11disc}.
Let $k$ be maximal such that $z_{k}^{n_{k}} \dots z_1^{n_1} z_0^{n_0}$ has non-zero coefficient in $\mathtt{S}$ for some $k\geq 1$, and for
such $k$ fix the maximal $n_k$. By \eqref{eq:zt n partial 3}, $y_k z_{k}^{n_{k}-1} \dots z_1^{n_1} z_0^{n_0}$
has non-zero coefficient in $\partial_3(\mathtt{S})$, and $\partial_3(\mathtt{S})$ is also a non-trivial linear
combination of elements of $B$, a contradiction.
Then $B$ is a basis of $\cB(\lstr( -1, \ghost))$ and $\cBt=\cB(\lstr( -1, \ghost))$.
The computation of $\GK$ follows from the Hilbert series at once.
\epf

\subsubsection{The Nichols algebra $\cB(\lstr_{-}( 1, \ghost))$}\label{subsubsection:lstr--11disc}

\begin{prop} \label{pr:lstr--11disc} Let $\ghost \in \N$. The algebra
	$\cB(\lstr_{-}( 1, \ghost))$ is presented by generators $x_1,x_2, x_3$ and relations \eqref{eq:rels-B(V(-1,2))-1},
	\eqref{eq:rels-B(V(-1,2))-2}, \eqref{eq:lstr-rels&11disc-1} and
	\begin{align}
		z_{1+2\ghost}&=0, \label{eq:lstr-rels&-11disc-1} \\
		x_{21}z_0& = q_{12}^2 \, z_0x_{21},  \label{eq:lstr-rels&-11disc-2} \\
		z_{2k+1}^2&=0, &  0\le & k < \ghost, \label{eq:lstr-rels&-11disc-3} \\
		z_{2k} z_{2k+1}&= q_{12}^{-1} \, z_{2k+1}z_{2k}, & 0\le & k < \ghost. \label{eq:lstr-rels&-11disc-4}
	\end{align}
	The set
	\begin{align*}
		B=\{ x_1^{m_1} x_{21}^{m_2} x_2^{m_3} z_{2\ghost}^{n_{2\ghost}} \dots z_1^{n_1} z_0^{n_0}: m_1, n_{2k+1} \in\{0,1\}, m_2, m_3, n_{2k} \in\N_0 \}
	\end{align*}
	is a basis of $\cB(\lstr_{-}( 1, \ghost))$ and $\GK \cB(\lstr_{-}( 1, \ghost)) = \ghost+3$.
\end{prop}

\pf
Relations \eqref{eq:rels-B(V(-1,2))-1}, \eqref{eq:rels-B(V(-1,2))-2}, \eqref{eq:lstr-rels&11disc-1} and
\eqref{eq:lstr-rels&-11disc-1} are 0 in $\cB(\lstr( -1, \ghost))$
being annihilated by $\partial_i$, $i=1,2,3$, and \eqref{eq:lstr-rels&-11disc-2}, \eqref{eq:lstr-rels&-11disc-3},
\eqref{eq:lstr-rels&-11disc-4} hold by Lemma \ref{lemma:relations L(pm1,pm1,G)}.
Hence the quotient $\cBt$ of $T(V)$ by these relations projects onto $\cB(\lstr_{-}(1,\ghost))$.
Then $z_{t} z_{t+1}= q_{21}q_{22} \, z_{t+1}z_{t}$ for all $t$ by Lemma \ref{lemma:zt zk} \vii, so
\eqref{eq:bracket ztzk} holds in $\cBt$ by Lemma \ref{lemma:zt zk}\vi.

We claim that the subspace $I$ spanned by $B$ is a right ideal of $\cBt$. Indeed,
$Ix_1\subseteq I$ follows by Lemma \ref{lemma:x1, x12 commute with zt}, $Ix_2\subseteq I$ by \eqref{eq:lstr-rels&-11disc-1},
\eqref{eq:bracket ztzk} and \eqref{eq:lstr-rels&-11disc-3}, and $I x_3\subseteq I$ by definition. Since $1\in I$, $\cBt$ is spanned by $B$.

To prove that $\cBt \simeq \cB(\lstr_{-}( 1, \ghost))$, it remains to show that
$B$ is linearly independent in $\cB(\lstr_{-}( 1, \ghost))$. For, suppose that there is a non-trivial linear combination $\mathtt{S}$
of elements of $B$ in $\cB(\lstr_{-}(1,\ghost))$, say of minimal degree.
As in the proof of Proposition \ref{pr:-1block},
each vector $x_1^{m_1} x_{21}^{m_2} x_2^{m_3} z_{2\ghost}^{n_{2\ghost}} \dots z_1^{n_1} z_0^{n_0}$ in $\mathtt{S}$
with non-trivial coefficient satisfies $m_i=0$, $i=1,2,3$.
Let $k$ be maximal such that $z_{k}^{n_{k}} \dots z_1^{n_1} z_0^{n_0}$ has non-zero coefficient in $\mathtt{S}$ for some $k\geq 1$, and for
such $k$ fix the maximal $n_k$. By \eqref{eq:zt n partial 3} and \eqref{eq:yn zt commute}, $y_k z_{k}^{n_{k}-1} \dots z_1^{n_1} z_0^{n_0}$
has non-zero coefficient in $\partial_3(\mathtt{S})$, and $\partial_3(\mathtt{S})$ is also a non-trivial linear
combination of elements of $B$, a contradiction.
Then $B$ is a basis of $\cB(\lstr_{-}(1, \ghost))$ and $\cBt=\cB(\lstr_{-}(1, \ghost))$.
The computation of $\GK$ follows from the Hilbert series at once.
\epf

\subsubsection{The Nichols algebra $\cB(\lstr_{-}( -1, \ghost))$}\label{subsubsection:lstr--1-1disc}

\begin{prop} \label{pr:lstr-1-1disc} Let $\ghost \in \N$. The algebra
	$\cB(\lstr_{-}( -1, \ghost))$ is presented by generators $x_1,x_2, x_3$ and relations \eqref{eq:rels-B(V(-1,2))-1},
	\eqref{eq:rels-B(V(-1,2))-2}, \eqref{eq:lstr-rels&11disc-1}, \eqref{eq:lstr-rels&-11disc-1}, \eqref{eq:lstr-rels&-11disc-2}
	and
	\begin{align}
		z_{2k}^2&=0, &  0\le & k \le \ghost, \label{eq:lstr-rels&-1-1disc-1} \\
		z_{2k-1} z_{2k}&= -q_{12}^{-1} z_{2k}z_{2k-1}, & 0< & k \le \ghost. \label{eq:lstr-rels&-1-1disc-2}
	\end{align}
	The set
	\begin{align*}
		B=\{ x_1^{m_1} x_{21}^{m_2} x_2^{m_3} z_{2\ghost}^{n_{2\ghost}} \dots z_1^{n_1} z_0^{n_0}: m_1, n_{2k} \in\{0,1\}, m_2, m_3, n_{2k-1} \in\N_0 \}
	\end{align*}
	is a basis of $\cB(\lstr_{-}( -1, \ghost))$ and $\GK \cB(\lstr_{-}( -1, \ghost)) = \ghost+2$.
\end{prop}

\pf
Analogous to Proposition \ref{pr:lstr-11disc}.
\epf

\subsubsection{The Nichols algebra $\cB(\lstr( \omega, 1))$}\label{subsubsection:lstr-1omega1}

\begin{remark}
	As in the previous cases, \eqref{eq:lstr-rels&11disc-1} and
	\begin{align}\label{eq:lstr1omega1-qserre}
		z_2 & =0
	\end{align}
	hold in $\cB(\lstr( \omega, 1))$. As $q_{22}=\omega\in\G'_3$ we also have
	\begin{align}\label{eq:lstr1omega1-z0cube}
		z_0^3 & =0.
	\end{align}
\end{remark}

Let $z_{1,0} :=z_1z_0-q_{12}q_{22}z_0z_1$.

\begin{remark}
	The following equations hold in $\cB(\lstr( \omega, 1))$ by Lemma \ref{le:-1bpz}
	\begin{align}
		g_1 \cdot z_{1,0} &= q_{12}^2 z_{1,0}, & g_2 \cdot z_{1,0} &= q_{21}q_{22}^2 z_{1,0}, \label{eq:lstr1omega1-gi on z} \\
		\partial_1(z_{1,0}) &= \partial_2(z_{1,0}) =0, & \partial_3(z_{1,0}) &= (1-q_{22}^2) z_{1,0}. \label{eq:lstr1omega1-derivations z}
	\end{align}
\end{remark}

\begin{lemma}\label{lemma:lstr1omega1 - qcommutators=0}
	Let $\cB$ be a quotient algebra of $T(V)$. Assume that
	\eqref{eq:lstr-rels&11disc-1}, \eqref{eq:lstr1omega1-qserre} and \eqref{eq:lstr1omega1-z0cube}
	hold in $\cB$. Then the following relations also hold:
	\begin{align}
		z_1z_{1,0} &= q_{12}\omega ^2 z_{1,0} z_1, &z_{1,0}z_0 &= q_{12}\omega ^2 z_0 z_{1,0}, \label{eq:lstr1omega1 - qcommutators 1}
		\\ \label{eq:lstr1omega1 - qcommutators 2}
		x_2 z_{1,0} &= q_{12}^2 z_{1,0} x_2+q_{12}(1-\omega )z_1^2,  &
		x_1z_{1,0} &= q_{12}^2 z_{1,0} x_1.
	\end{align}
\end{lemma}

\begin{proof}
	The second equation in \eqref{eq:lstr1omega1 - qcommutators 1} follows because
	\begin{align*}
		z_{1,0}z_0 &-q_{12}\omega ^2 z_0 z_{1,0} = (z_1 x_3-q_{12}\omega  x_3 z_1)x_3 - q_{12}\omega  x_3 (z_1 x_3-q_{12}\omega  x_3 z_1) \\
		&= x_2x_3^3 -q_{12}(1+\omega +\omega ^2)\left( x_3x_2x_3^2 + q_{12} x_3^2x_2x_3 \right) - q_{12}^3x_3^3 x_2=0.
	\end{align*}
	
	As $x_2z_1=q_{12}z_1x_2$ by \eqref{eq:lstr1omega1-qserre}, we have that
	\begin{align*}
		x_2 z_{1,0} &- q_{12}^2 z_{1,0} x_2 = x_2 (z_1 x_3-q_{12}\omega  x_3 z_1)- q_{12}^2 (z_1 x_3-q_{12}\omega  x_3 z_1) x_2  \\
		&= q_{12} z_1 x_2x_3- q_{12}\omega  (z_1+q_{12}x_3x_2) z_1 - q_{12}^2 z_1 x_3x_2 +q_{12}^2\omega  x_3x_2 z_1  \\
		&=q_{12}(1-\omega )z_1^2.
	\end{align*}
	so the first equation in \eqref{eq:lstr1omega1 - qcommutators 2} holds.
	Using it and \eqref{eq:lstr1omega1 - qcommutators 1} we compute
	\begin{align*}\allowdisplaybreaks
		z_1z_{1,0} &- q_{12}\omega ^2 z_{1,0} z_1 = (x_2x_3-q_{12}x_3x_2) z_{1,0} - q_{12}\omega ^2 z_{1,0} z_1 \\
		&= q_{21}\omega x_2 z_{1,0} x_3- q_{12} x_3 \left( q_{12}^2 z_{1,0} x_2+q_{12}(1-\omega )z_1^2 \right) - q_{12}\omega ^2 z_{1,0} z_1 \\
		&= q_{21}\omega  \left( q_{12}^2 z_{1,0} x_2+q_{12}(1-\omega )z_1^2 \right) x_3 - q_{12}^2\omega  z_{1,0} x_3x_2 \\
		& \qquad - q_{12}^2(1-\omega )x_3z_1^2 - q_{12}\omega ^2 z_{1,0} z_1 \\
		&= q_{12} \omega  z_{1,0} z_1 +\omega(1-\omega ) \left( z_1 z_{1,0} + q_{12}q_{22} z_{1,0} z_1+q_{12}^2q_{22}^2 \right)   \\
		& \qquad - q_{12}^2(1-\omega )x_3z_1^2 - q_{12}\omega ^2 z_{1,0} z_1
		= \omega(1-\omega ) \left( z_1z_{1,0} - q_{12}\omega ^2 z_{1,0} z_1 \right)
	\end{align*}
	and the first equation in \eqref{eq:lstr1omega1 - qcommutators 1} holds.
	Finally notice that $x_1z_i=q_{12} z_ix_1$, $i=0,1$, as
	in Lemma \ref{lemma:x1, x12 commute with zt}, giving
	the second equation in \eqref{eq:lstr1omega1 - qcommutators 2}.
\end{proof}

\begin{lemma}\label{lemma:lstr1omega1 - z3=0}
	In $\cB(\lstr( \omega, 1))$,
	\begin{align}\label{eq:lstr-rels&1omega1}
		z_1^3&=z_{1,0}^3=0.
	\end{align}
\end{lemma}

\begin{proof}
	As $\partial_i(z_1^3)=\partial_i(z_{1,0}^3)=0$ for $i=1,2$, it remains to prove that $\partial_3$ annihilates them.
	For the first relation, we use Lemma \ref{le:-1bpz},
	\begin{align*}
		\partial_3(z_1^3) &= \mu_1 x_1 (g_2\cdot z_1^2) + \mu_1 z_1 x_1 (g_2\cdot z_1) + \mu_1  z_1^2 x_1 \\
		&= \mu_1 q_{21}^2 (1+\omega+\omega^2)x_1z_1^2 =0.
	\end{align*}
	By Lemma \ref{lemma:lstr1omega1 - qcommutators=0}, $z_1z_{1,0} = q_{12}\omega^2 z_{1,0} z_1$, and using \eqref{eq:lstr1omega1-derivations z},
	\begin{align*}
		\partial_3(z_{1,0}^3) &= (1-\omega^2) z_1 (g_2\cdot z_{1,0}^2) + (1-\omega^2) z_{1,0} z_1 (g_2\cdot z_{1,0}) + (1-\omega^2)  z_{1,0}^2 z_1 \\
		&= (1-\omega^2) q_{21}^2 (1+\omega+\omega^2)z_1z_{1,0}^2 =0;
	\end{align*}
	so  $z_{1,0}^3=0$.
\end{proof}

\begin{prop} \label{pr:lstr1omega1} Let $\omega \in \G'_3$. The algebra
	$\cB(\lstr( \omega, 1))$ is presented by generators $x_1,x_2, x_3$ and relations  \eqref{eq:rels B(V(1,2))},
	\eqref{eq:lstr-rels&11disc-1}, \eqref{eq:lstr1omega1-qserre}, \eqref{eq:lstr1omega1-z0cube} and \eqref{eq:lstr-rels&1omega1}.
	The set
	\begin{align*}
		B=\{ x_1^{m_1} x_2^{m_2} z_1^{n_1} z_{1,0}^{n_2} z_0^{n_3}: m_i\in\N_0, 0 \le n_j\le 2 \}
	\end{align*}
	is a basis of $\cB(\lstr( \omega, 1))$ and $\GK \cB(\lstr( \omega, 1)) = 2$.
\end{prop}

\pf
Relations \eqref{eq:rels B(V(1,2))}, \eqref{eq:lstr-rels&11disc-1}, \eqref{eq:lstr1omega1-qserre}, \eqref{eq:lstr1omega1-z0cube}
are 0 in $\cB(\lstr( 1, \ghost))$ being annihilated by $\partial_i$, $i=1,2,3$, and \eqref{eq:lstr-rels&1omega1} holds by Lemma
\ref{lemma:lstr1omega1 - z3=0}. Hence the quotient $\cBt$ of $T(V)$
by \eqref{eq:rels B(V(1,2))}, \eqref{eq:lstr-rels&11disc-1}, \eqref{eq:lstr1omega1-qserre}, \eqref{eq:lstr1omega1-z0cube} and
\eqref{eq:lstr-rels&1omega1} projects onto $\cB(\lstr( \omega, 1))$. We claim that the subspace $I$ spanned by $B$ is a right
ideal of $\cBt$. Indeed, $I x_3\subseteq I$ by definition and $Ix_1\subseteq I$, $Ix_2\subseteq I$ follow by
Lemma \ref{lemma:lstr1omega1 - qcommutators=0}. Since $1\in I$, $\cBt$ is spanned by $B$.

To prove that $\cBt \simeq \cB(\lstr( \omega, 1))$, it remains to show that
$B$ is linearly independent in $\cB(\lstr( \omega, 1))$. For, suppose that there is a non-trivial linear combination $\mathtt{S}$
of elements of $B$ in $\cB(\lstr( \omega, 1))$, say of minimal degree. As in the proof of Proposition \ref{pr:lstr-11disc},
each vector $x_1^{m_1} x_2^{m_2} z_1^{n_1} z^{n_2} z_0^{n_3}$ in $\mathtt{S}$ with non-trivial coefficient satisfies $m_1=m_2=0$. As
\begin{align*}
	\partial_3(z_1^{n_1} z^{n_2} z_0^{n_3}) &= (n_3)_{\omega}\, z_1^{n_1} z^{n_2} z_0^{n_3-1}
	+ q_{21}^{n_2-1}(1-\omega^2)\omega^{n_3}(n_2)_{\omega} z_1^{n_1+1} z^{n_2-1} z_0^{n_3} \\
	& \qquad + \mu_1q_{21}^{n_1+n_2-1}\omega^{2n_2+n_3}(n_1)_{\omega} x_1 z_1^{n_1-1} z^{n_2} z_0^{n_3},
\end{align*}
$\partial_3(\mathtt{S})$ is also a non-trivial linear combination of elements of $B$, a contradiction.
Then $B$ is a basis of $\cB(\lstr( \omega, 1))$ and $\cBt=\cB(\lstr( \omega, 1))$.
The computation of $\GK$ follows from the Hilbert series at once.
\epf

\subsection{Mild interaction}\label{subsection:mild}
\emph{We assume in this Subsection that $q_{12}q_{21}=-1$}. Then we may  assume that $\epsilon = -1$.
For, if $\epsilon = 1$, then $\GK \NA (V) = \infty$ by Lemma \ref{lemma:points-trivial-braiding}, since $x_1$, $x_3$ span a braided vector subspace of diagonal type with
braiding matrix $(q_{ij})_{1\le i,j\le 2}$.

We keep the notation in \S \ref{subsubsection:YD3-notation};
so, $ \NA (V)\simeq  K\# \NA (V_1)$,
 $K \simeq \cB(K^1)$,  $K^1=\ad_c \NA (V_1) (V_2)$.
 We first describe $K^1$ in detail. Let $z_n$ as in \eqref{eq:zn} and
\begin{align*}
f_n&=\ad_c x_1 (z_n),&
n&\in \N _0.
\end{align*}

\begin{lemma} \label{le:-K^1} The following hold in $\NA(V)$ for all $n\in\N_0$:
\begin{align}\label{eq:-1block+point-mild-Kgen}
K^1 = &\,\langle  z_m,f_m: \, m\in \N_0 \rangle, &
&\ad_cx_{21}(K^1) = 0,
\\ \label{eq:-1block+point-mild}
g_1\cdot z_{2n} &= q_{12}z_{2n},& g_1\cdot z_{2n+1} &= -q_{12}z_{2n+1}+q_{12}f_{2n},
\\\label{eq:-1block+point-mild-bis}
g_2\cdot z_{2n} &= q_{21}^{2n}q_{22}z_{2n},& g_2\cdot z_{2n+1} &= q_{21}^{2n+1}q_{22}(z_{2n+1}+af_{2n}),
\\\label{eq:-1block+point-mild-ter}
g_1\cdot f_n &= (-1)^{n+1}q_{12}f_n,& g_2\cdot f_n &=q_{21}^{n+1}q_{22}f_n,
\\\label{eq:-1block+point-mild-quator}
\ad_cx_1 &(f_n) = 0,& \ad_cx_2 &(f_n) = -f_{n+1}.
\end{align}
\end{lemma}

\begin{proof}
The claims on $g_1 \cdot z_0$, $g_2 \cdot z_0$ follow by definition. We record that
\begin{align*}
g_1 \cdot z_1 &= g_1 \cdot \ad_cx_2 (x_3) = \ad_c(-x_2+x_1) q_{12}x_3 = -q_{12}z_1+q_{12}f_0,
\\
g_1 \cdot f_0 &= g_1 \cdot \ad_cx_1 (x_3) = \ad_c(-x_1) q_{12}x_3 = -q_{12}f_0.
\end{align*}

We start by proving that $\ad_cx_{21} (x_3) = 0$. By definition,
\begin{align*}
\ad_cx_{21} (x_3) &= \ad_cx_1 (z_1) + \ad_cx_2 (f_0).
\end{align*}
If $i\in \I_2$, then $\partial_i( \ad_cx_{21}(x_3)) = 0$ because
$\partial_i(x_3)=\partial_i(z_1)=\partial_i(f_0)=0$.
Moreover,
\begin{align} \label{eq:partial3z1}
\partial_3(z_1)&=x_2-q_{12}g_2 \cdot x_2=x_2-q_{12}q_{21}(x_2+ax_1)=2x_2+ax_1,
\\ \label{eq:partial3f0}
\partial_3(f_0)&=x_1-q_{12}g_2 \cdot x_1=2x_1;
\end{align}
therefore
\begin{align*}
   \partial_3( \ad_cx_1 (z_1))&=
  x_1\partial_3(z_1)-q_{12}\partial_3(-z_1+f_0)g_2 \cdot x_1\\
  &=x_1(2x_2+ax_1)+(-2x_2-ax_1+2x_1)x_1 =2x_1x_2-2x_2x_1.
\end{align*}
On the other hand,
\begin{align*}
  \partial_3( \ad_cx_2 (f_0))&=
  x_2\partial_3(f_0)-\partial_3(-q_{12}f_0)g_2 \cdot x_2\\
  &=x_2\cdot 2x_1+q_{12}2x_1q_{21}(x_2+ax_1) =2x_2x_1-2x_1x_2.
\end{align*}
  This implies that $\ad_cx_{21} (x_3) = 0$.

\smallbreak
Since $K^1=\ad_c\NA (V_1)(x_3)$ and $x_1^ax_2^b x_{21}^c$, $a \in \{0, 1\}$,  $b,c\in \N_0$, is a basis of $\NA (V_1)$,
one has
$K^1 = \langle  z_n,f_n: \, n\in \N_0 \rangle$. Also,
$x_{21}\NA (V_1)\subseteq \NA (V_1)x_{21}$; hence
$$ \ad_c x_{21} (K^1) = \ad_c (x_{21}\NA (V_1))(x_3)
   \subseteq (\ad_c\NA (V_1)x_{21})(x_3)=0 $$
which proves \eqref{eq:-1block+point-mild-Kgen}.
We proceed with \eqref{eq:-1block+point-mild-quator}. Since $x_1^2=0$,  we have
\begin{align*}
\ad_cx_1 (f_n) &= (\ad_cx_1)^2 (z_n) = (\ad_cx_1^2) (z_n) = 0,& n&\in \N _0.
\end{align*}
Moreover, because of \eqref{eq:-1block+point-mild-Kgen} and the definitions, we have
\begin{align*}
  \ad_cx_2 (f_n) &=\ad_c (x_2x_1) (z_n)\\
  &=\ad_c (x_{21}-x_1x_2) (z_n)  =-\ad_cx_1 (z_{n+1}) = -f_{n+1}.
\end{align*}

It remains to prove \eqref{eq:-1block+point-mild}, \eqref{eq:-1block+point-mild-bis} and \eqref{eq:-1block+point-mild-ter}.
We proceed by induction on $n$. Assume that $g_1 \cdot z_{2n}$ and
$g_2 \cdot z_{2n}$ are known for a given $n\in \N_0$. Then
\begin{align*}
  g_1 \cdot f_{2n}&= g_1 \cdot (\ad_cx_1)(z_{2n})
	=(\ad_c (-x_1))(q_{12}z_{2n}) =-q_{12}f_{2n};
\\
  g_2 \cdot f_{2n} &= g_2 \cdot (\ad_cx_1)(z_{2n})
	=(\ad_c (q_{21}x_1))(q_{21}^{2n}q_{22}z_{2n})
  =q_{21}^{2n+1}q_{22}f_{2n}.
\end{align*}
Then we obtain that
\begin{align*}
  g_1 \cdot z_{2n+1} &= g_1 \cdot \ad_cx_2(z_{2n})
	=\ad_c (-x_2+x_1)(q_{12}z_{2n})
  =-q_{12}z_{2n+1}+q_{12}f_{2n};
\\
  g_2 \cdot z_{2n+1}&=g_2 \cdot \ad_cx_2 (z_{2n})
	= \ad_c q_{21}(x_2+ax_1)(q_{21}^{2n}q_{22}z_{2n})
\\ &= q_{21}^{2n+1}q_{22}(z_{2n+1}+af_{2n}).
\end{align*}
The inductive step for the action on $f_{2n+1}$, $z_{2n+1}$, is
similar; use \eqref{eq:-1block+point-mild-quator}.
\end{proof}

\begin{lemma} \label{le:fn}
If $n\in \N _0$, then
\begin{align*}
\partial_3(f_{2n})&=2\prod _{i=1}^n(i-a) x_{21}^nx_1,\\
\partial_3(f_{2n+1})&=2\prod _{i=1}^n(i-a) (2x_1x_{21}^nx_2-x_{21}^{n+1})
\end{align*}
In particular, if $N\in \N _0$ and $a\notin \I_N$ then $f_n\ne 0$ for all $n$, $0\le n \le 2N+1$.
\end{lemma}

\begin{proof}
We proceed by induction on $n$. For $n=0$ the claim holds by \eqref{eq:partial3f0}.
Let $n\ge 0$. Then $g_1 \cdot f_n=(-1)^{n+1}q_{12}f_n$ by \eqref{eq:-1block+point-mild-ter}
and $f_{n+1}=-(\ad_cx_2)f_n$ by \eqref{eq:-1block+point-mild-quator}.
Hence
\begin{align*}
\partial_3(f_{n+1})&=\partial_3(-x_2f_n+(-1)^{n+1}q_{12}f_nx_2)\\
  &=-x_2\partial_3(f_n)+(-1)^{n+1}q_{12}q_{21}\partial_3(f_n)(x_2+ax_1).
\end{align*}
Let now $n\in \N _0$ such that the claim holds for $\partial_3(f_{2n})$. Then
\eqref{eq:rels-B(V(-1,2))-2}
and \eqref{eq:rels-B(V(-1,2))-dos} imply that
\begin{align*}
-x_2x_{21}^nx_1+x_{21}^nx_1(x_2+ax_1)
&=-(x_{21}^nx_2+nx_1x_{21}^n)x_1+x_1x_{21}^nx_2\\
&=-x_{21}^n(x_{21}-x_1x_2)+x_1x_{21}^nx_2\\
&=2x_1x_{21}^nx_2-x_{21}^{n+1}.
\end{align*}
Thus the claim for $\partial_3(f_{2n+1})$ holds. Moreover,
\begin{align*}
-x_2&(2x_1x_{21}^nx_2-x_{21}^{n+1})-(2x_1x_{21}^nx_2-x_{21}^{n+1})(x_2+ax_1)\\
&=-2(x_{21}-x_1x_2)x_{21}^nx_2+(x_{21}^{n+1}x_2+(n+1)x_1x_{21}^{n+1})\\
&\quad-2x_1x_{21}^nx_2^2+x_{21}^{n+1}x_2-2ax_1x_{21}^n(x_{21}-x_1x_2)+ax_1x_{21}^{n+1}\\
&=-2x_{21}^{n+1}x_2+2x_1x_{21}^nx_2^2+x_{21}^{n+1}x_2+(n+1)x_1x_{21}^{n+1}\\
&\quad -2x_1x_{21}^nx_2^2+x_{21}^{n+1}x_2-ax_1x_{21}^{n+1}\\
&=(n+1-a)x_1x_{21}^{n+1}.
\end{align*}
{}From this the claim for $\partial_3(f_{2n+2})$ follows. The rest is clear.
\end{proof}

Here is the main result of this Subsection; recall that $q_{12}q_{21} = - 1$.

\smallbreak
\begin{theorem} \label{thm:pm1bp-mild}  $\GK \NA (V)$ is finite if and only if   $a = 1$ and $q_{22} = - 1$; in this case $\GK \NA (V) =  2$.
\end{theorem}

The proof  is split in several Lemmas.
By Lemma \ref{lemma:points-trivial-braiding}, $q_{22} \neq 1$.
A combination of Lemma~\ref{le:fn} with Lemma \ref{lemma:rosso-lemma19-gral}
leads to the first step.

\begin{lemma}\label{lemma:mild-not-discrete}
If $a\notin \N$, then $\GK \NA (V)=\infty $.
\end{lemma}

\begin{proof}
Let $N\in \N $. We claim that if $a\notin \I_N$, then the elements
\begin{align*}
&f_{i_1}\cdots f_{i_k},&  k& \in \N _0,& 0 \le i_1< &\cdots <i_k\le 2N+1
\end{align*}
are linearly independent in $\NA (V)$. Indeed, this follows from
Lemma~\ref{le:fn}; use
that $\partial _1(f_n)=\partial_2(f_n)=0$ for all $n\in \N _0$ and
\begin{align*}
x_1f_n &= (-1)^{n+1}q_{12}f_nx_1,&
x_{21}f_n &= q_{12}^2f_nx_{21},
\end{align*}
cf.  Lemma \ref{le:-K^1}.
Since $f_i \in \NA^{i+2}(V)$ for all $i\in \N _0$,  Lemma \ref{lemma:rosso-lemma19-gral} applies.
\end{proof}

Now the braided vector space $K^1$ is not of diagonal type, as the next Lemma
shows.
But for our further analysis of $K^1$, again the quotient Hopf algebra
$B=\NA (V_1)/(x_1)\# \ku \Gamma $
turns out to be very useful.

\begin{lemma} \label{le:--K^1} Let $\delta ':K^1\to B\otimes K^1$ be the
	coaction
  $\delta '=(\pi \otimes \id )\Delta _{\NA (V)\#\ku \Gamma }$,
  where $\pi _{12}:\NA (V)\#\ku \Gamma \to B$ is the canonical algebra
	projection
  with $\pi(x_3)=\pi (x_1)=0$.
  Then for all $n\in \N _0$,
\begin{align*}
\delta '(f_{2n})&=g_1^{2n+1}g_2\otimes f_{2n},\\
\delta '(f_{2n+1})&=g_1^{2n+2}g_2\otimes f_{2n+1}
-2x_2g_1^{2n+1}g_2\otimes f_{2n}.
\end{align*}
\end{lemma}

\begin{proof}
	Induction on $n$ using
	\eqref{eq:-1block+point-mild-ter}
	and
	\eqref{eq:-1block+point-mild-quator}.
\end{proof}

\begin{lemma}\label{lemma:mild-odd-order}
If $a\ne 1$, and $q_{22}\notin \G'_2 \cup \G'_3$, then $\GK \NA (V)=\infty $.
\end{lemma}

\begin{proof}
Lemma~\ref{le:fn} and the assumption $a\ne 1$ imply that $f_0,f_2\ne 0$.
  Since $\ad_cx_1(f_n)=0$ for all $n\ge 0$, Lemma~\ref{le:--K^1} and
equations
\begin{align*}
g_1g_2 \cdot f_0&=q_{22}f_0,& g_1g_2 \cdot f_2&=q_{21}^2q_{22} f_2,\\
g_1^3g_2 \cdot f_0&=q_{12}^2q_{22}f_0,& g_1^3g_2 \cdot f_2&=q_{22} f_2
\end{align*}
imply that
$W = \ku f_0+\ku f_2$ is a braided subspace of $K^1$ of diagonal type
with braiding matrix $(p_{ij})_{i,j\in \I_2}$ with respect to the basis
$f_0,f_2$, where $p_{11}=p_{22}=q_{22}$ and $p_{12}p_{21}=q_{22}^2$.
Thus $\GK \NA (W) = \infty$ by  the assumptions on
$q_{22}$. Indeed, if $q_{22}\notin \G_{\infty}$, then $W$
does not admit all reflections;
while if $\ord q_{22} = N > 3$, then $W$ is of Cartan type with Cartan matrix
$\begin{pmatrix} 2 & 2-N \\ 2-N & 2\end{pmatrix}$.
Thus Theorem \ref{thm:nichols-diagonal-finite-gkd} applies.
Finally, if $q_{22}=1$, then $U:=\ku x_1+\ku x_3$ is a braided vector space of diagonal
type with braiding matrix $(p_{ij})_{i,j\in \I_2}$, where $p_{11}=-1 = p_{12}p_{21}$,  $p_{22}=1$. Then $\GK \NA (U)=\infty $
by Lemma \ref{lemma:points-trivial-braiding}.
\end{proof}

Another approach to the calculation of $\GK \NA (V)$ excludes further possibilities.
Recall the braided Hopf algebra $\Bdiag$, cf. \S \ref{subsection:filtr-nichols}.

\begin{lemma}\label{lemma:mild-odd-order-2}
If  $q_{22}\in \G'_N$, where $N \ge 3$, then $\GK \NA (V)=\infty $.
\end{lemma}

\begin{proof} \emph{Assume first that $N$ is odd}.
Here $\Bdiag $ is a braided graded Hopf algebra of diagonal type
generated by $U:=\sum _{i=1}^3\ku x_i$. The braiding matrix $(p_{ij})_{i,j\in
\I_3}$ of $U$ satisfies
$$ p_{11}=p_{13}p_{31}=p_{22}=p_{23}p_{32}=-1,\quad p_{12}p_{21}=1, \quad
p_{33}=q_{22}. $$
Therefore the reflection $R_3(U)$ is well-defined and has a braiding matrix
$(p'_{ij})_{i,j\in \I_3}$ such that
$p'_{ij}=-q_{22}$ for all $i,j\in \{1,2\}$. Since $-q_{22}\in \G'_{2N}$ and
$2N\ge 6$, we conclude that $\GK \NA (R_3(U))=\infty $, by Theorem \ref{thm:nichols-diagonal-finite-gkd}.
Hence
$$\GK \NA (V)\ge \GK \Bdiag \ge \GK \NA (U)=\GK \NA (R_3(U))=\infty
.$$

\emph{Assume next that $N = 2M$ is even}. Then $U$ as in the proof above
has generalized Dynkin diagram
\begin{align}\label{eq:Dynkin-mild} &&
&\xymatrix{\overset{-1}{\underset{1} {\circ}} \ar  @{-}[r]^{-1}  & \overset{q} {\underset{3} {\circ}}  \ar  @{-}[r]^{-1}  & \overset{-1}{\underset{2} {\circ}} ,}& q &\in \G'_{2M}.
\end{align}
This is of Cartan type, with Cartan matrix
$A =\begin{pmatrix} 2 & 0 & -1 \\ 0 & 2 & -1 \\ -M & -M & 2\end{pmatrix}$. If $M = 2$, then $A$ is of affine type and $\GK \NA (U)=\infty$ by Theorem \ref{thm:nichols-diagonal-finite-gkd}. Similarly,
if $M > 3$, then $A$ contains a rank 2 submatrix of affine or indefinite type, so
$\GK \NA (U)=\infty$ by Theorem \ref{thm:nichols-diagonal-finite-gkd}. If $M = 3$,
then $A$ is of hyperbolic type we conclude that $\GK \NA (U) = \infty$,
from the Hypothesis  \ref{hyp:nichols-diagonal-finite-gkd}.
\end{proof}

\begin{lemma}
If $q_{22}=-1$ and $a\ne 1$, then $\GK \NA (V)=\infty $.
\end{lemma}

\begin{proof}
  First, $\partial_3(z_0)=1$;
	$\partial_3(z_1)=2x_2+ax_1$,  \eqref{eq:partial3z1}; and
\begin{align*}
	\partial_3(z_2)&=x_2\partial_3(z_1)-\partial_3(-q_{12}z_1+q_{12}f_0)
	q_{21}(x_2+ax_1)\\
  &=x_2(2x_2+ax_1)+(2x_1-2x_2-ax_1)(x_2+ax_1)\\
  &=(2-a)x_1x_2-ax_2x_1
\end{align*}
by \eqref{eq:-1block+point-mild} and \eqref{eq:partial3f0}
for any values of $q_{22}$ and $a$. Therefore
\begin{align} \label{eq:partial3z2}
	\partial_3(z_2)=2x_1x_2-ax_{21}
\end{align}
and hence $z_2\ne 0$ in $\NA (V)$.
Assume that $a\ne 1$ and that $\ad_cx_2^2(x_3)=0$
in $\Bdiag $. Since $\NA (V)\simeq K\#\NA (V_1)$ via the
multiplication map
and since $z_2\in K=\ker \partial _1\cap \ker \partial _2$,
\eqref{eq:-1block+point-mild-Kgen} implies that
there exists $\lambda \in \ku^{\times}$
such that $z_2=\lambda f_1$ in $K$.
Then $\partial_3(f_1)$ is a multiple of $2x_1x_2-ax_{21}$,
a contradiction to $a\ne 1$ and Lemma~\ref{le:fn}.

Now $z_2=\ad_cx_2^2(x_3)$ is a non-zero primitive element in  $\Bdiag $ of
$\Gamma$-degree  $g_1^2g_2$.
Consider the canonical Hopf algebra filtration of $\Bdiag $ with generators
$x_1$, $x_2$, $x_3$ and $z_2$ of degree $1$ and let $\tilde{B}_1$ be the
associated braided graded Hopf algebra. It is generated in degree $1$.
Let $\tilde{B}_2$ be the Nichols algebra quotient of $\tilde{B}_1$. The Dynkin
diagram of $\tilde{B}_2$ is a cycle where all vertices and all edges have label
$-1$ since $q_{22}=-1$.
Thus $\tilde{B}_2$ is of Cartan type with an affine Cartan matrix.
Thus $\GK \tilde{B}_2=\infty $ by Theorem \ref{thm:nichols-diagonal-finite-gkd}. This implies
the claim.
\end{proof}

\subsubsection{The Nichols algebra  $\cB(\cyc_1)$}\label{subsubsection:nichols-mild}

To prove Theorem \ref{thm:pm1bp-mild}, it remains the case $a=1$, $q_{22}= -1$.
Recall that the corresponding braided vector space is denoted $\cyc_1$, cf. page \pageref{page:cyclope}.
Recall the relations of the super Jordan plane:
\begin{align*}
	\tag{\ref{eq:rels-B(V(-1,2))-1}} &x_1^2,
	\\
	\tag{\ref{eq:rels-B(V(-1,2))-2}} & x_2x_{21}- x_{21}x_2 - x_1x_{21}.
\end{align*}

\begin{lemma} \label{le:a1q22even} The following hold in $\NA (V)$:
	\begin{align}
x_2f_0+q_{12}f_0x_2&=-f_1 \label{eq:nichols-mild-relation1}\\
x_2z_1+q_{12}z_1x_2&=\frac 1 2 f_1+q_{12}f_0x_2, \label{eq:nichols-mild-relation2}\\
x_2f_1 &=q_{12}f_1x_2, \label{eq:nichols-mild-relation3}\\
z_0^2 &=0, & f_0^2 &= 0, \label{eq:nichols-mild-relation4}\\
z_1^2 &=0, & f_1^2 &= 0  \label{eq:nichols-mild-relation5}.
	\end{align}
\end{lemma}

\begin{proof}
Note that \eqref{eq:nichols-mild-relation1} follows by \eqref{eq:-1block+point-mild-quator} for
$n=0$. Now $z_2,f_1\in K$ and $\partial_3(z_2)=2x_1x_2-x_{21}$ by \eqref{eq:partial3z2}.
Thus $2z_2=f_1$ by Lemma~\ref{le:fn}; i.e.
\begin{align}\label{eq:nichols-mild-2z2=f1}
\frac 1 2 f_1 = z_2 = \ad_cx_2(z_1)= x_2z_1 +q_{12} (z_1-f_0)x_2
\end{align}
by \eqref{eq:-1block+point-mild}, and \eqref{eq:nichols-mild-relation2} follows.

Next we claim that \eqref{eq:nichols-mild-relation3} holds.
Since $\partial_i(f_2)=0$ for all $i\in \I_2 $, Lemma~\ref{le:fn} implies that $f_2=0$.
By \eqref{eq:-1block+point-mild-quator} and \eqref{eq:-1block+point-mild-ter},
$0=\ad_cx_2(f_1) = x_2f_1-q_{12}f_1x_2$.

Now we prove that \eqref{eq:nichols-mild-relation4} holds.
As $q_{22}=-1$, $z_0^2=x_3^2=0$. Therefore
$$ f_0^2=(x_1x_3-q_{12}x_3x_1)^2=x_1x_3x_1x_3+q_{12}^2x_3x_1x_3x_1. $$
Moreover, $\partial_1(f_0^2)=\partial_2(f_0^2)=0$ since $f_0\in K$. Further,
$$ \partial_3(f_0^2)=x_1x_3x_1+x_1g_2\cdot (x_1x_3)+q_{12}^2g_2\cdot (x_1x_3x_1)+q_{12}^2x_3x_1g_2\cdot x_1=0 $$
since $g_2\cdot x_1=q_{21}x_1$, $x_1^2=0$, $q_{12}q_{21}=-1$, and $g_2\cdot x_3=-x_3$.

Finally we claim that \eqref{eq:nichols-mild-relation5} holds.
First we prove that $\partial_3(z_1^2)=0$. To do so we  need the following formula,
for which we use \eqref{eq:nichols-mild-relation1} and \eqref{eq:nichols-mild-relation2}:
\begin{align*}
	(2x_2+x_1)(g_2 \cdot z_1)&=(2x_2+x_1)(-q_{21})(z_1+f_0)\\
	&=-q_{21}\Big(2(\ad_cx_2)(z_1+f_0)+(\ad_cx_1)(z_1+f_0)\\
	&\qquad +(g_1 \cdot (z_1+f_0))(2x_2+x_1)\Big)\\
	&=-q_{21}\Big( f_1-2f_1+f_1+0 +(-q_{12}z_1)(2x_2+x_1)\Big)\\
	&=-z_1(2x_2+x_1).
\end{align*}
Now we conclude from \eqref{eq:partial3z1} that
\begin{align*}
	\partial_3(z_1^2)&=\partial_3(z_1)(g_2 \cdot z_1)+z_1\partial_3(z_1)\\
	&=(2x_2+x_1)(g_2 \cdot z_1)+z_1(2x_2+x_1)=0.
\end{align*}
This in turn implies that $z_1^2=0$. Similarly, since $g_2 \cdot
f_1=-q_{21}^2f_1$ and $\partial_3(f_1)=2(x_1x_2-x_2x_1)$ by (the proof of)
Lemma~\ref{le:-K^1}, we obtain from \eqref{eq:nichols-mild-relation3} and
the first equation in \eqref{eq:-1block+point-mild-quator} that
\begin{align*}
	\partial_3(f_1^2)&=\partial_3(f_1)(g_2 \cdot f_1)+f_1\partial_3(f_1)\\
	&=2(x_1x_2-x_2x_1)(-q_{21}^2f_1)+2f_1(x_1x_2-x_2x_1)\\
	&=2(-q_{21}^2)q_{12}^2f_1(x_1x_2-x_2x_1)+2f_1(x_1x_2-x_2x_1)=0,
\end{align*}
and then $f_1^2=0$.
\end{proof}

\begin{remark}\label{rem:nichols-mild-relations}
By definition of $f_0$, $f_1$, $z_1$, we have the identities
\begin{align*}
x_1z_0 & =f_0+q_{12}z_0x_1, & x_1z_1 & =f_1-q_{12}z_1x_1+q_{12}f_0x_1, & x_2z_0 & =z_1+q_{12}z_0x_2.
\end{align*}
\end{remark}

\begin{lemma}\label{lem:nichols-mild-relations}
Let $\cB$ be a quotient algebra of $T(V)$. Assume that
\eqref{eq:rels-B(V(-1,2))-1}, \eqref{eq:rels-B(V(-1,2))-2}, \eqref{eq:nichols-mild-relation1},
\eqref{eq:nichols-mild-relation2}, \eqref{eq:nichols-mild-relation3}, \eqref{eq:nichols-mild-relation4}
and \eqref{eq:nichols-mild-relation5} hold in $\cB$. Then the following relations also hold:
\begin{align}
x_1f_0 &= -q_{12}f_0 x_1, & x_1 f_1&=q_{12} f_1x_1, \label{eq:nichols-mild-1} \\
x_{21}z_1 &= q_{12}^2z_1x_{21}+2x_2f_1-x_1f_1-2x_{21}f_0, & x_{21}z_0&=q_{12}^2 z_0x_{21}, \label{eq:nichols-mild-2} \\
x_{21}f_0 &= q_{12}^2 f_0x_{21}, & x_{21}f_1 &= q_{12}^2 f_1x_{21},\label{eq:nichols-mild-3} \\
z_1z_0 &= -q_{12} z_1z_0, & f_0z_0 &=-q_{12} z_0f_0, \label{eq:nichols-mild-4} \\
f_1z_0 &= -q_{12}^2z_0f_1 -2q_{12}f_0z_1, & f_0z_1 &=- z_1f_0, \label{eq:nichols-mild-5} \\
f_1z_1 &= q_{12}z_1f_1-q_{12}f_0f_1, & f_1f_0 &=q_{12} f_0 f_1, \label{eq:nichols-mild-6}
\end{align}

In particular these relations hold in $\cB(V)$.
\end{lemma}

\begin{proof}
Note that \eqref{eq:nichols-mild-1} follows from \eqref{eq:rels-B(V(-1,2))-1}. For \eqref{eq:nichols-mild-2} and \eqref{eq:nichols-mild-3},
\begin{align*}
x_{21}z_0 &= (x_1x_2+x_2x_1)z_0 = x_1 (z_1+q_{12}z_0x_2) + x_2 (f_0+q_{12}z_0x_1) \\
&= f_1-q_{12}z_1x_1+q_{12}f_0x_1 +q_{12} (f_0+q_{12}z_0x_1)x_2 -(q_{12}f_0x_2+f_1)\\
& \qquad +q_{12}(z_1+q_{12}z_0x_2)x_1 = q_{12}^2 z_0x_{21}, \\
x_{21}z_1 &=  x_1 \Big(-q_{12}z_1x_2+q_{12}f_0x_2+\frac{1}{2}f_1 \Big)+ x_2 (f_1+q_{12} f_0x_1 -q_{12}z_1 x_1) \\
&= -q_{12} (f_1+q_{12} f_0x_1 -q_{12}z_1 x_1) x_2 - q_{12}^2 f_0x_1x_2 + \frac{1}{2}q_{12}f_1x_1 - q_{12}f_1x_2 \\
& \qquad + q_{12}(-q_{12}f_0x_2-f_1)x_1 -q_{12} \Big(-q_{12}z_1x_2+q_{12}f_0x_2+\frac{1}{2}f_1 \Big) x_1 \\
&= q_{12}^2 z_1 x_{21} -2x_{21}f_0+2x_2f_1-x_1f_1, \\
x_{21}f_0 &= (x_1x_2+x_2x_1)f_0 = x_1(-q_{12}f_0x_2-f_1)-q_{12}x_2f_0x_1   \\
&= q_{12}^2 f_0x_1x_2-q_{12}f_1x_1-q_{12}(-q_{12}f_0x_2-f_1)x_1 = q_{12}^2 f_0x_1, \\
x_{21}f_1 &= (x_1x_2+x_2x_1)f_1 = -q_{12}x_1f_1x_2+q_{12}x_2f_1x_1=q_{12}^2f_1x_{21},
\end{align*}
where we use \eqref{eq:nichols-mild-relation1}, \eqref{eq:nichols-mild-relation2}, \eqref{eq:nichols-mild-relation3},
\eqref{eq:nichols-mild-1} and Remark \ref{rem:nichols-mild-relations}.

We prove \eqref{eq:nichols-mild-4} from the first equality of \eqref{eq:nichols-mild-relation4} since
\begin{align}\label{eq:qcommut-1}
	0&= \ad_c x_1 (x_3^2) = f_0z_0+q_{12}z_0f_0,
	\\ \label{eq:qcommut-2}
	0&=\ad_c x_2 (x_3^2) = z_1z_0+q_{12}z_0z_1.
\end{align}
From \eqref{eq:qcommut-1} and \eqref{eq:nichols-mild-relation1},
\begin{align}\label{eq:f1e0}
	0&=\ad_cx_2(f_0z_0+q_{12}z_0f_0)
	=-f_1z_0-q_{12}f_0z_1+q_{12}z_1f_0-q_{12}^2z_0f_1.
\end{align}
From \eqref{eq:qcommut-2} and \eqref{eq:nichols-mild-relation5} we conclude that
\begin{align*}
	0&=(\ad_cx_2)(z_1z_0+q_{12}z_0z_1)\\
	&=\frac{1}{2}f_1z_0+q_{12}(-z_1+f_0)z_1 +q_{12}z_1^2+\frac{1}{2}q_{12}^2z_0f_1\\
	&=\frac{1}{2}(f_1z_0+2q_{12}f_0z_1+q_{12}^2z_0f_1).
\end{align*}
Together with \eqref{eq:f1e0}, the last equation implies that \eqref{eq:nichols-mild-5} holds.
Hence
\begin{align*}
	0&=(\ad_cx_2)(f_0z_1+z_1f_0)\\
	&=-f_1z_1-q_{12}\frac 12f_0f_1+\frac 12 f_1f_0+q_{12}(-z_1+f_0)(-f_1)\\
	&=-f_1z_1+q_{12}z_1f_1-q_{12}f_0f_1.
\end{align*}
From the second equality of \eqref{eq:nichols-mild-relation4} we conclude that
\begin{align*}
	0 &= (\ad_c x_2)(f_0^2)=-f_1f_0+q_{12}f_0f_1,
\end{align*}
so \eqref{eq:nichols-mild-6} holds. The last statement follows by Lemma \ref{le:a1q22even}.
\end{proof}

\begin{lemma} \label{le:basisK}
The set $B=\{ f_1^{n_1} f_0^{n_2} z_1^{n_3} z_0^{n_4}: 0\le n_i\le 1, \,\, i\in \I_4\}$
is a basis of $K$. In particular, $\dim K = 16$.
\end{lemma}

\begin{proof}
First we claim that $K$ is spanned by the monomials $f_1^{n_1} f_0^{n_2} z_1^{n_3} z_0^{n_4}$
with $n_i\in\N_0$. By \eqref{eq:-1block+point-mild-Kgen},
$K$ is generated as an algebra by the $z_i$'s and the $f_i$'s; as $f_2=0$ and by
\eqref{eq:nichols-mild-2z2=f1} $2z_2=f_1$, it is enough to consider $z_0$, $z_1$,
$f_0$, $f_1$. Thus the claim follows since \eqref{eq:nichols-mild-4}, \eqref{eq:nichols-mild-5}, \eqref{eq:nichols-mild-6}
hold in $K \subset \cB(V)$. Thus $K$ is spanned by $B$ since the monomials with some $n_i\geq 2$ are 0 by
\eqref{eq:nichols-mild-relation4} and \eqref{eq:nichols-mild-relation5}.
We claim now that $B$ is linearly independent. Recall that $\partial_i(f_j)=\partial_i(z_j)=0$ for $i=1,2$, $j=0,1$, and
\begin{align*}
\partial_3(z_0)&=1, & \partial_3(z_1)&=2x_2+x_1, & \partial_3(f_0)&=2x_1, & \partial_3(f_1)&=2x_1x_2-2x_2x_1.
\end{align*}
Assume that $\sum\limits_{0\le n_i\le 1} a_{n_1n_2n_3n_4}\, f_1^{n_1} f_0^{n_2} z_1^{n_3} z_0^{n_4}=0$ for some $a_{n_1n_2n_3n_4}$. As
\begin{align*}
\partial_1\partial_2\partial_3(f_1f_0z_1z_0) &= 4g_1^2g_2\cdot (f_0z_1z_0)= -4q_{12}^2 \, f_0z_1z_0, \\
\partial_3\partial_1\partial_3(f_0z_1z_0) &= \partial_3(2q_{12}  \, z_1z_0-q_{12}\, f_0z_0)= -4x_2^2\neq 0,
\end{align*}
we have $f_1f_0z_1z_0\neq 0$. Thus $a_{1111}=0$ since $f_1f_0z_1z_0$ has degree 8 in $\cB(V)$ and
the other elements have degree $\le 7$. As $f_1f_0z_1\neq 0$ is the unique element in degree 7, $a_{1110}=0$.
In degree 6, $ a_{1101} \, f_1f_0z_0+a_{1011} \, f_1z_1z_0=0$. But
\begin{align*}
0 & = f_0 (a_{1101} \, f_1f_0z_0+a_{1011} \, f_1z_1z_0) = -q_{21} a_{1011} \, f_1f_0z_1z_0 , \\
0 & = z_1 (a_{1101} \, f_1f_0z_0+a_{1011} \, f_1z_1z_0) = -q_{21} (a_{1011}+a_{1101}) \, f_1f_0z_1z_0 ,
\end{align*}
by \eqref{eq:nichols-mild-6}, \eqref{eq:nichols-mild-relation4} and \eqref{eq:nichols-mild-relation5}, so $a_{1011}=a_{1101}=0$.
By a similar argument all the other $a_{n_1n_2n_3n_4}$ are also
 zero, so $B$ is a basis of $K$.
\end{proof}

Theorem \ref{thm:pm1bp-mild} is now proved.

\begin{remark}
The Lemma allows us to understand the structure of $K$--and hence of
$\NA (V)$--from a
different perspective. We observed that $\Bdiag $ is a braided Hopf algebra of
diagonal type. More precisely, it is of Cartan type with Dynkin diagram
of type $A_3$, where all labels are $-1$. Thus the corresponding Nichols
algebra $B$ has dimension $2^6$. The subalgebra $B'$ of $B$ generated
by $\ku x_1+\ku x_2$ has dimension $4$, and hence the right coinvariants of $B$ with
respect to $B'$ form a subalgebra of dimension $2^4=16$.
The proposition implies that $K$ has the same dimension as
$B^{\mathrm{co}\,B'}$, and hence the graded object associated with the
filtration on $K$ induced by the one on $\NA (V)$ is isomorphic as a braided
Hopf algebra to $B^{\mathrm{co}\,B'}$.
\end{remark}

We close this Subsection giving the presentation of $\cB(\cyc_1)$.

\begin{prop} \label{pr:nichols-mild} The algebra
$\cB(\cyc_1)$ is presented by generators $x_1,x_2, x_3$ and relations \eqref{eq:rels-B(V(-1,2))-1},
\eqref{eq:rels-B(V(-1,2))-2}, \eqref{eq:nichols-mild-relation1}, \eqref{eq:nichols-mild-relation2},
\eqref{eq:nichols-mild-relation3}, \eqref{eq:nichols-mild-relation4} and \eqref{eq:nichols-mild-relation5}. The set
\begin{align*}
B=\{ x_1^{m_1} x_{21}^{m_2} x_2^{m_3} f_1^{n_1} f_0^{n_2} z_1^{n_3} z_0^{n_4}: m_1,n_i \in\{0,1\}, m_2,m_3 \in\N_0\}
\end{align*}
is a basis of $\cB(\cyc_1)$ and $\GK \cB(\cyc_1) = 2$.
\end{prop}

\pf
Relations \eqref{eq:rels-B(V(-1,2))-1},  \eqref{eq:rels-B(V(-1,2))-2}, \eqref{eq:nichols-mild-relation1},
\eqref{eq:nichols-mild-relation2}, \eqref{eq:nichols-mild-relation3}, \eqref{eq:nichols-mild-relation4}
and \eqref{eq:nichols-mild-relation5} are 0 in $\cB(\cyc_1)$, see Lemma \ref{le:a1q22even}.
Hence the quotient $\cBt$ of $T(V)$ by these relations projects onto $\cB(\cyc_1)$.
We claim that the subspace $I$ spanned by $B$ is a right ideal of $\cBt$. Indeed, $I x_3\subseteq I$
by definition while $Ix_1\subseteq I$, $Ix_2\subseteq I$ follow by Lemma \ref{lem:nichols-mild-relations}
and Remark \ref{rem:nichols-mild-relations}. Since $1\in I$, $\cBt$ is spanned by $B$.

To prove that $\cBt \simeq \cB(\cyc_1)$, it remains to show that
$B$ is linearly independent in $\cB(\cyc_1)$. For, suppose that there is a non-trivial linear combination $\mathtt{S}$
of elements of $B$ in $\cB(\cyc_1)$, say of minimal degree. As in the proof of Proposition \ref{pr:-1block},
each vector $x_1^{m_1} x_{21}^{m_2} x_2^{m_3} f_1^{n_1} f_0^{n_2} z_1^{n_3} z_0^{n_4}$ in $\mathtt{S}$
with non-trivial coefficient satisfies $m_i=0$, $i=1,2,3$. But then we obtain a contradiction with Lemma \ref{le:basisK}.
Thus $B$ is a basis of $\cB(\cyc_1)$ and $\cBt=\cB(\cyc_1)$.
The computation of $\GK$ follows from the Hilbert series at once.
\epf

\section{One block and several points}\label{sec:yd-dim>3}
\subsection{The setting}\label{subsection:YD>3-setting}
Let $\Gamma $ be an abelian group. In this Section we consider $V \in \ydG$, $\dim V  >3$ as in \eqref{eq:bradinig-generalform} and \eqref{eq:bradinig-generalform1}, with one block  and several points.
We seek to determine when $\GK \NA(V) < \infty$.
 Recall that by Theorem \ref{theorem:blocks}, the block is assumed of the form
 $\cV(\epsilon, 2)$, with $\epsilon^2 = 1$.

For a more suggestive presentation, we introduce the notation
\begin{align*}
\I_{2,\theta} &= \I_{\theta} - \{1\},&
\Iw_\theta &= \I_{\theta} \cup \{\fudos\},&  &\theta \in\N.
\end{align*}

Let
$g_1, \dots, g_\theta\in \Gamma$,  $\chi_1, \dots, \chi_\theta\in \widehat\Gamma$ and $\eta: \Gamma \to \ku$ a $(\chi_1, \chi_1)$-derivation.
Let $\cV_{g_1}(\chi_1, \eta) \in \ydG$ be the indecomposable with basis $(x_i)_{i\in\Iw_1}$ and action given by \eqref{equation:basis-triangular-gral}-- but with $\fudos$ instead of 2;
while $\ku_{g_j}^{\chi_j}\in \ydG$ is irreducible with basis $(x_j)$, $j \in \I_{2,\theta}$.
Let
\begin{align*}
V = \cV_{g_1}(\chi_1, \eta) \oplus \ku_{g_2}^{\chi_2} \oplus \cdots \oplus \ku_{g_\theta}^{\chi_\theta}.
\end{align*}
Thus $(x_i)_{i\in\Iw_\theta}$ is a basis of $V$. We suppose that $V$ is not of diagonal type, hence
$\eta(g_1) \neq 0$; we may assume that  $\eta(g_1) = 1$ by normalizing $x_1$.
Let
\begin{align*}
q_{ij}&= \chi_j(g_i),& &i, j\in \I_{\theta};& a_j&= q_{j1}^{-1}\eta(g_j),&  j &\in \I_\theta.
\end{align*}
Let $\lfloor i\rfloor$ be the largest integer $\leq i$.
Then the braiding in the basis $(x_i)_{i\in\Iw_\theta}$ is
\begin{align}\label{eq:braiding-block-several-point}
c(x_i \otimes x_j) &= \begin{cases}
q_{\lfloor i\rfloor j} x_j  \otimes x_i, &i\in \Iw_{\theta},\, j\in \I_{\theta};\\
q_{\lfloor i\rfloor 1} (x_{\fudos} + a_{\lfloor i\rfloor} x_1) \otimes x_{i}, & i\in \Iw_{\theta},\, j =\fudos.
\end{cases}
\end{align}
Let $\epsilon :=q_{11}$. Notice that $\NA( \cV_{g_1}(\chi_1, \eta) \oplus \ku_{g_j}^{\chi_j}) \hookrightarrow \NA(V)$
for all $j \in \I_{2,\theta}$, thus we may apply the results from \S \ref{sec:yd-dim3}.
By Theorem \ref{theorem:blocks}, we may assume that $\epsilon^2 = 1$, thus $a_1 = \epsilon$.
The \emph{interaction} and the \emph{ghost} between the block and the points are the  vectors $ \inc = (\inc_{j})_{j \in \I_{2,\theta}}$, $\ghost = (\ghost_{j})_{j \in \I_{2,\theta}}$ given by
\begin{align}
 \inc_{j}&= q_{1j}q_{j1},& &\ghost_{j} =
\begin{cases} -2(a_{j})_{j \in \I_{2,\theta}}, &\epsilon = 1, \\
(a_{j})_{j \in \I_{2,\theta}}, &\epsilon = -1, \end{cases} & j &\in \I_{2,\theta}.
\end{align}
The interaction is \emph{strong} if there exists $h \in \I_{2,\theta}$ such that $\inc_{h} \notin \{\pm 1\}$;
when it is not strong, it is
\begin{align*}
\text{\emph{weak} if } \inc_{h}&= 1,& &\forall h \in \I_{2,\theta}; &\text{\emph{mild}, otherwise.}
\end{align*}

We say that the ghost is \emph{discrete} if $\ghost \in \N_0^{\I_{2,\theta}} - \{0\}$.

\bigbreak
We can present our main object of interest in the language of braided vector spaces.
Given $(q_{ij})_{i,j \in \I_{\theta}}$, with $q_{11}^2 = 1$, and $\ghost \in \ku^{\I_{2,\theta}}$, we set $a_1 = \epsilon = q_{11}$
and consider the braided vector space  $(V, c)$ of dimension $\theta + 1$,
with a basis $(x_i)_{i\in\Iw_{\theta}}$ and braiding given  by \eqref{eq:braiding-block-several-point}.
This braided vector space $(V, c)$ can be realized as a Yetter-Drinfeld module
$\cV_{g_1}(\chi_1, \eta) \oplus \bigoplus_{j \in \I_{\theta}} \ku_{g_j}^{\chi_j}$
over some abelian group $\Gamma$ as described above; for instance $\Gamma = \Z^{\theta}$ would do.
Such a realization will be called \emph{principal}.

The braided subspace $V_1$ spanned by $x_1, x_{\fudos}$ is $\simeq \cV(\epsilon, 2)$, while
$V_{\diag}$ spanned by $(x_i)_{i\in\I_{2,\theta}}$ is of diagonal type. Obviously,
\begin{align}\label{eq:v=v1+v2}
V &= V_1 \oplus V_{\diag}.
\end{align}
Let $\X$ be the set of connected components of the generalized Dynkin diagram of the matrix $\bq = (q_{ij})_{i, j \in \I_{2, \theta}}$.
If $J\in \X$, then we set  $J' = \I_{2, \theta} - J$,
\begin{align*}
V_J &= \sum_{j \in J} \ku_{g_j}^{\chi_j},&
\ghost_J &= (\ghost_{j})_{j \in J},& \inc_J
&= (q_{1h}q_{h1})_{h \in J}.
\end{align*}

As before,  $J$ could have weak, mild or strong interaction $\inc_J$.

\begin{table}[ht]
	\caption{A block and several points, finite $\GK$, weak interaction, $\epsilon = 1$;
		here  $\omega \in \G'_3$ and $\gkv_J = \GK \toba(K_J)$. The meaning of $K_J$  is explained
		in \S \ref{subsubsec:algK-block-points} below.} \label{tab:finiteGK-block-points}
	\begin{center}
		\begin{tabular}{|c|c|c|c|c|}
			\hline   $V_J$ & {\scriptsize type} & $\ghost_J$  &  $K_J$ & {\scriptsize $\gkv_J$}   \\
			\hline
			$\overset{1}{\circ}$	& \scriptsize{$A_{1}$} & \scriptsize{discrete} & \scriptsize{$(A_1)^{\ghost_J + 1}$} & \scriptsize{$\ghost_J + 1$}
			\\ \hline	
			$\overset{-1}{\circ}$	& \scriptsize{$A_{1}$} & \scriptsize{discrete} & \scriptsize{$(A_1)^{\ghost_J + 1}$} & $0$
			\\ \hline
			$\overset{\omega}{\circ}$	& \scriptsize{$A_{1}$} & 1 & \scriptsize{$A_2$} & $0$
			\\ \hline
			$\xymatrix{\overset{-1}{\circ} \ar  @{-}[r]^{-1}  & \overset{-1}{\circ}} \dots \xymatrix{
				\overset{-1}{\circ} \ar  @{-}[r]^{-1}  & \overset{-1}{\circ} }$
			& \scriptsize{$A_{\theta -1}$}	& $(1, 0,\dots, 0)$     &\scriptsize{$A_3$, $\theta = 3$ } & 0		
			\\ \cline{4-4}
			& & & \scriptsize{$D_\theta$, $\theta > 3$} &
			\\ \hline
			$\xymatrix{\overset{-1}{\circ} \ar  @{-}[r]^{-1}  & \overset{-1}{\circ}}$ 	& \scriptsize{$A_{2}$} & $(2, 0)$     & \scriptsize{$D_4$}& 0
			\\ \hline
			$\xymatrix{\overset{-1}{\circ} \ar  @{-}[r]^{\omega}  & \overset{-1}{\circ}}$ 	& \scriptsize{$\mathfrak{sl}(2\vert 1)$}
			& $(1, 0)$     &\scriptsize{$\mathfrak g(2,3)$} & 0
			\\ \hline
			$\xymatrix{\overset{-1}{\circ} \ar  @{-}[r]^{\omega^2 }  & \overset{\omega}{\circ}}$ & \scriptsize{$\mathfrak{sl}(2\vert 1)$}
			& $(1, 0)$  & \scriptsize{$\mathfrak{sl}(2\vert 2)$} & 0	
			\\ \cline{3-5}
			&  &  $(0, 1)$   & \scriptsize{$\mathfrak g(2,3)$} & 0	
			
			\\ \hline
			$\xymatrix{\overset{-1}{\circ} \ar  @{-}[r]^{\omega}  & \overset{\omega^2}{\circ} \ar  @{-}[r]^{\omega}  & \overset{\omega^2}{\circ} }$ & \scriptsize{$\mathfrak{sl}(1\vert 3)$}
			& $(1,0, 0)$ & \scriptsize{$\mathfrak g(3,3)$} & 0
			\\ \hline
			$\xymatrix{\overset{-1}{\circ} \ar  @{-}[r]^{\omega}  & \overset{\omega^2}{\circ} \ar  @{-}[r]^{\omega^2}  & \overset{\omega}{\circ} }$ & \scriptsize{$\mathfrak{osp} (2,4)$}
			& $(1,0, 0)$ & \scriptsize{$\mathfrak g(3,3)$}	& 0	
			\\ \hline
			$\xymatrix{\overset{-1}{\circ} \ar  @{-}[r]^{r^{-1}}  & \overset{r}{\circ}}$, $r \notin \G_{\infty}$ & \scriptsize{$\mathfrak{sl}(2\vert 1)$}
			&  $(1, 0)$   & \scriptsize{$\mathfrak{sl}(2\vert 2)$} & 2
			\\ \hline
			$\xymatrix{\overset{-1}{\circ} \ar  @{-}[r]^{r^{-1}}  & \overset{r}{\circ}}$, $r \in \G'_N, N > 3$ & \scriptsize{$\mathfrak{sl}(2\vert 1)$}
			&    $(1, 0)$    & \scriptsize{$\mathfrak{sl}(2\vert 2)$} & 0
			\\ \hline
		\end{tabular}
	\end{center}

\end{table}

Here are the main results of this Section.

\begin{theorem}\label{thm:points-block-eps1} Let $V$ be a braided vector space with braiding \eqref{eq:braiding-block-several-point}.  Assume that $\epsilon = 1$. Then the interaction  is weak and
 the following are equivalent:

\begin{enumerate}[leftmargin=*,label=\rm{(\roman*)}]

\item $\GK \NA (V) < \infty$.

\smallbreak
\item\label{item:points-block-ii} For every $J \in \X$, either  $\ghost_J = 0$, or else
 $V_J$ is  as in Table \ref{tab:finiteGK-block-points}.

\end{enumerate}
Furthermore, if \ref{item:points-block-ii} holds, then
\begin{align}\label{eq:points-block-gkd}
\GK \toba(V) &= 2 + \sum_{J\in \X} \GK \toba(K_J).
\end{align}
\end{theorem}

After some preliminaries in this Subsection, we prove Theorem \ref{thm:points-block-eps1}
in \S \ref{subsec:pf-main-block-points-weak}.

\begin{theorem}\label{thm:points-block-eps-1} Let $V$ be a braided vector space with braiding \eqref{eq:braiding-block-several-point}. Assume that $\epsilon = -1$.
	Then the following are equivalent:
	
	\begin{enumerate}[leftmargin=*,label=\rm{(\roman*)}]
		
		\item $\GK \NA (V) < \infty$.
		
		\smallbreak
		\item\label{item:points-block-ii-1} For $J \in \X$, either of the following holds:
	\end{enumerate}
	
	\begin{enumerate}[leftmargin=*,label=\rm{(\alph*)}]
		\smallbreak\item\label{item:points-block-weak-0}
		The interaction of $J$ is weak and $\ghost_J = 0$.

		\smallbreak
		\item\label{item:points-block-weak-eps-1}
		The interaction  of $J$ is weak, $J = \{i\}$,  $\ghost_i$ discrete and $q_{ii}  = \pm 1$.

		\smallbreak
		\item\label{item:points-block-mild} The interaction  of $J$ is mild, $J = \{i\}$,
		$\ghost_i = 1$ and $q_{ii}  = - 1$.
		
		\smallbreak
		\item\label{item:points-block-mild-cyc2} The interaction  of $J$ is mild, $J = \{i, j\}$ has Dynkin diagram
		$\xymatrix{\overset{-1}{\circ} \ar  @{-}[r]^{-1}  & \overset{-1}{\circ}} $, i.e. is of type $A_2$,
		$(\ghost_i, \ghost_j) = (1,0)$ and $(\inc_i, \inc_j) = (-1, 1)$.
	\end{enumerate}
	Furthermore, if there is one component $J$ with mild interaction, i.e. of type \ref{item:points-block-mild} or \ref{item:points-block-mild-cyc2},
	and no components disconnected from the block, i.e. of type \ref{item:points-block-weak-0}, then the diagram is connected, i.e. $J = \I_{2,\theta}$.
	
If \ref{item:points-block-ii-1} holds, then
	\begin{align}\label{eq:points-block-gkd-1}
	\GK \toba(V) &= 2 + \sum_{J\in \X} \GK \toba(K_J).
	\end{align}
\end{theorem}

We prove Theorem \ref{thm:points-block-eps-1} in \S \ref{subsec:pf-main-block-points-mild}.

\bigbreak
To start with the proofs of Theorems \ref{thm:points-block-eps1} and \ref{thm:points-block-eps-1},  we infer from Theorem \ref{thm:point-block}, since $\cV_{g_1}(\chi_1, \eta) \oplus \ku_{g_j}^{\chi_j} \hookrightarrow V$,
$j\in \I_{2, \theta}$, that the interaction is not strong, $q_{ii} \in \G_2 \cup \G_3$
 whenever $\ghost_i \neq 0$
 (according to Table \ref{tab:finiteGK-block-point})
 and the ghost $\ghost$ is either 0 or discrete. On the other hand, $\GK \toba(V_{\diag}) < \infty$.
 If $q_{ii} = 1$, then $q_{ij}q_{ji} = 1$ for all $j\in \I_{2,\theta}$, $j\neq i$
 by Lemma \ref{lemma:points-trivial-braiding}, so that the connected component containing $i$ is the singleton $\{i\}$.
 \label{page:qii1}

Now, if $J\in \X$ is a point, then the result follows from Theorem
\ref{thm:point-block}. Thus we need to analyze those $J$ with  $\vert J \vert  \geq 2$.

Next, if $J\in \X$ has weak interaction and $\ghost_J = 0$, then
	\begin{align*}
	\NA(V) \simeq \NA(V_1 \oplus V_{J'})  \, \underline{\otimes} \, \NA(V_J),
	\end{align*}
	hence $\GK \NA(V) = \GK\NA(V_1 \oplus V_{J'}) + \GK \NA(V_J)$ by Lemma \ref{lemma:GKdim-smashproduct}.
Therefore, we  consider  those $J \in \X$ with weak interaction only when $\ghost_J \neq 0$.

\subsection{Proof of Theorem \ref{thm:points-block-eps1} ($\epsilon = 1$)} \label{subsec:pf-main-block-points-weak}

After some preliminaries on the algebra $K$, we reduce to connected components in Corollary \ref{cor:conncomp}, and then deal with the case $\vert J \vert = 2$
in \S \ref{subsub:J=2},  and with $\vert J \vert > 2$ in \S \ref{subsub:J>2}.

\subsubsection{Weak interaction and the algebra $K$}\label{subsubsec:algK-block-points}
Here we assume  that the interaction is weak, but $\epsilon$ could be $\pm 1$.
We shall use the results and notations from the preceding Sections, but with a caveat: $\fudos$ replaces 2 when appropriate,
e. g. $x_{\fudos 1} = x_{\fudos}x_1 - \epsilon x_1 x_{\fudos}$.
As in \S \ref{sec:yd-dim3}, let
$K=\NA (V)^{\mathrm{co}\,\NA (V_1)}$; again
\begin{align*}
\NA(V) &\simeq K \# \NA (V_1);& K &\simeq \NA(K^1)& &\text{and}& K^1 &= \ad_c\NA (V_1) (V_{\diag})
\in {}^{\NA (V_1)\# \ku \Gamma}_{\NA (V_1)\# \ku \Gamma}\mathcal{YD},
\end{align*}
with the coaction \eqref{eq:coaction-K^1} and the adjoint action.
 We introduce
\begin{align}\label{eq:zjn}
z_{j,n} &:= (ad_c x_{\fudos})^n x_j,& j&\in\I_{2, \theta},& n&\in\N_0.
\end{align}

We have that
\begin{align}
\label{eq:1block+points-action}
  g_1\cdot z_{j,n} &= \epsilon^nq_{1j}z_{j,n},& &\text{(by Lemma \ref{le:-1bpz})}
  \\\label{eq:1block+points-bis}
  g_i\cdot z_{j,n} &= q_{i1}^nq_{ij}z_{j,n}, &&
\end{align}
for all $i,j\in \I_{2,\theta}$, $n\in \N_0$. Indeed, $g_i\cdot z_{j,0} = g_i\cdot x_{j} = q_{ij} x_{j}$; and
\begin{align*}
g_i\cdot z_{j,n+1} &= g_i\cdot (x_{\fudos} z_{j,n}  - \epsilon^{n} q_{1j}z_{j,n} x_{\fudos}) \\
&=q_{i1}(x_{\fudos}+a_ix_1)q_{i1}^nq_{ij}z_{j,n}  - \epsilon^{n}  q_{1j}q_{i1}^{n+1}q_{ij}z_{j,n}(x_{\fudos}+a_ix_1)\\
&\overset{\eqref{eq:-1block+point}} =q_{i1}^{n+1}q_{ij}(x_{\fudos}z_{j,n} - \epsilon^{n} q_{1j}z_{j,n}x_{\fudos}) = q_{i1}^{n+1}q_{ij}z_{j,n+1}.
\end{align*}

\begin{remark}\label{remark:GK-block-points-K}
As in Remark \ref{remark:GK-block-point-K}, if	$\GK \toba(V_1) < \infty$
and the action of $\toba(V_1)$ on $K$ is locally finite, then by Lemma \ref{lemma:GKdim-smashproduct}, $\GK\NA(V) =  \GK K + 2$.
\end{remark}

\begin{lemma}\label{lemma:braiding-K-weak-block-points} Assume  that the interaction is weak and that $\ghost$ is discrete.
Then the  braided vector space $K^1$ is of diagonal type with respect to the basis
\begin{align}\label{eq:block-points-baseK1}
(z_{j,n})_{j\in \I_{2,\theta}, 0\le n\le \vert 2a_j\vert}
\end{align}
with braiding matrix
\begin{align*}
(p_{im, jn})_{\substack{i,j \in \I_{2,\theta}, \\ 0\le m\le \vert 2a_i\vert,
0\le n \le 2\vert a_j \vert}} &= (\epsilon^{nm} q_{i1}^nq_{1j}^m  q_{ij})_{\substack{ i,j \in \I_{2,\theta}, \\ 0\le m\le \vert 2a_i\vert,
0\le n \le 2\vert a_j \vert}}.
\end{align*}
Hence, the corresponding generalized  Dynkin diagram  has labels
\begin{align*}
p_{im,im}&= \epsilon^m q_{ii},& p_{im, jn}p_{jn, im}&= q_{ij}q_{ji}, & &(i,m)\neq (j,n).
\end{align*}
\end{lemma}

\pf We have  $\ad_cx_{\fudos} z_{j,n} = z_{j,n + 1}$ by definition and $\ad_cx_1 z_{j,n} = 0$ by \eqref{eq:adx1-zn} for all $j,n$;
hence $K^1$ is generated by the family \eqref{eq:block-points-baseK1}; and this family is linearly independent,
because the $z_{j,n}$'s are homogeneous of distinct degrees, and  are $\neq 0$ by \eqref{eq:derivations-zn}.
By Lemma \ref{le:zcoact}
the coaction \eqref{eq:coaction-K^1} on $z_{j,n}$, $n\in \N _0$, is given by \eqref{eq:coact-zjn}, when $\epsilon = 1$,
and by \eqref{eq:coact-zjn-even}, \eqref{eq:coact-zjn-odd}, when $\epsilon = -1$:

\begin{align} \label{eq:coact-zjn}
\delta (z_{j,n}) &= \sum _{k=0}^n \nu_{k,n}\, x_1^{n-k}g_1^{k} g_j \otimes z_{j,k}.
\\ \label{eq:coact-zjn-even}
\delta (z_{j, 2n}) &=
\sum _{k=1}^n k\binom n k \mu_{k,n} \,
x_1x_{\fudos 1}^{n-k}g_1^{2k-1}g_j\otimes z_{j, 2k-1}
\\ \notag&\qquad  + \sum _{k=0}^n \binom n k\mu_{k, n} x_{\fudos 1}^{n-k}g_1^{2k}g_j\otimes z_{j, 2k},
\\ \label{eq:coact-zjn-odd}
\delta (z_{j, 2n+1}) &=
\sum _{k=0}^n \binom nk \mu_{k, n+1} \,   x_1x_{\fudos 1}^{n-k}g_1^{2k}g_j\otimes z_{j, 2k}
\\ \notag& \qquad + \sum _{k=0}^n \binom nk \mu_{k+1, n+1} x_{\fudos 1}^{n-k}g_1^{2k+1}g_j\otimes z_{j, 2k+1}.
\end{align}
We  compute
\begin{align*}
c(z_{i,m} \otimes z_{j,n}) &= g_1^mg_i \cdot z_{j,n} \otimes z_{i,m} = \epsilon^{nm} q_{i1}^n q_{1j}^m  q_{ij} \, z_{j,n} \otimes z_{i,m},
\end{align*}
by Lemmas \ref{le:zcoact} and \ref{le:-1bpz}, \eqref{eq:adx1-zn} and \eqref{eq:adx12-zn}.
\epf

Let $K_J$ be the braided vector subspace of $K^1$ spanned by $(z_{j,n})_{\substack{j\in J,\\ 0\le n\le \vert 2a_j\vert}}$.

\begin{coro}\label{cor:conncomp}
The braided subspaces corresponding to the connected components of the Dynkin diagram of $K^1$ are  $K_J$, $J \in \X$. Hence
\begin{align}\label{eq:points-block-gkd-K}
\GK K = \GK \toba(K^1) &=  \sum_{J\in \X} \GK \toba(K_J). \qed
\end{align}
\end{coro}
Observe that if $\ghost_J = 0$, then $K_J = V_J$.

\subsubsection{$\vert J \vert  = 2$}\label{subsub:J=2}
Recall that $\ghost_J \neq 0$.
Below we denote $\imath\, = \sqrt{-1}$.

\begin{lemma}\label{lemma:KJ-card2}
Assume that $\vert J \vert  = 2$.
Then $\GK \NA (K_J)$ is finite if and only if it appears in Table \ref{tab:finiteGK-block-points}.
\end{lemma}

\pf Say $J = \{i, j\}$; we set $q:= q_{ij}q_{ji} \neq 1$ for brevity.
Recall that  $\epsilon = 1$.
 There are several possibilities:

 \bigbreak
 \begin{enumerate}\renewcommand{\theenumi}{\alph{enumi}}\renewcommand{\labelenumi}{(\theenumi)}
  \item\label{case:diagram-block-points-casob}
 $q_{ii} = \omega \in \G'_3$,  $q_{jj} = -1$, $\ghost_j > 0$.
\end{enumerate}

Let $U$ be the braided vector subspace of $K_J$  spanned by  $1 = z_{i,0}$, $2 = z_{j,1}$, $3 = z_{j,0}$, i. e. the central diagram in
\eqref{eq:diagram-block-points-casob}. We apply the reflections
at various vertices as described below:

\begin{align}\label{eq:diagram-block-points-casob}
\xymatrix{ & \overset{-\omega q^2}{\underset{2}{\circ}} \ar  @{-}[ld]_{q^{-1}\omega^2}  \ar  @{-}[d]^{q^{4}\omega^2}
& \ar@/^0.5pc/@{<->}[r]^{1}_{\ast} &  & \overset{-1}{\underset{2}{\circ}} \ar  @{-}[ld]_{q}
& \ar@/^0.5pc/@{<->}[r]^{3} &  & \overset{-1}{\underset{2}{\circ}} \ar  @{-}[ld]_{q}
\\
\overset{\omega}{\underset{1}{\circ}} \ar  @{-}[r]_{q^{-1}\omega^2 }  & \overset{-\omega q^2}{\underset{3}{\circ}}
&&   \overset{\omega}{\underset{1}{\circ}} \ar  @{-}[r]^{q}  & \overset{-1}{\underset{3}{\circ}}
&&
\overset{-q\omega}{\underset{1}{\circ}} \ar  @{-}[r]^{q^{-1}}  & \overset{-1}{\underset{3}{\circ}} \ar@/^0.5pc/@{<->}[d]^{2}
\\
& & & & \overset{\mp \imath\, \omega}{\underset{2}{\circ}} \ar  @{-}[d]^{-\omega^2}  \ar  @{-}[ld]_{q^{5} \omega^2}
& \ar@/^0.5pc/@{<->}[r]^{1}_{\#} &
& \overset{-1}{\underset{2}{\circ}} \ar  @{-}[ld]_{q^{-1}}
\\
& &
& \overset{q^2\omega}{\underset{1}{\circ}} \ar  @{-}[r]^{q^{5} \omega^2}  & \overset{\mp \imath\, \omega}{\underset{3}{\circ}}
& & \overset{q^2\omega}{\underset{1}{\circ}} \ar  @{-}[r]^{q^{-1}}  & \overset{-1}{\underset{3}{\circ}}
}
\end{align}
$\ast$: only for $q\neq \omega^2$; $\#$: only when $q = \pm \imath\,$ or $q = \pm \imath\,\omega^2$.

 Then the following possibilities are excluded since $\GK \toba (U) = \infty$:

 \medbreak
\begin{itemize} [leftmargin=*] 
 \item $q \notin \G_{\infty}$. By Theorem \ref{thm:nichols-diagonal-finite-gkd} since $\xymatrix{\overset{\omega}{\circ} \ar  @{-}[r]^{q}  & \overset{-1}{\circ}}$ has an infinite root system.

  \medbreak
  \item $-\omega q^2 = 1$, that is $q = \pm \imath\, \omega$. By Lemma \ref{lemma:points-trivial-braiding}
  applied to $\cR^1 (U)$.

 \medbreak
  \item $-\omega q^2 = -1$, that is $q = \pm\omega$. By Lemma \ref{lemma:points-trivial-braiding}
  applied to $\cR^2\cR^3 (U)$.

   \medbreak
 \item $-\omega q^2 \in \G'_N$, $N > 3$, but $q\neq \omega^2$. By Theorem \ref{thm:nichols-diagonal-finite-gkd} applied to $\cR^1 (U)$, since the subdiagram spanned by 2 and 3 is of Cartan type but not finite.

   \medbreak
 \item $-\omega q^2 \in \G'_3$, that is $q = \pm \imath\,$ or $q = \pm \imath\,\omega^2$.
By Theorem \ref{thm:nichols-diagonal-finite-gkd} applied to $\cR^1\cR^2\cR^3 (U)$,
 since the subdiagram spanned by 2 and 3 is of Cartan type but not finite.

  \end{itemize}

It remains the following case:

 \begin{itemize} [leftmargin=*] \renewcommand{\labelitemi}{$\circ$}
 \item $q = \omega^2$. Then $\dim \toba (U) < \infty$ by  \cite[Table 2, row 8]{H-classif}, type super A.
\end{itemize}

We argue according to $\ghost_j$.
If $\ghost_j = 1$ and $\ghost_i = 0$, then $U = K_J$.
There are two more cases: ($\alpha$)  $\ghost_j > 1$: then $K_J$ contains a braided vector subspace $U'$
spanned by  $z_{i,0}$, $z_{j,0}$, $z_{j,1}$, $z_{j,2}$:  ($\beta$) $\ghost_j = 1 = \ghost_i$,
then $K_J$ is spanned by $z_{i,0}$, $z_{i,1}$, $z_{j,0}$, $z_{j,1}$. The respective diagrams are

\begin{align}\label{eq:tridente-omega}
& (\alpha) \quad\xymatrix{ &\overset{-1}{\circ}\ar  @{-}[d]_{\omega^{2}} &
\\ \overset{-1}{\circ} \ar  @{-}[r]^{\omega^{2}}  &\overset{\omega}{\circ} \ar  @{-}[r]^{\omega^{2}}  & \overset{-1}{\circ},}
& & (\beta) \quad \xymatrix{\overset{\omega}{\circ}\ar  @{-}[d]_{\omega^{2}} \ar  @{-}[rd]^{\omega^{2}} \ar  @{-}[r]^{\omega^{2}}
	& \overset{-1}{\circ}\ar  @{-}[d]^{\omega^{2}}
	\\ \overset{-1}{\circ} \ar  @{-}[r]^{\omega^{2}}  & \overset{\omega}{\circ}.}
\end{align}
Hence  $\GK \toba(V) = \infty$ in these cases by Hypothesis \ref{hyp:nichols-diagonal-finite-gkd}.

\bigbreak
\begin{enumerate}\renewcommand{\theenumi}{\alph{enumi}}\renewcommand{\labelenumi}{(\theenumi)}
\setcounter{enumi}{1}
\item
 $q_{ii} = \omega \in \G'_3$, $q_{jj} = -1$, $\ghost_j = 0$ and thus $\ghost_i = 1$.
\end{enumerate}

Then $K_J$ is spanned by the vertices $1 = z_{i,0}$, $2 = z_{i,1}$, $3 = z_{j,0}$,
and its diagram is the left-hand side in \eqref{eq:diagram-block-points-casoa}.
We apply the reflection at 3, that is

\begin{align}\label{eq:diagram-block-points-casoa}
\xymatrix{\overset{\omega}{\underset{2}{\circ}}\ar  @{-}[d]_{\omega^{2}} \ar  @{-}[rd]^{q} &
\\ \overset{\omega}{\underset{1}{\circ}} \ar  @{-}[r]^{q}  &
\overset{-1}{\underset{3}{\circ}}} &
&\xymatrix{\ar@/^1pc/@{<->}[rr]^{3} & & }  &
&\xymatrix{\overset{-\omega q}{\underset{2}{\circ}}\ar  @{-}[d]_{\omega^{2} q^2} \ar  @{-}[rd]^{q^{-1}} &
\\ \overset{-\omega q}{\underset{1}{\circ}} \ar  @{-}[r]^{q^{-1}}  &
\overset{-1}{\underset{3}{\circ}}.}
\end{align}

Then the following possibilities are excluded looking at $\cR^3(K_J)$:

 \medbreak
\begin{itemize} [leftmargin=*] 
 \item $-\omega q = 1$, i. e. $q = -\omega^2$. Then $\GK \toba (K_J) = \infty$ by Lemma \ref{lemma:points-trivial-braiding}.

 \medbreak
 \item $-\omega q \in \G'_N$, $N > 3$. Then $\GK \toba (K_J) = \infty$ by Theorem \ref{thm:nichols-diagonal-finite-gkd}.

 \medbreak
  \item $-\omega q \notin \G_{\infty}$. Then $\GK \toba (K_J) = \infty$ by Theorem \ref{thm:nichols-diagonal-finite-gkd}.

\item Either $-\omega q = \omega$, i. e.  $q = -1$; or else $-\omega q = \omega^2$, i. e. $q = -\omega$.
Here  $\GK \toba(V) = \infty$ by Hypothesis \ref{hyp:nichols-diagonal-finite-gkd}
because the braiding matrices are
\begin{align}\label{eq:triangle-omega}
\xymatrix{\overset{\omega}{\circ}\ar  @{-}[d]_{\omega^{2}} \ar  @{-}[rd]^{-1} &
\\ \overset{\omega}{\circ} \ar  @{-}[r]^{-1}  &
\overset{-1}{\circ},} & &\hspace{30pt} &
\xymatrix{\overset{\omega}{\circ}\ar  @{-}[d]_{\omega^{2}} \ar  @{-}[rd]^{-\omega} &
\\ \overset{\omega}{\circ} \ar  @{-}[r]^{-\omega}  &
\overset{-1}{\circ}.}
\end{align}
\end{itemize}

The last possibility gives a positive answer:

\begin{itemize} [leftmargin=*] \renewcommand{\labelitemi}{$\circ$}
	\item $-\omega q = -1$, i. e. $q = \omega^2$. Then $\dim \toba (K_J) < \infty$ by  \cite[Table 2, row 15]{H-classif}.
\end{itemize}

 \bigbreak
\begin{enumerate}\renewcommand{\theenumi}{\alph{enumi}}\renewcommand{\labelenumi}{(\theenumi)}
 \setcounter{enumi}{2}
  \item
 $q_{ii} = q_{jj} =\omega \in \G'_3$,  $\ghost_i = 1$.
\end{enumerate}

We consider the braided vector subspace $U$ of $K_J$ corresponding to the
subdiagram spanned by the vertices $1 = z_{i,0}$, $2 = z_{i,1}$, $3 = z_{j,0}$, that is the central diagram in
\eqref{eq:diagram-block-points-casoc}. We apply the reflections as described below:
\begin{align}\label{eq:diagram-block-points-casoc}
\xymatrix{\overset{\omega}{\underset{2}{\circ}}\ar  @{-}[d]_{\omega^{2}} \ar  @{-}[rd]^{q^2\omega^2} &
\\ \overset{\omega}{\underset{1}{\circ}} \ar  @{-}[r]^{q^{-1}\omega^2 \quad}  &
\overset{q^2\omega^2}{\underset{3}{\circ}}} &
&\xymatrix{\ar@/^1pc/@{<->}[r]^{1}_{\#}  & } &
&\xymatrix{\overset{\omega}{\underset{2}{\circ}}\ar  @{-}[d]_{\omega^{2}} \ar  @{-}[rd]^{q} &
\\ \overset{\omega}{\underset{1}{\circ}} \ar  @{-}[r]^{q}  &
\overset{\omega}{\underset{3}{\circ}}} &
&\xymatrix{\ar@/^1pc/@{<->}[r]^{3}_{\#}  & } &
&\xymatrix{\overset{\omega^2 q^2}{\underset{2}{\circ}}\ar  @{-}[d]_{\omega q^4} \ar  @{-}[rd]^{q^{-1}\omega^2} &
\\ \overset{\omega^2 q^2}{\underset{1}{\circ}} \ar  @{-}[r]^{q^{-1} \omega^2}  &
\overset{\omega}{\underset{3}{\circ}}.}
\end{align}
$\#$: only for $q\notin \G'_3$.
Many possibilities are excluded as we discuss now:

 \medbreak
\begin{itemize} [leftmargin=*] 
 \item $\omega^2 q^2 = 1$, i. e. $q = -\omega^2$. Then $\GK \toba (\cR^3(U)) = \infty$ by Lemma \ref{lemma:points-trivial-braiding}.

 \medbreak
 \item $\omega^2 q^2 = -1$, i. e. $q = \pm \imath\omega^2$. Then $\cR^3(U)$ has diagram
\begin{align*}
\xymatrix{\overset{-1}{\underset{2}{\circ}} \ar  @{-}[rd]^{\mp \imath} &
\\ \overset{-1}{\underset{1}{\circ}} \ar  @{-}[r]^{\mp \imath}  &
\overset{\omega}{\underset{3}{\circ}}}
\end{align*}
that appeared in \eqref{eq:diagram-block-points-casob}; hence $\GK \toba (U) = \infty$ by case \eqref{case:diagram-block-points-casob}.

\medbreak
 \item $\omega^2 q^2 = \omega$, i. e. $q = -\omega$ (the possibility $q=\omega$ is discussed separately).
 Then $\cR^1(U)$  contains a subdiagram of type $A_1^{(1)}$, since its  diagram is
\begin{align*}
\xymatrix{\overset{\omega}{\circ} \ar  @{-}[rd]^{\omega}\ar  @{-}[d]^{\omega^2} &
\\ \overset{\omega}{\circ} \ar  @{-}[r]^{-\omega}  &
\overset{\omega}{\circ}.}
\end{align*}
Hence  $\GK \toba (U) = \infty$ by Theorem \ref{thm:nichols-diagonal-finite-gkd}.

 \medbreak
 \item $\omega^2 q^2 \in \G'_N$, $N > 3$. Then $\GK \toba (\cR^3(U)) = \infty$ by Theorem \ref{thm:nichols-diagonal-finite-gkd}.

 \medbreak
  \item $\omega^2 q^2 \notin \G_{\infty}$. Then $\GK \toba (\cR^3(U)) = \infty$ by Theorem \ref{thm:nichols-diagonal-finite-gkd}.

 \medbreak
 \item If $q = \omega$, respectively $q = \omega^2$, then $U$ is of type $A_2^{(1)}$, respectively $U$ has a subdiagram of  type $A_1^{(1)}$, hence  $\GK \toba (U) = \infty$ by  Theorem \ref{thm:nichols-diagonal-finite-gkd}.

\medbreak
 \item $\omega^2 q^2 = \omega^2$, i. e. $q = -1$ (we assumed $q\neq 1$).
The braiding matrix is
\begin{align}\label{eq:triangle-omega-bis}
\xymatrix{\overset{\omega}{\circ}\ar  @{-}[d]_{\omega^{2}} \ar  @{-}[rd]^{-1} &
\\ \overset{\omega}{\circ} \ar  @{-}[r]^{-1}  &
\overset{\omega}{\circ}.}
\end{align}

\end{itemize}
The subdiagram $\xymatrix{\overset{\omega}{\circ} \ar  @{-}[r]^{-1}  &
\overset{\omega}{\circ}}$ has infinite root system, hence  $\GK \toba(V) = \infty$ by  Theorem \ref{thm:nichols-diagonal-finite-gkd}.

 \bigbreak
\begin{enumerate}\renewcommand{\theenumi}{\alph{enumi}}\renewcommand{\labelenumi}{(\theenumi)}
 \setcounter{enumi}{3}
  \item
 $q_{ii} = \omega \in \G'_3$, $q_{jj} = \omega^2$, $\ghost_i = 1$.
\end{enumerate}

We consider the braided vector subspace $U$ corresponding to the
subdiagram spanned by the vertices $1 = z_{i,0}$, $2 = z_{i,1}$, $3 = z_{j,0}$, that is the left-hand diagram in
\eqref{eq:diagram-block-points-casod}. We apply the reflections as described below:
\begin{align}\label{eq:diagram-block-points-casod}
\xymatrix{\overset{\omega}{\underset{2}{\circ}}\ar  @{-}[d]_{\omega^{2}} \ar  @{-}[rd]^{q} &
\\ \overset{\omega}{\underset{1}{\circ}} \ar  @{-}[r]^{q}  &
\overset{\omega^2}{\underset{3}{\circ}}} &
&\xymatrix{\ar@/^1pc/@{<->}[rr]^{3}_{\#} & & }  &
&\xymatrix{\overset{q^2}{\underset{2}{\circ}}\ar  @{-}[d]_{q^4} \ar  @{-}[rd]^{q^{-1}\omega} &
\\ \overset{q^2}{\underset{1}{\circ}} \ar  @{-}[r]^{q^{-1} \omega}  &
\overset{\omega^2}{\underset{3}{\circ}}.}
\end{align}
$\#$: only for $q\notin \G'_3$.
Many possibilities are excluded as we discuss now:

 \medbreak
\begin{itemize} [leftmargin=*] 
 \item $q = -1$. Then $\GK \toba (\cR^3(U)) = \infty$ by Lemma \ref{lemma:points-trivial-braiding}.

  \medbreak
 \item $q^2 \in \G'_N$, $N > 3$. Then $\GK \toba (\cR^3(U)) = \infty$ by Theorem \ref{thm:nichols-diagonal-finite-gkd}.

 \medbreak
  \item $q^2 \notin \G_{\infty}$. Then $\GK \toba (\cR^3(U)) = \infty$ by Theorem \ref{thm:nichols-diagonal-finite-gkd}.

\medbreak
 \item $q^2 = -1$. Then $\cR^3(U)$ has diagram
\begin{align*}
\xymatrix{\overset{-1}{\underset{2}{\circ}} \ar  @{-}[rd]^{\mp \imath \omega} &
\\ \overset{-1}{\underset{1}{\circ}} \ar  @{-}[r]^{\mp \imath \omega}  &
\overset{\omega^2}{\underset{3}{\circ}}}
\end{align*}
that appeared in \eqref{eq:diagram-block-points-casob}; hence $\GK \toba (U) = \infty$ by case \eqref{case:diagram-block-points-casob}.

\medbreak \item $q = \pm\omega,  \pm\omega^2$.
The braiding matrices are respectively
\begin{align}\label{eq:triangle-omega-tris}
\xymatrix{\overset{\omega}{\circ}\ar  @{-}[d]_{\omega^{2}} \ar  @{-}[rd]^{\pm\omega} &
\\ \overset{\omega}{\circ} \ar  @{-}[r]^{\pm\omega}  &
\overset{\omega^{2}}{\circ},} & &\hspace{25pt} &
\xymatrix{\overset{\omega}{\circ}\ar  @{-}[d]_{\omega^{2}} \ar  @{-}[rd]^{\pm\omega^{2}} &
\\ \overset{\omega}{\circ} \ar  @{-}[r]^{\pm\omega^{2}}  &
\overset{\omega^{2}}{\circ}.}
\end{align}

If $q \in \G'_3$, then these matrices are of indefinite Cartan type; by Hypothesis \ref{hyp:nichols-diagonal-finite-gkd}, $\GK \toba (U) = \infty$.
If $-q \in \G'_3$, then the subdiagram $\xymatrix{\overset{\omega}{\circ} \ar  @{-}[r]^{q}  &
	\overset{\omega^2}{\circ}}$ has infinite root system, hence  $\GK \toba(V) = \infty$ by  Theorem \ref{thm:nichols-diagonal-finite-gkd}.

\end{itemize}

\bigbreak
\begin{enumerate}\renewcommand{\theenumi}{\alph{enumi}}\renewcommand{\labelenumi}{(\theenumi)}
 \setcounter{enumi}{4}
  \item
 $q_{ii} = \omega \in \G'_3$, $q_{jj} \notin \G_2 \cup \G_3$, thus $\ghost_j = 0$ and  $\ghost_i = 1$.
\end{enumerate}

We set $r = q_{jj}$. We consider the braided vector subspace $U$ corresponding to the subdiagram spanned by the vertices $1 = z_{i,0}$, $2 = z_{i,1}$, $3 = z_{j,0}$. The braiding matrix is
\begin{align}\label{eq:triangle-omega-general roots}
\xymatrix{\overset{\omega}{\circ}\ar  @{-}[d]_{\omega^{2}} \ar  @{-}[rd]^{q} &
	\\ \overset{\omega}{\circ} \ar  @{-}[r]^{q}  &
	\overset{r}{\circ}.}
\end{align}

 \medbreak
\begin{itemize} [leftmargin=*] 
 \item Either $q \notin\G_{\infty}$ or $r \notin\G_{\infty}$.
 Here $\GK \toba (U) = \infty$ by Theorem \ref{thm:nichols-diagonal-finite-gkd}.
 Indeed, if $r\neq q^{-1}$, then we consider the subdiagram $\xymatrix{\overset{\omega}{\circ} \ar  @{-}[r]^{q}  &
 	\overset{r}{\circ}}$; while for $r= q^{-1}$, $\GK \toba (\cR^1(U)) = \infty$, cf. the subdiagram $\xymatrix{\overset{\omega}{\circ} \ar  @{-}[r]^{q^2\omega}  &
 	\overset{\omega q}{\circ}}$:
\begin{align}\label{eq:diagram-block-points-casoe-semigeneric}
\xymatrix{\overset{\omega}{\underset{2}{\circ}}\ar  @{-}[d]_{\omega^{2}} \ar  @{-}[rd]^{q} &
\\ \overset{\omega}{\underset{1}{\circ}} \ar  @{-}[r]^{q}  &
\overset{q^{-1}}{\underset{3}{\circ}}} &
&\xymatrix{\ar@/^1pc/@{<->}[rr]^{1} & & }  &
&\xymatrix{\overset{\omega}{\underset{2}{\circ}}\ar  @{-}[d]_{\omega^2} \ar  @{-}[rd]^{q^{2}\omega} &
\\ \overset{\omega}{\underset{1}{\circ}} \ar  @{-}[r]_{\omega^2 q^{-1}}  &
\overset{\omega q}{\underset{3}{\circ}}.}
\end{align}

\medbreak
 \item $q, r \in\G_{\infty}$, $r\notin \G_2\cup\G_3$.
Here $\GK \toba (U) = \infty$ by Hypothesis \ref{hyp:nichols-diagonal-finite-gkd}.
\end{itemize}

\bigbreak
\begin{enumerate}\renewcommand{\theenumi}{\alph{enumi}}\renewcommand{\labelenumi}{(\theenumi)}
 \setcounter{enumi}{5}
  \item\label{case:diagram-block-points-casof}
 $q_{ii} = -1$, $ q_{jj} \neq 1$. We omit $ q_{jj} \in \G_3$ because this was considered before.
 If $q_{jj} \neq - 1$, $\ghost_j = 0$ hence $\ghost_i \geq 1$. So, we can assume always $\ghost_i \geq 1$ without loss of generality.
\end{enumerate}

We set $r = q_{jj}$. We consider the braided vector subspace $U$ corresponding to the
subdiagram spanned by the vertices $1 = z_{i,0}$, $2 = z_{i,1}$, $3 = z_{j,0}$, that is the central diagram in
\eqref{eq:diagram-block-points-casof}. We apply the reflections as described below:
\begin{align}\label{eq:diagram-block-points-casof}
\xymatrix{\overset{-1}{\underset{2}{\circ}} \ar  @{-}[rd]^{q} &
\\ \overset{-1}{\underset{1}{\circ}} \ar  @{-}[r]^{q^{-1}}  &
\overset{-qr}{\underset{3}{\circ}}}
&
&\xymatrix{\ar@/^1pc/@{<->}[r]^{1}  & }  &
&
\xymatrix{\overset{-1}{\underset{2}{\circ}} \ar  @{-}[rd]^{q} &
\\ \overset{-1}{\underset{1}{\circ}} \ar  @{-}[r]^{q}  &
\overset{r}{\underset{3}{\circ}}} &
&\xymatrix{\ar@/^1pc/@{<->}[r]^{3}_{\#}  & }  &
&\xymatrix{\overset{-q^{N-1}r}{\underset{2}{\circ}}\ar  @{-}[d]_{q^{2(N-1)}r^2} \ar  @{-}[rd]^{q^{-1} r^2} &
\\ \overset{-q^{N-1}r}{\underset{1}{\circ}} \ar  @{-}[r]_{\quad q^{-1} r^2}  &
\overset{r}{\underset{3}{\circ}}.}
\end{align}
$\#$: only for $r \in \G'_N$, $q\notin r^{\Z}$.
There are some positive answers:

 \begin{itemize} [leftmargin=*] \renewcommand{\labelitemi}{$\circ$}
 \item $q = r = -1$, $(\ghost_i, \ghost_j) = (1, 0)$ or $(2,0)$. Then $K_J$ is of Cartan type $A_3$ or $D_4$ respectively,
 hence $\dim \toba(K_J) < \infty$.

\item $q \in \G'_3$, $r = -1$, $(\ghost_i, \ghost_j) = (1, 0)$;  $\dim \toba(K_J) < \infty$ by   \cite[Table 2, row 15]{H-classif}.
It can be shown that the root system of this Nichols algebra is isomorphic to that of the modular Lie superalgebra $\mathfrak g(2,3)$.

\item $r \notin \G_2 \cup \G_3$, $q  = r^{-1}$. Then $\ghost_j = 0$.
If $\ghost_i = 1$, then $K_J$ is of super type $A$, hence $\dim \toba(K_J) < \infty$ when $r\in \G_{\infty}$,
while $\GK \toba(K_J) = 2$ when $r\notin \G_{\infty}$,
see   \cite[Table 2, row 9]{H-classif}.

\end{itemize}

 \medbreak
\begin{itemize} [leftmargin=*] 

 \item $q = r = -1$, $(\ghost_i, \ghost_j) \notin\{ (1, 0),(2,0)\}$. Then $K_J$
 has a subdiagram of affine Cartan type $A_3^{(1)}$ or $D_4^{(1)}$,
 thus $\GK \toba(K_J) = \infty$ by Theorem \ref{thm:nichols-diagonal-finite-gkd}.

\item $q \in \G'_3$, $r = -1$, $(\ghost_i, \ghost_j) \neq (1, 0)$. Then  $K_J$
 contains a braided subspace $U'$ with diagram of either of the following types:
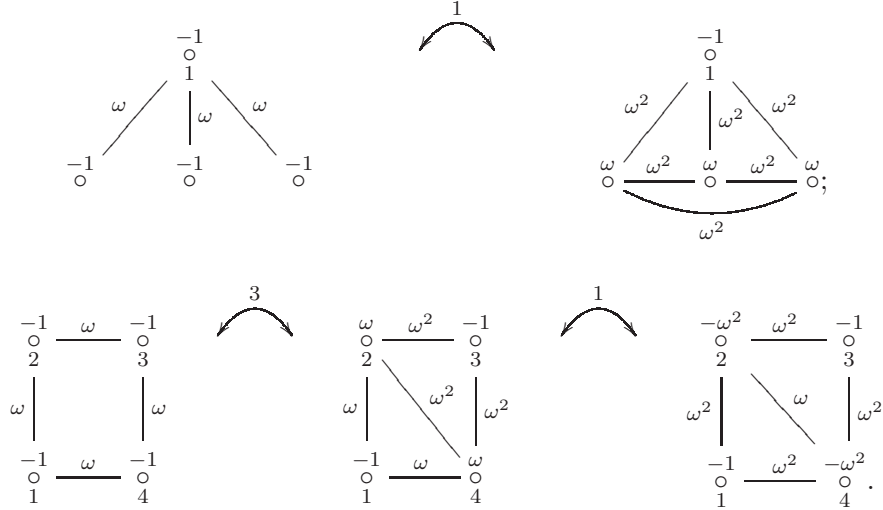
\begin{align*}
\xymatrix{&\overset{-1}{\underset{1}{\circ}} \ar  @{-}[rd]^{\omega}\ar  @{-}[ld]_{\omega}\ar  @{-}[d]^{\omega} &
\\ \overset{-1} {\circ}  & \overset{-1} {\circ}  &	 \overset{-1}{\circ}}
&
&\xymatrix{\ar@/^1pc/@{<->}[r]^{1}  & }  &
&
\xymatrix{&\overset{-1}{\underset{1}{\circ}} \ar  @{-}[rd]^{\omega^2}\ar  @{-}[ld]_{\omega^2}\ar  @{-}[d]^{\omega^2} &
\\ \overset{\omega} {\circ} \ar  @{-}[r]^{\omega^2} \ar@/_1pc/  @{-}[rr]_{\omega^2} & \overset{\omega} {\circ} \ar  @{-}[r]^{\omega^2}
&\overset{\omega}{\circ};}
\end{align*}
\begin{align*}
\xymatrix{\overset{-1}{\underset{2}{\circ}} \ar  @{-}[r]^{\omega}\ar  @{-}[d]_{\omega} & \overset{-1}{\underset{3}{\circ}} \ar  @{-}[d]^{\omega}
\\ \overset{-1} {\underset{1}{\circ}} \ar  @{-}[r]^{\omega} & \overset{-1} {\underset{4}{\circ}} }
&
&\xymatrix{\ar@/^1pc/@{<->}[r]^{3}  & }  &
&
\xymatrix{\overset{\omega}{\underset{2}{\circ}} \ar  @{-}[r]^{\omega^2}\ar  @{-}[rd]^{\omega^2}\ar  @{-}[d]_{\omega}
& \overset{-1}{\underset{3}{\circ}} \ar  @{-}[d]^{\omega^2}
\\ \overset{-1} {\underset{1}{\circ}} \ar  @{-}[r]^{\omega} & \overset{\omega} {\underset{4}{\circ}} }
&
&\xymatrix{\ar@/^1pc/@{<->}[r]^{1}  & }  &
&
\xymatrix{\overset{-\omega^2}{\underset{2}{\circ}} \ar  @{-}[r]^{\omega^2}\ar  @{-}[rd]^{\omega}\ar  @{-}[d]_{\omega^2}
& \overset{-1}{\underset{3}{\circ}} \ar  @{-}[d]^{\omega^2}
\\ \overset{-1} {\underset{1}{\circ}} \ar  @{-}[r]^{\omega^2} & \overset{-\omega^2} {\underset{4}{\circ}}. }
\end{align*}
In the first case, $\cR^1(U')$ contains a subdiagram of affine Cartan type $A_2^{(1)}$, hence
$\GK \toba(V_J) = \infty$ by Theorem \ref{thm:nichols-diagonal-finite-gkd}. In the second, the diagonal of
$\cR^1\cR^3(U')$ has infinite $\GK$ by  Theorem \ref{thm:nichols-diagonal-finite-gkd}.

\medbreak
\item $q \notin \G_2 \cup \G_3$, $r = -1$; then
$\GK \toba (\cR^3(U)) = \infty$ by Theorem \ref{thm:nichols-diagonal-finite-gkd}.

\medbreak
\item $r \notin \G_2 \cup \G_3$, $q  = r^{-1}$, $\ghost_j = 0$, $\ghost_i \geq 1$.
Then  $K_J$  contains a braided subspace $U'$ with diagram of the following type:
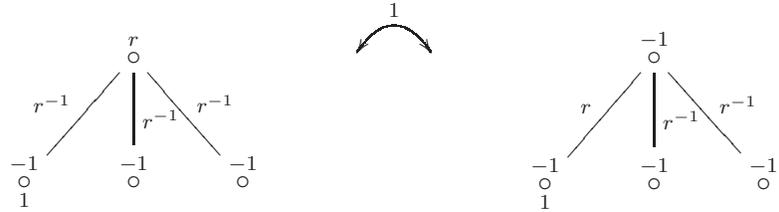
\begin{align*}
\xymatrix{&\overset{r} {\circ} \ar  @{-}[rd]^{r^{-1}}\ar  @{-}[ld]_{r^{-1}}\ar  @{-}[d]^{r^{-1}} &
\\ \overset{-1} {\underset{1}{\circ}}  & \overset{-1} {\circ}  &	 \overset{-1}{\circ}}
&
&\xymatrix{\ar@/^1pc/@{<->}[r]^{1}  & }  &
&
\xymatrix{&\overset{-1} {\circ} \ar  @{-}[rd]^{r^{-1}}\ar  @{-}[ld]_{r}\ar  @{-}[d]^{r^{-1}} &
\\ \overset{-1} {\underset{1}{\circ}}  & \overset{-1} {\circ}  &	 \overset{-1}{\circ}.}
\end{align*}
Erasing the vertex 1 of $\cR^1(U')$ we are in the previous case, therefore  $\GK \toba(K_J) = \infty$.

\medbreak
\item $r \notin \G_{\infty}$, $q  \neq r^{-1}$. Then  $\GK \toba(K_J) = \infty$ by  Theorem \ref{thm:nichols-diagonal-finite-gkd} (applied to $\cR^1(U)$
when $q\in r^{-\N}$) or Theorem \ref{thm:nichols-diagonal-finite-gkd}.

\medbreak
\item $r \in \G_{\infty}$ and $q  \notin \G_{\infty}$. Then  $\GK \toba(K_J) = \infty$ by  Theorem \ref{thm:nichols-diagonal-finite-gkd}
applied to $\cR^1(U)$.

\medbreak
\item $q, r \in\G_{\infty}$, $r\notin \G_2\cup\G_3$, $r\neq q^{-1}$.
The braiding matrix is
\begin{align}\label{eq:casof-general roots}
\xymatrix{\overset{-1}{\circ}  \ar  @{-}[rd]^{q} &
\\ \overset{-1}{\circ} \ar  @{-}[r]^{q}  &
\overset{r}{\circ}.}
\end{align}\end{itemize}
Here $\GK \toba (U) = \infty$ by Hypothesis \ref{hyp:nichols-diagonal-finite-gkd}.
\epf

\subsubsection{$\vert J \vert  > 2$}\label{subsub:J>2}
We now conclude the consideration of $J\in \X$ with $\ghost_J \neq 0$. Recall that $\epsilon=1$ and $\omega \in \G'_3$.

\begin{lemma}\label{lemma:KJ-card-geq3}
Assume that $\vert J \vert  > 2$.
Then $\GK \NA (K_J)$ is finite if and only if it appears in Table \ref{tab:finiteGK-block-points}.
\end{lemma}

\pf
We proceed  by induction on $\vert J \vert$.

\begin{caso2}\label{caso2:a}
$\vert J \vert  =3$. Say $J = \{i, j, k \}$; set $q:= q_{ij}q_{ji} \neq 1$,
$r:= q_{jk}q_{kj} \neq 1$, $u = q_{ik}q_{ki}$.
\end{caso2}

By Lemma \ref{lemma:KJ-card2} for each pair of connected vertices, at most one of them has positive ghost, while the other has ghost $0$.  Also every connected Dynkin subdiagram of rank 2 with non-zero ghost, appears in Table \ref{tab:finiteGK-block-points}.

  \medbreak
 \begin{enumerate}[leftmargin=*]\renewcommand{\theenumi}{\alph{enumi}}\renewcommand{\labelenumi}{\theenumi.}
  \item\label{case:diagram-block-2points-casob}
 $\ghost_j > 0$. Thus $\ghost_j \in\{ 1,2 \}$, $\ghost_i =\ghost_k =0$.
\end{enumerate}

Let $U$ be the braided vector subspace spanned by $1 = z_{i,0}$, $2 = z_{j,0}$, $3 = z_{j,1}$, $4 = z_{k,0}$. Let $t = q_{jj}^2$.  Then $\GK \toba (U) = \infty$ by Hypothesis \ref{hyp:nichols-diagonal-finite-gkd},  since
there is no square in the list in \cite{H-classif} and the Dynkin
 diagram of $U$ is
\begin{align*}
\xymatrix{ \overset{q_{ii}} {\circ} \ar  @{-}[r]^{q} \ar  @{-}[d]_{q}
\ar @{-}[rd]_(.3){u}	& \overset{q_{jj}} {\circ} \ar  @{-}[d]^{r}
\ar  @{-}[ld]^(.3){t}|\hole
\\ \overset{q_{jj}} {\circ} \ar  @{-}[r]_{r}  & \overset{q_{kk}} {\circ}.}
\end{align*}

This case excluded, we may assume that $\ghost_i > 0$ (otherwise $\ghost_k > 0$ but then we exchange $i$ and $k$) and $u = 1$ (otherwise we exchange $i$ and $j$ and reduce to case \ref{case:diagram-block-2points-casob}).

  \medbreak
 \begin{enumerate}[leftmargin=*]\renewcommand{\theenumi}{\alph{enumi}}\renewcommand{\labelenumi}{\theenumi.}
 \setcounter{enumi}{1}
  \item\label{case:diagram-block-3points-casoc}
 $q_{ii} = \omega$. Thus $\ghost_i=1$, $\ghost_j=0$, $q=\omega^2$, $q_{jj}=-1$.
\end{enumerate}
Set $s=q_{kk}$. Let $U$ be the braided vector subspace spanned by $1 = z_{i,0}$, $2 = z_{i,1}$, $3 = z_{j,0}$, $4 = z_{k,0}$.
Its generalized Dynkin diagram is
\begin{align}\label{eq:3pt-case1c}
\xymatrix{  & \overset{\omega}{\circ} \ar  @{-}[d]^{\omega^2} \ar  @{-}[dl]_{\omega^2} &  \\
\overset{\omega}{\circ}\ar  @{-}[r]^{\omega^2}  & \overset{-1}{\circ} \ar  @{-}[r]^{r}  & \overset{s}{\circ}.}
\end{align}

 By Hypothesis \ref{hyp:nichols-diagonal-finite-gkd},  $\GK \toba(V) = \infty$, since no diagram with 4 points like this appears in the list in \cite{H-classif}.

 \medbreak
 So, we may assume that $\ghost_i > 0$, $u = 1$ and $q_{ii} = -1$.

 \medbreak
\begin{enumerate}[leftmargin=*]\renewcommand{\theenumi}{\alph{enumi}}\renewcommand{\labelenumi}{\theenumi.}
 \setcounter{enumi}{2}
  \item\label{case:diagram-block-3points-casod}
 $\ghost_k > 0$. By case \ref{case:diagram-block-3points-casoc} we may assume $q_{kk}=-1$.
\end{enumerate}
Let $U$ be the braided vector subspace spanned by $z_{i,0}$, $z_{i,1}$, $z_{j,0}$, $z_{k,0}$, $z_{k,1}$.
Its generalized Dynkin diagram is
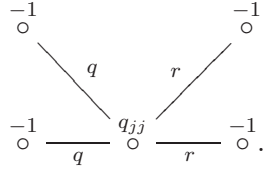
\begin{align*}
\xymatrix{ \overset{-1}{\circ} \ar  @{-}[dr]^{q} &  & \overset{-1}{\circ} \ar  @{-}[dl]_{r} \\
\overset{-1}{\circ}\ar  @{-}[r]_{q}  & \overset{q_{jj}}{\circ} \ar  @{-}[r]_{r}  & \overset{-1}{\circ}.}
\end{align*}
 By Hypothesis \ref{hyp:nichols-diagonal-finite-gkd},  $\GK \toba(V) = \infty$, since no diagram with 5 points like this appears in the list in \cite{H-classif}.

 \medbreak
\begin{enumerate}[leftmargin=*]\renewcommand{\theenumi}{\alph{enumi}}\renewcommand{\labelenumi}{\theenumi.}
 \setcounter{enumi}{3}
  \item\label{case:diagram-block-3points-casoe}
  So, we may assume that $\ghost_i > 0$, $\ghost_j=\ghost_k=0$, $u = 1$, $q_{ii} = -1$.
\end{enumerate}
Let $U$ be the braided vector subspace spanned by $1 = z_{i,0}$, $2 = z_{i,1}$, $3 = z_{j,0}$, $4 = z_{k,0}$.
There are some cases with $\dim \toba(K_J) < \infty$:

\begin{itemize} [leftmargin=*] \renewcommand{\labelitemi}{$\circ$}
\item $q = q_{jj} = r = s = -1$, $\ghost_i = 1$. Then $K_J$ is of Cartan type $D_4$.

\item $q = q_{jj}^{-1} \in \G'_3$, $r = s^{-1}=q^{\pm 1}$, $\ghost_i= 1$; this is \cite[Table 3, row 18]{H-classif}.
\end{itemize}

The remaining possibilities are excluded as we discuss now:

 \medbreak
\begin{itemize} [leftmargin=*] 
 \item $q = q_{jj} = r = s = -1$, $\ghost_i = 2$. Then $K_J$ is of Cartan type $D_4^{(1)}$, hence $\GK \toba(K_J) = \infty$
 by Theorem \ref{thm:nichols-diagonal-finite-gkd}.

 \medbreak
 \item $q \notin \G_{\infty}$. We may assume $q_{jj}=q^{-1}$ by Lemma \ref{lemma:KJ-card2}  and $r=q^{-m}$
 for some $m\in\N$ by Theorem \ref{thm:nichols-diagonal-finite-gkd}. Then $\cR^1 U$ and $\cR^3\cR^1 U$ have diagrams
 \begin{align*}
& 
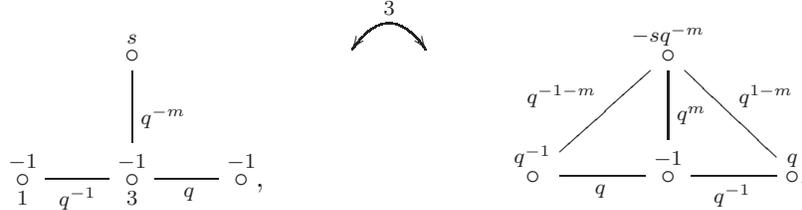
\xymatrix{ & \overset{s}{\circ} \ar  @{-}[d]^{q^{-m}} &   \\
\overset{-1}{\underset{1}{\circ}}\ar  @{-}[r]_{q^{-1}}  & \overset{-1}{\underset{3}{\circ}} \ar  @{-}[r]_{q}  & \overset{-1}{\circ},}
&
&\xymatrix{\ar@/^1pc/@{<->}[r]^{3} & }
&
& \xymatrix{ & \overset{-sq^{-m}}{\circ} \ar  @{-}[d]^{q^{m}}  \ar  @{-}[dl]_{q^{-1-m}}  \ar  @{-}[dr]^{q^{1-m}}  &   \\
\overset{q^{-1}}{\circ}\ar  @{-}[r]_{q}  & \overset{-1}{\circ} \ar  @{-}[r]_{q^{-1}}  & \overset{q}{\circ}.}
\end{align*}
 Thus $\GK \toba (\cR^3\cR^1(U)) = \infty$ by Theorem \ref{thm:nichols-diagonal-finite-gkd}.

 \medbreak
 \item $q_{jj}\neq-1$, either $r\notin\G_\infty$ or $s\notin\G_\infty$. We may assume $q\in\G_\infty$ by the previous step; moreover,
 $q_{jj}=q^{-1}$ by Lemma \ref{lemma:KJ-card2}, and $q=\omega$, $rs=1$ by Theorem \ref{thm:nichols-diagonal-finite-gkd}. Then $\GK \toba (\cR^3\cR^1(U)) =
 \infty$ by Theorem \ref{thm:nichols-diagonal-finite-gkd}.

 \medbreak
 \item $q_{jj}=-1$, $q=\omega$, either $r\notin\G_\infty$ or $s\notin\G_\infty$. As $\cR^1 U$ has the vertex 3 labeled with $\omega$,
 we may assume $rs=1$ by Theorem \ref{thm:nichols-diagonal-finite-gkd}. Then $\GK \toba (\cR^3(U)) = \infty$ by Theorem \ref{thm:nichols-diagonal-finite-gkd}.

 \medbreak
 \item $q_{jj}=q=-1$, either $r\notin\G_\infty$ or $s\notin\G_\infty$. We may assume $r=s^{-n}$, $n\in\N$,
 by Theorem \ref{thm:nichols-diagonal-finite-gkd}. Then $\GK \toba (\cR^1\cR^3(U)) = \infty$ by Theorem \ref{thm:nichols-diagonal-finite-gkd}.

\medbreak
 \item Either $q=q_{jj}^{-1},r,s\in\G_\infty$, or else $q = \omega$, $q_{jj}=-1$, $r,s\in\G_\infty$.
\begin{align}\label{eq:3pt-case1e}
& \xymatrix{ & \overset{-1}{\circ} \ar  @{-}[d]^{q} &   \\
\overset{-1}{\circ}\ar  @{-}[r]_{q}  & \overset{q^{-1}}{\circ} \ar  @{-}[r]_{r}  & \overset{s}{\circ},}
&
& \xymatrix{ & \overset{-1}{\circ} \ar  @{-}[d]^{\omega} &   \\
\overset{-1}{\circ}\ar  @{-}[r]_{\omega}  & \overset{-1}{\circ} \ar  @{-}[r]_{r}  & \overset{s}{\circ}.}
\end{align}
By Hypothesis \ref{hyp:nichols-diagonal-finite-gkd},  $\GK \toba(U) = \infty$.

\end{itemize}

\bigbreak

\begin{caso2} $\vert J \vert  =4$. Say $J = \{i, j, k,\ell \}$.
\end{caso2}

Assume $\ghost_i>0$ and $q_{ij}q_{ji}\neq 1$. By the previous Step, $\ghost_i=1$, $\ghost_j=0$, $q_{ik}q_{ki}=q_{i\ell}q_{\ell i}=1$,
$q_{ii}=-1$ and $q_{ij}q_{ji}=q_{jj}^{-1}\in\{-1,\omega\}$. We may assume $q_{jk}q_{kj}\neq 1$ since $J$ is connected,
so the subdiagram corresponding to $i,j,k$ appears in Table \ref{tab:finiteGK-block-points}; in particular $\ghost_k =0$.
Let $U$ be the subspace of
$K_J$ spanned by the vertices $1 = z_{i,0}$, $2 = z_{i,1}$, $3 = z_{j,0}$, $4 = z_{k,0}$, $5 = z_{\ell,0}$.

\bigbreak
 \begin{enumerate}\renewcommand{\theenumi}{\alph{enumi}}\renewcommand{\labelenumi}{(\theenumi)}
  \item\label{case:diagram-block-4points-casoa}
 $q_{j\ell}q_{\ell j}\neq 1$. Thus $U$ has a diagram of the following shape:
\end{enumerate}

\begin{align*}
\xymatrix{ \overset{-1}{\circ} \ar  @{-}[dr]^{q_{jj}^{-1}} &  & \overset{q_{\ell\ell}}{\circ} \ar  @{-}[dl]_{r} \ar@{-}[d]^{t} \\
	\overset{-1}{\circ}\ar  @{-}[r]_{q_{jj}^{-1}}  & \overset{q_{jj}}{\circ} \ar  @{-}[r]_{s}  & \overset{q_{kk}}{\circ},}
\end{align*}
where $r\neq 1$, $s \neq 1$. Then $\GK \toba(K_J) = \infty$ by Hypothesis \ref{hyp:nichols-diagonal-finite-gkd}, since there is no diagram in \cite[Table 4]{H-classif} like this.

 \bigbreak
 \begin{enumerate}\renewcommand{\theenumi}{\alph{enumi}}\renewcommand{\labelenumi}{(\theenumi)}
 \setcounter{enumi}{1}
  \item\label{case:diagram-block-4points-casob}
 $q_{j\ell}q_{\ell j} =1$. Thus  $q_{k\ell}q_{\ell k}\neq 1$.
\end{enumerate}

There is one case with $\dim \toba(K_J) < \infty$:

\begin{itemize} [leftmargin=*] \renewcommand{\labelitemi}{$\circ$}
\item $q_{jj} = q_{kk} = q_{\ell\ell} = q_{ij}q_{ji} =  q_{jk}q_{kj} = q_{k\ell}q_{\ell k} = -1$, $\ghost_\ell=0$.
Then $U = K_J$ is of Cartan type $D_5$.
\end{itemize}

\medskip
In the remaining cases $\GK \toba(K_J) = \infty$.

\medbreak
\begin{itemize} [leftmargin=*] 
 \item $q_{jj} = q_{kk} = q_{\ell\ell} = q_{ij}q_{ji} =  q_{jk}q_{kj} = q_{k\ell}q_{\ell k} = -1$, $\ghost_\ell=1$.
 Then $K_J$ is of Cartan type $D_5^{(1)}$, so Theorem \ref{thm:nichols-diagonal-finite-gkd} applies.

 \medbreak\item $q_{jj} = q_{kk} = q_{ij}q_{ji} =  q_{jk}q_{kj} = -1$, either $q_{k\ell}q_{\ell k}\notin \G_\infty$ or $q_{\ell\ell}\notin \G_\infty$.
 If $q_{k\ell}q_{\ell k}\neq q_{\ell\ell}^{-1}, q_{\ell\ell}^{-2}$, then apply Theorem \ref{thm:nichols-diagonal-finite-gkd}; otherwise \emph{ditto}  to $\cR^4(U)$.

 \medbreak\item $q_{jj} = \omega^2$, $q_{ij}q_{ji}=\omega$, $q_{jk}q_{kj}=q_{kk}^{-1}=\omega^{\pm 1}$, either $q_{k\ell}q_{\ell k}\notin \G_\infty$
 or $q_{\ell\ell}\notin \G_\infty$. The argument is similar to the previous case.
\end{itemize}

\medbreak
Set $q=q_{k\ell}q_{\ell k}$, $r=q_{\ell\ell}$.

\begin{itemize} [leftmargin=*]
 \item $q_{jj} = q_{kk} = q_{ij}q_{ji} =  q_{jk}q_{kj} = -1$, $q, r \in \G_\infty$; or

 \item $q_{jj} = \omega^2$, $q_{ij}q_{ji}=\omega$, $q_{jk}q_{kj}=q_{kk}^{-1}=\omega^{\pm 1}$, $q, r \in \G_\infty$. The Dynkin diagrams are
\begin{align}\label{eq:4pt-case2b}
& \xymatrix{  & \overset{-1}{\circ} \ar  @{-}[d]^{-1} &  &  \\
\overset{-1}{\circ}\ar  @{-}[r]_{-1}  & \overset{-1}{\circ} \ar  @{-}[r]_{-1}  & \overset{-1}{\circ}
\ar  @{-}[r]_{q}  & \overset{r}{\circ} }
&
& \xymatrix{  & \overset{-1}{\circ} \ar  @{-}[d]^{\omega} &  &  \\
\overset{-1}{\circ}\ar  @{-}[r]_{\omega}  & \overset{\omega^2}{\circ} \ar  @{-}[r]_{\omega^{\pm 1}}  & \overset{\omega^{\mp 1}}{\circ}
\ar  @{-}[r]_{q}  & \overset{r}{\circ} .}
\end{align}
By Hypothesis \ref{hyp:nichols-diagonal-finite-gkd}, $\GK \toba(K_J) = \infty$.
\end{itemize}

\bigbreak

\begin{caso2}
$\vert J \vert = m >4$. Say $J=\{2,\dots, m + 1\}$, $\ghost_{2}>0$.
\end{caso2}

We argue recursively. Since the Dynkin diagram spanned by $J$ is connected, there is a connected subdiagram $\widetilde J$ with $\vert J \vert = m-1$ and $2 \in \widetilde J$; we may assume that $\widetilde J = \{2,\dots, m\}$; thus this subdiagram appears in
Table \ref{tab:finiteGK-block-points}, and up to renumbering $q_{jj}=-1$, $\ghost_j=0$, $q_{jk}q_{kj}=1$
if $|k-j|\geq 2$, $q_{jk}q_{kj} = -1$ if $|k-j|= 1$, $j,k\in \I_{2,m}$. By the recursive hypothesis, $m + 1$ is  connected with no $j< m-1$.
Hence either $m +1$ is connected only
with $m$, or else $q_{m - 1 \, m + 1}q_{m + 1 \, m - 1} = q_{m+1\, m+1} = -1$, $\ghost_{m + 1}=0$.
In one case we have that $\dim \toba(K_J) < \infty$:

\begin{itemize} [leftmargin=*] \renewcommand{\labelitemi}{$\circ$}
\item $q_{m \, m+1}q_{m+1 \, m}= q_{m+1\, m+1}=-1$, $q_{m-1 \, m+1}q_{m+1 \, m-1}=1$,
$\ghost_{m+1}=0$. Then $K_J$ is of Cartan type $D_m$.
\end{itemize}

\medskip
In all other cases $\GK \toba(K_J) = \infty$.
Set $q=q_{m \, m+1}q_{m+1 \, m}$, $r=q_{m+1\, m+1}$, $s=q_{m-1 \, m+1}q_{m+1 \, m-1}$.
By Theorem \ref{thm:nichols-diagonal-finite-gkd}, we discard:

\medbreak
\begin{itemize} [leftmargin=*] 
 \item $q= r=-1$, $s=1$, $\ghost_{m+1}=1$. Then $K_J$ is of  type $D_{m+1}^{(1)}$.

 \item $s=r=-1$, $q=1$.  Then $K_J$ is of  type $D_m^{(1)}$.

 \item $s=1$, either $q\notin\G_\infty$ or  else $r\notin\G_\infty$ (if $q = r^{-1}$, or $r^{-2}$, consider $\cR^{m+1}(K_J)$).

\end{itemize}

Hypothesis \ref{hyp:nichols-diagonal-finite-gkd} covers the remaining cases, see \cite[Table 4]{H-classif}. \epf

\begin{table}[ht]
	\caption{$V = \cV(1,2) \oplus V_J$,  $\vert J \vert > 1$} \label{tab:finiteGK-block-points-names}
	\begin{center}
		\begin{tabular}{|c|c|}
			\hline   $V_J$ & $V$    \\
			\hline
			$\xymatrix{\overset{-1}{\circ} \ar  @{-}[r]^{r^{-1}}  & \overset{r}{\circ}}$; $r \in \ku^{\times}$ & $\lstr(A(1\vert 0)_1; r)$
			\\ \hline
			$\xymatrix{\overset{-1}{\circ} \ar  @{-}[r]^{\omega}  & \overset{-1}{\circ}}$; $\omega \in \G'_3$ 	& $\lstr(A(1\vert 0)_2; \omega)$
			\\ \hline
			$\xymatrix{\overset{\omega}{\circ} \ar  @{-}[r]^{\omega^2 }  &\overset{-1}{\circ}} $; $\omega \in \G'_3$ &   $\lstr(A(1\vert 0)_3; \omega)$
			\\ \hline
			$\xymatrix{\overset{-1}{\circ} \ar  @{-}[r]^{\omega}  & \overset{\omega^2}{\circ} \ar  @{-}[r]^{\omega}  & \overset{\omega^2}{\circ} }$; $\omega \in \G'_3$ & $\lstr(A(2 \vert 0)_1; \omega)$
			\\ \hline
			$\xymatrix{\overset{-1}{\circ} \ar  @{-}[r]^{\omega}  & \overset{\omega^2}{\circ} \ar  @{-}[r]^{\omega^2}  & \overset{\omega}{\circ} }$; $\omega \in \G'_3$ & $\lstr(D(2 \vert 1); \omega)$	
			\\ \hline
			$\xymatrix{\overset{-1}{\circ} \ar  @{-}[r]^{-1}  & \overset{-1}{\circ}}$,  $\ghost = (2,0)$ 	& $\lstr(A_{2}, 2)$
			\\ \hline			
			$\xymatrix{\overset{-1}{\circ} \ar  @{-}[r]^{-1}  & \overset{-1}{\circ}} \dots \xymatrix{
				\overset{-1}{\circ} \ar  @{-}[r]^{-1}  & \overset{-1}{\circ} }$
			& $\lstr(A_{\theta -1})$, $\theta > 2$	 		
			\\ \hline
		\end{tabular}
	\end{center}
\end{table}

\subsection{The Nichols algebras with finite $\GK$, $V_{\diag}$ connected}\label{subsection:points-block-presentation}
Let $V = V_1 \oplus V_{\diag}$ be as discussed before \eqref{eq:v=v1+v2}. We assume that the Dynkin diagram of $V_{\diag}$ is connected, i.e. $\X = \{J\}$ and $V_{\diag} = V_J$, where $J = \I_{2, \theta}$, and  that $\vert J\vert > 1$, as the case $\vert J\vert =  1$ was treated in \S \ref{sec:yd-dim3}.
We provide a presentation by generators and relations and exhibit an explicit PBW basis of $\toba(V)$, for all $V$ as in
Theorem \ref{thm:points-block-eps1}, see Table \ref{tab:finiteGK-block-points}.
We give names to these braided vector spaces in Table \ref{tab:finiteGK-block-points-names}.
The subspace $V_1\oplus \ku x_2$ is a braided vector subspace of type
\begin{itemize}
  \item $\lstr(-1, 2)$ when $V$ is of type $\lstr(A_{2}, 2)$,
  \item $\lstr(\omega, 1)$ when $V$ is of type $\lstr(A(1\vert 0)_3; \omega)$, or
  \item $\lstr(-1, 1)$ for all the other cases.
\end{itemize}
Thus the subalgebra generated by $V_1\oplus \ku x_2$ is a Nichols algebra of the corresponding type. We recall next the relations of each of them. For this, remember the change of index with respect to \S  \ref{sec:yd-dim3}; the $2$ and $3$ there are now $\fudos$ and $2$.
As in \eqref{eq:xijk}, we set   $x_{i_1i_2 \dots i_M} = \ad_c x_{i_1}\, x_{i_2 \dots i_M}$.

\smallbreak
First, the defining relations of $\cB(\lstr(-1, 2))$ are
\begin{align}
\tag{\ref{eq:rels B(V(1,2))}} &x_{\fudos}x_1-x_1x_{\fudos}+\frac{1}{2}x_1^2,
\\
\tag{\ref{eq:lstr-rels&11disc-1}} & x_1x_2-q_{12} \, x_2x_1,
\\
\label{eq:lstr-rels&11disc-qserre-g=2} & (\ad_c x_{\fudos})^3 x_2,
\\
\label{eq:lstr-rels&1-1disc-g=2} & x_2^2, \, x_{\fudos 2}^2, \, x_{\fudos\fudos 2}^2.
\end{align}
Here \eqref{eq:lstr-rels&11disc-qserre-g=2} and \eqref{eq:lstr-rels&1-1disc-g=2} are the specializations of \eqref{eq:lstr-rels&11disc-qserre} and \eqref{eq:lstr-rels&1-1disc} at $\ghost=2$, respectively.

Second, the defining relations of $\cB(\lstr(\omega, 1))$ are \eqref{eq:rels B(V(1,2))}, \eqref{eq:lstr-rels&11disc-1},
\begin{align}
\label{eq:lstr1omega1-qserre-g=1} & (\ad_c x_{\fudos})^2 x_2,
\\
\tag{\ref{eq:lstr1omega1-z0cube},\ref{eq:lstr-rels&1omega1}} & x_2^3, \, x_{\fudos 2}^3, \, [x_{\fudos 2}, x_2]_c^3.
\end{align}
Third, the defining relations of $\cB(\lstr(-1, 1))$ are \eqref{eq:rels B(V(1,2))}, \eqref{eq:lstr-rels&11disc-1}, \eqref{eq:lstr1omega1-qserre-g=1} and
\begin{align}
\label{eq:lstr-rels&1-1disc-g=1} & x_2^2, \, x_{\fudos 2}^2.
\end{align}
Here \eqref{eq:lstr1omega1-qserre-g=1} and \eqref{eq:lstr-rels&1-1disc-g=1} are the specializations of \eqref{eq:lstr-rels&11disc-qserre} and \eqref{eq:lstr-rels&1-1disc} at $\ghost=1$, respectively.

\begin{remark}
We have
\begin{align}\label{eq:lstr-rels-ghost-int-trivial}
x_1x_j &= q_{1j} x_jx_1, & x_{\fudos} x_j &= q_{1j} x_j x_{\fudos},& j &\in \I_{3, \theta},
\end{align}
since $q_{1j}q_{j1}=1$ and $\ghost_j=0$.
\end{remark}

\subsubsection{The Nichols algebra $\cB(\lstr(A(1\vert 0)_1; r))$, $r\in\G_N'$, $N\geq 3$} \label{subsubsection:lstr-a(10)1}
The subalgebra generated by $x_2$, $x_3$ is a Nichols algebra of type $A(1\vert 0)_1$. Thus,
\begin{align}\label{eq:lstr-rels-a10-diagpart}
(\ad_c x_3)^2 x_2 &=0, & x_3^N &= 0.
\end{align}

\medskip

Let $W$ be a braided vector space of diagonal type with Dynkin diagram
$\xymatrix{ \overset{-1}{\circ} \ar  @{-}[r]^{r^{-1}}  &\overset{r}{\circ} \ar  @{-}[r]^{r^{-1}}  & \overset{-1}{\circ}}$.
By \cite[Theorem 3.1]{Ang-crelle}, $\cB(W)$ is presented by generators $y_1$, $y_2$, $y_3$ and relations
\begin{align}\label{eq:lstr-rels-a10-K}
& (\ad_c y_2)^2 y_1, & & (\ad_c y_2)^2 y_3, & & (\ad_c y_1) y_3, & &y_1^2, & &y_3^2, &  &y_2^N, & &y_{123}^N.
\end{align}
Here is a basis of $\cB(W)$:
\begin{align*}
B_W=\{ y_{1}^{n_{1}}y_{12}^{n_{12}}y_{123}^{n_{123}}y_{2}^{n_{2}}y_{23}^{n_{23}}y_{3}^{n_{3}} :
n_{1}, n_{12}, n_{23}, n_{3}\in\{0,1\}, 0\le n_{2}, n_{123} < N \}.
\end{align*}

\begin{remark}\label{rem:relations lstr-a(10)1}
By Lemma \ref{lemma:braiding-K-weak-block-points}, $K^1$ is isomorphic to $W$ as braided vector spaces. Hence there exists an  isomorphism of braided Hopf algebras
$\psi:\cB(W)\to K$ such that $\psi(y_1)=x_{\fudos 2}$, $\psi(y_2)=x_{3}$, $\psi(y_3)=x_{2}$. Let
\begin{align}\label{eq:lstr-def-ztt-1}
\ztt_{12}&=x_{\fudos 23}=\psi(y_{12}), & \ztt_{123}&=[x_{\fudos 2}, x_{23}]_c=\psi(y_{123}).
\end{align}
Thus, the following identity holds in $\cB(\lstr(A(1\vert 0)_1; r))$:
\begin{align}\label{eq:lstr-rels-a10-1}
\ztt_{123}^N&=0,
\end{align}
and the following set is a basis of $K$:
\begin{align*}
B_K=\{ x_{\fudos 2}^{n_{1}}\ztt_{12}^{n_{2}}\ztt_{123}^{n_{3}}x_{3}^{n_{4}}x_{23}^{n_{5}}x_{2}^{n_{6}} :
0 \le n_{1}, n_{2}, n_{5}, n_{6} < 2, \,  0\le n_{3}, n_{4} < N \}.
\end{align*}
\end{remark}

\begin{lemma}\label{lemma:other relations lstr-a(10)1}
Let $\cB$ be a quotient algebra of $T(V)$. Assume that \eqref{eq:lstr1omega1-qserre-g=1}, \eqref{eq:lstr-rels&1-1disc-g=1},
\eqref{eq:lstr-rels-ghost-int-trivial}, \eqref{eq:lstr-rels-a10-diagpart}, \eqref{eq:lstr-rels-a10-1} hold in $\cB$.
Then there exists an algebra map $\phi:\cB(W)\to \cB$ such that $\phi(y_1)=x_{\fudos,2}$, $\phi(y_2)=x_{3}$, $\phi(y_3)=x_{2}$.
\end{lemma}
\pf
Let $\Phi:T(W)\to \cB$ be the algebra map defined as $\phi$ on the $y_i$'s. We claim that $\Phi$ annihilates
all the relations in \eqref{eq:lstr-rels-a10-K}, and the Lemma follows.
The second and the sixth relations are annihilated by \eqref{eq:lstr-rels-a10-diagpart} while the last is \eqref{eq:lstr-rels-a10-1}.
The fourth and the fifth relations are annihilated because of \eqref{eq:lstr-rels&1-1disc}, and for the third relation we apply
Lemma \ref{lemma:zt zk} (ii). Finally,
\begin{align*}
\Phi \left( (\ad_c y_2)^2 y_1 \right) &= (\ad_c x_3)^2 x_{\fudos,2} =q_{31}^2 (\ad_c x_{\fudos})(\ad_c x_3)^2 x_{2}=0,
\end{align*}
where we use \eqref{eq:lstr-rels-ghost-int-trivial} and \eqref{eq:lstr-rels-a10-diagpart}.
\epf

\begin{prop} \label{pr:lstr-a(10)1}
The algebra $\cB(\lstr(A(1\vert 0)_1; r))$ is presented by generators $x_i$, $i\in \Iw_3$, and relations
\eqref{eq:rels B(V(1,2))}, \eqref{eq:lstr-rels&11disc-1},
\eqref{eq:lstr1omega1-qserre-g=1}, \eqref{eq:lstr-rels&1-1disc-g=1},
\eqref{eq:lstr-rels-ghost-int-trivial}, \eqref{eq:lstr-rels-a10-diagpart}, \eqref{eq:lstr-rels-a10-1}.
The set
\begin{multline*}
B =\big\{ x_1^{m_1} x_{\fudos}^{m_2} x_{\fudos 2}^{n_{1}}  \ztt_{12}^{n_{2}}\ztt_{123}^{n_{3}}x_{3}^{n_{4}}x_{23}^{n_{5}}x_{2}^{n_{6}} : \\
0 \le n_{1}, n_{2}, n_{5}, n_{6} < 2, \,  0\le n_{3}, n_{4} < N, \, 0 \le m_1, m_2 \big\}
\end{multline*}
is a basis of $\cB(\lstr(A(1\vert 0)_1; r))$ and $\GK \cB(\lstr(A(1\vert 0)_1; r)) = 2$.
\end{prop}
\pf
We first prove that $B$ is a basis of $\cB := \cB(\lstr(A(1\vert 0)_1; r))$: since $\cB\simeq K\#\cB(V_1)$, see \S\ref{subsubsec:algK-block-points},
the claim follows from
Remark \ref{rem:relations lstr-a(10)1} and Proposition \ref{pr:-1block}. Then  $\GK \cB = 2$ by computing the Hilbert series.

Relations \eqref{eq:rels B(V(1,2))}, \eqref{eq:lstr-rels&11disc-1}, \eqref{eq:lstr-rels&11disc-qserre}, \eqref{eq:lstr-rels&1-1disc},
\eqref{eq:lstr-rels-ghost-int-trivial}, \eqref{eq:lstr-rels-a10-diagpart}, \eqref{eq:lstr-rels-a10-1} hold in $\cB$ as we have discussed at the
beginning of the subsection. Hence the quotient $\cBt$ of $T(V)$ by these relations projects onto $\cB$.

We claim that the subspace $I$ spanned by $B$ is a left ideal of $\cBt$. Indeed, $x_1I\subseteq I$ by definition, and
$x_{\fudos}I\subseteq I$ by \eqref{eq:rels B(V(1,2))}. By Lemma \ref{lemma:other relations lstr-a(10)1},
\begin{align*}
x_3\phi(B_W)&=\phi(y_2 B_W)\subset \phi(B_W), & x_{\fudos 2}\phi(B_W)&=\phi(y_1 B_W)\subset \phi(B_W) \\
x_2\phi(B_W) &=\phi(y_3 B_W)\subset \phi(B_W).
\end{align*}
As $I = \underset{m_1,m_2}\sum \ku \, x_1^{m_1} x_{\fudos}^{m_2} \phi(B_W)$, we have that
\begin{align*}
x_3 I & = \sum_{m_1,m_2} \ku \, x_3 x_1^{m_1} x_{\fudos}^{m_2} \phi(B_W) = \sum_{m_1,m_2} \ku \,  x_1^{m_1} x_{\fudos}^{m_2} x_3 \phi(B_W)\subset I,\\
x_2 I & = \sum_{m_1,m_2} \ku \, x_2 x_1^{m_1} x_{\fudos}^{m_2} \phi(B_W) \\
& = \sum_{m_1,m_2} \ku \,  x_1^{m_1} x_{\fudos}^{m_2} x_2 \phi(B_W) + \ku \,  x_1^{m_1} x_{\fudos}^{m_2-1} x_{\fudos2} \phi(B_W)\subset I,
\end{align*}
by \eqref{eq:lstr-rels-ghost-int-trivial}, \eqref{eq:lstr-rels&1-1disc}. Since $1\in I$, $\cBt$ is spanned by $B$. Thus
$\cBt \simeq \cB$ since $B$ is a basis of $\cB$.
\epf

\subsubsection{The Nichols algebra $\cB(\lstr(A(1\vert 0)_1; r))$, $r\notin\G_\infty$} \label{subsubsection:lstr-a(10)1-bis}
The subalgebra generated by $x_2$, $x_3$ is a Nichols algebra of type $A(1\vert 0)_1$. Thus,
\begin{align}\label{eq:lstr-rels-a10-diagpart-bis}
x_{332} &=0.
\end{align}

\medskip

Let $W$ be a braided vector space of diagonal type with Dynkin diagram
$\xymatrix{ \overset{-1}{\circ} \ar  @{-}[r]^{r^{-1}}  &\overset{r}{\circ} \ar  @{-}[r]^{r^{-1}}  & \overset{-1}{\circ}}$.
By \cite[Theorem 3.1]{Ang-crelle}, $\cB(W)$ is presented by generators $y_1$, $y_2$, $y_3$ and relations
\begin{align}\label{eq:lstr-rels-a10-K-bis}
& (\ad_c y_2)^2 y_1, & & (\ad_c y_2)^2 y_3, & & (\ad_c y_1) y_3, & &y_1^2, & &y_3^2.
\end{align}
Here is a basis of $\cB(W)$:
\begin{align*}
B_W=\{ y_{1}^{n_{1}}y_{12}^{n_{12}}y_{123}^{n_{123}}y_{2}^{n_{2}}y_{23}^{n_{23}}y_{3}^{n_{3}} :
0 \le n_{1}, n_{2}, n_{5}, n_{6} < 2, \,  0\le n_{3}, n_{4} \}.
\end{align*}

\begin{remark}\label{rem:relations lstr-a(10)1-bis}
\begin{sloppypar}
By Lemma \ref{lemma:braiding-K-weak-block-points}, $K^1$ is isomorphic to $W$ as braided vector spaces.
Hence there exists an isomorphism of braided Hopf algebras
${\psi: \cB(W)\to K}$ such that $\psi(y_1)=x_{\fudos 2}$, $\psi(y_2)=x_{3}$, $\psi(y_3)=x_{2}$. Let
$\ztt_{12}$, $\ztt_{123}$ be as in \eqref{eq:lstr-def-ztt-1}. Thus, the following set is a basis of $K$:
\end{sloppypar}
\begin{align*}
B_K=\{ x_{\fudos 2}^{n_{1}}\ztt_{12}^{n_{2}}\ztt_{123}^{n_{3}}x_{3}^{n_{4}}x_{23}^{n_{5}}x_{2}^{n_{6}} :
0 \le n_{1}, n_{2}, n_{5}, n_{6} < 2, \,  0\le n_{3}, n_{4} \}.
\end{align*}
\end{remark}

\begin{prop} \label{pr:lstr-a(10)1-bis}
The algebra $\cB(\lstr(A(1\vert 0)_1; r))$ is presented by generators $x_i$, $i\in \Iw_3$, and relations
\eqref{eq:rels B(V(1,2))}, \eqref{eq:lstr-rels&11disc-1}, \eqref{eq:lstr1omega1-qserre-g=1}, \eqref{eq:lstr-rels&1-1disc-g=1},
\eqref{eq:lstr-rels-ghost-int-trivial}, \eqref{eq:lstr-rels-a10-diagpart-bis}. The set
\begin{align*}
B =\{ x_1^{m_1} x_{\fudos}^{m_2} x_{\fudos 2}^{n_{1}} & \ztt_{12}^{n_{2}}\ztt_{123}^{n_{3}}x_{3}^{n_{4}}x_{23}^{n_{5}}x_{2}^{n_{6}} :
0 \le n_{1}, n_{2}, n_{5}, n_{6} < 2, \,  0\le m_1,m_2, n_{3}, n_{4} \}
\end{align*}
is a basis of $\cB(\lstr(A(1\vert 0)_1; r))$ and $\GK \cB(\lstr(A(1\vert 0)_1; r)) = 4$.
\end{prop}
\pf
Analogous to Proposition \ref{pr:lstr-a(10)1}.
\epf

\subsubsection{The Nichols algebra $\cB(\lstr(A(1\vert 0)_2; \omega))$} \label{subsubsection:lstr-a(10)2}

The subalgebra generated by $x_2$, $x_3$ is a Nichols algebra of type $A(1\vert 0)_2$. Thus,
\begin{align}\label{eq:lstr-rels-a10-2-diagpart}
x_2^2 &=0, & x_3^2 &= 0, & x_{23}^3 &= 0.
\end{align}

\medskip

Let $W$ be a braided vector space of diagonal type with Dynkin diagram
$\xymatrix{ \overset{-1}{\circ} \ar  @{-}[r]^{\omega}  &\overset{-1}{\circ} \ar  @{-}[r]^{\omega}  & \overset{-1}{\circ}}$.
By \cite[Theorem 3.1]{Ang-crelle}, $\cB(W)$ is presented by generators $y_1$, $y_2$, $y_3$ and relations
\begin{align}\label{eq:lstr-rels-a10-2-K-1}
& (\ad_c y_1) y_3, & &y_1^2, & &y_2^2, & &y_3^2,  & &[[y_{12},y_{123}]_c,y_2]_c, \\ \label{eq:lstr-rels-a10-2-K-2}
&[y_{123},y_2]_c^3, & &y_{12}^3, & y_{123}^6, & &&y_{23}^3, & & [[y_{123},y_{23}]_c,y_2]_c.
\end{align}
Here is a basis of $\cB(W)$:
\begin{multline*}
\big\{ y_1^{n_1} y_{12}^{n_2} [y_{12},[y_{12},y_{123}]_c]_c^{n_3} [y_{12},y_{123}]_c^{n_4} y_{123}^{n_5} [y_{123},y_2]_c^{n_6}
[y_{123},y_{23}]_c^{n_7} y_2^{n_8} y_{23}^{n_9} y_3^{n_{10}}: \\
0\le n_1,n_3,n_4,n_7,n_8,n_{10}<2, \, 0\le n_2,n_6,n_9<3, \, 0\le n_5<6 \big\}.
\end{multline*}

\begin{remark}\label{rem:relations lstr-a(10)2}
\begin{sloppypar}
By Lemma \ref{lemma:braiding-K-weak-block-points}, $K^1$ is isomorphic to $W$ as braided vector spaces.
Hence there exists an isomorphism of braided Hopf algebras
${\psi: \cB(W)\to K}$ such that $\psi(y_1)=x_{\fudos 2}$, $\psi(y_2)=x_{3}$, $\psi(y_3)=x_{2}$. Let
$\ztt_{12}$, $\ztt_{123}$ be as in \eqref{eq:lstr-def-ztt-1}. Thus, the following identities hold in
$\cB(\lstr(A(1\vert 0)_2; \omega))$:
\begin{align}\label{eq:lstr-rels-a10-2}
\ztt_{12}^3&=0, & \ztt_{123}^6&=0, & [\ztt_{123},x_3]_c^3&=0,
\end{align}
and the following set is a basis of $K$: $B_K=$
\begin{align*}
\big\{ x_{\fudos 2}^{n_1} \ztt_{12}^{n_2} [\ztt_{12},[\ztt_{12},\ztt_{123}]_c]_c^{n_3} [\ztt_{12},\ztt_{123}]_c^{n_4} \ztt_{123}^{n_5} [\ztt_{123},x_3]_c^{n_6}
[\ztt_{123},x_{32}]_c^{n_7} x_3^{n_8} x_{32}^{n_9} x_2^{n_{10}}:  \\
0\le n_1,n_3,n_4,n_7,n_8,n_{10}<2, \, 0\le n_2,n_6,n_9<3, \, 0\le n_5<6 \big\}.
\end{align*}
\end{sloppypar}
\end{remark}

\begin{lemma}\label{lemma:other relations lstr-a(10)2}
Let $\cB$ be a quotient algebra of $T(V)$. Assume that \eqref{eq:lstr1omega1-qserre-g=1}, \eqref{eq:lstr-rels&1-1disc-g=1},
\eqref{eq:lstr-rels-ghost-int-trivial}, \eqref{eq:lstr-rels-a10-2-diagpart}, \eqref{eq:lstr-rels-a10-2} hold in $\cB$.
Then there exists an algebra map $\phi:\cB(W)\to \cB$ such that $\phi(y_1)=x_{\fudos,2}$, $\phi(y_2)=x_{3}$, $\phi(y_3)=x_{2}$.
\end{lemma}
\pf
Let $\Phi:T(W)\to \cB$ be the algebra map defined as $\phi$ on the $y_i$'s. We claim that $\Phi$ annihilates
all the relations in \eqref{eq:lstr-rels-a10-2-K-1} and \eqref{eq:lstr-rels-a10-2-K-2}, and the Lemma follows.
For \eqref{eq:lstr-rels-a10-2-K-1}, $\Phi$ annihilates the first relation by Lemma \ref{lemma:zt zk} (ii), the second by
\eqref{eq:lstr-rels&1-1disc-g=1}, the third and the fourth relations by \eqref{eq:lstr-rels-a10-2-diagpart}; for the last,
notice that $[x_{\fudos 23}, x_3]_c =0$ since $x_3^2=0$, so
\begin{align*}
\Phi \left( [[y_{12},y_{123}]_c,y_2]_c \right) &= [[x_{\fudos 23},[x_{\fudos 23}, x_2]_c]_c,x_3]_c = [x_{\fudos 23},[x_{\fudos 23}, x_{23}]_c]_c \\
& = (\ad_c x_{\fudos})^2 x_{23}^3=0,
\end{align*}
c.f. \eqref{eq:lstr-rels-a10-2-diagpart}. For \eqref{eq:lstr-rels-a10-2-K-2}, the first three relations are annihilated by \eqref{eq:lstr-rels-a10-2}, and the fourth by \eqref{eq:lstr-rels-a10-2-diagpart}; for the last relation,
\begin{align*}
\Phi \left( [[y_{32},y_{321}]_c,y_2]_c \right) &= [[[x_{\fudos 2},x_{23}]_c, x_{23}]_c,x_3]_c
= - q_{23} [[[x_{\fudos 23},x_{2}]_c, x_{23}]_c,x_3]_c \\
& = q_{23}^2 [[[x_{\fudos 23},x_{2}]_c, x_3]_c,x_{23}]_c = q_{23}^2 [[x_{\fudos 23},x_{23}]_c, x_{23}]_c=0,
\end{align*}
where we use \eqref{eq:lstr-rels-ghost-int-trivial} and \eqref{eq:lstr-rels-a10-2-diagpart}.
\epf

\begin{prop} \label{pr:lstr-a(10)2}
The algebra $\cB(\lstr(A(1\vert 0)_2; \omega))$ is presented by generators $x_i$, $i\in \Iw_3$, and relations
\eqref{eq:rels B(V(1,2))}, \eqref{eq:lstr-rels&11disc-1},
\eqref{eq:lstr1omega1-qserre-g=1}, \eqref{eq:lstr-rels&1-1disc-g=1},
\eqref{eq:lstr-rels-ghost-int-trivial}, \eqref{eq:lstr-rels-a10-2-diagpart}, \eqref{eq:lstr-rels-a10-2}. The set
\begin{multline*}
B =\big\{ x_1^{m_1} x_{\fudos}^{m_2} x_{\fudos 2}^{n_1} \ztt_{12}^{n_2} [\ztt_{12},[\ztt_{12},\ztt_{123}]_c]_c^{n_3}
[\ztt_{12},\ztt_{123}]_c^{n_4} \ztt_{123}^{n_5}  \\
[\ztt_{123},x_3]_c^{n_6} [\ztt_{123},x_{32}]_c^{n_7} x_3^{n_8} x_{32}^{n_9} x_2^{n_{10}}:  0 \le m_1, m_2, \\
0\le n_2,n_6,n_9<3, \, 0\le n_5<6, \,0\le n_1,n_3,n_4,n_7,n_8,n_{10}<2 \big\}
\end{multline*}
is a basis of $\cB(\lstr(A(1\vert 0)_2; \omega))$ and $\GK \cB(\lstr(A(1\vert 0)_2; \omega)) = 2$.
\end{prop}
\pf
We first prove that $B$ is a basis of $\cB := \cB(\lstr(A(1\vert 0)_2; \omega))$: since $\cB\simeq K\#\cB(V_1)$, see \S\ref{subsubsec:algK-block-points},
the claim follows from
Remark \ref{rem:relations lstr-a(10)2} and Proposition \ref{pr:-1block}. Then  $\GK \cB = 2$ by computing the Hilbert series.

Relations \eqref{eq:rels B(V(1,2))}, \eqref{eq:lstr-rels&11disc-1}, \eqref{eq:lstr1omega1-qserre-g=1}, \eqref{eq:lstr-rels&1-1disc-g=1},
\eqref{eq:lstr-rels-ghost-int-trivial}, \eqref{eq:lstr-rels-a10-2-diagpart}, \eqref{eq:lstr-rels-a10-2} hold in $\cB$ as we have discussed at the
beginning of the subsection. Hence the quotient $\cBt$ of $T(V)$ by these relations projects onto $\cB$.

We claim that the subspace $I$ spanned by $B$ is a left ideal of $\cBt$. Indeed, $x_1I\subseteq I$ by definition, and
$x_{\fudos}I\subseteq I$ by \eqref{eq:rels B(V(1,2))}. By Lemma \ref{lemma:other relations lstr-a(10)2},
\begin{align*}
x_3\phi(B_W)&=\phi(y_2 B_W)\subset \phi(B_W), & x_{\fudos 2}\phi(B_W)&=\phi(y_1 B_W)\subset \phi(B_W) \\
x_2\phi(B_W) &=\phi(y_3 B_W)\subset \phi(B_W).
\end{align*}
As $I = \underset{m_1,m_2}\sum \ku \, x_1^{m_1} x_{\fudos}^{m_2} \phi(B_W)$ and \eqref{eq:lstr-rels-ghost-int-trivial}, \eqref{eq:lstr-rels&1-1disc}
hold in $\cBt$, the same computation of Proposition \ref{pr:lstr-a(10)1} applies to prove that $x_3 I \subset I$, $x_2 I \subset I$.
Since $1\in I$, $\cBt$ is spanned by $B$. Thus $\cBt \simeq \cB$ since $B$ is a basis of $\cB$.
\epf

\subsubsection{The Nichols algebra $\cB(\lstr(A(1\vert 0)_3; \omega))$} \label{subsubsection:lstr-a(10)3}

The subalgebra generated by $x_2$, $x_3$ is a Nichols algebra of diagonal type $A(1\vert 0)_3$. Thus,
\begin{align}\label{eq:lstr-rels-a10-3-diagpart}
x_2^3 &=0, & x_3^2 &= 0, & x_{223} &= 0.
\end{align}

\medskip

Let $W$ be a braided vector space of diagonal type with Dynkin diagram
$$\xymatrix{ \overset{\omega}{\circ} \ar@{-}[r]^{\omega^2} \ar@{-}[rd]_{\omega^2} &\overset{-1}{\circ} \ar@{-}[d]^{\omega^2}  \\
& \overset{\omega}{\circ}.}$$ 
By \cite[Theorem 3.1]{Ang-crelle}, $\cB(W)$ is presented by generators $y_1$, $y_2$, $y_3$ and relations
\begin{align}\label{eq:lstr-rels-a10-3-K-1}
& y_{112}, \qquad y_{113}, & &y_2^2, \qquad y_3^3,
\\ \label{eq:lstr-rels-a10-3-K-2}
&y_{331}, \qquad y_{332}, & &y_1^3, \qquad y_{13}^3,
\\ \label{eq:lstr-rels-a10-3-K-3}
&y_{123}-q_{23}\omega [y_{13},y_2]_c  -q_{12}(1-\omega^2) y_2y_{13}, & & y_{123}^6.
\end{align}
Here is a basis of $\cB(W)$:
\begin{multline*}
B_W=\{ y_1^{a_1} y_{12}^{a_2} [y_{12},y_{13}]_c^{a_3} y_{123}^{a_4} [y_{123},y_{13}]_c^{a_5}
[y_{13},y_{23}]_c^{a_6} y_{13}^{a_7} y_2^{a_8} y_{23}^{a_9} y_3^{a_{10}}:\\
0\le n_2,n_3,n_5,n_6,n_8,n_9<2, \, 0\le n_1,n_7,n_{10}<3, \, 0\le n_4<6 \big\}.
\end{multline*}

\begin{remark}\label{rem:relations lstr-a(10)3}
\begin{sloppypar}
By Lemma \ref{lemma:braiding-K-weak-block-points}, $K^1$ is isomorphic to $W$ as braided vector spaces.
Hence there exists an isomorphism of braided Hopf algebras
${\psi: \cB(W)\to K}$ such that $\psi(y_1)=x_{\fudos 2}$, $\psi(y_2)=x_{3}$, $\psi(y_3)=x_{2}$. Let
$\ztt_{12}$, $\ztt_{123}$ be as in \eqref{eq:lstr-def-ztt-1}, $\ztt_{13}=\psi(y_{13})$.
Thus, the following identities hold in $\cB(\lstr(A(1\vert 0)_3; \omega))$:
\begin{align}\label{eq:lstr-rels-a10-3}
x_{\fudos 2}x_{\fudos 23}&+q_{13}q_{23}x_{\fudos 23}x_{\fudos 2}=0, &  \ztt_{123}^6&=0,
\end{align}
and the following set is a basis of $K$:
\begin{multline*}
B_K=\big\{ x_{\fudos 2}^{n_1} \ztt_{12}^{n_2} [\ztt_{12},\ztt_{13}]_c^{n_3} \ztt_{123}^{n_4} [\ztt_{123},\ztt_{13}]_c^{n_5} [\ztt_{13},x_{32}]_c^{n_6}
\ztt_{13}^{n_7} x_3^{n_8} x_{32}^{n_9} x_2^{n_{10}}:\\
0\le n_2,n_3,n_5,n_6,n_8,n_9<2, \, 0\le n_1,n_7,n_{10}<3, \, 0\le n_4<6 \big\}.
\end{multline*}
\end{sloppypar}
\end{remark}

\begin{lemma}\label{lemma:other relations lstr-a(10)3}
Let $\cB$ be a quotient algebra of $T(V)$. Assume that \eqref{eq:lstr1omega1-qserre-g=1}, \eqref{eq:lstr1omega1-z0cube},
\eqref{eq:lstr-rels&1omega1}, \eqref{eq:lstr-rels-ghost-int-trivial}, \eqref{eq:lstr-rels-a10-3-diagpart},
\eqref{eq:lstr-rels-a10-3} hold in $\cB$.
Then there exists an algebra map $\phi:\cB(W)\to \cB$ such that $\phi(y_1)=x_{\fudos,2}$, $\phi(y_2)=x_{3}$, $\phi(y_3)=x_{2}$.
\end{lemma}
\pf
Let $\Phi:T(W)\to \cB$ be the algebra map defined as $\phi$ on the $y_i$'s. We claim that $\Phi$ annihilates
all the relations in \eqref{eq:lstr-rels-a10-3-K-1}, \eqref{eq:lstr-rels-a10-3-K-2} and \eqref{eq:lstr-rels-a10-3-K-3},
and the Lemma follows.
For \eqref{eq:lstr-rels-a10-3-K-1}, $\Phi$ annihilates the third and the fourth relations by \eqref{eq:lstr-rels-a10-3-diagpart}, the first by
\eqref{eq:lstr-rels-a10-3} and the second by \eqref{eq:lstr1omega1-qserre-g=1}, \eqref{eq:lstr1omega1-z0cube}.

For \eqref{eq:lstr-rels-a10-3-K-2}, $\Phi$ annihilates the first relation by \eqref{eq:lstr1omega1-z0cube}, the second by \eqref{eq:lstr-rels-a10-3-diagpart},
while the third and the fourth relation follow by \eqref{eq:lstr-rels&1omega1}. Finally, for \eqref{eq:lstr-rels-a10-3-K-3},
$\Phi$ applies the first relation to $(\ad_c x_{\fudos})x_{223}=0$, c.f. \eqref{eq:lstr-rels-a10-3-diagpart}, and the second relation to
$\ztt_{123}^6=0$,  c.f. \eqref{eq:lstr-rels-a10-3}.
\epf

\begin{prop} \label{pr:lstr-a(10)3}
The algebra $\cB(\lstr(A(1\vert 0)_3; \omega))$ is presented by generators $x_i$, $i\in \Iw_3$, and relations
\eqref{eq:rels B(V(1,2))}, \eqref{eq:lstr-rels&11disc-1}, \eqref{eq:lstr1omega1-qserre-g=1}, \eqref{eq:lstr1omega1-z0cube},
\eqref{eq:lstr-rels&1omega1}, \eqref{eq:lstr-rels-ghost-int-trivial}, \eqref{eq:lstr-rels-a10-3-diagpart}, \eqref{eq:lstr-rels-a10-3}.
The set
\begin{multline*}
B =\big\{ x_1^{m_1} x_{\fudos}^{m_2}x_{\fudos 2}^{n_1} \ztt_{12}^{n_2} [\ztt_{12},\ztt_{13}]_c^{n_3} \ztt_{123}^{n_4} [\ztt_{123},\ztt_{13}]_c^{n_5} [\ztt_{13},x_{32}]_c^{n_6} \ztt_{13}^{n_7} x_3^{n_8} x_{32}^{n_9} x_2^{n_{10}}:  \\
0 \le m_1, m_2,  \, 0\le n_2,n_3,n_5,n_6,n_8,n_9<2, \, 0\le n_1,n_7,n_{10}<3, \, 0\le n_4<6 \big\}
\end{multline*}
is a basis of $\cB(\lstr(A(1\vert 0)_3; \omega))$ and $\GK \cB(\lstr(A(1\vert 0)_3; \omega)) = 2$.
\end{prop}
\pf
We first prove that $B$ is a basis of $\cB := \cB(\lstr(A(1\vert 0)_3; \omega))$: since $\cB\simeq K\#\cB(V_1)$, the claim follows from
Remark \ref{rem:relations lstr-a(10)3} and Proposition \ref{pr:-1block}. Then  $\GK \cB = 2$ by computing the Hilbert series.

Relations \eqref{eq:rels B(V(1,2))}, \eqref{eq:lstr-rels&11disc-1}, \eqref{eq:lstr1omega1-qserre-g=1}, \eqref{eq:lstr1omega1-z0cube},
\eqref{eq:lstr-rels&1omega1}, \eqref{eq:lstr-rels-ghost-int-trivial}, \eqref{eq:lstr-rels-a10-3-diagpart}, \eqref{eq:lstr-rels-a10-3}
hold in $\cB$ as we have discussed at the beginning of the subsection. Hence the quotient $\cBt$ of $T(V)$ by these relations projects onto $\cB$.

We claim that the subspace $I$ spanned by $B$ is a left ideal of $\cBt$. Indeed, $x_1I\subseteq I$ by definition, and
$x_{\fudos}I\subseteq I$ by \eqref{eq:rels B(V(1,2))}. By Lemma \ref{lemma:other relations lstr-a(10)3},
\begin{align*}
x_3\phi(B_W)&=\phi(y_2 B_W)\subset \phi(B_W), & x_{\fudos 2}\phi(B_W)&=\phi(y_1 B_W)\subset \phi(B_W) \\
x_2\phi(B_W) &=\phi(y_3 B_W)\subset \phi(B_W).
\end{align*}
As $I = \underset{m_1,m_2}\sum \ku \, x_1^{m_1} x_{\fudos}^{m_2} \phi(B_W)$ and \eqref{eq:lstr-rels-ghost-int-trivial}, \eqref{eq:lstr-rels&1-1disc}
hold in $\cBt$, the same computation of Proposition \ref{pr:lstr-a(10)1} applies to prove that $x_3 I \subset I$, $x_2 I \subset I$.
Since $1\in I$, $\cBt$ is spanned by $B$. Thus $\cBt \simeq \cB$ since $B$ is a basis of $\cB$.
\epf

\subsubsection{The Nichols algebra $\cB(\lstr(A(2\vert 0)_1; \omega))$} \label{subsubsection:lstr-a(20)1}

The subalgebra generated by $x_2$, $x_3$, $x_4$ is a Nichols algebra of type $A(2\vert 0)_1$. Thus,
\begin{align}\label{eq:lstr-rels-a(20)1-diagpart}
x_{24}&=0, & x_{332}&=0, & x_{334}&=0, & x_{443}&=0, \\ \label{eq:lstr-rels-a(20)1-diagpart-2}
x_2^2 &=0, & x_3^3&=0, & x_{34}^3&=0, & x_4^3&=0.
\end{align}

\medskip

Let $W$ be a braided vector space of diagonal type with Dynkin diagram
$$
\xymatrix{ \overset{-1}{\circ} \ar@{-}[r]^{\omega}  &\overset{\omega^2}{\circ} \ar@{-}[r]^{\omega}  \ar@{-}[d]^{\omega^2}
& \overset{\omega}{\circ}  \\ & \overset{-1}{\circ} & }.
$$
By \cite[Thm. 3.1]{Ang-crelle}, $\cB(W)$ is presented by generators $y_i$, $i\in\I_4$, and relations
\begin{align}
&y_{13}, &&y_{221}, && y_{224}, &&  y_{14}, && y_{442}, && y_{223},
\label{eq:lstr-rels-a(20)1-K-1} \\
&y_{34}, &&y_1^2, && y_2^3, && y_{123}^3, && [y_{123},y_{1234}]_c^3, && [y_{123},[y_{1234},y_2]_c]_c^3,
\label{eq:lstr-rels-a(20)1-K-2}\\
&y_{24}^3, && y_3^2, && y_4^3, && y_{1234}^3, && [y_{1234},y_2]_c^3, && [y_{1234},[y_{1234},y_2]_c]_c^3.
\label{eq:lstr-rels-a(20)1-K-3}
\end{align}
Here is a basis of $\cB(W)$:
\begin{multline*}
B_W=\{ y_1^{n_1} y_{12}^{n_2} [y_{12},y_{1234}]_c^{n_3} y_{123}^{n_4} [y_{123},y_{1234}]_c^{n_5}
[y_{123},[y_{1234}, y_2]_c]_c^{n_6} y_{1234}^{n_7}
\\
[y_{1234}, [y_{1234}, y_2]_c]_c^{n_8} [y_{1234}, y_2]_c^{n_{9}} [y_{1234}, y_{23}]_c^{n_{10}} y_{124}^{n_{11}} y_2^{n_{12}} y_{23}^{n_{13}} y_{234}^{n_{14}} y_{24}^{n_{15}} y_3^{n_{16}} y_4^{n_{17}}:
\\
0\le n_{1},n_{2},n_{3},n_{10},n_{11},n_{13},n_{14},n_{15} <2,
\\
0\le n_{4},n_{5},n_{6},n_{7},n_{8},n_{9},n_{12},n_{16},n_{17} <3 \}.
\end{multline*}

\begin{remark}\label{rem:relations lstr-a(20)1}
\begin{sloppypar}
By Lemma \ref{lemma:braiding-K-weak-block-points}, $K^1$ is isomorphic to $W$ as braided vector spaces.
Hence there exists an isomorphism of braided Hopf algebras
${\psi: \cB(W)\to K}$ such that $\psi(y_1)=x_{\fudos 2}$, $\psi(y_2)=x_{3}$, $\psi(y_3)=x_{2}$, $\psi(y_4)=x_{4}$.
Let $\ztt_{12}$, $\ztt_{123}$ be as in \eqref{eq:lstr-def-ztt-1},
\begin{align}\label{eq:lstr-def-ztt-2}
\ztt_{1234}&=\psi(y_{1234})=[x_{\fudos 23}, x_{24}]_c.
\end{align}
Thus, the following identities hold in $\cB(\lstr(A(2\vert 0)_1; \omega))$:
\begin{align}\label{eq:lstr-rels-a(20)1}
[x_{\fudos 23},x_2]_c^3&=0, & [[x_{\fudos 23},x_2]_c, [x_{\fudos 23},x_{24}]_c]_c^3&=0,  \notag\\
[x_{\fudos 23},x_{24}]_c^3&=0, & [ [x_{\fudos 23},x_2]_c, [[x_{\fudos 23},x_{24}]_c,x_2]_c]_c^3&=0,\\
[[x_{\fudos 23},x_{24}]_c,x_2]_c^3&=0, & [ [x_{\fudos 23},x_{24}]_c, [[x_{\fudos 23},x_{24}]_c,x_2]_c]_c^3&=0, \notag
\end{align}
and the following set is a basis of $K$:
\begin{multline*}
B_K=\{ x_{\fudos 2}^{n_1} \ztt_{12}^{n_2} [\ztt_{12},\ztt_{1234}]_c^{n_3} \ztt_{123}^{n_4} [\ztt_{123},\ztt_{1234}]_c^{n_5}
[\ztt_{123},[\ztt_{1234}, x_3]_c]_c^{n_6} \ztt_{1234}^{n_7}
\\
[\ztt_{1234}, [\ztt_{1234}, x_3]_c]_c^{n_8} [\ztt_{1234}, x_3]_c^{n_{9}} [\ztt_{1234}, x_{32}]_c^{n_{10}} x_{\fudos 234}^{n_{11}} x_3^{n_{12}} x_{32}^{n_{13}} x_{324}^{n_{14}} x_{34}^{n_{15}} x_2^{n_{16}} x_4^{n_{17}}:
\\
0\le n_{1},n_{2},n_{3},n_{10},n_{11},n_{13},n_{14},n_{15} <2,
\\
0\le n_{4},n_{5},n_{6},n_{7},n_{8},n_{9},n_{12},n_{16},n_{17} <3 \}.
\end{multline*}
\end{sloppypar}
\end{remark}

\begin{lemma}\label{lemma:other relations lstr-a(20)1}
Let $\cB$ be a quotient algebra of $T(V)$. Assume that \eqref{eq:lstr1omega1-qserre-g=1}, \eqref{eq:lstr-rels&1-1disc-g=1},
\eqref{eq:lstr-rels-ghost-int-trivial}, \eqref{eq:lstr-rels-a(20)1-diagpart}, \eqref{eq:lstr-rels-a(20)1-diagpart-2}, \eqref{eq:lstr-rels-a(20)1} hold in $\cB$.
Then there exists an algebra map $\phi:\cB(W)\to \cB$ such that $\phi(y_1)=x_{\fudos,2}$, $\phi(y_2)=x_{3}$, $\phi(y_3)=x_{2}$, $\phi(y_4)=x_4$.
\end{lemma}
\pf
Let $\Phi:T(W)\to \cB$ be the algebra map defined as $\phi$ on the $y_i$'s. We claim that $\Phi$ annihilates
all the relations in \eqref{eq:lstr-rels-a(20)1-K-1}, \eqref{eq:lstr-rels-a(20)1-K-2} and \eqref{eq:lstr-rels-a(20)1-K-2},
and the Lemma follows. All the relations in \eqref{eq:lstr-rels-a(20)1-K-1} and the first relation in \eqref{eq:lstr-rels-a(20)1-K-2}
are annihilated by $\Phi$ because of \eqref{eq:lstr1omega1-qserre-g=1}, \eqref{eq:lstr-rels&1-1disc-g=1}, \eqref{eq:lstr-rels-ghost-int-trivial}, \eqref{eq:lstr-rels-a(20)1-diagpart}. For the remaining relations, we use \eqref{eq:lstr-rels-a(20)1-diagpart-2}, \eqref{eq:lstr-rels-a(20)1}.
\epf

\begin{prop} \label{pr:lstr-a(20)1}
The algebra $\cB(\lstr(A(2\vert 0)_1; \omega))$ is presented by generators $x_i$, $i\in \Iw_4$, and relations
\eqref{eq:rels B(V(1,2))}, \eqref{eq:lstr-rels&11disc-1},
\eqref{eq:lstr1omega1-qserre-g=1}, \eqref{eq:lstr-rels&1-1disc-g=1},
\eqref{eq:lstr-rels-ghost-int-trivial}, \eqref{eq:lstr-rels-a(20)1-diagpart}, \eqref{eq:lstr-rels-a(20)1-diagpart-2}, \eqref{eq:lstr-rels-a(20)1}.
The set
\begin{multline*}
B =\big\{ x_1^{m_1} x_{\fudos}^{m_2}
x_{\fudos 2}^{n_1} \ztt_{12}^{n_2} [\ztt_{12},\ztt_{1234}]_c^{n_3} \ztt_{123}^{n_4} [\ztt_{123},\ztt_{1234}]_c^{n_5}
[\ztt_{123},[\ztt_{1234}, x_3]_c]_c^{n_6} \ztt_{1234}^{n_7}
\\
[\ztt_{1234}, [\ztt_{1234}, x_3]_c]_c^{n_8} [\ztt_{1234}, x_3]_c^{n_{9}} [\ztt_{1234}, x_{32}]_c^{n_{10}} x_{\fudos 234}^{n_{11}} x_3^{n_{12}} x_{32}^{n_{13}} x_{324}^{n_{14}} x_{34}^{n_{15}} x_2^{n_{16}} x_4^{n_{17}}:
\\
0 \le m_1, m_2, \, 0\le n_{1},n_{2},n_{3},n_{10},n_{11},n_{13},n_{14},n_{15} <2,
\\
0\le n_{4},n_{5},n_{6},n_{7},n_{8},n_{9},n_{12},n_{16},n_{17} <3 \big\}
\end{multline*}
is a basis of $\cB(\lstr(A(2\vert 0)_1; \omega))$ and $\GK \cB(\lstr(A(2\vert 0)_1; \omega)) = 2$.
\end{prop}
\pf
We first prove that $B$ is a basis of $\cB := \cB(\lstr(A(2\vert 0)_1; \omega))$: since $\cB\simeq K\#\cB(V_1)$, the claim follows from
Remark \ref{rem:relations lstr-a(20)1} and Proposition \ref{pr:-1block}. Then  $\GK \cB = 2$ by computing the Hilbert series.

Relations \eqref{eq:rels B(V(1,2))}, \eqref{eq:lstr-rels&11disc-1},\eqref{eq:lstr1omega1-qserre-g=1}, \eqref{eq:lstr-rels&1-1disc-g=1},
\eqref{eq:lstr-rels-ghost-int-trivial}, \eqref{eq:lstr-rels-a(20)1-diagpart}, \eqref{eq:lstr-rels-a(20)1-diagpart-2}, \eqref{eq:lstr-rels-a(20)1}
hold in $\cB$ as we have discussed at the beginning of the subsection. Hence the quotient $\cBt$ of $T(V)$ by these relations projects onto $\cB$.

We claim that the subspace $I$ spanned by $B$ is a left ideal of $\cBt$. Indeed, $x_1I\subseteq I$ by definition, and
$x_{\fudos}I\subseteq I$ by \eqref{eq:rels B(V(1,2))}. By Lemma \ref{lemma:other relations lstr-a(20)1},
\begin{align*}
x_3\phi(B_W)&=\phi(y_2 B_W)\subset \phi(B_W), & x_{\fudos 2}\phi(B_W)&=\phi(y_1 B_W)\subset \phi(B_W), \\
x_4\phi(B_W)&=\phi(y_4 B_W)\subset \phi(B_W), & x_2\phi(B_W) &=\phi(y_3 B_W)\subset \phi(B_W).
\end{align*}
As $I = \underset{m_1,m_2}\sum \ku \, x_1^{m_1} x_{\fudos}^{m_2} \phi(B_W)$ and \eqref{eq:lstr-rels-ghost-int-trivial}, \eqref{eq:lstr-rels&1-1disc}
hold in $\cBt$, the same computation of Proposition \ref{pr:lstr-a(10)1} applies to prove that $x_4 I \subset I$, $x_3 I \subset I$, $x_2 I \subset I$.
Since $1\in I$, $\cBt$ is spanned by $B$. Thus $\cBt \simeq \cB$ since $B$ is a basis of $\cB$.
\epf

\subsubsection{The Nichols algebra $\cB(\lstr(D(2\vert 1); \omega))$} \label{subsubsection:lstr-D(21)}

The subalgebra generated by $x_2$, $x_3$, $x_4$ is a Nichols algebra of type $D(2\vert 1)$. Thus,
\begin{align}\label{eq:lstr-rels-d(21)-diagpart-1}
x_{24}&=0, & x_{332}&=0, & x_{443}&=0, & [[x_{234}&,x_3]_c,x_3]_c=0,
\\ \label{eq:lstr-rels-d(21)-diagpart-2}
x_2^2&=0, & x_3^3&=0, & x_{34}^3&=0, & x_{334}^3&=0, \quad x_4^3=0.
\end{align}

\medskip
Let $W$ be a braided vector space of diagonal type with Dynkin diagram
$$
\xymatrix{ \overset{-1}{\circ} \ar@{-}[r]^{\omega}  &\overset{\omega^2}{\circ} \ar@{-}[r]^{\omega^2}  \ar@{-}[d]^{\omega}
& \overset{\omega}{\circ}  \\ & \overset{-1}{\circ} & }.
$$
By \cite[Thm. 3.1]{Ang-crelle}, $\cB(W)$ is presented by generators $y_i$, $i\in\I_4$, and relations
\begin{align}
&y_1^2, & & y_{13}, & &y_{34}, & &y_{221}, & &[[y_{124},y_2]_c,y_2]_c, \label{eq:lstr-rels-d(21)-K-1} \\
&y_3^2, & & y_{14}, & &y_{442}, & &y_{223}, &  & [[y_{324},y_2]_c,y_2]_c, \label{eq:lstr-rels-d(21)-K-2} \\
&y_2^3, &&  y_{224}^3, && y_{123}^3, && [y_{1234}, y_2]_c^3, && [[y_{1234}, y_2]_c, y_2]_c^3, \label{eq:lstr-rels-d(21)-K-3} \\
&y_4^3, &&  y_{24}^3, && y_{1234}^3, && [y_{1234}, y_{224}]_c^3. \label{eq:lstr-rels-d(21)-K-4}
\end{align}
Here is a basis of $\cB(W)$:
\begin{multline*}
B_W=\{ y_1^{n_1} y_{12}^{n_2} y_{123}^{n_3} y_{1234}^{n_4} [y_{1234},y_2]_c^{n_5}
[[y_{1234}, y_2]_c, y_2]_c^{n_6} [y_{1234},y_{224}]_c^{n_7}
\\
y_{124}^{n_8} [y_{124}, y_2]_c^{n_{9}} y_{2}^{n_{10}} y_{23}^{n_{11}} y_{234}^{n_{12}} [y_{234},y_2]_c^{n_{13}} y_{224}^{n_{14}} y_{24}^{n_{15}} y_3^{n_{16}} y_4^{n_{17}}:
\\
0\le n_{1},n_{2},n_{8},n_{9},n_{11},n_{12},n_{13},n_{16} <2,
\\
0\le n_{3},n_{4},n_{5},n_{6},n_{7},n_{10},n_{14},n_{15},n_{17} <3 \}.
\end{multline*}

\begin{remark}\label{rem:relations lstr-d(21)}
\begin{sloppypar}
By Lemma \ref{lemma:braiding-K-weak-block-points}, $K^1$ is isomorphic to $W$ as braided vector spaces.
Hence there exists an isomorphism of braided Hopf algebras
${\psi: \cB(W)\to K}$ such that $\psi(y_1)=x_{\fudos 2}$, $\psi(y_2)=x_{3}$, $\psi(y_3)=x_{2}$, $\psi(y_4)=x_{4}$.
Let $\ztt_{12}$, $\ztt_{123}$, $\ztt_{1234}$ be as in \eqref{eq:lstr-def-ztt-1}, \eqref{eq:lstr-def-ztt-2}.
Thus, the following identities hold in $\cB(\lstr(D(2\vert 1); \omega))$:
\begin{align}\label{eq:lstr-rels-d(21)}
& && [x_{\fudos 23},x_2]_c^3=0, && [[x_{\fudos 23},x_{24}]_c,x_3]_c^3=0, \\
&[[[x_{\fudos 23},x_{24}]_c, x_3]_c, x_3]_c^3=0, && [x_{\fudos 23},x_{24}]_c^3=0, && [[x_{\fudos 23},x_{24}]_c,x_{334}]_c^3=0. \notag
\end{align}
and the following set is a basis of $K$:
\begin{multline*}
B_K=\{ x_{\fudos 2}^{n_1} \ztt_{12}^{n_2} \ztt_{123}^{n_3} \ztt_{1234}^{n_4} [\ztt_{1234},x_3]_c^{n_5}
[[\ztt_{1234}, x_3]_c, x_3]_c^{n_6} [\ztt_{1234},x_{334}]_c^{n_7}
\\
x_{\fudos 234}^{n_8} [x_{\fudos 234}, x_3]_c^{n_{9}} x_{3}^{n_{10}} x_{32}^{n_{11}} x_{324}^{n_{12}} [x_{324},x_3]_c^{n_{13}} x_{334}^{n_{14}} x_{34}^{n_{15}} x_2^{n_{16}} x_4^{n_{17}}:
\\
0\le n_{1},n_{2},n_{8},n_{9},n_{11},n_{12},n_{13},n_{16} <2,
\\
0\le n_{3},n_{4},n_{5},n_{6},n_{7},n_{10},n_{14},n_{15},n_{17} <3 \}.
\end{multline*}
\end{sloppypar}
\end{remark}

\begin{lemma}\label{lemma:other relations lstr-d(21)}
Let $\cB$ be a quotient algebra of $T(V)$. Assume that \eqref{eq:lstr1omega1-qserre-g=1}, \eqref{eq:lstr-rels&1-1disc-g=1},
\eqref{eq:lstr-rels-ghost-int-trivial}, \eqref{eq:lstr-rels-d(21)-diagpart-1}, \eqref{eq:lstr-rels-d(21)-diagpart-2},
\eqref{eq:lstr-rels-d(21)} hold in $\cB$.
Then there exists an algebra map $\phi:\cB(W)\to \cB$ such that $\phi(y_1)=x_{\fudos,2}$, $\phi(y_2)=x_{3}$, $\phi(y_3)=x_{2}$, $\phi(y_4)=x_4$.
\end{lemma}
\pf
Let $\Phi:T(W)\to \cB$ be the algebra map defined as $\phi$ on the $y_i$'s. We claim that $\Phi$ annihilates
all the relations in \eqref{eq:lstr-rels-d(21)-K-1}, \eqref{eq:lstr-rels-d(21)-K-2}, \eqref{eq:lstr-rels-d(21)-K-3}, \eqref{eq:lstr-rels-d(21)-K-4},
and the Lemma follows. For the last relation in \eqref{eq:lstr-rels-d(21)-K-1},
\begin{align*}
\Phi\left( [[y_{124},y_2]_c,y_2]_c \right) &= \left[ \left[ x_{\fudos 234}, x_3 \right]_c , x_3 \right]_c = (\ad_c x_{\fudos}) [[x_{234},x_3]_c,x_3]_c =0,
\end{align*}
by \eqref{eq:lstr-rels-ghost-int-trivial} and the last relation in \eqref{eq:lstr-rels-d(21)-diagpart-1}.
All the other relations in \eqref{eq:lstr-rels-d(21)-K-1} and all the relations in \eqref{eq:lstr-rels-d(21)-K-2}
follow directly by \eqref{eq:lstr1omega1-qserre-g=1}, \eqref{eq:lstr-rels&1-1disc-g=1}, \eqref{eq:lstr-rels-ghost-int-trivial}, \eqref{eq:lstr-rels-d(21)-diagpart-1}.
For \eqref{eq:lstr-rels-d(21)-K-3} and \eqref{eq:lstr-rels-d(21)-K-4}, we use \eqref{eq:lstr-rels-d(21)-diagpart-2}, \eqref{eq:lstr-rels-d(21)}.
\epf

\begin{prop} \label{pr:lstr-d(21)}
The algebra $\cB(\lstr(D(2\vert 1); \omega))$ is presented by generators $x_i$, $i\in \Iw_4$, and relations
\eqref{eq:rels B(V(1,2))}, \eqref{eq:lstr-rels&11disc-1}, \eqref{eq:lstr1omega1-qserre-g=1}, \eqref{eq:lstr-rels&1-1disc-g=1},
\eqref{eq:lstr-rels-ghost-int-trivial}, \eqref{eq:lstr-rels-d(21)-diagpart-1}, \eqref{eq:lstr-rels-d(21)-diagpart-2},
\eqref{eq:lstr-rels-d(21)}. The set
\begin{multline*}
B =\{ x_1^{m_1} x_{\fudos}^{m_2} x_{\fudos 2}^{n_1} \ztt_{12}^{n_2} \ztt_{123}^{n_3} \ztt_{1234}^{n_4} [\ztt_{1234},x_3]_c^{n_5}
[[\ztt_{1234}, x_3]_c, x_3]_c^{n_6} [\ztt_{1234},x_{334}]_c^{n_7}
\\
x_{\fudos 234}^{n_8} [x_{\fudos 234}, x_3]_c^{n_{9}} x_{3}^{n_{10}} x_{32}^{n_{11}} x_{324}^{n_{12}} [x_{324},x_3]_c^{n_{13}} x_{334}^{n_{14}} x_{34}^{n_{15}} x_2^{n_{16}} x_4^{n_{17}}:
\\
0\le m_1,m_2, \, 0\le n_{1},n_{2},n_{8},n_{9},n_{11},n_{12},n_{13},n_{16} <2,
\\
0\le n_{3},n_{4},n_{5},n_{6},n_{7},n_{10},n_{14},n_{15},n_{17} <3 \}.
\end{multline*}
is a basis of $\cB(\lstr(D(2\vert 1); \omega))$ and $\GK \cB(\lstr(D(2\vert 1); \omega)) = 2$.
\end{prop}
\pf
Analogous to Proposition \ref{pr:lstr-a(20)1}.
\epf

\subsubsection{The Nichols algebra $\cB(\lstr(A_2, 2))$} \label{subsubsection:lstr-a-22}

The subalgebra generated by $x_2$, $x_3$ is a Nichols algebra of type $A_2$. Thus,
\begin{align}\label{eq:lstr-rels-a-22-diagpart}
x_2^2 &=0, & x_3^2 &=0, & x_{32}^2&=0.
\end{align}

\medskip

Let $W$ be a braided vector space of diagonal type with Dynkin diagram
$$
\xymatrix{ \overset{-1}{\circ} \ar@{-}[r]^{-1}  &\overset{-1}{\circ} \ar@{-}[r]^{-1}  \ar@{-}[d]^{-1}
& \overset{-1}{\circ}  \\ & \overset{-1}{\circ} & }.
$$
By \cite[Thm. 3.1]{Ang-crelle}, $\cB(W)$ is presented by generators $y_i$, $i\in\I_4$, and relations
\begin{align}
& [y_{123},y_2]_c, && y_{13}, && y_1^2, && y_{12}^2,  && y_{123}^2, && [y_{124}, y_3]_c^2, \label{eq:lstr-rels-a-22-K-1} \\
& [y_{124},y_2]_c, && y_{34}, && [[y_{124}, y_3]_c,y_2]_c^2, && y_{124}^2, && y_2^2, && y_{23}^2, \label{eq:lstr-rels-a-22-K-2} \\
& [y_{324},y_2]_c, && y_{14}, && [y_{24}, y_3]_c^2, && y_{24}^2, && y_3^2, && y_4^2.\label{eq:lstr-rels-a-22-K-3}
\end{align}
Here is a basis of $\cB(W)$: $B_W=$
\begin{align*}
\{  y_1^{n_1} y_{12}^{n_2}  y_{123}^{n_3} y_{1234}^{n_4} [y_{1234},y_2]_c^{n_5} y_{124}^{n_6}
y_2^{n_7} y_{23}^{n_8} y_{234}^{n_9} y_{24}^{n_{10}} y_3^{n_{11}} y_4^{n_{12}} : \, 0 \le n_i <2 \}.
\end{align*}

\begin{remark}\label{rem:relations lstr-a-22}
\begin{sloppypar}
By Lemma \ref{lemma:braiding-K-weak-block-points}, $K^1$ is isomorphic to $W$ as braided vector spaces.
Hence there exists an isomorphism of braided Hopf algebras
${\psi: \cB(W)\to K}$ such that $\psi(y_1)=x_{\fudos\fudos 2}$, $\psi(y_2)=x_{3}$, $\psi(y_3)=x_{\fudos 2}$, $\psi(y_4)=x_2$.
Thus, the following identities hold in $\cB(\lstr(A_2,2))$:
\begin{align}\label{eq:lstr-rels-a-22}
& & & x_{\fudos\fudos 23}^2=0, &
& [x_{\fudos\fudos 23}, x_{\fudos 2}]_c^2=0, &
& [x_{\fudos\fudos 23}, x_2]_c, x_{\fudos 2}]_c^2=0, \\ \notag
& [x_{\fudos\fudos 23},x_2]_c^2=0, &
& x_{3\fudos 2}^2=0, &
& [x_{32},x_{\fudos 2}]_c^2=0, &
& [[x_{\fudos\fudos 23}, x_2]_c, x_{\fudos 2}]_c, x_3]_c^2=0,
\end{align}
and the following set is a basis of $K$:
\begin{multline*}
B_K=\{ x_{\fudos\fudos 2}^{n_1} x_{\fudos\fudos 23}^{n_2}  [x_{\fudos\fudos 23}, x_{\fudos 2}]_c^{n_3}
[[x_{\fudos\fudos 23},x_2]_c,x_{\fudos 2}]_c^{n_4} [[[x_{\fudos\fudos 23},x_2]_c,x_{\fudos 2}]_c,x_3]_c^{n_5} \\
[x_{\fudos\fudos 23}, x_2]_c^{n_6} x_3^{n_7} x_{3\fudos 2}^{n_8} [x_{32},x_{\fudos 2}]^{n_9} x_{32}^{n_{10}}
x_{\fudos 2}^{n_{11}} x_2^{n_{12}} : \, 0 \le n_i <2 \}.
\end{multline*}
\end{sloppypar}
\end{remark}

\begin{lemma}\label{lemma:other relations lstr-a-22}
Let $\cB$ be a quotient algebra of $T(V)$. Assume that \eqref{eq:lstr-rels&11disc-qserre-g=2}, \eqref{eq:lstr-rels&1-1disc-g=2},
\eqref{eq:lstr-rels-ghost-int-trivial}, \eqref{eq:lstr-rels-a-22-diagpart}, \eqref{eq:lstr-rels-a-22} hold in $\cB$.
Then there exists an algebra map $\phi:\cB(W)\to \cB$ such that $\phi(y_1)=x_{\fudos\fudos 2}$, $\phi(y_2)=x_{3}$,
$\phi(y_3)=x_{\fudos 2}$, $\phi(y_4)=x_2$.
\end{lemma}
\pf
Let $\Phi:T(W)\to \cB$ be the algebra map defined as $\phi$ on the $y_i$'s. We claim that $\Phi$ annihilates
all the relations in \eqref{eq:lstr-rels-a-22-K-1}, \eqref{eq:lstr-rels-a-22-K-2}, \eqref{eq:lstr-rels-a-22-K-3},
and the Lemma follows. By \eqref{eq:lstr-rels-ghost-int-trivial} and \eqref{eq:lstr-rels-a-22},
\begin{align*}
\Phi & ([y_{123},y_2]_c) =[[x_{\fudos\fudos 23}, x_{\fudos 2}]_c,x_3]_c = [x_{\fudos\fudos 23}, x_{\fudos 2 3}]_c \\
&= (x_{\alpha}x_{\fudos 2 3}+q_{12}q_{13} x_{\fudos 2 3}x_{\alpha})x_{\fudos 2 3}
-q_{12}q_{13} x_{\fudos 2 3}(x_{\alpha}x_{\fudos 2 3}+q_{12}q_{13} x_{\fudos 2 3}x_{\alpha})=0,
\end{align*}
since $x_{3\fudos 2}^2=0$ implies that $x_{\fudos 23}^2=0$, and $x_3^2=0$. Now $[x_{\fudos 23},x_{23}]_c=0$ because
$x_{23}^2=0$, and using again that $x_{\fudos 23}^2=0$,
\begin{align*}
\Phi([y_{124},y_2]_c) &= [[x_{\fudos\fudos 23}, x_2]_c,x_3]_c = [x_{\fudos\fudos 23}, x_{23}]_c=-2q_{12}q_{13}\, x_{\fudos 23}^2 =0.
\end{align*}
Similarly,
\begin{align*}
\Phi([y_{324},y_2]_c) &= [[x_{\fudos 23},x_2]_c,x_3]_c= [x_{\fudos 23},x_{23}]_c=0.
\end{align*}
Finally, $\Phi$ applies the remaining relations to \eqref{eq:lstr-rels&11disc-qserre-g=2}, \eqref{eq:lstr-rels&1-1disc-g=2},
\eqref{eq:lstr-rels-ghost-int-trivial}, \eqref{eq:lstr-rels-a-22-diagpart}, \eqref{eq:lstr-rels-a-22},
and the claim follows.
\epf

\begin{prop} \label{pr:lstr-a-22}
The algebra $\cB(\lstr(A_2,2))$ is presented by generators $x_i$, $i\in \Iw_3$, and relations
\eqref{eq:rels B(V(1,2))}, \eqref{eq:lstr-rels&11disc-1}, \eqref{eq:lstr-rels&11disc-qserre-g=2}, \eqref{eq:lstr-rels&1-1disc-g=2},
\eqref{eq:lstr-rels-ghost-int-trivial}, \eqref{eq:lstr-rels-a-22-diagpart}, \eqref{eq:lstr-rels-a-22}.
The set $B=$
\begin{multline*}
\{ x_1^{m_1} x_{\fudos}^{m_2} x_{\fudos\fudos 2}^{n_1} x_{\fudos\fudos 23}^{n_2}  [x_{\fudos\fudos 23}, x_{\fudos 2}]_c^{n_3}
[[x_{\fudos\fudos 23},x_2]_c,x_{\fudos 2}]_c^{n_4} [[[x_{\fudos\fudos 23},x_2]_c,x_{\fudos 2}]_c,x_3]_c^{n_5} \\
[x_{\fudos\fudos 23}, x_2]_c^{n_6} x_3^{n_7} x_{3\fudos 2}^{n_8} [x_{32},x_{\fudos 2}]^{n_9} x_{32}^{n_{10}}
x_{\fudos 2}^{n_{11}} x_2^{n_{12}} : \, 0\le m_1,m_2, \, 0\le n_i <2 \}.
\end{multline*}
is a basis of $\cB(\lstr(A_2,2))$ and $\GK \cB(\lstr(A_2,2)) = 2$.
\end{prop}
\pf
Analogous to Proposition \ref{pr:lstr-a(20)1}.
\epf

\subsubsection{The Nichols algebra $\cB(\lstr(A_{\theta - 1}))$} \label{subsubsection:lstr-a-n}

The subalgebra generated by $x_i$, $i\in\I_{2,\theta}$, is a Nichols algebra of type $A_{\theta - 1}$. Thus,
\begin{align}\label{eq:lstr-rels-a-n-diagpart}
x_{ij}^2 &=0, &  &2\le i\le j\le\theta,\\
\notag
[x_{k-1 \, k \, k+1}, x_k]&=0, & & 3\le k <\theta.
\end{align}

\medskip

Let $W$ be a braided vector space of diagonal type with Dynkin diagram
$$
\xymatrix{ \overset{-1}{\circ} \ar@{-}[r]^{-1}  &\overset{-1}{\circ} \ar@{-}[r]^{-1}  \ar@{-}[d]^{-1}
& \overset{-1}{\circ}  \\ & \overset{-1}{\circ} & } \dots
\xymatrix{ \overset{-1}{\circ} \ar@{-}[r]^{-1}  & \overset{-1}{\circ}\\ & }.
$$
Set as above $y_{ij}$, $i\leq j$, and for $k<\ell\in\I_\theta$,
\begin{align*}
\ch_{1\ell} & = [y_1, y_{3\ell}]_c, & \ch_{k\ell} & = [\ch_{1\ell}, y_{2k}], \quad k>1.
\end{align*}
We order the letters $y_{ij}$, $\ch_{k\ell}$ as follows:
\begin{itemize}
  \item $y_{ij}<y_{k\ell}$, $\ch_{ij}<\ch_{k\ell}$ if either $i<k$ or else $i=k$, $j<\ell$,
  \item $y_{1j}<\ch_{k\ell}<y_{mn}$ for all $j,k,\ell,m,n\in\I_\theta$, $m\geq 2$.
\end{itemize}
Given $\overline{a}=(a_n)\in\{0,1\}^{\theta(\theta-1)}$, let $\mathbf{y}^{\overline{a}}$ be the product of $y_{ij}^{a_n}$,
$\ch_{k\ell}^{\, a_n}$ with the previous order, where $n$ is the position of the letter $y_{ij}$,
$\ch_{k\ell}$.
By \cite[Thm. 3.1]{Ang-crelle}, $\cB(W)$ is presented by generators $y_i$, $i\in\I_\theta$, and relations
\begin{align} \label{eq:lstr-rels-a-n-K-1}
& y_{ij}^2, &  &[y_{k-1 \, k \, k+1}, y_k], \quad 3\le k <\theta,\\ \label{eq:lstr-rels-a-n-K-2}
& \ch_{k\ell}^2, &  &[y_{1 34}, y_3].
\end{align}
Here is a basis of $\cB(W)$: $B_W=\{ \mathbf{y}^{\overline{a}}: \overline{a}\in\{0,1\}^{\theta(\theta-1)} \}.$

\begin{remark}\label{rem:relations lstr-a-n}
\begin{sloppypar}
By Lemma \ref{lemma:braiding-K-weak-block-points}, $K^1$ is isomorphic to $W$ as braided vector spaces.
Hence there exists an isomorphism of braided Hopf algebras ${\psi: \cB(W)\to K}$ such that  $\psi(y_1)=x_{\fudos 2}$,
$\psi(y_j)=x_j$, $j\geq 2$. Set as above $x_{ij}$, $2\le i\le j$, and
\begin{align*}
x_{1j}&=\psi(y_{1j})=[x_{\fudos 2}, y_{2 j}]_c, &
\ya_{k\ell} & = \psi(\ch_{k\ell}) = [\ya_{1\ell}, x_{2k}], \quad k>1, \\
\ya_{1\ell} &=\psi(\ch_{1\ell}) = [x_{\fudos 2}, x_{3\ell}]_c, &
\mathbf{x}^{\overline{a}} &=\psi( \mathbf{y}^{\overline{a}}), \quad \overline{a}\in\{0,1\}^{\theta(\theta-1)}.
\end{align*}
Thus, the following identities hold in $\cB(\lstr(A_{\theta - 1}))$:
\begin{align}\label{eq:lstr-rels-a-n}
& x_{1j}^2, &  &\ya_{k\ell}^2, & j,k,\ell\in\I_\theta, & k<\ell,
\end{align}
and $B_K=\{ \mathbf{x}^{\overline{a}}: \overline{a}\in\{0,1\}^{\theta(\theta-1)} \}$ is a basis of $K$.
\end{sloppypar}
\end{remark}

\begin{lemma}\label{lemma:other relations lstr-a-n}
Let $\cB$ be a quotient algebra of $T(V)$. Assume that \eqref{eq:lstr1omega1-qserre-g=1}, \eqref{eq:lstr-rels&1-1disc-g=1},
\eqref{eq:lstr-rels-ghost-int-trivial}, \eqref{eq:lstr-rels-a-n-diagpart}, \eqref{eq:lstr-rels-a-n} hold in $\cB$.
Then there exists an algebra map $\phi:\cB(W)\to \cB$ such that $\phi(y_1)=x_{\fudos 2}$, $\phi(y_j)=x_j$, $j\geq 2$.
\end{lemma}
\pf
Let $\Phi:T(W)\to \cB$ be the algebra map defined as $\phi$ on the $y_i$'s. We claim that $\Phi$ annihilates
all the relations in \eqref{eq:lstr-rels-a-n-K-1}, \eqref{eq:lstr-rels-a-n-K-2},
and the Lemma follows. By \eqref{eq:lstr-rels-ghost-int-trivial} and \eqref{eq:lstr-rels-a-n-diagpart},
\begin{align*}
\Phi([y_{134},y_3]_c) &= [[x_{\fudos 2}, x_{34}]_c,x_3]_c = (\ad_c x_{\fudos})[x_{234},x_3]_c=0.
\end{align*}
Finally, $\Phi$ applies the remaining relations to \eqref{eq:lstr-rels-ghost-int-trivial}, \eqref{eq:lstr-rels-a-n-diagpart},
\eqref{eq:lstr-rels-a-n}, and the claim follows.
\epf

\begin{prop} \label{pr:lstr-a-n}
The algebra $\cB(\lstr(A_{\theta - 1}))$ is presented by generators $x_i$, $i\in \Iw_{\theta}$, and relations
\eqref{eq:rels B(V(1,2))}, \eqref{eq:lstr-rels&11disc-1},
\eqref{eq:lstr1omega1-qserre-g=1}, \eqref{eq:lstr-rels&1-1disc-g=1},
\eqref{eq:lstr-rels-ghost-int-trivial}, \eqref{eq:lstr-rels-a-n-diagpart}, \eqref{eq:lstr-rels-a-n}.
The set $B =\{ x_1^{m_1} x_{\fudos}^{m_2} \mathbf{x}^{\overline{a}}: \, m_i\in\N_0, \, \overline{a}\in\{0,1\}^{\theta(\theta-1)} \}$
is a basis of $\cB(\lstr(A_{\theta - 1}))$ and $\GK \cB(\lstr(A_{\theta - 1})) = 2$.
\end{prop}
\pf
Analogous to Proposition \ref{pr:lstr-a(20)1}.
\epf

\subsection{Proof of Theorem \ref{thm:points-block-eps-1} ($\epsilon = -1$)} \label{subsec:pf-main-block-points-mild}

Here $\cB(V_1)$ is a super Jordan plane. Let $j \in \I_{2, \theta}$. If $\GK \NA(V) < \infty$, then we see from Table \ref{tab:finiteGK-block-point} the following possibilities:

\begin{table}[ht]
	\caption{}\label{tab:possible-block-points}
	\begin{center}
		\begin{tabular}{|c|c|c|}
			\hline \scriptsize{interaction} $\inc_j = q_{1j}q_{j1}$ &  $q_{jj}$  & $\ghost_j$   \\
			\hline
			$1$ &  $\in \ku^{\times}$ & $0$
			\\ \hline
			$1$ & $\pm 1$ &   \small{discrete}
			\\ \hline
			$-1$  & $-1$  &   1
			
			\\\hline
		\end{tabular}
	\end{center}
\end{table}

We start by a result that will be applied several times.

\begin{lemma}\label{lemma:superJordan-2-mild}
	Let $W = V_1 \oplus U$ be a direct sum of braided vector spaces, where $V_1$ and $U$ have dimension 2, $V_1$ is a $-1$-block
	with basis $x_1, x_{\fudos}$
	and $U$ is of diagonal type with respect to a basis $x_2, x_3$. Assume that both $x_2$ and $x_3$ have mild interaction with $V_1$.
	Then $\GK \cB(W) = \infty$.
\end{lemma}

\pf
By Theorem \ref{thm:pm1bp-mild}, $q_{22} = q_{33} = -1$. Set $q = q_{23}q_{32}$.
We consider the flag of braided subspaces:
$0 =\cV_0 \subsetneq \cV_1  \subsetneq \cV_2 =W$, where $\cV_1$ is spanned by $(x_i)_{i\in \I_3}$.
Let  $\Bdiag := \gr \cB (W)$, a pre-Nichols algebra of  $\cV^{\textrm{diag}}$, see \S \ref{subsection:filtr-nichols}. Thus $\GK \toba(W) = \infty$ since the Dynkin diagram of $\cV^{\textrm{diag}}$ is:
\begin{align*}
\xymatrix{ \overset{-1} {\underset{2}{\circ}} \ar  @{-}[r]^{-1} \ar  @{-}[d]_{-1}
	\ar @{-}[rd]_{q}	& \overset{-1} {\underset{\fudos}{\circ}} \ar  @{-}[d]^{-1}
	\\ \overset{-1} {\underset{1}{\circ}} \ar  @{-}[r]_{-1}  & \overset{-1} {\underset{3}{\circ}},}
\end{align*}
by Hypothesis \ref{hyp:nichols-diagonal-finite-gkd} (or by Theorem \ref{thm:nichols-diagonal-finite-gkd} if $q = \pm 1$).
\epf

\begin{coro} If $\GK \NA(V) < \infty$, then  $V_{\diag}$ has at most one connected component with mild interaction.
\end{coro}

\subsubsection{Connected components of $V_{\diag}$} Let $J \in \X$. We assume that $\ghost_J \neq 0$.
One of the implications in Theorem \ref{thm:points-block-eps-1} follows from \S \ref{sec:yd-dim3}, so we assume that
$\GK \NA(V_1 \oplus V_J) < \infty$.

\begin{lemma}\label{lemma:conncomp-eps-1-weak}
If the interaction of $J$ is weak, then $\vert J \vert = 1$.
\end{lemma}

\pf Let $i \in J$ such that $\ghost_i > 0$.
Then $q_{ii} = \pm 1$ by Theorem \ref{thm:pm1bp}.
If $q_{ii} = 1$, then $J = \{i\}$ as already discussed, cf. page \pageref{page:qii1}.
If $q_{ii} = -1$ and $\vert J \vert >1$, then there is $j\in J - \{i\}$ with $q := q_{ij}q_{ji} \neq 1$. By Lemma \ref{lemma:braiding-K-weak-block-points},
$p_{i1,i1} = 1$ and $p_{i1,j0} = q$; hence $\GK \toba(K_J) =\infty$ by Lemma \ref{lemma:points-trivial-braiding}.
\epf

\emph{Assume from the rest of this Subsection that  the interaction of $J$ is mild.}

\smallbreak
In the next proof, we denote  $x_{i_1i_2 \dots i_M} = \ad_c x_{i_1}\, x_{i_2 \dots i_M}$, cf. \eqref{eq:xijk}.

\begin{lemma}\label{lemma:block-points-eps-1-aux1}
If $\vert J \vert > 1$, then $\vert J \vert = 2$, say  $J=\{2,3\}$; and
$q_{22}=q_{23}q_{32}=q_{33}=-1$, $\inc_2 = -1$, $\inc_3 =1$, $\ghost_2 = 1$, $\ghost_3 = 0$.
\end{lemma}

\pf First we assume that $\vert J\vert = 2$, say $J=\{ 2,3 \}$.
By Lemma \ref{lemma:superJordan-2-mild}, we  may suppose that $\inc_2 = -1$ (thus $q_{22} = - 1$), $\inc_3 = 1$.
Set $q = q_{23}q_{32}$,  $r = q_{33}$.
We consider the flag of braided subspaces:
$0 =\cV_0 \subsetneq \cV_1  \subsetneq \cV_2 = V_1 \oplus V_J$, where $\cV_1$ is spanned by $(x_i)_{i\in \I_3}$.
Let  $\Bdiag := \gr \cB (V_1 \oplus V_J)$, a pre-Nichols algebra of  $\cV^{\textrm{diag}}$, see \S \ref{subsection:filtr-nichols}. Now the Dynkin diagram of $U = \cV^{\textrm{diag}}$ has vertices $\{1, \fudos, 2,3 \}$
and edges as follows:
\begin{align*}
\xymatrix{ & \overset{-1}{\underset{1}{\circ}} \ar  @{-}[d]^{-1} &   \\
	\overset{-1}{\underset{\fudos}{\circ}}\ar  @{-}[r]_{-1}  & \overset{-1}{\underset{2}{\circ}} \ar  @{-}[r]_{q}  & \overset{r}{\underset{3}{\circ}}.}
\end{align*}

\begin{casso}
	$(q, r) \neq (-1, -1)$. We distinguish two cases:
\end{casso}

\medbreak
\begin{itemize} [leftmargin=*] 
	\item Either $q \notin\G_{\infty}$ or $r \notin\G_{\infty}$. If $q\neq r^{-1}, r^{-2}$, then
	$\GK \toba (U) = \infty$ by Theorem \ref{thm:nichols-diagonal-finite-gkd}. For $q= r^{-1}$, respectively $q=r^{-2}$, we apply the
	reflections as described below:
	
	
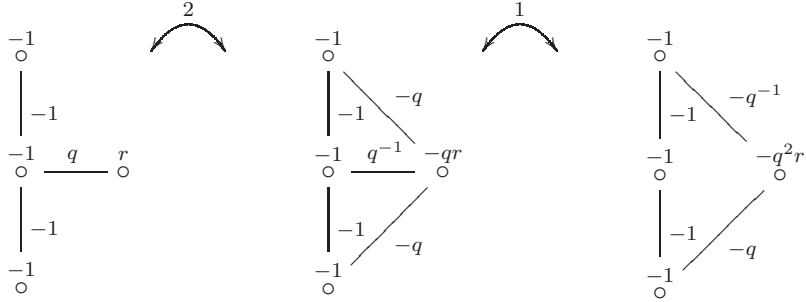
\begin{align*}
	\xymatrix{  \overset{-1}{\circ} \ar  @{-}[d]^{-1}  &  \\
		\overset{-1}{\circ} \ar  @{-}[d]^{-1} \ar  @{-}[r]^{q}  & \overset{r}{\circ} \\
		\overset{-1}{\circ}  & }
	&
	\xymatrix{\ar@/^1pc/@{<->}[r]^{2} & }
	&
	\xymatrix{  \overset{-1}{\circ} \ar  @{-}[d]^{-1} \ar  @{-}[dr]^{-q}  &  \\
		\overset{-1}{\circ} \ar  @{-}[d]^{-1} \ar  @{-}[r]^{q^{-1}}  & \overset{-qr}{\circ} \\
		\overset{-1}{\circ} \ar  @{-}[ur]_{-q} & }
	&
	\xymatrix{\ar@/^1pc/@{<->}[r]^{1} & }
	&
	\xymatrix{  \overset{-1}{\circ} \ar  @{-}[d]^{-1} \ar  @{-}[dr]^{-q^{-1}}  &  \\
		\overset{-1}{\circ} \ar  @{-}[d]^{-1}   & \overset{-q^2r}{\circ} \\
		\overset{-1}{\circ} \ar  @{-}[ur]_{-q} & }
	\end{align*}
	Then $\GK \toba (\cR^1\cR^2(U)) = \infty$, respectively $\GK \toba (\cR^2(U)) = \infty$, by Theorem \ref{thm:nichols-diagonal-finite-gkd}.
	
	\item $q, r \in\G_{\infty}$. Thus $\GK \toba(V_1 \oplus V_J) = \infty$
	by Hypothesis \ref{hyp:nichols-diagonal-finite-gkd}.
\end{itemize}

\begin{casso}
	$(q, r) = (-1, -1)$, $\ghost_ 3 > 0$.
\end{casso}

For simplicity of the proof of this Step, we consider a different flag of braided subspaces:
$0 =\cV_0 \subsetneq \cV_1  \subsetneq \cV_2 = \cV_3  \subsetneq \cV_4 = V_1 \oplus V_J$, where $\cV_1$
is spanned by $(x_i)_{i\in \I_2}$ and $\cV_2$ is spanned by $(x_i)_{i\in \Iw_2}$.
Let  $\Bdiag := \gr \cB (V_1 \oplus V_J)$, a pre-Nichols algebra of  $\cV^{\textrm{diag}}$, see \S \ref{subsection:filtr-nichols},
with the same Dynkin diagram:
\begin{align*}
\xymatrix{ & \overset{-1}{\underset{1}{\circ}} \ar  @{-}[d]^{-1} &   \\
	\overset{-1}{\underset{\fudos}{\circ}}\ar  @{-}[r]_{-1}  & \overset{-1}{\underset{2}{\circ}} \ar  @{-}[r]_{-1}  & \overset{-1}{\underset{3}{\circ}}.}
\end{align*}

As in \S \ref{sec:yd-dim3}, see \eqref{eq:zn}, we set $z_n := (\ad_c x_{\fudos})^n x_3$, $n\in\N_0$.
Then $z_1,z_2 \neq 0$ by Lemma \ref{lemma:derivations-zn}, and $z_2\in\cB^3_5$. By direct computation,
\begin{align*}
\Delta(z_2) &= z_2\otimes 1 -a x_{\fudos 1}\otimes x_3+a  x_1\otimes z_1 + 1\otimes z_2.
\end{align*}
Suppose that $z_2\in\cB^3_4$. Notice that $\cB^3_4$ is spanned by the monomials in letters $x_i$, $i\in \Iw_2$
with at most one $x_{\fudos}$, since $T(V)^3_4$ is spanned by these monomials, cf. Remark \ref{obs:filtered} and Lemma \ref{lemma:filtered}.
We check that all these monomials are written as a linear combination of $x_1 x_{\fudos 1}$, $x_1 x_{\fudos} z_0$, $x_1 z_1$,
$x_{\fudos 1} z_0$, $z_1 z_0$, where we use the defining relations of $\cB(\lstr_{-}( -1, \ghost_3))$, cf. Proposition \ref{pr:lstr-1-1disc}.
Then $z_2$ is a linear combination of these elements and we get a contradiction with Proposition \ref{pr:lstr-1-1disc}.
Thus $z_2\in\cB^3_5-\cB^3_4$.

Let $z$ be class of $z_2$ in $\Bdiag$. Then $z$ is a non-zero primitive element in $\Bdiag$.
Let $\widetilde{\cB}_1$ be the subalgebra of $\Bdiag$ generated by $Z = \langle x_1,x_{\fudos},x_2,x_3,z \rangle$,
and consider the algebra filtration of $\widetilde{\cB }_1$, such that the generators have degree one.
This is a Hopf algebra filtration, hence the associated graded algebra  $\widetilde{\cB}_2$ is a braided Hopf algebra.
The Nichols algebra $\NA(Z)$ is a subquotient of $\widetilde{\cB}_2$.
Then $\GK \NA(Z)\le
  \GK \widetilde{\cB}_2\le \GK \widetilde{\cB}_1\le \GK\Bdiag
=\GK \cB (V_1 \oplus V_J)$.
The braided vector space $Z$ is of Cartan type $D_4^{(1)}$, since it has diagram
$$ \xymatrix{ \overset{-1}{\underset{1}{\circ}} \ar  @{-}[rd]^{-1} & &  \overset{-1}{\underset{z}{\circ}} \\
\overset{-1}{\underset{\fudos}{\circ}}\ar@{-}[r]_{-1}  & \overset{-1}{\underset{2}{\circ}} \ar@{-}[r]_{-1}  \ar@{-}[ru]_{-1}
& \overset{-1}{\underset{3}{\circ}},} $$
so $\GK\NA(Z)=\infty $ by Theorem \ref{thm:nichols-diagonal-finite-gkd}, and then $\GK \cB (V_1 \oplus V_J)=\infty $.

\begin{casso}
	$\vert J\vert \geq 3$.
\end{casso}

We assume that $\vert J\vert = 3$, say $J=\{ 2,3,4 \}$, and obtain a contradiction. Then the general case follows.
By Lemma \ref{lemma:superJordan-2-mild}, we  may suppose that $\inc_2 = -1$, $\inc_3 = \inc_4 = 1$,
and by the previous steps $q_{22} = q_{23}q_{32} = q_{33} = - 1$, $\ghost_3=0$. Set $q = q_{24}q_{42}$,
$r = q_{44}$, $s=q_{34}q_{43}$. We consider the flag of braided subspaces:
$0 =\cV_0 \subsetneq \cV_1  \subsetneq \cV_2 = V_1 \oplus V_J$, where $\cV_1$ is spanned by $(x_i)_{i\in \I_4}$.
Let  $\Bdiag := \gr \cB (V_1 \oplus V_J)$, a pre-Nichols algebra of  $\cV^{\textrm{diag}}$, see \S \ref{subsection:filtr-nichols}.
Now the Dynkin diagram of $U = \cV^{\textrm{diag}}$ has vertices $\{1, \fudos, 2,3,4 \}$
and edges as follows:
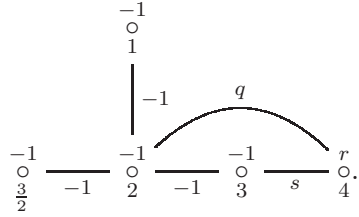
\begin{align*}
\xymatrix{ & \overset{-1}{\underset{1}{\circ}} \ar  @{-}[d]^{-1} &  & \\
	\overset{-1}{\underset{\fudos}{\circ}}\ar  @{-}[r]_{-1}  & \overset{-1}{\underset{2}{\circ}} \ar@/^2pc/@{-}[rr]^{q} \ar  @{-}[r]_{-1}
& \overset{-1}{\underset{3}{\circ}}  \ar  @{-}[r]_{s} & \overset{r}{\underset{4}{\circ}}.}
\end{align*}
By Hypothesis \ref{hyp:nichols-diagonal-finite-gkd}, $q=1$ (otherwise the vertex 2 would have valence 4, contradicting \cite{H-classif}) and $r=s=-1$ (the only posible extension of $D_4$ in the list in \cite{H-classif}). Thus $\cV_1$ is of Cartan type $A_4$ and
\begin{align*}
B=\{ x_1^{n_1}x_{12}^{n_2}x_{123}^{n_3}x_{1234}^{n_4}x_{2}^{n_5}x_{23}^{n_6}x_{234}^{n_7}x_{3}^{n_8}x_{34}^{n_9}x_{4}^{n_{10}}: 0\le n_i\le 1 \}
\end{align*}
is a basis of $\cB(\cV_1)$. Let $B_4$ be the subset of $B$ of those elements of degree 4.

Notice that $\cB^4_j=0$ for $j\le 3$ and $\cB^4_4$ is spanned by the monomials in letters $x_i$, $i\in \I_4$,
since $T(V)^4_j=0$ and $T(V)^4_4$ is spanned by these monomials, cf. Remark \ref{obs:filtered}
and Lemma \ref{lemma:filtered}. The subalgebra generated by $\cV_1$ is the Nichols algebra of
$\cV_1$; thus $B_4$ is a basis of $\cB^5_5$.

We observe that $(\ad_c x_{\fudos}) x_3=0$ because $\inc_3 =1$ and $\ghost_3 =0$.
We next claim that $\partial_2(x_{\fudos23})=0$. Indeed, $\partial_2(x_{\fudos23}) =$
\begin{align*}
&= \partial_2(x_{\fudos2}x_3 - q_{13}q_{23} x_3 x_{\fudos2}) \overset{\eqref{eq:partial3z1}}{=}
(2x_{\fudos} + x_1) q_{23} x_3 -  q_{13}q_{23} x_3 (2x_{\fudos} + x_1) \\
&=  q_{23}\left(2(x_{\fudos} x_3 -  q_{13}x_3 x_{\fudos}) + (x_1x_3 -  q_{13}x_3 x_1) \right)  = 0.
\end{align*}
We next introduce $u=[x_{\fudos23},x_2]_c$. We observe that $\partial_1(u)=0$, and
\begin{align*}
\partial_2(u) &= \partial_2(x_{\fudos23} x_2+q_{12}q_{32} x_2 x_{\fudos23})
= x_{\fudos23}+q_{12}q_{32}\, g_2\cdot x_{\fudos23} \\
& = x_{\fudos23}+q_{12}q_{32}q_{21}q_{22}q_{23} ( x_{\fudos23}+x_{123})= -x_{123}, \\
\partial_3(u) &= \partial_3(x_{\fudos23} x_2+q_{12}q_{32} x_2 x_{\fudos23})
= 2 q_{32} (x_{\fudos 2} x_2+q_{12}x_2 x_{\fudos 2})=0.
\end{align*}
where we use \eqref{eq:nichols-mild-4} for the last equality.
Suppose that $u\in\cB^4_4$; i.e. $u$ is a linear combination of the elements of $B_4$. As $\partial_4(x_{1234})=2x_{123}$,
$\partial_4(x_{234})=2x_{23}$, $\partial_4(x_{34})=2x_{3}$, $\partial_4(x_4)=1$, and $\partial_4$ annihilates $u$
and the remaining PBW generators, the elements of $B_4$ containing $x_{1234}$, $x_{234}$, $x_{34}$, $x_4$ have coefficient 0.
For the elements of the shape $x_1 y \in B_4$, $\partial_1(x_1y)$ gives $y$ up to a non-zero scalar. As $\partial_1$ annihilates
$u$ and the remaining elements of $B_4$, the elements $x_1 y \in B_4$ also have coefficient 0. Thus
\begin{align}\label{eq:mild-2points-combination-u}
u &= b_1 \, x_{12}x_{23}+b_2 \, x_{12}x_2x_3 +  b_3 \, x_{123}x_2  +  b_4 \, x_{123}x_3+ b_5 \, x_2x_{23}x_3,
\end{align}
for some $b_i\in\ku$. Applying $\partial_2$ to this equality gives
\begin{align*}
-x_{123} &= q_{23}(b_2-2 b_1) \, x_{1}x_{23}-2b_2q_{23} \, x_{1}x_2x_3 + b_3 \, x_{123} - b_5 q_{23}^2 \, x_{23}x_3,
\end{align*}
so $b_1=b_2=b_5=0$, $b_3=-1$. Applying $\partial_3$ to \eqref{eq:mild-2points-combination-u},
\begin{align*}
0 &= - 2q_{32}\, x_{12}x_2  -2 b_4 \, x_{12}x_3 +  b_4 \, x_{123},
\end{align*}
so we have a contradiction. Thus $u \in \cB^4_5-\cB^4_4$. Next we compute
\begin{align*}
\Delta(x_{\fudos23}) &= x_{\fudos23}\otimes 1+ 2 x_{\fudos2} \otimes x_3 + (2x_{\fudos}+x_1)\otimes x_{23}+1\otimes x_{\fudos23}, \\
\Delta(u) &= u\otimes 1 - 2 x_{21}\otimes x_3x_2 + 1\otimes u.
\end{align*}
Let $\overline{u}$ be the class of $u$ in $\Bdiag$. Thus $\overline{u}$ is a non-zero primitive element in $\Bdiag$.
Let $\widetilde{\cB}_1$ be the subalgebra of $\Bdiag$ generated by $Z = \langle x_4,  \overline{u}\rangle$,
and consider the algebra filtration of
$\widetilde{\cB }_1$, such that the generators $\overline{u}$ and $x_4$ have degree one. This is a Hopf algebra filtration, hence
the associated graded algebra  $\widetilde{\cB}_2$ is a braided Hopf algebra,  and $\NA(Z)$ is a subquotient of $\widetilde{\cB}_2$. Then
$ \GK \NA(Z)\le \GK \widetilde{\cB}_2\le \GK \widetilde{\cB}_1\le \GK\Bdiag
=\GK \cB (V_1 \oplus V_J)$.
The braided vector space $Z$ has diagram $\xymatrix{ \overset{-1}{\circ} \ar@{-}[r]_{-1} & \overset{1}{\circ}}$,
so $\GK\NA(Z)=\infty $ by Lemma \ref{lemma:points-trivial-braiding}, and then $\GK \cB (V_1 \oplus V_J)=\infty $.
\epf

\subsubsection{The Nichols algebra  $\cB(\cyc_2)$}\label{subsubsection:nichols-mild-cyc2}

Assume that $\vert J \vert = 2$, say  $J=\{2,3\}$; and
$q_{22}=q_{23}q_{32}=q_{33}=-1$, $\inc_2 = -1$, $\inc_3 =1$, $\ghost_2 = 1$, $\ghost_3 = 0$.
Recall that the corresponding braided vector space is denoted $\cyc_2$.
Recall the relations of $\cB(\cyc_1)$, with the change of index with respect to \S  \ref{sec:yd-dim3}
(the $2$ and $3$ there are now $\fudos$ and $2$):
\begin{align*}
\tag{\ref{eq:rels-B(V(-1,2))-1}} &x_1^2,
\\
\tag{\ref{eq:rels-B(V(-1,2))-2}} & x_{\fudos}x_{\fudos 1}- x_{\fudos 1}x_{\fudos} - x_1x_{\fudos 1},
\\
\tag{\ref{eq:nichols-mild-relation1}} & x_{\fudos}x_{12}+q_{12}x_{12}x_{\fudos}+x_{1 \fudos 2} \\
\tag{\ref{eq:nichols-mild-relation2}} & x_{\fudos}x_{\fudos 2}+q_{12}x_{\fudos 2}x_{\fudos}-\frac 1 2 x_{1 \fudos 2} - q_{12}x_{12}x_{\fudos}, \\
\tag{\ref{eq:nichols-mild-relation3}} & x_{\fudos}x_{1 \fudos 2} -q_{12}x_{1 \fudos 2}x_{\fudos}, \\
\tag{\ref{eq:nichols-mild-relation4}} & x_2^2, \qquad x_{12}^2, \\
\tag{\ref{eq:nichols-mild-relation5}} & x_{\fudos 2}^2, \qquad x_{1 \fudos 2}^2.
\end{align*}

\begin{remark}\label{rem:nichols-mild-cyc2-defrelations}
We have
\begin{align}\label{eq:nichols-mild-cyc2-ghost-int-trivial}
x_1x_3 &= q_{13} x_3x_1, & x_{\fudos} x_3 &= q_{13} x_3 x_{\fudos},
\end{align}
since $q_{13}q_{31}=1$ and $\ghost_3=0$.
The subalgebra generated by $x_1$, $x_2$, $x_3$ is a Nichols algebra of diagonal type $A_3$. Thus,
\begin{align}
\label{eq:nichols-mild-cyc2-A3-1}  x_3^2  &=0, &  x_{123}^2 & =0, \\
\label{eq:nichols-mild-cyc2-A3-2}  x_{23}^2 & =0, & [x_{123}, x_2]_c &= 0.
\end{align}
\end{remark}

\begin{lemma} \label{lem:nichols-mild-cyc2-defrelations} The following hold in $\NA (V)$:
	\begin{align}
x_{\fudos 23}^2 &=0, \label{eq:nichols-mild-cyc2-relation1}\\
x_{1\fudos 23}^2 &=0, \label{eq:nichols-mild-cyc2-relation2}\\
[x_{\fudos 2}, x_{123}]_c^2 &=0, \label{eq:nichols-mild-cyc2-relation3}\\
x_{123}x_{\fudos 23} &=-x_{\fudos 23}x_{123}, \label{eq:nichols-mild-cyc2-relation4}\\
[x_1, [x_{\fudos 23}, x_2]_c]_c &= q_{12}q_{13} x_{123} x_{12}, \label{eq:nichols-mild-cyc2-relation5}\\
[x_1, [x_{\fudos 2}, x_{123}]_c]_c &= 2q_{12}x_{12}x_{1\fudos 23}, \label{eq:nichols-mild-cyc2-relation6}\\
[x_{\fudos}, [x_{\fudos 2}, x_{123}]_c]_c &= 2q_{12}x_{12}x_{1\fudos 23}-2q_{12}x_{\fudos 2}x_{1\fudos 23}. \label{eq:nichols-mild-cyc2-relation7}
\end{align}
\end{lemma}

\begin{proof}
We claim that $\partial_i(x_{\fudos 23}^2)=0$ for all $i\in\Iw_3$. For $i\neq 3$, we notice that
$\partial_1(x_{\fudos 23})=\partial_{\fudos}(x_{\fudos 23})=\partial_2(x_{\fudos 23})=0$. For $i=3$,
$\partial_3(x_{\fudos 23})=2 x_{\fudos 2}$, and as $x_{\fudos 2}^2=0$, we deduce that
\begin{align}\label{eq:nichols-mild-cyc2-other-rel1}
x_{\fudos 2}x_{\fudos 23}=-q_{13}q_{23}x_{\fudos 23}x_{\fudos 2}
\end{align}
Thus
\begin{align*}
\partial_3(x_{\fudos 23}^2)& = 2x_{\fudos 23}x_{\fudos 2}-2q_{31}q_{32}x_{\fudos 2}x_{\fudos 23}=0.
\end{align*}
Analogously we prove  that $x_{1\fudos 23}^2=0$, since $\partial_1$, $\partial_2$ and $\partial_{\fudos}$ annihilate $x_{1\fudos 23}$, $\partial_3(x_{1\fudos 23})=2x_{1\fudos 2}$ and as $x_{1\fudos 2}^2=0$ we have
\begin{align}\label{eq:nichols-mild-cyc2-other-rel2}
x_{1\fudos 2}x_{1\fudos 23}&= x_{1\fudos 2}(x_{1\fudos 2}x_3- q_{13}^2q_{23}^2x_3x_{1\fudos 2}) = -q_{13}^2q_{23}^2x_{1\fudos 23}x_{1\fudos 23}.
\end{align}
Thus \eqref{eq:nichols-mild-cyc2-relation1} and \eqref{eq:nichols-mild-cyc2-relation2} hold in $\cB(V)$. For \eqref{eq:nichols-mild-cyc2-relation4},
we use \eqref{eq:nichols-mild-cyc2-relation1} and \eqref{eq:nichols-mild-cyc2-A3-1}:
\begin{align*}
0 & = g_1\cdot x_{\fudos 23}^2 = x_{\fudos 23}^2 -x_{123}x_{\fudos 23}-x_{\fudos 23}x_{123}+x_{123}^2
= - x_{123}x_{\fudos 23}-x_{\fudos 23}x_{123}.
\end{align*}

Let $u = [x_{\fudos 23}, x_2]_c$, $v=[x_{\fudos 2},x_{123}]_c$. By the second relation in \eqref{eq:nichols-mild-cyc2-A3-2}
\begin{align*} g_1\cdot u &= q_{12}^2q_{13}([x_{123},x_2]_c-u)= -q_{12}^2q_{13} \, u \implies
[x_1, u]_c = x_1 u +q_{12}^2q_{13} u x_1. \end{align*}
We claim that \eqref{eq:nichols-mild-cyc2-relation5} holds in $\cB(V)$.
By direct computation,
\begin{align}\label{eq:nichols-mild-derivations-u}
\partial_1(u)&=\partial_{\fudos}(u)=\partial_3(u)=0, & \partial_2(u)&=-x_{123}.
\end{align}
Thus, $\partial_1([x_1, u]_c)=\partial_{\fudos}([x_1, u]_c)=\partial_3([x_1, u]_c) = 0$, and $\partial_2([x_1, u]_c) =$ \begin{align*} & = x_1 \partial_2(u)+q_{12}^2 q_{13} \partial_2(u) \, g_2\cdot x_1 = -x_1 x_{123}+ q_{12}q_{13} \, x_{123} x_1 = 2q_{12}q_{13} \, x_{123} x_1 \end{align*} since $x_1x_{123}=-q_{12}q_{13} \, x_{123}x_1$, as $\ad_c^2 x_1 =0$, by (3.7). Also,
\begin{align*}
\partial_2(x_{123} x_{12}) & = 2\, x_{123} x_{12}, & \partial_3(x_{123} x_{12}) & = 2q_{31}q_{32} \, x_{12}^2=0.
\end{align*}

Now we claim that $\partial_i(x_1v-q_{12}^2q_{13} vx_1)= 2q_{12} \partial_i(x_{12}x_{1\fudos 23})$
for all $i\in\Iw_3$, i.e. \eqref{eq:nichols-mild-cyc2-relation6} holds in $\cB(V)$.
By direct computation, $\partial_1(v)=\partial_{\fudos}(v)=0$,
\begin{align*}
\partial_2(v) &= \partial_2(x_{\fudos 2}x_{123}+q_{13}q_{23}x_{123}x_{\fudos 2}) \\
&= -q_{21}q_{23} (2x_{\fudos}+x_1)x_{123}+q_{13}q_{23} \, x_{123} (2x_{\fudos}+x_1) \\
& = -2q_{31}q_{12} x_{1\fudos 23}, \\
\partial_3(v) &= 2 x_{\fudos 2}x_12+ 2q_{13}q_{31}q_{23}q_{32} \, x_{12}x_{\fudos 2} = -4x_{\fudos 2}x_{12},
\end{align*}
where we use \eqref{eq:nichols-mild-relation1}-\eqref{eq:nichols-mild-relation5}. Thus,
$\partial_1$, $\partial_{\fudos}$ annihilate both sides of \eqref{eq:nichols-mild-cyc2-relation6}, and
\begin{align*}
\partial_2(x_1v-q_{12}^2q_{13} vx_1) & = 2q_{21}q_{23} (x_1x_{1\fudos 23}+q_{12}q_{13}x_{1\fudos 23}x_1)= 4q_{21}q_{23} \, x_1x_{1\fudos 23}, \\
\partial_3(x_1v-q_{12}^2q_{13} vx_1) & = -4x_1 x_{\fudos 2}x_{12}+4q_{12}^2 \, x_{\fudos 2}x_{12}x_1=4q_{12}x_{12}x_{1\fudos 2}, \\
\partial_2(x_{12}x_{1\fudos 23}) &= -2q_{21}^2q_{23} \, x_1 x_{1\fudos 23},\\
\partial_3(x_{12}x_{1\fudos 23}) &= 2x_{12}x_{1\fudos 2}.
\end{align*}
Similarly \eqref{eq:nichols-mild-cyc2-relation7} holds in $\cB(V)$ since $\partial_1$, $\partial_{\fudos}$ annihilate both sides, and
\begin{align*}
\partial_2(x_{\fudos}v-q_{12}^2q_{13} vx_{\fudos}) & = 4q_{21}q_{23} \, x_{\fudos}x_{1\fudos 23}+2q_{21}q_{23} \, x_1x_{1\fudos 23} \\
&= 2q_{12}\partial_2(x_{12}x_{1\fudos 23}-x_{\fudos 2}x_{1\fudos 23}),\\
\partial_3(x_{\fudos}v-q_{12}^2q_{13} vx_{\fudos}) & = -4x_1 x_{\fudos 2}x_{12}+4q_{12}^2 \, x_{\fudos 2}x_{12}x_1=4q_{12}x_{12}x_{1\fudos 2}, \\
&= 2q_{12}\partial_3(x_{12}x_{1\fudos 23}-x_{\fudos 2}x_{1\fudos 23}).
\end{align*}

Finally we claim that $\partial_i(v^2)=0$ for all $i\in\Iw_3$. It suffices to prove that $\partial_2(v^2)=\partial_3(v^2)=0$. By direct computation,
\begin{align}\label{eq:nichols-mild-cyc2-other-rel3}
[x_{1\fudos 23}, x_2]_c  &= -q_{12}q_{13} v.
\end{align}
From this equation and \eqref{eq:nichols-mild-cyc2-relation2},
\begin{align}\label{eq:nichols-mild-cyc2-other-rel4}
x_{1\fudos 23} v = q_{12}^2q_{32} \, v x_{1\fudos 23}.
\end{align}
Thus,
\begin{align*}
\partial_2(v^2) &= -2q_{31}q_{12} \, v x_{1\fudos 23} -2q_{31}q_{12}q_{21}^2q_{23} \, x_{1\fudos 23}v=0.
\end{align*}
As $x_2^2=0$, $v x_2= q_{12}^2 x_2 v$. Using this equation and \eqref{eq:nichols-mild-cyc2-relation6}, \eqref{eq:nichols-mild-cyc2-relation7},
\begin{align}\label{eq:nichols-mild-cyc2-other-rel5}
x_{12} v &= q_{13}q_{23} \, v x_{12}, & x_{\fudos 2} v &= q_{13}q_{23} \, v x_{\fudos 2}.
\end{align}
Then we compute:
\begin{align*}
\partial_3(v^2) &= 4q_{31}^2q_{32}^2 \, x_{12}x_{\fudos 2} x_{1\fudos 23} -4 \, v x_{12}x_{\fudos 2}=0.
\end{align*}
Thus \eqref{eq:nichols-mild-cyc2-relation3} holds in $\cB(V)$.
\end{proof}

We fix the notation $u = [x_{\fudos 23}, x_2]_c$, $v=[x_{\fudos 2},x_{123}]_c$ for the rest of this Subsection,
as in the proof of last Lemma.

\begin{remark}\label{rem:nichols-mild-cyc2-relations}
By definition of $x_{1\fudos 23}$, $x_{123}$, $x_{\fudos 23}$, $v$, $u$, $x_{23}$, we have
\begin{align*}
x_1 x_{\fudos 23} &= x_{1\fudos 23} + q_{12}q_{13} (x_{123} - x_{\fudos 23}) x_{1}, &
x_1 x_{23} &= x_{123}+q_{12}q_{13} x_{23} x_{1}, \\
x_{\fudos} x_{23} &= x_{\fudos 23}+q_{12}q_{13} x_{23} x_{\fudos}, &
x_{\fudos 2} x_{123} &= v - q_{13}q_{23}x_{123} x_{\fudos 2}, \\
x_{\fudos 23} x_2 &= u-q_{12}q_{32} x_2 x_{\fudos 23}, &
x_2 x_{3} &= x_{23}+q_{23} x_{3} x_2.
\end{align*}
\end{remark}

\begin{lemma}\label{lem:nichols-mild-cyc2-relations}
Let $\cB$ be a quotient algebra of $T(V)$. Assume that \eqref{eq:rels-B(V(-1,2))-1}, \eqref{eq:rels-B(V(-1,2))-2},
\eqref{eq:nichols-mild-relation1}, \eqref{eq:nichols-mild-relation2}, \eqref{eq:nichols-mild-relation3}, \eqref{eq:nichols-mild-relation4}, \eqref{eq:nichols-mild-relation5}, \eqref{eq:nichols-mild-cyc2-relation1}, \eqref{eq:nichols-mild-cyc2-relation2}, \eqref{eq:nichols-mild-cyc2-relation3}, \eqref{eq:nichols-mild-cyc2-relation4}, \eqref{eq:nichols-mild-cyc2-relation5}, \eqref{eq:nichols-mild-cyc2-relation6}
and \eqref{eq:nichols-mild-cyc2-relation7} hold in $\cB$. Then the following relations also hold:
\begin{align*}
x_1 x_{1\fudos 23} &= q_{12}q_{13} x_{1\fudos 23} x_{1}, &
x_1 x_{123} &= -q_{12}q_{13} x_{123} x_{1}, \\
x_1 v &= q_{12}^2q_{13} vx_1-2q_{12}x_{12}x_{1\fudos 23}, &
x_1 x_{3} &= q_{13} x_{3} x_{1}, \\
x_1 u &= - q_{12}^2q_{13} u x_{1} + q_{12}q_{13} x_{123}x_{12}, &
x_{\fudos 1} x_{3} &= q_{13}^2 x_{3} x_{\fudos 1} \\
x_{\fudos 1} x_{1\fudos 23} &= q_{12}^2q_{13}^2 x_{1\fudos 23} x_{\fudos 1}, &
x_{\fudos 1} x_{123} &= q_{12}^2q_{13}^2 x_{123} x_{\fudos 1}, \\
x_{\fudos 1} x_{\fudos 23} &= q_{12}^2q_{13}^2 x_{\fudos 23} x_{\fudos 1}-q_{12}^2q_{13}^2 x_{123} x_{\fudos 1}, &
x_{\fudos 1} v &= q_{12}^4q_{13}^2 v x_{\fudos 1}, \\
x_{\fudos 1} u &= q_{12}^4q_{13}^2 u x_{\fudos 1}, &
x_{\fudos 1} x_{23} &= q_{12}^2 q_{13}^2 x_{23} x_{\fudos 1}, \\
x_{\fudos} x_{123} &= -q_{12}q_{13} x_{123} x_{\fudos} - x_{1\fudos 23},  &
x_{\fudos} x_{3} &= q_{13} x_{3} x_{\fudos}, \\
x_{\fudos} x_{1\fudos 23} &= q_{12}q_{13} x_{1\fudos 23} x_{\fudos},
\end{align*}
\begin{align*}
x_{\fudos} x_{\fudos 23} &= -q_{12}q_{13} x_{\fudos 23} x_{\fudos} + q_{12}q_{13}x_{123}x_{\fudos} +\frac{1}{2} x_{1\fudos 23}, \\
x_{\fudos} v &= q_{12}^2q_{13} v x_{\fudos} +2q_{12}x_{12}x_{1\fudos 23}-2q_{12}x_{\fudos 2}x_{1\fudos 23},\\
x_{\fudos} u &= -q_{12}^2q_{13} u x_{\fudos} +q_{12}^2q_{13} (x_{123}-x_{\fudos 23})x_{\fudos 2}-\frac{q_{12}q_{13}}{2} v, \\
x_{12} x_{\fudos 23} &= -q_{13}q_{23} x_{\fudos 23} x_{12} + x_{12}x_{123}-q_{13}q_{23} v, \\
x_{1 \fudos 2} x_{\fudos 23} &= q_{12}q_{13}^2q_{23} x_{\fudos 23} x_{1 \fudos 2} -2q_{13}^2q_{23}^2 x_{12}x_{1\fudos 23}, \\
x_{1 \fudos 2} u &= q_{12}^3q_{13}^2q_{23} u x_{1 \fudos 2} -2q_{12}q_{13}^3q_{23} x_{12}v, \\
x_{\fudos 2} x_{1\fudos 23} &= q_{21}q_{23}q_{13} x_{1\fudos 23} x_{\fudos 2} + 2q_{13}q_{23}x_{12}x_{1\fudos 23}, \\
x_{12} x_{23} &= -q_{12}q_{13}q_{23} x_{23} x_{12} -2q_{12} x_2x_{123}, \\
x_{1 \fudos 2} x_{23} &= -q_{21}q_{13}^2q_{23} x_{23} x_{1 \fudos 2} + q_{12}q_{13}q_{23}v, \\
x_{1\fudos 23} x_2  &= -q_{12}^2q_{32} x_2 x_{1\fudos 23} -q_{12}q_{13} v,
\end{align*}
\begin{align*}
x_{12} x_{1\fudos 23} &= q_{13}q_{23}q_{21} x_{1\fudos 23} x_{12}, &
x_2 x_{123} &= -q_{21}q_{23} x_{123} x_2, \\
x_{12} x_{123} &= -q_{13}q_{23} x_{123} x_{12}, &
x_{12} v &= q_{13}q_{23} v x_{12}, \\
x_{1 \fudos 2} x_{3} &= q_{13}^2q_{23} x_{3} x_{1 \fudos 2} + x_{1\fudos 23} &
x_{12} u &= q_{12}q_{13}q_{23} u x_{12}, \\
x_{12} x_{3} &= q_{13}q_{23} x_{3} x_{12} + x_{123},  &
x_{\fudos 2} v &= q_{13}q_{13} v x_{\fudos 2}, \\
x_{\fudos 2} x_{\fudos 23} &= -q_{12}q_{23} x_{\fudos 23} x_{\fudos 2}, &
x_{1 \fudos 2} v &= q_{12}^2q_{13}^2q_{23} v x_{1 \fudos 2}, \\
x_{\fudos 2} x_{3} &= q_{13}q_{23} x_{3} x_{\fudos 2} + x_{\fudos 23}, &
x_{\fudos 2} u &= q_{12}q_{13} u x_{\fudos 2},\\
x_{1 \fudos 2} x_{1\fudos 23} &= -q_{13}^3q_{23} x_{1\fudos 23} x_{1 \fudos 2}, &
v x_2 &=  q_{12}^2 q_{32} x_2 v,\\
x_{\fudos 2} x_{23} &= -q_{12}q_{13}q_{23} x_{23} x_{\fudos 2} -q_{23}u, &
x_2 u &= q_{21}q_{23} u x_2,\\
x_{1 \fudos 2} x_{123} &= q_{12}q_{13}^2q_{23} x_{123} x_{1 \fudos 2}, &
x_2 x_{23} &= -q_{23} x_{23} x_2,
\end{align*}
\begin{align*}
x_{1\fudos 23} x_{123} &= q_{12}q_{13} x_{123} x_{1\fudos 23} ,&
x_{1\fudos 23} v &= q_{12}^2q_{32} v x_{1\fudos 23} , \\
x_{1\fudos 23} x_{\fudos 23} &= q_{12}q_{13} x_{\fudos 23} x_{1\fudos 23} ,&
x_{1\fudos 23} u &= -q_{12}^3q_{13}q_{32} u x_{1\fudos 23} , \\
x_{1\fudos 23} x_{23} &= -q_{12}^2q_{13}^2 x_{23} x_{1\fudos 23},&
x_{1\fudos 23} x_{3} &= -q_{13}^2q_{23} x_{3} x_{1\fudos 23}, \\
x_{123} v &= q_{31}q_{32} v x_{123}, &
x_{123} x_{\fudos 23} &= -x_{\fudos 23} x_{123} , \\
x_{123} u &= q_{12}q_{32} u x_{123} , &
x_{123} x_{23} &= -q_{12}q_{13} x_{23} x_{123}, \\
x_{123} x_{3} &= -q_{13}q_{23} x_{3} x_{123}, &
v x_{\fudos 23} &= -q_{12}q_{13} x_{\fudos 23} v , \\
v u &= q_{12}^2q_{13} u v , &
v x_{23} &= q_{12}^2q_{13}^2q_{23} x_{23} v , \\
v x_{3} &= -q_{13}^2q_{23}^2 x_{3} v , &
x_{\fudos 23} u &= q_{12}q_{32} u x_{\fudos 23}, \\
x_{\fudos 23} x_{23} &= -q_{12}q_{13} x_{23} x_{\fudos 23}, &
x_{\fudos 23} x_{3} &= -q_{13}q_{23} x_{3} x_{\fudos 23}, \\
u x_{23} &= q_{12}q_{13} x_{23} u, &
u x_{3} &= -q_{13}q_{23}^2 x_{3} u, \\
x_{23} x_{3} &= -q_{23} x_{3} x_{23}.
\end{align*}

In particular these relations hold in $\cB(V)$.
\end{lemma}

\begin{proof}
The equations are proved recursively on the degree from the defining relations.
\end{proof}

\begin{lemma} \label{le:basisK-cyc2}
The set
\begin{align*}
B_K=&\big\{x_{12}^{n_1} x_{1\fudos 2}^{n_2} x_{1\fudos 23}^{n_3} [x_{123}, x_{\fudos 2}]_c^{n_4} x_{123}^{n_5}  x_{\fudos 2}^{n_6} x_{\fudos 23}^{n_7}
 [x_{\fudos 23}, x_2]_c^{n_8} x_2^{n_9}  x_{23}^{n_{10}} x_3^{n_{11}}:
\\ & 0 \le n_1, n_2, n_3,n_4, n_5, n_6, n_7, n_9, n_{10}, n_{11}  \le 1 \big\}.
\end{align*}
is a basis of $K$ and $\GK K = 1$.
\end{lemma}

\begin{proof}
First we claim that $K$ is spanned by the monomials
\begin{align}\label{eq:cyclop2-generating-subspace}
x_{12}^{n_1} x_{1\fudos 2}^{n_2} x_{1\fudos 23}^{n_3} v^{n_4} x_{123}^{n_5}  x_{\fudos 2}^{n_6} x_{\fudos 23}^{n_7}
 u^{n_8} x_2^{n_9}  x_{23}^{n_{10}} x_3^{n_{11}}, \, n_i\in\N_0.
\end{align}
Using the relations in Lemma \ref{lem:nichols-mild-cyc2-relations} and Remark \ref{rem:nichols-mild-cyc2-relations}, $K^1= \ad_c\NA (V_1) (V_{\diag})$
is spanned by $x_{12}$, $x_{1\fudos 2}$, $x_{\fudos 2}$, $x_2$, $x_3$ since the subspace generated by them is stable by $\ad_c x_1$, $\ad_c x_{\fudos}$,
$\ad_c x_{\fudos 1}$.
Then $x_{1\fudos 23}$, $v$, $x_{123}$, $x_{\fudos 23}$, $u$, $x_{23}$ also belong to $K$. Thus the subspace generated by \eqref{eq:cyclop2-generating-subspace}
is contained in $K$. On the other hand, the subspace generated by \eqref{eq:cyclop2-generating-subspace} is a left ideal by Lemma \ref{lem:nichols-mild-cyc2-relations} and Remark \ref{rem:nichols-mild-cyc2-relations}; as $1$ belongs to this subspace, it is $K$.
Now we can restrict all the $n_i$, $i\neq 8$, in \eqref{eq:cyclop2-generating-subspace} to the set $\{0,1\}$ because of Lemma \ref{lem:nichols-mild-cyc2-relations}, \eqref{eq:nichols-mild-relation4}, \eqref{eq:nichols-mild-relation5},
\eqref{eq:nichols-mild-cyc2-A3-1} and \eqref{eq:nichols-mild-cyc2-A3-2}. Thus $K$ is spanned by $B_K$.

We claim now that $B_K$ is linearly independent. From \eqref{eq:nichols-mild-derivations-u},
\begin{align*}
\partial_1\partial_2\partial_3\partial_2(u^n) &= \sum_{k=0}^n u^k \partial_1\partial_2\partial_3\partial_2(u) g_1g_2^2g_3 \cdot u^{n-1-k} =-4n\, u^{n-1},
\end{align*}
for all $n\in\N$. We claim that $x_{12} x_{1\fudos 2} x_{\fudos 2} u^{n} x_2\neq 0$ for all $n\in\N_0$. Indeed, it holds for $n=0$ by Lemma \ref{le:basisK},
and recursively
\begin{align*}
\partial_1\partial_2\partial_3\partial_2(x_{12}x_{1\fudos 2}x_{\fudos 2}u^{n}x_2) &= -4nq_{21}q_{23} \, x_{12} x_{1\fudos 2} x_{\fudos 2} u^{n-1} x_2 \neq 0.
\end{align*}
Set $\ur_n=x_{12} x_{1\fudos 2} x_{1\fudos 23} v x_{123} x_{\fudos 2} x_{\fudos 23} u^{n} x_2  x_{23} x_3$.
We claim that
\begin{align*}
\partial_3(\ur_n) &= x_{12} x_{1\fudos 2} x_{1\fudos 23} v x_{123} x_{\fudos 2} x_{\fudos 23} u^{n} x_2  x_{23}.
\end{align*}
Indeed $\partial_3$ annihilates $x_{12}$, $x_{1\fudos 2}$, $x_{\fudos 2}$, $u$, $x_2$, and
\begin{align*}
\partial_3(v)&=-4x_{\fudos}x_{12}, & \partial_3(x_{1\fudos 23}) &=2x_{1\fudos 2}, & \partial_3(x_{123}) &=2 x_{12}, \\
\partial_3(x_{\fudos 23}) &=2 x_{\fudos}-x_{12}, & \partial_3(x_{23}) &=2x_2, & \partial_3(x_3) &=1.
\end{align*}
As $\partial_3$ is a skew derivation, $\partial_3(\ur_n)$ is the sum of 11 terms, each of them obtained by applying $\partial_3$
either to a PBW generator different to $u$ or else to $u^n$, and applying $g_3$ to the generators on the right. Here, 5 of these summands
are 0 because $\partial_3$ annihilates the corresponding PBW generator, and another 5 summands are 0 because we can reorder the PBW
generators to obtain the square of a PBW generator which is zero, see Lemma \ref{lem:nichols-mild-cyc2-relations},
\eqref{eq:nichols-mild-relation4}, \eqref{eq:nichols-mild-relation5}, \eqref{eq:nichols-mild-cyc2-A3-1} and \eqref{eq:nichols-mild-cyc2-A3-2}.
The exception is the summand containing $\partial_3(x_3) =1$.
By a similar argument,
\begin{align*}
& \partial_1^2\partial_2\partial_3(x_{12} x_{1\fudos 2} x_{1\fudos 23} v x_{123} x_{\fudos 2} x_{\fudos 23} u^{n} x_2  x_{23}) =
4 \chi_1^{n+5}\chi_2^{2n+7}\chi_3^{n+4}(g_1^2g_2g_3) \\
&\qquad x_{12} x_{1\fudos 2} v x_{123} x_{\fudos 2} x_{\fudos 23} u^{n} x_2  x_{23}, \\
& \partial_1^2\partial_2\partial_3\partial_2(x_{12} x_{1\fudos 2} v x_{123} x_{\fudos 2} x_{\fudos 23} u^{n} x_2  x_{23}) =
8q_{31}q_{32} \chi_1^{n+3}\chi_2^{2n+5}\chi_3^{n+3}(g_1^2g_2^2g_3) \\
&\qquad x_{12} x_{1\fudos 2} x_{123} x_{\fudos 2} x_{\fudos 23} u^{n} x_2  x_{23}, \\
& \partial_{\fudos}\partial_2\partial_3(x_{12} x_{1\fudos 2} x_{123} x_{\fudos 2} x_{\fudos 23} u^{n} x_2  x_{23}) =
4 \chi_1^{n}\chi_2^{2n+2}\chi_3^{n+1}(g_1g_2g_3) \\
&\qquad x_{12} x_{1\fudos 2} x_{123} x_{\fudos 2} u^{n} x_2  x_{23}, \\
& \partial_1\partial_2\partial_3(x_{12} x_{1\fudos 2} x_{123} x_{\fudos 2} u^{n} x_2  x_{23}) =
4 \chi_1^{n+1}\chi_2^{2n+3}\chi_3^{n+1}(g_1g_2g_3) \\
&\qquad x_{12} x_{1\fudos 2} x_{\fudos 2} u^{n} x_2  x_{23}, \\
& \partial_2\partial_3(x_{12} x_{1\fudos 2} x_{\fudos 2} u^{n} x_2  x_{23}) = 2 x_{12} x_{1\fudos 2} x_{\fudos 2} u^{n} x_2,
\end{align*}
so $\ur_n\neq 0$ for all $n\in\N$. By Lemma \ref{lem:nichols-mild-cyc2-relations}, and \eqref{eq:nichols-mild-relation4}, \eqref{eq:nichols-mild-relation5}, \eqref{eq:nichols-mild-cyc2-A3-1}, \eqref{eq:nichols-mild-cyc2-A3-2},
\begin{multline*}
x_{12}^{n_1} x_{1\fudos 2}^{n_2} x_{1\fudos 23}^{n_3} v^{n_4} x_{123}^{n_5}  x_{\fudos 2}^{n_6} x_{\fudos 23}^{n_7}
u^{n_8} x_2^{n_9}  x_{23}^{n_{10}} x_3^{n_{11}} \\
\cdot x_3^{1-n_{11}} x_{23}^{1-n_{10}} x_2^{1-n_9} x_{\fudos 23}^{1-n_7} x_{\fudos 2}^{1-n_6} x_{123}^{1-n_5} v^{1-n_4}x_{1\fudos 23}^{1-n_3} x_{1\fudos 2}^{1-n_2} x_{12}^{1-n_1}
\end{multline*}
gives $\ur_{n_8}$ multiplied by a non-zero scalar. Let $\mathtt{S}$ be a non-trivial linear combination of elements of $B_K$. Let
$N$ be the minimum of $\sum_{i\neq 8} n_i$ between the elements of $B_K$ with non-zero coefficient and pick
\begin{align*}
&x_{12}^{m_1} x_{1\fudos 2}^{m_2} x_{1\fudos 23}^{m_3} v^{m_4} x_{123}^{m_5}  x_{\fudos 2}^{m_6} x_{\fudos 23}^{m_7}
u^{m_8} x_2^{m_9}  x_{23}^{m_{10}} x_3^{m_{11}}, && \sum_{i\neq 8} m_i=N,
\end{align*}
with non-zero coefficient. Then
\begin{align}\label{eq:S-mult-by-product}
& \mathtt{S} \cdot x_3^{1-m_{11}} x_{23}^{1-m_{10}} x_2^{1-m_9} x_{\fudos 23}^{1-m_7} x_{\fudos 2}^{1-m_6} x_{123}^{1-m_5} v^{1-m_4}x_{1\fudos 23}^{1-m_3} x_{1\fudos 2}^{1-m_2} x_{12}^{1-m_1}
\end{align}
is a sum of $\ur_n$'s with non-zero coefficient, one for each term of $B_K$ such that $n_i=m_i$, $i\neq 8$, $n_8=n$, and non-zero
coefficient in $\mathtt{S}$. As the $\ur_n$'s are non-zero and $\gr \ur_n=4n+26$, \eqref{eq:S-mult-by-product} is non-zero, and $\mathtt{S}$
is so. Thus the claim follows, so $B_K$ is a basis of $K$.
\end{proof}

We close this Subsection giving the presentation of $\cB(\cyc_2)$.

\begin{prop} \label{pr:nichols-mild-cyc2} The algebra
$\cB(\cyc_2)$ is presented by generators $x_i$, $i\in\Iw_3$, and relations \eqref{eq:rels-B(V(-1,2))-1}, \eqref{eq:rels-B(V(-1,2))-2},
\eqref{eq:nichols-mild-relation1}, \eqref{eq:nichols-mild-relation2}, \eqref{eq:nichols-mild-relation3}, \eqref{eq:nichols-mild-relation4}, \eqref{eq:nichols-mild-relation5}, \eqref{eq:nichols-mild-cyc2-relation1}, \eqref{eq:nichols-mild-cyc2-relation2}, \eqref{eq:nichols-mild-cyc2-relation3}, \eqref{eq:nichols-mild-cyc2-relation4}, \eqref{eq:nichols-mild-cyc2-relation5}, \eqref{eq:nichols-mild-cyc2-relation6}
and \eqref{eq:nichols-mild-cyc2-relation7}. The set
\begin{align*}
&\big\{x_1^{a_1} x_{1\fudos}^{a_2} x_{\fudos}^{a_3} x_{12}^{a_4} x_{1\fudos 2}^{a_5}
x_{1\fudos 23}^{a_6} [x_{123}, x_{\fudos 2}]_c^{a_7} x_{123}^{a_8}  x_{\fudos 2}^{a_9} x_{\fudos 23}^{a_{10}}
 [x_{\fudos 23}, x_2]_c^{a_{11}} x_2^{a_{12}}  x_{23}^{a_{13}}
 x_3^{a_{14}}:
\\ & 0 \le a_1, a_4, a_5, a_6,a_7, a_8, a_9, a_{10}, a_{12}, a_{13}, a_{14}  \le 1 \big\}.
\end{align*}
is a basis of $\cB(\cyc_2)$ and $\GK \cB(\cyc_2) = 3$.
\end{prop}

\pf
Relations \eqref{eq:rels-B(V(-1,2))-1}, \eqref{eq:rels-B(V(-1,2))-2}, \eqref{eq:nichols-mild-relation1}, \eqref{eq:nichols-mild-relation2}, \eqref{eq:nichols-mild-relation3}, \eqref{eq:nichols-mild-relation4}, \eqref{eq:nichols-mild-relation5}, \eqref{eq:nichols-mild-cyc2-relation1}, \eqref{eq:nichols-mild-cyc2-relation2}, \eqref{eq:nichols-mild-cyc2-relation3}, \eqref{eq:nichols-mild-cyc2-relation4},
\eqref{eq:nichols-mild-cyc2-relation5}, \eqref{eq:nichols-mild-cyc2-relation6} and \eqref{eq:nichols-mild-cyc2-relation7} are 0 in $\cB(\cyc_1)$,
see Lemma \ref{lem:nichols-mild-cyc2-defrelations} and Remark \ref{rem:nichols-mild-cyc2-defrelations}.
Hence the quotient $\cBt$ of $T(V)$ by these relations projects onto $\cB(\cyc_2)$.
We claim that the subspace $I$ spanned by $B$ is a right ideal of $\cBt$. Indeed, $I x_3\subseteq I$
by definition while $Ix_1\subseteq I$, $Ix_{\fudos}\subseteq I$, $Ix_2\subseteq I$ follow by Lemma \ref{lem:nichols-mild-cyc2-relations}
and Remark \ref{rem:nichols-mild-cyc2-relations}. Since $1\in I$, $\cBt$ is spanned by $B$.

To prove that $\cBt \simeq \cB(\cyc_2)$, it remains to show that
$B$ is linearly independent in $\cB(\cyc_2)$. For, suppose that there is a non-trivial linear combination $\mathtt{S}$
of elements of $B$ in $\cB(\cyc_2)$, say of minimal degree. As in the proof of Proposition \ref{pr:-1block},
each vector in $\mathtt{S}$ with non-trivial coefficient satisfies $a_i=0$, $i=1,2,3$. But then we obtain a contradiction with Lemma \ref{le:basisK-cyc2}.
Thus $B$ is a basis of $\cB(\cyc_2)$ and $\cBt=\cB(\cyc_2)$. The computation of $\GK$ follows from the Hilbert series at once.
\epf

\subsubsection{Several components}\label{subsubsec:components-block-points-eps-1}

\begin{lemma}\label{lemma:superJordan-2-mild+weak}
	Let $W = V_1 \oplus U$ be a direct sum of braided vector spaces, where $V_1$ and $U$ have dimension 2, $V_1$ is a $-1$-block
	with basis $x_1, x_{\fudos}$ and $U$ is of diagonal type with respect to a basis $x_2$, $x_3$, $q_{23}q_{32}=1$. Assume that $x_2$ has mild interaction with $V_1$, and $x_3$ has weak interaction with $\ghost_3>0$. Then $\GK \cB(W) = \infty$.
\end{lemma}

\pf
By Lemmas \ref{le:basisK} and \ref{lemma:derivations-zn}, $u_1:=x_{12}$, $u_2:=x_{1\fudos 2}$,
$u_3:=x_{\fudos 3}$ are non-zero elements of $K^1$. Moreover they are linearly independent since
$u_2$ has degree 3, $u_1$, $u_3$ have degree 2 and
\begin{align*}
\partial_2(u_1) &=2x_1, & \partial_2(u_3)&=0, & \partial_3(u_3)&= x_1.
\end{align*}
By direct computation,
\begin{align*}
\delta(u_1) &= g_1g_2 \otimes u_1 + 2x_1 g_2\otimes x_2, \\
\delta(u_2) &= g_1^2g_2 \otimes u_2 + x_1 g_1g_2\otimes (2x_{\fudos 2}-u_1)+2(x_1x_{\fudos}-x_{\fudos}x_1) g_2 \otimes x_2, \\
\delta(u_3) &= g_1g_3 \otimes u_3 + x_1 g_3\otimes x_3.
\end{align*}
Notice that $\ad_c x_1 u_2=0$ by \eqref{eq:-1block+point}, and $\ad_c x_1 u_1=\ad_c x_1 u_3=0$ since $x_1^2=0$.
Thus $U=\ku u_1\oplus \ku u_2\oplus \ku u_3$ is a braided vector subspace of $K^1$ of diagonal type with Dynkin diagram
\begin{align*}
\xymatrix{ \overset{-q_{33}} {\underset{3}{\circ}} \ar  @{-}[d]_{-1} \ar @{-}[rd]^{-1}	&
\\ \overset{-1}{\underset{1}{\circ}} \ar  @{-}[r]_{-1}  & \overset{-1} {\underset{2}{\circ}}.}
\end{align*}
If $q_{33}=-1$, then $\GK\cB(U)=\infty$ by Lemma \ref{lemma:points-trivial-braiding}. If $q_{33}=1$, then $\GK\cB(U)=\infty$ by Theorem \ref{thm:nichols-diagonal-finite-gkd} since the braiding is of affine type $A_2^{(1)}$. In any case, $\GK K=\infty$.
\epf

\subsubsection{The Nichols algebras with finite $\GK$, several connected components in $V_{\diag}$}\label{subsubsection:points-block-presentation-manycomponents}

Let $V = V_1 \oplus V_{\diag}$ be a braided vector space as described in \S \ref{subsection:YD>3-setting}, where the braided subspace $V_1\simeq \cV(\epsilon, 2)$ is a block with $\epsilon^2=1$, while
$V_{\diag}$ is of diagonal type. We assume that the interaction between $V_1$ and $V_{\diag}$ is weak.
Let $\X$ be the set of connected components of the generalized Dynkin diagram of $V_{\diag}$.
If $J\in \X$, then we denote by $\D_J$ its generalized Dynkin diagram and by $\ghost_J$ its ghost, see the paragraph after \eqref{eq:v=v1+v2}.
Then we denote
\begin{align*}
V = \lstr((\D_J)_{J\in \X}, (\ghost_J)_{j\in \X}).
\end{align*}
Notice that given a finite set $\X$ and collections $(\D_J)_{J\in \X}$ of generalized Dynkin diagrams and $(\ghost_J)_{j\in \X}$ of appropriate vectors, then there is a unique braided vector space $V = \lstr((\D_J)_{J\in \X}, (\ghost_J)_{j\in \X})$. We assume in this \S \/
that for every $J \in \X$:
\begin{itemize} [leftmargin=*]\renewcommand{\labelitemi}{$\circ$}
	\item If $\ghost_J \neq 0$, then $(\D_J, \ghost_J)$ (or the corresponding abbreviation) belongs either to Table \ref{tab:finiteGK-block-point} or to Table \ref{tab:finiteGK-block-points}. In this case, we denote by $\cR_J$ the defining relations of $\NA(\lstr(\D_J, \ghost_J))$ as described in those Tables, and by $B_J$ a basis of $\NA(K_J)$ as described in the previous Subsections.
	
	\item If $\ghost_J = 0$, then $\GK \NA(V_J) < \infty$.
In this case, we denote by $\cR_J$ a set of defining relations of $\NA(V_J)$ and by $B_J$ a basis of $\NA(V_J)$. (If Conjecture \ref{conjecture:nichols-diagonal-finite-gkd} is valid, then these relations are known from \cite{Ang-crelle}).
\end{itemize}

Let $B_0$ be a basis of $\NA(V_1)$ as described in Propositions \ref{pr:1block}, \ref{pr:-1block}. We fix an order of the elements of $\X$: $J_1,\dots,J_{|\X|}$.

Our goal is to describe a presentation of $\NA(V) = \NA(\lstr ((\D_J)_{J\in \X}, (\ghost_J)_{j\in \X}))$. We follow the same path as in previous related subsections.

\begin{lemma}\label{lem:several-comp-def-relations}
Assume that the interaction is weak and that $\ghost$ is discrete.
Let $i,j\geq 2$ be such that $q_{ij}q_{ji}=1$. Then
\begin{align}\label{eq:several-comp-def-relations}
z_{i,m}x_j &= q_{1j}^mq_{ij} \, x_j z_{i,m} & & 0 \le m\le 2|a_i|.
\end{align}
\end{lemma}
\pf
By Lemma \ref{lemma:braiding-K-weak-block-points}, $K^1$ is a braided vector space of diagonal type with matrix $(p_{ks,\ell t})$.
As $p_{im,j0}p_{j0,im}=q_{ij}q_{ji}=1$, the corresponding elements $z_{i,m}$ and $x_j$ satisfy $\ad_c z_{i,m} x_j=0$.
\epf

\begin{lemma}\label{lem:several-comp-other-relations}
Assume that the interaction is weak and that $\ghost$ is discrete.
Let $i,j\geq 2$ be such that $q_{ij}q_{ji}=1$, $\ghost_i\leq\ghost_j$ and
$\cB$ a quotient algebra of $T(V)$ such that \eqref{eq:several-comp-def-relations} and
$z_{i,2|a_i|+1}=z_{j,2|a_j|+1}=0$ hold in $\cB$. Then
\begin{align}\label{eq:several-comp-other-relations}
z_{i,m}z_{j,n} &= p_{im,jn} \, z_{j,n} z_{i,m} & &\mbox{for all }m, n \in\N_0.
\end{align}
\end{lemma}
\pf
By induction on $n$. For $n=0$, if $m\le 2|a_i|$, then \eqref{eq:several-comp-other-relations} holds
by \eqref{eq:several-comp-def-relations}, while if $m>2|a_i|$ then $z_{i,m}=0$ since $z_{i,2|a_i|+1}=0$.

Assume that \eqref{eq:several-comp-other-relations} holds for a fixed $n$ and all $m\in\N_0$.
By direct computation and inductive hypothesis,
\begin{align*}
0 &= \ad_c x_{\fudos} \left[ z_{i,m}, z_{j,n}  \right]_c =  \left[ \ad_c x_{\fudos} z_{i,m}, z_{j,n}  \right]_c
+ \epsilon^m q_{1i} z_{i,m} \ad_c x_{\fudos}(z_{j,n}) \\
& \qquad -\epsilon^{mn} q_{1j}^m q_{i1}^n q_{ij} \ad_c x_{\fudos}(z_{j,n}) z_{i,m} \\
&= \left[ z_{i,m+1}, z_{j,n}  \right]_c + \epsilon^m q_{1i} \left( z_{i,m} z_{j,n+1}
- \epsilon^{m(n+1)} q_{1j}^m q_{i1}^{n+1} q_{ij} z_{j,n+1} z_{i,m} \right) \\
&= \epsilon^m q_{1i} \left[ z_{i,m}, z_{j,n+1} \right]_c,
\end{align*}
so \eqref{eq:several-comp-other-relations} holds for $n+1$ and all $m\in\N_0$.
\epf

\begin{prop} \label{pr:several-comp-presentation}
The algebra $\cB(V)$ is presented by generators $x_i$, $i\in \Iw_{\theta}$, and relations $\cR_J$, $J\in\X$,
\begin{align}\label{eq:several-comp-def-relations-2}
z_{i,m}x_j &= q_{1j}^mq_{ij} \, x_j z_{i,m} & & i,j \mbox{ not connected, } \ghost_i\le \ghost_j, \, 0 \le m\le 2|a_i|.
\end{align}
The set $B =\big\{ b_0b_1\cdots b_{|\X|}: \, b_0\in B_0, \, b_i \in B_{J_i} \big\}$
is a basis of $\cB(V)$ and $\GK \cB(V) = 2 + \sum_{J\in \X} \GK \toba(K_J)$.
\end{prop}
\pf
We first prove that $B$ is a basis of $\cB := \cB(V)$: it follows since $\cB\simeq K\#\cB(V_1)$, see \S\ref{subsubsec:algK-block-points},
and $K\simeq \cB(K^1)\simeq \underline{\otimes}_{J\in\X}\cB(K_J)$ by Corollary \ref{cor:conncomp}.
Then $\GK \cB = 2 + \sum_{J\in \X} \GK \toba(K_J)$ by computing the Hilbert series.

Relations $\cR_J$ and \eqref{eq:several-comp-def-relations-2} hold in $\cB$ by previous Subsections and Lemma
\ref{lem:several-comp-def-relations}. Hence the quotient $\cBt$ of $T(V)$ by these relations projects onto $\cB$.

We claim that the subspace $I$ spanned by $B$ is a left ideal of $\cBt$. Indeed, $x_1I\subseteq I$,
$x_{\fudos}I\subseteq I$ since $x_1 b_0, x_{\fudos}b_0$ belongs to the subspace spanned by $B_0$.
Let $i\geq 2$, $b_0\in B_0$, $b_j \in B_{J_j}$. Then $i\in J_k$ for some $k$. If $\ghost_k\neq 0$, then
\begin{align*}
x_i b_0 & \in \sum_{b_0'\in B_0}\sum_{n=0}^{2|a_i|} \ku \, b_0'z_{i,n}, & \mbox{using the relations in }&\cR_k.
\end{align*}
If $\ghost_k=0$, then $\ku x_i b_0= \ku b_0 x_i$. In any case,
\begin{align*}
x_i b_0 b_1\dots b_{|\X|} & \in \sum_{b_0'\in B_0}\sum_{n=0}^{2|a_i|} \ku \, b_0'z_{i,n} b_1\dots b_{|\X|} \\
& = \sum_{b_0'\in B_0}\sum_{n=0}^{2|a_i|} \ku \, b_0' b_1\dots b_{k-1} z_{i,n} b_k b_{k+1} \dots b_{|\X|} \\
& \subseteq \sum_{b_0'\in B_0}\sum_{b_k'\in B_k} \ku \, b_0' b_1\dots b_{k-1} b_k' b_{k+1} \dots b_{|\X|} \subseteq I
\end{align*}
by Lemma \ref{lem:several-comp-other-relations} and the relations in $\cR_k$. Thus $x_i I\subseteq I$ for all $i\in\Iw_{\theta}$.
Since $1\in I$, $\cBt$ is spanned by $B$. Thus $\cBt \simeq \cB$ since $B$ is a basis of $\cB$.
\epf

\section{Two blocks}\label{section:YD>3-2blocks}
\subsection{The setting}
Let $\Idd_j = \{j, j + \tfrac{1}{2}\}$,
$\Idd = \Idd_1 \cup \Idd_2 = \{1, \fudos, 2, \futres\}$.
Let $g_i\in \Gamma$,  $\chi_i\in \widehat\Gamma$ and $\eta_i: \Gamma \to \ku$ a $(\chi_i, \chi_i)$-derivation,
$i\in \I_2$.
Let $\cV_{g_i}(\chi_i, \eta_i) \in \ydG$ be the indecomposable with basis $(x_h)_{h\in\Idd_i}$
and action given by
\eqref{equation:basis-triangular-gral}.
Let
\begin{align*}
V &= V_1 \oplus V_2,& &\text{where } V_i = \cV_{g_i}(\chi_i, \eta_i), i\in \I_2.
\end{align*}
Thus $(x_h)_{h\in\Idd}$ is a basis of $V$.
Let
\begin{align*}
q_{ij}&= \chi_j(g_i),&  a_{ij} &= q_{ij}^{-1}\eta_j(g_i),&  i,j &\in \I_\theta.
\end{align*}

We suppose that $V$ is neither of diagonal type nor of the form  \emph{1 block $\&$ 2 points}, hence
$\as_{ii} \neq 0$, $i\in \I_2$; we may assume that  $\eta_i(g_i) = 1$ by normalizing $x_i$.
If $i \in \I_2$, then set $\ghost_i =
\begin{cases} -2a_{ji}, &q_{ii} = 1, \\
a_{ji}, &q_{ii} = -1 \end{cases}$ where $j\neq i$. The
\emph{ghost} is $\ghost = (\ghost_1, \ghost_2)$. Thus $\ghost_1$ is the ghost of the braiding of the block 1 with the point 2, and vice versa for $\ghost_2$.

We seek to know when $\GK \toba(V) < \infty$.
Since $\NA(V_i \oplus \ku_{g_j}^{\chi_j}) \hookrightarrow \NA(V)$
for $i\neq j \in \I_2$, we may apply the results from \S 5, 6
and assume that
\begin{align}
q_{ii}^2 &= 1 \text{ hence } \as_{ii} = q_{ii} =: \epsilon_i, & & q_{12}q_{21}\in\{\pm 1\},& &\ghost\in \N_0^2.
\end{align}

Here is the main result of this Section.

\begin{theorem}\label{theorem:2blocks}
The following are equivalent:
\begin{enumerate}
	\item $\GK \toba(V) < \infty$.
	\item $c^2_{\vert V_1 \otimes V_2} = \id$.
\end{enumerate}
\end{theorem}

From what we have already explained, $c^2_{\vert V_1 \otimes V_2} = \id$ implies that
$$\GK \toba(V) = \GK \toba(V_1) + \GK \toba(V_2) = 4.$$
So, let us  assume that $c^2_{\vert V_1 \otimes V_2} \neq \id$.
Again, we can codify the situation in the language of braided vector spaces.
With the previous conventions,  the braiding is given in the  basis $(x_i)_{i\in \Idd}$ by $(c(x_i \otimes x_j))_{i,j\in \Idd} =$
\begin{align}\label{eq:braiding-2blocks}
\begin{pmatrix}
\epsilon_1 x_1 \otimes x_1&  (\epsilon_1 x_{\fudos} + x_1) \otimes x_1& q_{12} x_2  \otimes x_1 &  q_{12}(x_{\futres} + \as_{12} x_2) \otimes x_1
\\
\epsilon_1 x_1 \otimes x_{\fudos} & (\epsilon_1 x_{\fudos} + x_1) \otimes x_{\fudos}& q_{12} x_2  \otimes x_{\fudos} & q_{12}(x_{\futres} + \as_{12} x_2) \otimes x_{\fudos}
\\
q_{21} x_1 \otimes x_2 &  q_{21}(x_{\fudos} + \as_{21} x_1) \otimes x_2& \epsilon_2 x_2  \otimes x_2 &  (\epsilon_2 x_{\futres} + x_2) \otimes x_2
\\
q_{21} x_1 \otimes x_{\futres} &  q_{21}(x_{\fudos} + \as_{21} x_1) \otimes x_{\futres}& \epsilon_2 x_2  \otimes x_{\futres} & (\epsilon_2 x_{\futres} + x_2) \otimes x_{\futres}
\end{pmatrix}.
\end{align}

Conversely, given a braided vector space with a braiding \eqref{eq:braiding-2blocks} we may define principal realizations over abelian groups.

Here is our first reduction:

\begin{lemma}
	If $q_{12}q_{21} = -1$, then $\GK \toba(V) = \infty$.
\end{lemma}

\pf From the hypothesis, we conclude that $q_{11} = q_{22} = -1$. Consider the filtration given by the natural order of $\Idd$.
Then $\gr V$ is of diagonal type, actually of affine Cartan type, hence Theorem \ref{thm:nichols-diagonal-finite-gkd} applies.
\epf

Therefore we may assume that $q_{12}q_{21} = 1$ and consequently $\ghost \neq 0$. There are three alternatives:
\begin{align*}
\epsilon_1 &= \epsilon_2 = 1, &  \epsilon_1 &= -\epsilon_2 = 1,&  \epsilon_1 &= \epsilon_2 = -1.
\end{align*}
We dispose of the first two alternatives in \S \ref{subsec:2blocks-eps1} and the third in \S \ref{subsec:2blocks-eps-1}, see Lemmas \ref{lemma:2blocks-caso1}, \ref{lemma:2blocks-caso2} and \ref{lemma:2blocks-caso3}. Altogether these give a proof of Theorem \ref{theorem:2blocks}.

\smallbreak
Before starting,  a bit  of notation.
As in \S \ref{sec:yd-dim3}, let
$K=\NA (V)^{\mathrm{co}\,\NA (V_1)}$; again
\begin{align*}
\NA(V) &\simeq K \# \NA (V_1);& K &\simeq \NA(K^1)& &\text{and}& K^1 &= \ad_c\NA (V_1) (V_2)
\in {}^{\NA (V_1)\# \ku \Gamma}_{\NA (V_1)\# \ku \Gamma}\mathcal{YD},
\end{align*}
with the coaction \eqref{eq:coaction-K^1} and the adjoint action.
We shall use as customary the derivations $\partial_i$, $i \in \Idd$.
The next observation will be used often:

\begin{remark}\label{rem:der-often}
	If $i \in \Idd$ and $\partial_i(x) =0$, then $\partial_i(\ad_c x_i x) =0$.
\end{remark}

\subsection{$\epsilon_1 = 1$}\label{subsec:2blocks-eps1}
Here we deal with the alternatives $\epsilon_1 = 1$ and $\epsilon_2 = \pm 1$.
We distinguish the cases when either $\ghost_{2} = 0$, $\ghost_{1} \ge 1$, see
Lemma \ref{lemma:2blocks-caso1}; or else $\ghost_{2} \ge 1$, $\ghost_{1} \ge 0$, see Lemma
 \ref{lemma:2blocks-caso2}. In both situations, we introduce for all $n, m \in \N_0$, the following elements of $K^{1}$:
\begin{align}
z_n &= (\ad_{c} x_{\fudos})^n x_2,&  \sh_{m, n} &= (\ad_{c} x_{1})^m(\ad_{c} x_{\fudos})^n x_{\futres},&  w_n &=  \sh_{0,n}.
\end{align}
The definition of $z_n$ is consistent with \eqref{eq:zn}.
By definition, \eqref{eq:-1block+point}, and \eqref{eq:derivations-zn},
\begin{align}\label{eq:2blocks-property-minus1}
\ad_{c} x_{1} z_n &=0,  & \partial_{\futres}(z_n) &= 0, & n &\in \N_0.
\end{align}
Also $\partial_2(z_n) \neq 0 \iff 0\le n\le \ghost_1$  and
the family $(z_n)_{0\le n\le \ghost_1}$ is linearly independent, by Remark \ref{lemma:K-basis}.

By \eqref{eq:relations B(W) - case 1}
$x_{\fudos}x_1^n =x_1^n x_{\fudos}-\frac{n}{2}x_1^{n+1}$, for all $n \in \N$. Thus
\begin{align} \label{eq:2blocks-property0}
\ad_c x_1 \sh_{m, n} &= \sh_{m + 1, n},&  \ad_c x_{\fudos} \sh_{m, n} &= \sh_{m, n + 1} - \frac{m}{2}\sh_{m + 1, n},
\end{align}
Thus the $z_n$'s and the $\sh_{m, n}$'s generate $K^1$.

\begin{kase} $\ghost_{2} = 0$, $\ghost_{1} \ge 1$.
\end{kase}

\begin{lemma}\label{lemma:2blocks-caso1}
If $\epsilon_1 = 1$,  $\ghost_{2} = 0$ and $\ghost_{1} \ge 1$, then $\GK \cB(V) = \infty$.
\end{lemma}

We need some preliminaries before the proof of Lemma \ref{lemma:2blocks-caso1}.

\begin{lemma}\label{lemma:2blocks-auxiliar}
	\begin{enumerate}\renewcommand{\theenumi}{\alph{enumi}}\renewcommand{\labelenumi}{(\theenumi)}
		\item\label{item:2blocks-auxiliar-1} If $m > 0$, then $\sh_{m, n} =0$. Hence
\begin{align}\label{eq:2blocks-auxiliar-0}
\ad_cx_1(w_n) &= 0,& \forall&n \in \N_0.
\end{align}
		
		\medbreak
		\item\label{item:2blocks-auxiliar-2} Let $\varpi_{m} = (-1)^m \prod_{0\le k \le m-1} (\as_{21} + \frac k2)$.  We have for all $n\in \N_0$
		\begin{align}\label{eq:2blocks-auxiliar-1}
		\partial_1(w_n) &= \partial_{2} (w_n) = \partial_{\fudos} (w_n) = 0; & \partial_{\futres} (w_n) &= \varpi_n x_1^n;
		\\ g_1 \cdot w_n &= q_{12} w_n,& g_2 \cdot w_n &= q_{21}^n (\epsilon_2w_n + z_n).
		\end{align}
		
		\medbreak
		\item\label{item:2blocks-auxiliar-3} The elements $(z_n)_{0 \le n \le \ghost_1}$ and $(w_n)_{0 \le n \le \ghost_1}$ form a basis of $K^1$.
	\end{enumerate}
\end{lemma}

\pf \eqref{item:2blocks-auxiliar-1}: By induction on $n$. Clearly, if $\sh_{m, n} =0$, then $\sh_{m', n} =0$ for all $m' \geq m$.
Now $\sh_{1, 0} = \ad_c(x_1) x_{\futres} = 0$ because
$\toba (V_2 \oplus \ku x_1) \simeq \toba (V_2) \underline{\otimes} \toba (\ku x_1)$.
The recursive step follows because $\sh_{m, n + 1} = \ad_c x_{\fudos} \sh_{m, n} + \frac{m}{2}\sh_{m + 1, n}$  by \eqref{eq:2blocks-property0}. Now \eqref{eq:2blocks-auxiliar-0} follows from this and \eqref{eq:2blocks-property0}.
\eqref{item:2blocks-auxiliar-2}: By definition, $\partial_1(w_n) = \partial_{2} (w_n) =  0$; while $\partial_{\fudos} (w_n) = 0$ by definition and Remark \ref{rem:der-often}.
Now $g_1 \cdot w_0 = q_{12} w_0$ hence
\begin{align*}
g_1 \cdot w_{n + 1} &= g_1 \cdot(x_{\fudos}w_n - q_{12} w_n x_{\fudos}) = (x_{\fudos} + x_1)q_{12}w_n  - q_{12}^2 w_n  (x_{\fudos} + x_1)
\\
&= q_{12} w_{n + 1} + q_{12}(x_1w_n  - q_{12} w_nx_1) =  q_{12} (w_{n + 1} + \sh_{1,n}) = q_{12} w_{n + 1};
\\
\partial_{\futres} ( w_{n + 1}) &=  \partial_{\futres}(x_{\fudos}w_n - q_{12} w_n x_{\fudos}) = x_{\fudos} \partial_{\futres}(w_n) - q_{12} \partial_{\futres}(w_n) g_2 \cdot x_{\fudos}
\\&= x_{\fudos}\varpi_n x_1^n - q_{12} q_{21}\varpi_n x_1^n (x_{\fudos} + \as_{21} x_1)
\\ & = \varpi_n \left(x_{\fudos} x_1^n -  x_1^n x_{\fudos} - \as_{21} x_1^{n+1}\right)
\overset{\eqref{eq:relations B(W) - case 1}}{=} -\varpi_n(\frac{n}{2} + \as_{21}) x_1^{n+1}.
\end{align*}
Next we compute $g_2 \cdot w_0 = g_2 \cdot x_{\futres} = \epsilon_2x_{\futres} +  x_2$, while
\begin{align*}
g_2 \cdot w_{n + 1} &= g_2 \cdot(x_{\fudos}w_n - q_{12} w_n x_{\fudos})
\\&=
q_{21}^{n+1} \left((x_{\fudos} + \as_{21} x_1)  (\epsilon_2w_n + z_n) - q_{12}(\epsilon_2w_n + z_n)(x_{\fudos} + \as_{21} x_1)
\right)
\\&=
q_{21}^{n+1} \left( \epsilon_2 (x_{\fudos}w_n - q_{12} w_n x_{\fudos}) + (x_{\fudos} z_n  - q_{12}  z_nx_{\fudos}) \right.
\\ &\qquad \left. + \as_{21} \left(\epsilon_2( x_1 w_n - q_{12} w_nx_1) + (x_1 z_n - q_{12} z_n x_1)\right)
\right)
\\&\overset{\eqref{eq:2blocks-property-minus1}, \eqref{eq:2blocks-auxiliar-0}}{=} q_{21}^{n+1} (\epsilon_2w_{n + 1} + z_{n + 1})
\end{align*}
In the last equality, we use that $\ad_cx_1 w_n =  x_1 w_n - q_{12} w_nx_1$, by the last identity in \eqref{eq:2blocks-auxiliar-1}.
\eqref{item:2blocks-auxiliar-3}: first, $(z_n)_{0 \le n \le \ghost_1}$ and $(w_n)_{0 \le n}$ generate $K^1$ by  \eqref{item:2blocks-auxiliar-1}.
Second,  $w_n \neq 0 \iff 0 \le n \le \ghost_1$ by  \eqref{item:2blocks-auxiliar-2}. Finally,  since $\gr z_n = \gr w_n = n+1$, it suffices to prove that
$\{z_n, w_n\}$ for a fixed $0 \le n \le \ghost_1$ are linearly independent, what follows applying $\partial_2$ and $\partial_{\futres}$.
\epf

\emph{Proof of Lemma \ref{lemma:2blocks-caso1}.\/}
By Lemma \ref{le:zcoact}, the coaction \eqref{eq:coaction-K^1} on $z_n$ and $w_n$, $0 \le n \le \ghost_1$, is given by \eqref{eq:coact-zn}, that read in the present context
	as follows:
	\begin{align} \label{eq:coact-zn-2blocks}
	\delta (z_{n}) &= \sum _{k=0}^n \nu_{k,n}\, x_1^{n-k}g_1^{k} g_2 \otimes z_k,&
\delta (w_{n}) &= \sum _{k=0}^n \nu_{k,n}\, x_1^{n-k}g_1^{k} g_2 \otimes w_k.
	\end{align}

Let $p_{ij} = q_{12}^iq_{21}^j\epsilon_{2}$, $0\le i,j\le \ghost_1$.
The braided vector space  $K^1$ has braiding
\begin{align*}
c(z_i \ot z_j ) &= p_{ij} z_j \ot z_i, &
c(z_i \ot w_j ) &= p_{ij} (w_j + \epsilon_2 z_j)\ot z_i, \\
c(w_i \ot z_j ) &= p_{ij} z_j \ot w_i, & c(w_i \ot w_j ) &= p_{ij} (w_j + \epsilon_2 z_j) \ot w_i.
\end{align*}
Then the braided vector subspace $W = W_1 \oplus W_2$, where $W_1 = \langle z_0, w_0\rangle$ and $W_2 = \langle z_1\rangle$, has braiding in the
ordered basis $z_0, w_0, z_1$ given by
\begin{align}\label{eq:braiding-2blocks-caso1}
\begin{pmatrix}
\epsilon_2 z_0 \otimes z_0 &  (\epsilon_2 w_0 + z_0) \otimes z_0 & p_{01} z_1  \otimes z_0
\\
\epsilon_2 z_0 \otimes w_0 & (\epsilon_2 w_0 + z_0) \otimes w_0 & p_{01}  z_1  \otimes w_0
\\
p_{10} z_0 \otimes z_1 &  p_{10}(w_0 + \epsilon_2 z_0) \otimes z_1&  \epsilon_2 z_1  \otimes z_1
\end{pmatrix}.
\end{align}
That is, $W_1$ is an $\epsilon_2$-block, $W_2$ is a point with label $\epsilon_2$, the interaction is weak and $a = \epsilon_2$, cf. \eqref{eq:braiding-block-point}; thus
$\ghost = \begin{cases} -2, &\epsilon_2 = 1, \\ -1, &\epsilon_2 = -1 \end{cases}$ is negative. Hence $\GK \toba(W) = \infty$ by Lemma \ref{lemma:weak-not-discrete}.
\qed

\begin{kase}  $\ghost_{2} \ge 1$, $\ghost_{1} \ge 0$.
\end{kase}

\begin{lemma}\label{lemma:2blocks-caso2}
	If $\epsilon_1 = 1$, $\ghost_{2} \ge 1$ and $\ghost_{1} \ge 0$, then $\GK \cB(V) = \infty$.
\end{lemma}

We need some preliminaries before the proof of Lemma \ref{lemma:2blocks-caso2}. Let $(\mu_n)$ be the scalars defined in \eqref{eq:def-mu-n}, with $a_{21}$ in the place of $a$, and let $\kappa_n = (-1)^{n}\prod_{i= 0}^{n-1} \left(\dfrac{i}{2}+ a_{21}\right)$, $n \in \N_0$. It is easy to see that
\begin{align}
\kappa_n = \frac{(-1)^k}{2^k} \mu_{n},& &\text{where } k = \lfloor \frac n2\rfloor.
\end{align}
Then \eqref{eq:derivations-zn-eps1} says in the present context that
$\partial_2(z_{n}) = \kappa_n x_{1}^{n}$.

\begin{lemma}\label{lemma:2blocks-auxiliar-caso2}
	\begin{enumerate}\renewcommand{\theenumi}{\alph{enumi}}\renewcommand{\labelenumi}{(\theenumi)}
		\item\label{item:2blocks-auxiliar-11} If $m > 1$, then $\sh_{m, n} =0$. Hence
		for all $n \in \N_0$
		\begin{align}\label{eq:2blocks-auxiliar-10}
		\ad_cx_1(\sh_{1,n}) &= 0, \\
	\label{eq:2blocks-caso2-3}
	\sh_{1, n + 1} &= [x_{\fudos}, \sh_{1,n}]_c.
		\end{align}
		
		\medbreak
		\item\label{item:2blocks-auxiliar-12}   We have for all $n\in \N_0$
		\begin{align}\label{eq:2blocks-auxiliar-11}
		g_1 \cdot \sh_{1,n} &= q_{12} \sh_{1,n},& g_2 \cdot \sh_{1,n} &= q_{21}^{n + 1} \epsilon_2\sh_{1,n};
		\\ \label{eq:shn0}		\partial_1(\sh_{1,n}) &= \partial_{\fudos} (\sh_{1,n}) = \partial_{\futres} (\sh_{1,}n) = 0; &&\\
		\label{eq:shn}		\partial_{2} (\sh_{1,n}) &= \frac{a_{12} \kappa_{n+1}}{a_{21}} x_1^{n+1}. &&
		\end{align}
In particular, if $\ghost_1 = 0$, then $\sh_{1,n} \neq 0$ for every $n \in \N_0$.

		\medbreak
		\item\label{item:2blocks-auxiliar-13} For every $n \in \N_0$,
		$	\partial_1(w_{n}) = \partial_{\fudos} (w_{n}) = 0$,
\begin{align}\label{eq:wn-g1}
g_1 \cdot w_n &= q_{12} (w_n + a_{12}z_n + n\sh_{1, n-1}),\\
\label{eq:wn-g2}
g_2 \cdot w_n &= q_{21}^n (\epsilon_2w_n + z_n + n a_{21} \epsilon_2\sh_{1, n-1}),
\\ \label{eq:wn-partial2}
\partial_{2} (w_n) &= \frac{ a_{12} n \kappa_{n}}{a_{21}}  x_1^{n-1}(x_{\fudos} + a_{21}x_1),
\\ \label{eq:wn-partial52}
\partial_{\futres} (w_n) &= \kappa_{n} x_1^n.
\end{align}		
		\medbreak
		\item\label{item:2blocks-auxiliar-14}
Assume that $\ghost_1 > 0$.	
The elements $(z_n)_{0 \le n \le \ghost_1}$, $(w_n)_{0 \le n \le \ghost_1}$ and
$(\sh_{1,n})_{0 \le n \le \ghost_1 - 1}$ form a basis of $K^1$.
	\end{enumerate}
\end{lemma}

\pf
To start with \eqref{item:2blocks-auxiliar-11}, $\partial_1(\sh_{1, 0}) = \partial_{\fudos} (\sh_{1, 0}) = \partial_{\futres} (\sh_{1, 0}) = 0$, while
\begin{align}\label{eq:2blocks-caso2-2}
\partial_{2} (\sh_{1, 0}) &= -a_{12}x_1, &
g_1 \cdot \sh_{1, 0} &= q_{12} \sh_{1, 0},
&  g_2 \cdot \sh_{1, 0} &= q_{21} \epsilon_2 \sh_{1, 0}.
\end{align}
Indeed,
\begin{align*}
\partial_{2} (\sh_{1, 0}) &= \partial_{2} ( x_{1} x_{\futres} - q_{12}(x_{\futres} + a_{12} x_2)x_{1}) = - q_{12}a_{12}\partial_{2} (x_2x_{1}) = -  a_{12}x_{1}; \\
g_1 \cdot \sh_{1, 0} &= q_{12} \left( x_1 (x_{\futres} + a_{12} x_2)
- q_{12}(x_{\futres} + 2a_{12} x_2)x_{1} \right) \\
&= q_{12} \left(\sh_{1, 0} + a_{12} (x_1x_2 - q_{12} x_2x_1) \right)
= q_{12} \sh_{1, 0};
\\
g_2 \cdot \sh_{1, 0} &= q_{21} \left( x_1  (\epsilon_2x_{\futres} +  x_2)
- q_{12}(\epsilon_2x_{\futres} + (1 + a_{12}\epsilon_2) x_2) x_1\right)
\\
&=  q_{21} \left(\epsilon_2\sh_{1, 0} + x_1x_2 - q_{12} x_2x_1 \right) = q_{21} \epsilon_2\sh_{1, 0}.
\end{align*}
However, $\sh_{2, 0} = 0$:
clearly $\partial_1(\sh_{2, 0}) = \partial_{\fudos} (\sh_{2, 0}) = \partial_{\futres} (\sh_{2, 0}) = 0$, while
$\partial_{2} (\sh_{2, 0}) = \partial_{2} ( x_{1}  \sh_{1, 0} - q_{12} \sh_{1, 0} x_{1}) = -a_{12}x_1^2 + q_{12} a_{12}x_1 q_{21}x_1 =0$.

Hence $\sh_{m, 0} = 0$ for every $m \ge 2$ by definition, and \eqref{item:2blocks-auxiliar-11} holds
by \eqref{eq:2blocks-property0}. From this, again by \eqref{eq:2blocks-property0}, we get \eqref{eq:2blocks-caso2-3}.
Next we prove \eqref{eq:2blocks-auxiliar-11}.
Arguing recursively, since
\begin{align}\label{eq:shn-explicit}
\sh_{1, n + 1} = x_{\fudos} \sh_{1, n}- q_{12}\sh_{1, n}x_{\fudos}
\end{align}
and $\sh_{2, n} = x_{1} \sh_{1, n}- q_{12}\sh_{1, n}x_{1}$ by the inductive hypothesis, we have
\begin{align*}
g_1 \cdot \sh_{1, n + 1} &= q_{12} \left((x_{\fudos} + x_1) \sh_{1, n}
- q_{12}\sh_{1, n}(x_{\fudos} + x_1) \right) \\
&= q_{12} \left( \sh_{1, n + 1} + \sh_{2, n} \right) = q_{12} \sh_{1, n + 1}; \\
g_2 \cdot \sh_{1, n + 1} &= q_{21}^{n+1} \epsilon_2 \left((x_{\fudos} + a_{21}x_1) \sh_{1, n}
- q_{12}\sh_{1, n}(x_{\fudos} + a_{21} x_1) \right) \\
&= q_{21}^{n+1} \epsilon_2 \left( \sh_{1, n + 1} + a_{21} \sh_{2, n} \right) = q_{21}^{n+1}  \epsilon_2 \sh_{1, n + 1}.
\end{align*}

Now \eqref{eq:shn0} follows at once; while \eqref{eq:shn} for $n=0$ is just \eqref{eq:2blocks-caso2-2}.
If \eqref{eq:shn} holds for $n$, then
\begin{align*}
\partial_{2} (\sh_{1, n + 1}) &= \partial_{2} ( x_{\fudos}  \sh_{1, n} - q_{12} \sh_{1, n} x_{\fudos})
\\ &= \frac{a_{12} \kappa_{n+1}}{a_{21}}
\left(x_{\fudos} x^{n+1}  -  x^{n+1} (x_{\fudos} + a_{21} x_1)\right)
\\ &\overset{\eqref{eq:relations B(W) - case 1}}{=} \frac{a_{12} \kappa_{n+1}}{a_{21}}
\left(-\frac{n+1}{2}x_1^{n+2} - \as_{21} x_1^{n+2}\right) =  \frac{a_{12} \kappa_{n+2}}{a_{21}} x_1^{n+2}.
\end{align*}
Then the last statement in \eqref{item:2blocks-auxiliar-12} follows. To start with \eqref{item:2blocks-auxiliar-13}, observe that
$\partial_1(w_{1}) = \partial_{\fudos} (w_{1}) =  0$, and \emph{a fortiori}
$\partial_1(w_{n}) = \partial_{\fudos} (w_{n}) =  0$ recursively,
while
\begin{align*}
g_1 \cdot w_1 &= q_{12} (w_1 + a_{12}z_1 + \sh_{1, 0}),&
g_2 \cdot w_1 &= q_{21} (\epsilon_2w_1 + z_1 + a_{21} \epsilon_2\sh_{1, 0});\\
\partial_{2} (w_1) &= -a_{12}(x_{\fudos} + a_{21}x_1),&
\partial_{\futres} (w_1) &= -a_{21}x_1.
\end{align*}
Indeed,
\begin{align*}
g_1 \cdot w_1 &= q_{12} \left((x_{\fudos} + x_1) (x_{\futres} + a_{12} x_2)
- q_{12}(x_{\futres} + 2a_{12} x_2) (x_{\fudos} + x_1) \right) \\
&= q_{12} \left(w_1 + a_{12} (x_1x_2 - q_{12} x_2x_1) + a_{12} z_1 + \sh_{1, 0}\right);
\\
g_2 \cdot w_1 &= q_{21} \left( (x_{\fudos} + a_{21}x_1)  (\epsilon_2x_{\futres} +  x_2) \right. \\
&\qquad - \left. q_{12}(\epsilon_2x_{\futres} + (1 + a_{12}\epsilon_2) x_2)  (x_{\fudos} + a_{21}x_1)\right)
\\
&=  q_{21} \left(\epsilon_2w_1 + a_{21}(x_1x_2 - q_{12} x_2x_1) + z_1 + a_{21} \epsilon_2 \sh_{1, 0} \right);
\end{align*}
also,
\begin{align*}
\partial_{2} (w_1) &= \partial_{2} ( x_{\fudos} x_{\futres} - q_{12}(x_{\futres} + a_{12} x_2)x_{\fudos}) = - q_{12}a_{12}\partial_{2} (x_2x_{\fudos})\\ &= -  a_{12}(x_{\fudos} + a_{21}x_1);
\\
\partial_{\futres} (w_1) &= \partial_{\futres}
(x_{\fudos} x_{\futres} - q_{12}x_{\futres} x_{\fudos}) = x_{\fudos} - q_{12}q_{21}  (x_{\fudos} + a_{21} x_1).
\end{align*}
Next, assume that \eqref{eq:wn-g1} holds for $n$. Then
\begin{align*}
w_{n+1} = x_{\fudos} w_n - q_{12} (w_n + a_{12}z_n +n \sh_{1, n-1}) x_{\fudos}.
\end{align*}
We compute
\begin{align*}
g_1 \cdot w_{n+1} &= g_1 \cdot \left( x_{\fudos} w_n - q_{12} (w_n + a_{12}z_n +n \sh_{1, n-1}) x_{\fudos}\right) \\
&= q_{12} \left((x_{\fudos} + x_1) (w_n + a_{12}z_n +n \sh_{1, n-1}) \right. \\ & \left. \qquad
- q_{12}(w_n + 2a_{12}z_n + 2n \sh_{1, n-1}) (x_{\fudos} + x_1) \right) \\
&= q_{12} \big(w_{n+1} + a_{12} z_{n+1} + n(x_{\fudos} -q_{12}\sh_{1, n-1}x_{\fudos}) + x_1 w_n - q_{12} w_n x_1 \\
&  \qquad +
a_{12} (x_1z_n - q_{12}z_nx_1)  + n(x_{1}\sh_{1, n-1} -q_{12}\sh_{1, n-1}x_{1}) \big),
\end{align*}
and \eqref{eq:wn-g1} follows because $[x_1, z_n]_c =0= \sh_{2, n}$.
Similarly,
\begin{multline*}
	g_2 \cdot w_{n+1} = g_2 \cdot \left( x_{\fudos} w_n - q_{12} (w_n + a_{12}z_n +n \sh_{1, n-1}) x_{\fudos}\right) \\
	= q_{21}^{n+1} \left((x_{\fudos} + a_{21} x_1) (\epsilon_2w_n + z_n +n a_{21} \epsilon_2 \sh_{1, n-1}) \right. \\  \left.
	- q_{12}(\epsilon_2w_n + z_n +n a_{21} \epsilon_2 \sh_{1, n-1} + a_{12} \epsilon_2 z_n + n \epsilon_2 \sh_{1, n-1}) (x_{\fudos} + a_{21} x_1) \right) \\
	= q_{21}^{n+1} \left(\epsilon_2 w_{n+1} + z_{n+1} +
	(n+1)a_{21} \epsilon_2 \sh_{1, n}    +
	a_{21}[x_1, z_n]_c  + na_{21}^2 \epsilon_2 \sh_{2, n-1} \right),
\end{multline*}
and \eqref{eq:wn-g2} follows because $[x_1, z_n]_c =0= \sh_{2, n-1}$.
We check \eqref{eq:wn-partial2} for $n=2$:
\begin{align*}
\partial_{2} (w_{2}) & = \partial_{2} (x_{\fudos} w_1 - q_{12} (w_1 + a_{12}z_1 +  \sh_{1, 0}) x_{\fudos})
\\ &= x_{\fudos} \partial_{2} (w_1) - \partial_2(w_1 + a_{12}z_1 +  \sh_{1, 0})(x_{\fudos} + a_{21} x_1)
= -  a_{12} x_{\fudos}(x_{\fudos} + a_{21}x_1) \\ &- \left(-  a_{12} (x_{\fudos} + a_{21}x_1)
-  a_{12} a_{21} x_1 -a_{12} x_1 \right)(x_{\fudos} + a_{21} x_1)
\\ &= a_{12} \left(-x_{\fudos} + x_{\fudos} + (2a_{21} + 1 )x_1 \right)(x_{\fudos} + a_{21} x_1)
\end{align*}
as claimed. Set $P_n = \frac{n \kappa_{n}}{a_{21}}$. If \eqref{eq:wn-partial2} holds for $n$, then
\begin{align*}
\partial_{2} (w_{n+1}) & =    a_{12}  P_n x_{\fudos}  x_1^{n-1}(x_{\fudos} + a_{21}x_1) \\& - \left(a_{12} P_n   x_1^{n-1}(x_{\fudos} + a_{21}x_1) + a_{12}\kappa_n x_1^n + n \frac{a_{12} \kappa_{n}}{a_{21}} x_1^n \right)(x_{\fudos} + a_{21} x_1)
\\&= -a_{12}  \left(P_n(\tfrac{n - 1}{2} + a_{21}) +   \kappa_n + \frac{n\kappa_{n}}{a_{21}} \right) x_1^n (x_{\fudos} + a_{21} x_1)
\\&= -\frac{a_{12} \kappa_{n}}{a_{21}}  \left(n(\tfrac{n - 1}{2} + a_{21}) + a_{21} + n \right) x_1^n (x_{\fudos} + a_{21} x_1)
\\&= a_{12}P_{n+1}x_1^n (x_{\fudos} + a_{21} x_1).
\end{align*}

Next we deal with \eqref{eq:wn-partial52}; the case $ n=1$ was already settled. If \eqref{eq:wn-partial52} holds for $n$, then
\begin{align*}
\partial_{\futres} (w_{n+1}) & \overset{\heartsuit}{=} \partial_{\futres} ( x_{\fudos}  w_n - q_{12} w_n x_{\fudos})
	\\ &=\kappa_{n}
	\left(x_{\fudos} x^{n}  -  x^{n} (x_{\fudos} + a_{21} x_1)\right)
	\\ &\overset{\eqref{eq:relations B(W) - case 1}}{=} \kappa_{n}
	\left(-\frac{n}{2}x_1^{n+1} - \as_{21} x_1^{n+1}\right) = \kappa_{n+1}x_1^{n+1};
\end{align*}
here $\heartsuit$ holds because $\partial_{\futres} (z_n) = \partial_{\futres} (\sh_{1,n-1}) =0$. Hence \eqref{item:2blocks-auxiliar-13} is proved.

For \eqref{item:2blocks-auxiliar-14}, notice that $z_{\ghost_1+1}=w_{\ghost_1+1}=\sh_{\ghost_1}=0$ since $\kappa_{\ghost_1+1}=0$. Thus $(z_n)_{0 \le n \le \ghost_1}$, $(w_n)_{0 \le n \le \ghost_1}$ and
$(\sh_{1,n})_{0 \le n \le \ghost_1 - 1}$ span the vector space $K^1$. To prove that these elements are linearly independent, it suffices to consider elements of the same degree, $z_n$, $w_n$ and $\sh_{1,n-1}$. We use now $\partial_2$, $\partial_{\futres}$ and that $\kappa_n\neq 0$ for $n\le\ghost_1$.
\epf

We now compute the coaction \eqref{eq:coaction-K^1} given by $\delta =(\pi _{\NA (V_1)\#  \ku \Gamma}\otimes \id)\Delta _{\NA (V)\#  \ku \Gamma}$. By Lemma \ref{le:zcoact}, this is given by  \eqref{eq:coact-zn-2blocks} on $z_n$, $0 \le n \le \ghost_1$.

\begin{lemma}\label{lemma:coact-shn-wn}
 The coaction \eqref{eq:coaction-K^1} of $K^1 \in {}^{\NA (V_1)\# \ku \Gamma}_{\NA (V_1)\# \ku \Gamma}\mathcal{YD}$ is determined by \eqref{eq:coact-zn-2blocks},
\begin{align} \label{eq:coact-shn}
\delta (\sh_{1,n}) &= \sum _{h=0}^n A_{h,n}\, x_1^{n-h}g_1^{h+1} g_2 \otimes \sh_{1,h} - a_{12} B_{h,n}\, x_1^{n + 1 -h}g_1^{h} g_2 \otimes z_{h},
\\ \label{eq:coact-wn}
\delta (w_{n}) &= g_1^ng_2\otimes w_n + \sum_{j=0}^{n-1} \mathrm{a}_j \otimes w_j
+ \sum_{j=0}^{n-1} \mathrm{b}_j \otimes z_j + \sum_{j=0}^{n-2} \mathrm{c}_j \otimes \sh_{1,j},
\end{align}
where $\mathrm{a}_j,\mathrm{b}_j\in\cB(V_1)\#  \ku \Gamma$ have degree $n-j$, $\mathrm{c}_j\in\cB(V_1)\#  \ku \Gamma$ has degree $n-j-1$, and
\begin{align} \label{eq:Ahn}
A_{h,n} &=  \binom{n}{h} \frac{\kappa_{n+1}}{\kappa_{h+1}}, &
B_{h,n} &=  \binom{n}{h} \frac{\kappa_{n}}{\kappa_{h}}.
\end{align}

\end{lemma}

\pf
Since $\sh_{1,0} = x_1 x_{\futres} - q_{12}(x_{\futres} + a_{12}x_2) x_1$, we have
\begin{multline*}
\delta (\sh_{1,0})  = (x_1\otimes 1+g_1\otimes x_1)(g_2\otimes x_{\futres})
-q_{12}(g_2\otimes x_{\futres}) (x_1\otimes 1+g_1\otimes x_1)
\\-q_{12}a_{12} (g_2\otimes x_{2}) (x_1\otimes 1+g_1\otimes x_1)=
x_1g_2\otimes x_{\futres} + g_1g_2\otimes x_1x_{\futres}
\\
-q_{12} \left(g_2 x_1\otimes x_{\futres} + g_2g_1\otimes x_{\futres}x_1
+a_{12} g_2x_1\otimes x_{2} + a_{12} g_1g_2\otimes x_{2}x_1\right)
\\
= g_1g_2 \ot \sh_{1,0}- a_{12} x_1g_2 \ot z_0.
\end{multline*}
Hence $A_{0,0} = 1 = B_{0,0}$.
Now we assume that \eqref{eq:coact-shn} holds for $n$; hence
\begin{multline*}
\delta (\sh_{1,n+1}) \overset{\eqref{eq:shn-explicit}}{=} (x_{\fudos}\otimes 1+g_1\otimes x_{\fudos}) \\
\times \sum _{h=0}^n
\left(A_{h,n}\, x_1^{n-h}g_1^{h+1} g_2 \otimes \sh_{1,h} - a_{12} B_{h,n}\, x_1^{n + 1 -h}g_1^{h} g_2 \otimes z_{h} \right)
\\- q_{12} \sum _{h=0}^n
\left(A_{h,n}\, x_1^{n-h}g_1^{h+1} g_2 \otimes \sh_{1,h}  - a_{12} B_{h,n}\, x_1^{n + 1 -h}g_1^{h} g_2 \otimes z_{h}\right)
\\ \times (x_{\fudos}\otimes 1+g_1\otimes x_{\fudos}) =
I + II + III + IV + V + VI + VII + VIII,
\end{multline*}
where
\begin{align*}
I &= \sum _{h=0}^n
A_{h,n}\, x_{\fudos}x_1^{n-h}g_1^{h+1} g_2 \otimes \sh_{1,h}, \\
II &= \sum _{h=0}^n
A_{h,n}\, g_1x_1^{n-h}g_1^{h+1} g_2 \otimes x_{\fudos}\sh_{1,h},
\\
III &= - a_{12} \sum _{h=0}^n   B_{h,n}\, x_{\fudos} x_1^{n + 1 -h}g_1^{h} g_2 \otimes z_{h},\\
IV &= - a_{12} \sum _{h=0}^n  B_{h,n}\, g_1 x_1^{n + 1 -h}g_1^{h} g_2 \otimes x_{\fudos}z_{h},
\\
V &= -q_{12} \sum _{h=0}^n
A_{h,n}\, x_1^{n-h}g_1^{h+1} g_2 x_{\fudos} \otimes \sh_{1,h},\\
VI &=  -q_{12} \sum _{h=0}^n
A_{h,n}\, x_1^{n-h}g_1^{h+2} g_2 \otimes \sh_{1,h} x_{\fudos},
\\
VII &=  a_{12}q_{12} \sum _{h=0}^n   B_{h,n}\, x_1^{n + 1 -h}g_1^{h} g_2 x_{\fudos} \otimes z_{h}, \\
VIII &= a_{12}q_{12} \sum _{h=0}^n  B_{h,n}\, x_1^{n + 1 -h} g_1^{h + 1} g_2 \otimes z_{h} x_{\fudos}.
\end{align*}
Then by a direct computation using \eqref{eq:relations B(W) - case 1}, we have
\begin{align*}
I + V &=  -\sum _{h=0}^n  A_{h,n} (a_{21} + \tfrac{n+h+2}{2})\, x_1^{n + 1-h}g_1^{h+1} g_2 \otimes \sh_{1,h};
\\
II + VI &=  \sum _{h=0}^n A_{h,n} \, x_1^{n -h}g_1^{h+2} g_2 \otimes \sh_{1,h + 1};
\end{align*}
\begin{align*}
III + VII &= a_{12} \sum _{h=0}^n   B_{h,n} (a_{21} + \tfrac{n+h+1}{2})
\, x_1^{n + 2 -h}g_1^{h} g_2 \otimes z_{h};
\\
IV + VIII &=  -a_{12} \sum _{h=0}^n  B_{h,n}\, x_1^{n + 1 -h} g_1^{h + 1} g_2 \otimes z_{h +1}.
\end{align*}
Hence
\begin{align*}
I + V + II + VI &=  \sum _{h=0}^{n + 1}  A_{h,n + 1} \, x_1^{n+1 -h}g_1^{h+1} g_2 \otimes \sh_{1,h}, \\
III + VII + IV + VIII &= \sum _{h=0}^{n + 1} B_{h,n + 1} \, x_1^{n+2-h}g_1^{h} g_2 \otimes z_{h},
\end{align*}
where $A_{n+1,n + 1} = A_{n,n}$, $B_{n+1,n + 1} = B_{n,n}$
\begin{align}\label{eq:recursiveAn}
  A_{h,n + 1} &= A_{h-1,n}-A_{h,n} (a_{21} + \tfrac{n+h+2}{2}), \, A_{0, n+1} = -A_{0, n}(a_{21} + \tfrac{n+2}{2});
\\ \label{eq:recursiveBn}
B_{h,n + 1} &= B_{h-1,n}-B_{h,n} (a_{21} + \tfrac{n+h+1}{2}), \, B_{0, n+1} = -B_{0, n}(a_{21} + \tfrac{n+1}{2}).
\end{align}
From $A_{0,0}=1 = B_{0,0}$, \eqref{eq:recursiveAn} and \eqref{eq:recursiveBn}, \eqref{eq:Ahn} follows by induction.

For \eqref{eq:coact-wn}, note that $\delta(w_n)\in \cB(V_1)\#  \ku \Gamma\otimes K^1$. As
$(z_n)_{0 \le n \le \ghost_1}$, $(w_n)_{0 \le n \le \ghost_1}$ and $(\sh_{1,n})_{0 \le n \le \ghost_1 - 1}$ form a basis of $K^1$ by Lemma \ref{lemma:2blocks-auxiliar-caso2} and $\delta$ is a graded map,
\begin{align*}
\delta (w_{n}) &= \sum_{j=0}^{n} \mathrm{a}_j \otimes w_j
+ \sum_{j=0}^{n} \mathrm{b}_j \otimes z_j + \sum_{j=0}^{n-1} \mathrm{c}_j \otimes \sh_{1,j},
\end{align*}
for some $\mathrm{a}_j,\mathrm{b}_j\in\cB(V_1)\#  \ku \Gamma$ of degree $n-j$ and $\mathrm{c}_j\in\cB(V_1)\#  \ku \Gamma$ of degree $n-j-1$.
It remains to prove that $\mathrm{a}_n=g_1^ng_2$, $\mathrm{b}_n=\mathrm{c}_{n-1}=0$, which follows by induction.
\epf

\emph{Proof of Lemma \ref{lemma:2blocks-caso2}.\/}
Set $N=\ghost_1=-2a_{21}$. If $x\in\{w_N,z_N,\sh_{1,N-1}\}$, then
\begin{align*}
c(x\ot & w_{N}) = g_1^Ng_2\cdot w_{N}\ot x = q_{21}^N g_1^N \cdot (\epsilon_2 w_N+ z_N + Na_{21}\epsilon_2\sh_{1,N-1}) \ot x \\
&= q_{12}^Nq_{21}^N(\epsilon_2w_N+ (1+\epsilon_2 Na_{12}) z_N + N(N+a_{21}\epsilon_2)\sh_{1,N-1})  \ot x \\
&= (\epsilon_2w_N+ (1-2 \epsilon_2 a_{21}a_{12}) z_N -2a_{21}^2(\epsilon_2-2)\sh_{1,N-1}) \ot x , \\
c(x\ot & z_{N}) = g_1^Ng_2\cdot z_{N}\ot x = q_{12}^Nq_{21}^N\epsilon_2 z_{N}\ot x = \epsilon_2 z_{N}\ot x, \\
c(x\ot & \sh_{1,N-1}) = g_1^Ng_2\cdot \sh_{1,N-1}\ot x = \epsilon_2 \sh_{1,N-1}\ot x.
\end{align*}
by Lemmas \ref{lemma:2blocks-auxiliar-caso2} and \ref{lemma:coact-shn-wn}.
Set $y_N= (1-2 \epsilon_2 a_{21}a_{12}) z_N -2a_{21}^2(\epsilon_2-2)\sh_{1,N-1}$. Then $W = \langle y_N, w_N, z_N \rangle$ is a 3-dimensional braided vector subspace of $K^1$, with braiding in the ordered basis $y_N, w_N, z_N$ given by
\begin{align}\label{eq:braiding-2blocks-caso2}
\begin{pmatrix}
\epsilon_2 y_N \otimes y_N &  (\epsilon_2 w_N + y_N) \otimes y_N & \epsilon_2 z_N  \otimes y_N
\\
\epsilon_2 y_N \otimes w_N & (\epsilon_2 w_N + y_N) \otimes w_N & \epsilon_2  z_N  \otimes w_N
\\
\epsilon_2 y_N \otimes z_N &  \epsilon_2(w_N + \epsilon_2 y_N) \otimes z_N&  \epsilon_2 z_N  \otimes z_N
\end{pmatrix}.
\end{align}
Hence $W = W_1 \oplus W_2$, where $W_1 = \langle y_N, w_N\rangle$ and $W_2 = \langle z_N\rangle$, $W_1$ is an $\epsilon_2$-block, $W_2$ is a point with label $\epsilon_2$, the interaction is weak and $a = \epsilon_2$, cf. \eqref{eq:braiding-block-point}; thus
$\ghost = \begin{cases} -2, &\epsilon_2 = 1, \\ -1, &\epsilon_2 = -1 \end{cases}$ is negative, and $\GK \toba(W) = \infty$ by Lemma \ref{lemma:weak-not-discrete}.
\qed

\subsection{$\epsilon_1 = \epsilon_2 = -1$}\label{subsec:2blocks-eps-1}

\begin{lemma}\label{lemma:2blocks-caso3}
	If $\epsilon_1 = \epsilon_2 = -1$,   then $\GK \cB(V) = \infty$.
\end{lemma}

We consider again the scalars $(\mu_n)$ defined in \eqref{eq:def-mu-n}, with $a_{21}$ in the place of $a$.
As in Remark \ref{rem:xk qcommutes with z_k}, set $y_{2k} =x_{\fudos 1}^k$, $y_{2k+1} = x_1x_{\fudos 1}^k$, $k \in \N_0$. Then
\begin{align}\label{eq:yn-formula}
\mu_{n+1}y_{n+1} &=\mu_n \left( x_{\fudos}y_n-(-1)^n y_n(x_{\fudos}+a_{21}x_1) \right) & \mbox{for all }&n\in\N_0.
\end{align}
It follows by the proof of Lemma \ref{lemma:derivations-zn}.

\medbreak

We introduce for all $n, m \in \N_0$ and $a\in \Z/2$, the following elements of $K^{1}$:
\begin{align}
	z_n &= (\ad_{c} x_{\fudos})^n x_2,&  \sch_{a, m, n} &= (\ad_{c} x_{1})^a (\ad_{c} x_{\fudos 1})^m (\ad_{c} x_{\fudos})^n x_{\futres}.
\end{align}
The definition of $z_n$ is consistent with \eqref{eq:zn}, hence \eqref{eq:2blocks-property-minus1} holds
and the family $(z_n)_{0\le n\le \ghost_1}$ is linearly independent.
Moreover \eqref{eq:derivations-zn-eps1} says in the present context that $\partial_2(z_{n}) = \mu_n y_{n}$. We pick some of the $\sch_{a, m, n}$'s:
\begin{align}\label{eq:two-blocks-w-zh-ur}
w_n&:=\sch_{0,0,n}, & \zh_n&:=\sch_{1,0,n-1}, & \ur_n&:=\sch_{0,1,n-1}.
\end{align}

\medbreak

By \eqref{eq:rels-B(V(-1,2))-2} and \eqref{eq:rels-B(V(-1,2))-dos},
$x_{\fudos}x_{\fudos 1} = x_{\fudos 1}x_{\fudos} + x_1x_{\fudos 1}$ and $x_{\fudos 1}x_1 = x_1x_{\fudos}x_1 = x_1x_{\fudos 1}$. By \eqref{eq:-1block+point-bis}, $x_{\fudos} x^n_{\fudos 1} = x_{\fudos 1}^n x_{\fudos} + n  x_1 x^n_{\fudos 1}$. Thus
\begin{align} \label{eq:2blocks-property0-caso3}
	\ad_c x_1 \sch_{a, m, n} &= \sch_{a + 1, m, n}, & \ad_c x_{\fudos} \sch_{0, m, n} &= \sch_{0, m, n+1} +n \sch_{1, m, n},
\\ \label{eq:2blocks-property1-caso3}
	\ad_c x_{\fudos 1} \sch_{a, m, n} &= \sch_{a, m+1, n}, & \ad_c x_{\fudos} \sch_{1, m, n} &= \sch_{0,m+ 1, n} - \sch_{1, m, n+1}.
\end{align}
Thus the $z_n$'s and the $\sch_{a,m, n}$'s generate $K^1$.

\begin{lemma}\label{lemma:2blocks-auxiliar-caso3}
\begin{enumerate}\renewcommand{\theenumi}{\alph{enumi}}\renewcommand{\labelenumi}{(\theenumi)}
\item\label{item:2blocks-auxiliar-21} If $a+m>1$, then $\sch_{a,m,n}=0$. Hence for all $n \in \N$,
\begin{align}\label{eq:2blocks-auxiliar-action}
\ad_cx_{\fudos}(\zh_n) &=\ur_n-\zh_{n+1}, & \ad_cx_{\fudos}(\ur_n) &=\ur_{n+1}.
\end{align}
		
\medbreak

\item\label{item:2blocks-auxiliar-22}   We have for all $n\in \N$, $\ur_n=2\zh_{n+1}$. Hence,
\begin{align}\label{eq:2blocks-auxiliar-action-x1}
\ad_cx_1(w_n) &=\zh_{n+1}, & \ad_cx_1(\zh_n) &=0; \\
\label{eq:2blocks-auxiliar-action-xfudos}
\ad_cx_{\fudos}(w_n) &=w_{n+1}, & \ad_cx_{\fudos}(\zh_n) &=\zh_{n+1}.
\end{align}

\medbreak
\item\label{item:2blocks-auxiliar-23}   We have for all $n\in \N_0$
\begin{align}\label{eq:2blocks-auxiliar-zh0}
g_1 \cdot \zh_n &= (-1)^{n}q_{12} \zh_n,& g_2 \cdot \zh_n &= -q_{21}^{n} \zh_n;
\\ \label{eq:zh1}		\partial_1(\zh_n) &= \partial_{\fudos} (\zh_n) = \partial_{\futres} (\zh_n) = 0; &&\\
\label{eq:zh2}		\partial_{2} (\zh_n) &= \frac{a_{12} \mu_{n}}{a_{21}} y_{n}. &&
\end{align}
In particular, if $\ghost_1 = 0$, then $\zh_{n} \neq 0$ for every $n \in \N_0$.

\medbreak
\item\label{item:2blocks-auxiliar-24} For every $n \in \N_0$,
$	\partial_1(w_{n}) = \partial_{\fudos} (w_{n}) = 0$,
\begin{align}\label{eq:w-g1-case3}
g_1 \cdot w_n &= (-1)^n q_{12} (w_n + a_{12}z_n - n\zh_{n}),\\
\label{eq:w-g2-case3}
g_2 \cdot w_n &= -q_{21}^n (w_n - z_n + n a_{21} \zh_{n}),
\\ \label{eq:w-partial2-case3}
\partial_{2} (w_n) &=\frac{ a_{12} n \mu_{n}}{a_{21}}  y_{n-1}(x_{\fudos} + a_{21}x_1),
\\ \label{eq:w-partial52-case3}
\partial_{\futres} (w_n) &= \mu_n \, y_n.
\end{align}		

\medbreak
\item\label{item:2blocks-auxiliar-25}
Assume that $\ghost_1 > 0$.	
The elements $(z_n)_{0 \le n \le 2\ghost_1}$, $(w_n)_{0 \le n \le 2\ghost_1}$ and $(\zh_{n})_{1 \le n \le 2\ghost_1}$
form a basis of $K^1$.
\end{enumerate}
\end{lemma}

\pf
To start with \eqref{item:2blocks-auxiliar-21}, we notice that
\begin{align}\label{eq:ad-xfudos1}
(\ad_{c} x_{\fudos 1}) x &= x_{\fudos 1} x-x_1(g_1\cdot x)x_1-(g_1^2\cdot x)x_{\fudos 1}, & &x\in\cB(V).
\end{align}
Hence,
\begin{align*}
\sch_{0,1,0} &=x_{\fudos 1}x_{\futres}-q_{12} x_1x_{\futres}x_1-q_{12}^2(x_{\futres}+2a_{12}x_2)x_{\fudos 1}.
\end{align*}
As $\partial_1(x_{\fudos 1})=x_1$, $\partial_{\fudos}(x_{\fudos 1})=\partial_2(x_{\fudos 1})=\partial_{\futres}(x_{\fudos 1})=0$,
we have $\partial_{\fudos}(\sch_{0,1,0})=0$,
\begin{align*}
\partial_1(\sch_{0,1,0})&= x_1(g_1\cdot x_{\futres})-q_{12} g_1\cdot(x_{\futres}x_1) -q_{12} x_1x_{\futres}-q_{12}^2(x_{\futres}+2a_{12}x_2)x_1=0; \\
\partial_2(\sch_{0,1,0})&= -2a_{12}q_{12}^2q_{21}^2 x_{\fudos 1} = -2a_{12}x_{\fudos 1};\\
\partial_{\fudos}(\sch_{0,1,0})&= x_{\fudos 1}-q_{12}q_{21}x_1^2 -q_{12}^2q_{21}^2 x_{\fudos 1}=0.
\end{align*}
By direct computation, $g_1\cdot \sch_{0,1,0}=q_{12}\sch_{0,1,0}$. Hence,
$$ \sch_{1,1,0}=x_1\sch_{0,1,0}-q_{12}\sch_{0,1,0}x_1. $$
Thus $\partial_i(\sch_{1,1,0})=0$ for all $i\in\Idd$, so $\sch_{1,1,0}=0$. Now,
\begin{align*}
\sch_{0,2,0} &=x_{\fudos 1}\sch_{0,1,0}-q_{12} x_1\sch_{0,1,0}x_1-q_{12}^2\sch_{0,1,0} x_{\fudos 1}.
\end{align*}
Hence $\partial_i(\sch_{0,2,0})=0$ for $i=\fudos,\futres$. Also,
\begin{align*}
\partial_1(\sch_{0,1,0})&= q_{12} x_1 \sch_{0,1,0}+q_{12}^2 \sch_{0,1,0}x_1-q_{12} x_1 \sch_{0,1,0}-q_{12}^2\sch_{0,1,0} x_1=0; \\
\partial_2(\sch_{0,1,0})&= -2a_{12}x_{\fudos 1}^2 -x_1x_{\fudos 1}x_1+2a_{12}x_{\fudos 1}^2 =0.
\end{align*}
Thus $\sch_{0,2,0}=0$. We claim that $\sch_{1,1,n}=\sch_{0,2,n}=0$ for all $n\in\N_0$;
the case $n=0$ was already settled. If this holds for $n$, then
\begin{align*}
0 &= \ad_c x_{\fudos} \sch_{1,1,n}= \sch_{0,2,n}- \sch_{1,1,n+1}, \\
0 &= \ad_c x_{\fudos} \sch_{0,2,n}= \sch_{0,2,n+1}+n \, \sch_{1,2,n}= \sch_{0,2,n+1}+n \, \ad_c x_{\fudos 1}\sch_{1,1,n},
\end{align*}
by \eqref{eq:2blocks-property0-caso3} and \eqref{eq:2blocks-property1-caso3}, so $\sch_{1,1,n+1}=\sch_{0,2,n+1}=0$.

Finally we prove that $\sch_{a,m,n}=0$ for all $m,n\in\N_0$, $a=0,1$ such that $a+m>1$ by induction on $a+m$. The case $a+m=2$
was already settled, and the inductive step follows since by \eqref{eq:2blocks-property0-caso3} and
\eqref{eq:2blocks-property1-caso3},
\begin{align*}
\sch_{a+1,m,n}&=\ad_c x_1 \sch_{a,m,n}, & \sch_{a,m+1,n}&=\ad_c x_{\fudos 1} \sch_{a,m,n}.
\end{align*}
Finally \eqref{eq:2blocks-auxiliar-action} is a direct consequence of \eqref{eq:2blocks-property0-caso3} and \eqref{eq:2blocks-property1-caso3}.

As $\ur_1=\sch_{0,1,0}$, we have already seen that
\begin{align*}
\partial_1(\ur_{1}) &= \partial_{\fudos} (\ur_{1}) = \partial_{\futres} (\ur_{1}) = 0, \partial_2(\ur_{1})&=-2a_{12}x_{\fudos 1}.
\end{align*}
As $\partial_1(\zh_1) = \partial_{\fudos} (\zh_{1}) = \partial_{\futres} (\zh_{1}) = 0$, $\partial_2(\zh_{1})=-a_{12}x_{1}$, we have
\begin{align*}
\partial_1(\zh_{2}) &=\partial_{\fudos} (\zh_{2}) = \partial_{\futres} (\zh_{2}) = 0, & \partial_2(\zh_{2}) &=-a_{12}x_{\fudos 1}.
\end{align*}
Hence $\ur_1=2\zh_2$. If $\ur_n=2\zh_{n+1}$, then
\begin{align*}
\ur_{n+1} &= \ad_c x_{\fudos} \ur_n = 2 \ad_c x_{\fudos} \zh_{n+1} = 2 \left(\ur_{n+1}-\zh_{n+2} \right)
\end{align*}
by \eqref{eq:2blocks-auxiliar-action}, and we prove the inductive step. Now \eqref{eq:2blocks-auxiliar-action-x1} and \eqref{eq:2blocks-auxiliar-action-xfudos} follow by \eqref{eq:2blocks-property0-caso3} and \eqref{eq:2blocks-property1-caso3} and the previous equality.
This ends the proof of \eqref{item:2blocks-auxiliar-22}.

To start with \eqref{item:2blocks-auxiliar-23},
\begin{align*}
g_1\cdot \zh_1 &= g_1\cdot \ad_c x_{1}x_{\futres} = -\ad_c x_{1} \big( q_{12}(x_{\futres}+a_{12} x_2)\big) =-q_{12}\zh_1, \\
g_2\cdot \zh_1 &= g_2\cdot \ad_c x_{1}x_{\futres} = q_{21}\ad_c x_{\fudos 1} \big( -x_{\futres}+x_2)\big) =-q_{21}\zh_1.
\end{align*}
so \eqref{eq:2blocks-auxiliar-zh0} holds for $n=1$. If \eqref{eq:2blocks-auxiliar-zh0} holds for $n$, then
\begin{align*}
g_1\cdot \zh_{n+1} &= g_1\cdot \ad_c x_{\fudos}\zh_n =(-1)^{n}q_{12}\ad_c (-x_{\fudos}+x_1)\zh_n =(-1)^{n+1}q_{12}\zh_{n+1}, \\
g_2\cdot \zh_{n+1} &= g_2\cdot \ad_c x_{\fudos}\zh_n =-q_{21}^{n+1}\ad_c (x_{\fudos}+a_{21}x_1) \zh_n=-q_{21}^{n+1}\zh_{n+1}.
\end{align*}
The cases $ n=1,2$ of \eqref{eq:zh1} were already settled, while the recursive step follows from \eqref{eq:2blocks-auxiliar-action-xfudos}.
The cases $ n=1,2$ of \eqref{eq:zh2} were also settled. If \eqref{eq:zh2} holds for $n$, then
\begin{align*}
\partial_{2} (\zh_{n+1}) &= \partial_2 \left( x_{\fudos} \zh_n - (-1)^n q_{12}\zh_n x_{\fudos} \right) \\
&=\frac{a_{12} \mu_{n}}{a_{21}} \left( x_{\fudos} y_{n} - (-1)^n y_{n}(x_{\fudos}+a_{21}x_1) \right) \overset{\eqref{eq:yn-formula}}{=} \frac{a_{12} \mu_{n+1}}{a_{21}} y_{n+1}.
\end{align*}

Next we prove \eqref{item:2blocks-auxiliar-24}. Note that
$\partial_1(w_{1}) = \partial_{\fudos} (w_{1}) =  0$, and \emph{a fortiori}
$\partial_1(w_{n}) = \partial_{\fudos} (w_{n}) =  0$ recursively.
Now we compute
\begin{align*}
g_1 \cdot w_1 &= -q_{12} (w_1 + a_{12}z_1 -\zh_{1}),&
\partial_{2} (w_1) &= -a_{12}(x_{\fudos} + a_{21}x_1),\\
g_2 \cdot w_1 &= -q_{21} (w_1 - z_1 + a_{21} \zh_{1}),&
\partial_{\futres} (w_1) &= -a_{21}x_1.
\end{align*}
Assume that \eqref{eq:w-g1-case3} holds for $n$. Then
\begin{align*}
g_1 \cdot w_{n+1} &= g_1 \cdot \left(\ad_c x_{\fudos}\right) w_n \\
&= (-1)^n q_{12}\left(\ad_c (-x_{\fudos}+x_1)\right)
 (w_n + a_{12}z_n - n\zh_{n}) \\
&= (-1)^n q_{12}\left( -w_{n+1}- a_{12}z_{n+1} + n\zh_{n+1} \right. \\
& \qquad \left.+ \zh_{n+1} + a_{12}(\ad_c x_1)z_n - n (\ad_c x_1)\zh_{n} \right) \\
&= (-1)^{n+1} q_{12}(w_{n+1}+a_{12}z_{n+1}- (n+1)\zh_{n+1} \\
&\qquad -a_{12}(\ad_c x_1)z_n + n (\ad_c x_1)\zh_{n}),
\end{align*}
and \eqref{eq:w-g1-case3} follows because $(\ad_c x_1)z_n=(\ad_c x_1)\zh_{n} =0$.
Similarly, for \eqref{eq:w-g2-case3},
\begin{align*}
g_2 \cdot w_{n+1} &= -q_{21}^{n+1}\left(\ad_c (x_{\fudos}+a_{21}x_1)\right)
(w_n - z_n + n a_{21} \zh_{n})
\\
&= -q_{21}^{n+1} \left( w_{n+1} - z_{n+1} + n a_{21} \zh_{n+1} \right. \\
&\left. \qquad + a_{21}\zh_{n+1}-a_{21}(\ad_c x_1)z_n+n a_{21}^2 (\ad_c x_1)\zh_{n} \right) \\
&= -q_{21}^{n+1} \left( w_{n+1} - z_{n+1} + (n+1) a_{21} \zh_{n+1} \right).
\end{align*}
Notice that $\partial_2(w_1)=-a_{12}(x_{\fudos}+a_{21}x_1)$, so \eqref{eq:w-partial2-case3} holds for $n=1$ since $y_0=1$. If \eqref{eq:w-partial2-case3} holds for $n$, then
\begin{align*}
\partial_{2} (w_{n+1}) &=\partial_{2} \left( x_{\fudos} w_n-(-1)^n q_{12}(w_n + a_{12}z_n - n\zh_{n}) x_{\fudos} \right) \\
& = \frac{a_{12} \mu_{n}}{a_{21}} \left( n\, x_{\fudos} y_{n-1} -(-1)^n \left( n\,  y_{n-1}(x_{\fudos}+a_{21}x_1) +a_{21} y_n -n \, y_{n} \right) \right) \\
& \qquad \times (x_{\fudos}+a_{21}x_1) \overset{\eqref{eq:yn-formula}}{=} \frac{(n+1)a_{12} \mu_{n+1}}{a_{21}} y_{n}(x_{\fudos}+a_{21}x_1).
\end{align*}
Now $\partial_{\futres}(w_1)=-a_{21}x_1$, so \eqref{eq:w-partial52-case3} holds for $n=1$. If \eqref{eq:w-partial52-case3} holds for $n$, then
\begin{align*}
\partial_{\futres} (w_{n+1}) &= \partial_{\futres} \left( x_{\fudos} w_n-(-1)^n q_{12}(w_n + a_{12}z_n - n\zh_{n}) x_{\fudos} \right) \\
&= \mu_n \left( x_{\fudos}y_n-(-1)^n y_n(x_{\fudos}+a_{21}x_1) \right) \overset{\eqref{eq:yn-formula}}{=} \mu_{n+1}y_{n+1}.
\end{align*}

For \eqref{item:2blocks-auxiliar-25}, notice that $z_{2\ghost_1+1}=w_{2\ghost_1+1}=\zh_{2\ghost_1+1}=0$ since $\mu_{2\ghost_1+1}=0$. Thus $(z_n)_{0 \le n \le 2\ghost_1}$, $(w_n)_{0 \le n \le 2\ghost_1}$ and $(\zh_{n})_{1 \le n \le 2\ghost_1}$ span the vector space $K^1$. To prove that these elements are linearly independent, it suffices to consider elements of the same degree, $z_n$, $w_n$ and $\zh_{n}$. We use now $\partial_2$, $\partial_{\futres}$ and that $\kappa_n\neq 0$ for $n\le2\ghost_1$.
\epf

We now compute the coaction \eqref{eq:coaction-K^1} given by $\delta =(\pi _{\NA (V_1)\#  \ku \Gamma}\otimes \id)\Delta _{\NA (V)\#  \ku \Gamma}$. By Lemma \ref{le:zcoact}, this is given by  \eqref{eq:coact-zn-2blocks} on $z_n$, $0 \le n \le \ghost_1$.

\begin{lemma}\label{lemma:coact-zhn-wn-case3}
 The coaction \eqref{eq:coaction-K^1} of $K^1 \in {}^{\NA (V_1)\# \ku \Gamma}_{\NA (V_1)\# \ku \Gamma}\mathcal{YD}$ is given by \eqref{eq:coact-zn-2blocks},
\begin{align} \label{eq:coact-wn-case3}
\delta (w_{n}) &= g_1^ng_2\otimes w_n + \sum_{j=0}^{n-1} \mathrm{a}_j \otimes w_j
+ \sum_{j=0}^{n-1} \mathrm{b}_j \otimes z_j + \sum_{j=1}^{n-1} \mathrm{c}_j \otimes \zh_{j},
\\ \label{eq:coact-zhn-case3}
\delta (\zh_{n}) &= g_1^ng_2\otimes \zh_n + \sum_{j=0}^{n-1} \widetilde{\mathrm{a}}_j \otimes w_j
+ \sum_{j=0}^{n-1} \widetilde{\mathrm{b}}_j \otimes z_j + \sum_{j=1}^{n-1} \widetilde{\mathrm{c}}_j \otimes \zh_{j},
\end{align}
for some $\mathrm{a}_j,\mathrm{b}_j, \mathrm{c}_j,\widetilde{\mathrm{a}}_j, \widetilde{\mathrm{b}}_j, \widetilde{\mathrm{c}}_j \in\cB(V_1)\#  \ku \Gamma$ of degree $n-j$.
\end{lemma}
\pf
Analogous to \eqref{eq:coact-wn}. Indeed, $\delta(w_n), \delta(\zh_n)\in \cB(V_1)\#  \ku \Gamma\otimes K^1$. As
$(z_n)_{0 \le n \le \ghost_1}$, $(w_n)_{0 \le n \le \ghost_1}$ and $(\zh_{n})_{1 \le n \le \ghost_1}$ form a basis of $K^1$ by Lemma \ref{lemma:2blocks-auxiliar-caso3} and $\delta$ is a graded map,
\begin{align*}
\delta (w_{n}) &= \sum_{j=0}^{n} \mathrm{a}_j \otimes w_j
+ \sum_{j=0}^{n} \mathrm{b}_j \otimes z_j + \sum_{j=1}^{n} \mathrm{c}_j \otimes \zh_{j}, \\
\delta (\zh_{n}) &= \sum_{j=0}^{n} \widetilde{\mathrm{a}}_j \otimes w_j
+ \sum_{j=0}^{n} \widetilde{\mathrm{b}}_j \otimes z_j + \sum_{j=1}^{n} \widetilde{\mathrm{c}}_j \otimes \zh_{j},
\end{align*}
for some $\mathrm{a}_j,\mathrm{b}_j, \mathrm{c}_j,\widetilde{\mathrm{a}}_j, \widetilde{\mathrm{b}}_j, \widetilde{\mathrm{c}}_j\in\cB(V_1)\#  \ku \Gamma$ of degree $n-j$. We claim that
\begin{align*}
\mathrm{a}_n&=\widetilde{\mathrm{c}}_n=g_1^ng_2, & \mathrm{b}_n&=\mathrm{c}_{n}=\widetilde{\mathrm{a}}_n=\widetilde{\mathrm{b}}_n=0.
\end{align*}
The proof is direct, by induction on $n$.
\epf

\emph{Proof of Lemma \ref{lemma:2blocks-caso3}.\/}
Set $N=2\ghost_1=2a_{21}$. If $x\in\{w_N,z_N,\zh_{N}\}$, then
\begin{align*}
c(x\ot  w_{N})& = g_1^Ng_2\cdot w_{N}\ot x = -q_{21}^N g_1^N \cdot (w_N-z_N + Na_{21}\zh_{N}) \ot x \\
&= -(-1)^Nq_{12}^Nq_{21}^N(w_N+ (Na_{12}-1) z_N + N(a_{21}-N)\zh_{N})  \ot x \\
&= (-w_N+ (1-2a_{21}a_{12}) z_N +2a_{21}^2\zh_{N}) \ot x , \\
c(x\ot  z_{N}) &= g_1^Ng_2\cdot z_{N}\ot x = -(-1)^Nq_{12}^Nq_{21}^N z_{N}\ot x = -z_{N}\ot x, \\
c(x\ot  \zh_{N}) &= g_1^Ng_2\cdot \zh_{N}\ot x = -(-1)^Nq_{12}^Nq_{21}^N\zh_{N}\ot x=-\zh_{N}\ot x.
\end{align*}
by Lemmas \ref{lemma:2blocks-auxiliar-caso3} and \ref{lemma:coact-zhn-wn-case3}.
Set $y_N= (1-2a_{21}a_{12}) z_N +2a_{21}^2\zh_{N}$. Then $W = \langle y_N, w_N, z_N \rangle$ is a 3-dimensional braided vector subspace of $K^1$, with braiding in the ordered basis $y_N, w_N, z_N$ given by
\begin{align}\label{eq:braiding-2blocks-caso3}
\begin{pmatrix}
- y_N \otimes y_N & (- w_N + y_N) \otimes y_N & - z_N  \otimes y_N
\\
- y_N \otimes w_N & (- w_N + y_N) \otimes w_N & - z_N  \otimes w_N
\\
- y_N \otimes z_N &  -(w_N - y_N) \otimes z_N & - z_N  \otimes z_N
\end{pmatrix}.
\end{align}
Hence $W = W_1 \oplus W_2$, where $W_1 = \langle y_N, w_N\rangle$ and $W_2 = \langle z_N\rangle$, $W_1$ is an $-1$-block, $W_2$ is a point with label $-1$, the interaction is weak and $a = -1$, cf. \eqref{eq:braiding-block-point}; thus
$\ghost =-1$ is negative, and $\GK \toba(W) = \infty$ by Lemma \ref{lemma:weak-not-discrete}.
\qed

\section{Several blocks, several points}\label{sec:yd-dim>3-2blocks}

\subsection{Notations}\label{subsubsection:YD>3-severalblocks-pts}

Let $t,\theta \in \N$, $2 \leq t < \theta$. We fix the notation on the braided vector spaces we will consider in this Section:
\begin{align}\label{eq:notacio-several-blocks}
\Idd_k &= \{k, k + \tfrac{1}{2}\},& k&\in \I_t;
\\ \Idd &= \Idd_1 \cup \dots \cup \Idd_{t} \cup \I_{t+1,\theta};
\\
g_i&\in \Gamma,\, \chi_i\in \widehat\Gamma,& i&\in \I_{\theta};
\\ \eta_k&: \Gamma \to \ku \text{ a $(\chi_k, \chi_k)$-derivation}, \,\eta_k(g_k) = 1,&  k&\in \I_t.
\end{align}
For $k\in \I_t$, we set $V_k = \cV_{g_k}(\chi_k, \eta_k) \in \ydG$  with basis $(x_h)_{h\in\Idd_k}$
and action  	 	
\begin{align}\label{equation:basis-triangular-gral-k}
\gamma\cdot x_k &= \chi_k(\gamma) x_k,& \gamma\cdot x_{\fkdos}&=\chi_k(\gamma) x_{\fkdos} + \eta_k(\gamma)x_{k},& \gamma&\in \Gamma.
\end{align}
For $i\in \I_{t+1, \theta}$ we set $V_i = \ku_{g_{i}}^{\chi_{i}}$. Then we set
\begin{align}\label{eq:braiding-blocks-points}
V &= V_1 \oplus \dots \oplus V_t \oplus V_{t+1} \oplus \dots \oplus V_{\theta},&
c_{ij} &= c_{\vert V_i \otimes V_j}: V_i \otimes V_j \to V_j \otimes V_i.
\end{align}
Thus $(x_h)_{h\in\Idd}$ is a basis of $V$.  Let
\begin{align*}
q_{ij}&= \chi_j(g_i),& i, j&\in \I_{\theta};&  a_{ik} &= q_{ik}^{-1}\eta_k(g_i),& k\in \I_t, \,i&\in \I_{\theta}.
\end{align*}
We want to know when $\GK \toba(V) < \infty$. Thence, we may  assume that
\begin{enumerate}[leftmargin=*,label=\rm{(\Alph*)}]
	\item\label{item:A} $q_{kk}^2 = 1$ hence $a_{kk} = q_{kk} =: \epsilon_k$, $k\in \I_t$, see \S \ref{sec:yd-dim2};
	
	\item\label{item:B} $q_{kl}q_{lk}=1$, $a_{kl}=0$, $k\neq l\in\I_t$, by \S \ref{section:YD>3-2blocks}.
	Hence
	\begin{align}\label{eq:several-blocks-qcommute}
	x_ix_j &= q_{\lfloor i\rfloor \lfloor j\rfloor } x_j x_i, & i,j &\in \Idd_1 \cup \dots \cup \Idd_{t}, & \lfloor i\rfloor &\neq \lfloor j\rfloor.
	\end{align}
	
	\item\label{item:C} $\widetilde{q}_{ik} := q_{ik}q_{ki} = \pm 1$ for all $k\in \I_t$, $i\in \I_{t+1,\theta}$ by Lemma \ref{le:strong}.
	
    \item\label{item:D} $\ghost_{ik}\in\N_0$ for all $k\in \I_t$, $i\in \I_{t+1,\theta}$ by Lemmas \ref{lemma:weak-not-discrete} and \ref{lemma:mild-not-discrete}, where $$ \ghost_{ik} := \begin{cases} -2a_{ik}, &\epsilon_k = 1, \\ a_{ik}, & \epsilon_k = -1.\end{cases} $$

	\item\label{item:E} The corresponding flourished diagram is connected, see \S \ref{subsec:intro-sb}.
\end{enumerate}

\subsection{Several blocks, one point}\label{subsubsection:YD>3-severalblocks-1pt}
Here we assume that $\theta = t+1$.
Let \begin{align*}
a_k&:=a_{\theta \, k} = q_{\theta \, k}^{-1}\eta_k (g_{\theta}), &
\ghost_k &= \ghost_{\theta \, k},
& k &\in\I_t.
\end{align*}

\begin{theorem}\label{thm:variosblocks-unpto}
	Let $V$ be a braided vector space as in  \eqref{eq:braiding-blocks-points} with $\theta = t+1$.
We keep all the assumptions in \S \ref{subsubsection:YD>3-severalblocks-pts}.	Then the following are equivalent:
	\begin{enumerate}
		\item\label{item:variosblocks-unpto-1} $\GK \NA (V) < \infty$.
		
		\item\label{item:variosblocks-unpto-2}  $\ghost_k\in\N$, $\widetilde{q}_{k \, \theta} = 1$ for all $k\in \I_t$, and $\epsilon_{\theta}^2=1$.
			\end{enumerate}	
			If this happens, then $\GK \toba (V)$ is computed by \eqref{eq:variosblocks-unpto-gkd}.
\end{theorem}
Assume that $\GK \NA (V) < \infty$. Then $\widetilde{q}_{k \, \theta}\in \{\pm 1\}$ by Lemma \ref{le:strong}.
Also $a_k\in\N_0$ and $\epsilon_{\theta} \in \G_2 \cup \G_3$ by \S \ref{sec:yd-dim3}.
Observe that $a_k\neq 0$; otherwise, the block $k$ would be disconnected to $\theta$ and to all other blocks.
Here is our first reduction:

\begin{lemma}\label{lemma;several-blocks-1pt-weak}
	If there is $k \in \I_t$ such that $\widetilde{q}_{k \, \theta} = -1$; then $\GK \toba(V) = \infty$.
\end{lemma}

\pf It is enough to deal with $k =1$.
By Theorem \ref{thm:point-block}, we may assume $q_{11} = q_{\theta \, \theta} = -1$.

Suppose $\widetilde{q}_{2 \, \theta} = -1$. Consider the filtration given by the natural order of $\Idd$.
Then $\gr V$ is of diagonal type, and the Dynkin subdiagram spanned by $\Idd_1 \cup \Idd_2 \cup \{\theta\}$
is of affine Cartan type $D_4^{(1)}$, hence Theorem \ref{thm:nichols-diagonal-finite-gkd} applies.

Suppose $\widetilde{q}_{2 \, \theta} = 1$. Let $K=\NA (V)^{\mathrm{co}\,\NA (V_2)}$ and $V'= \oplus_{j\neq 2} V_{j}$. Then
\begin{align*}
\NA(V) &\simeq K \# \NA (V_2);& K &\simeq \NA(K^1)& &\text{and}& K^1 &= \ad_c\NA (V_2) (V')
\in {}^{\NA (V_2)\# \ku \Gamma}_{\NA (V_2)\# \ku \Gamma}\mathcal{YD},
\end{align*}
with coaction \eqref{eq:coaction-K^1} and the adjoint action, cf. \S \ref{subsubsection:YD3-notation}. Let $z_0=x_{\theta}$, $z_1=(\ad_c x_{\futres}) x_{\theta}$.
Notice that $z_1\neq 0$ since $a_2\neq0$, and $x_1$, $x_{\fudos}$, $z_0$, $z_1$ are linearly independent.
Let $U$ be the subspace of $K^1$ spanned by $x_1$, $x_{\fudos}$, $z_0$, $z_1$. As
\begin{align*}
\delta(z_0)&= g_3\otimes z_0, & \delta(z_1) &= g_2g_3\otimes z_1+x_2g_3\otimes z_0
\end{align*}
and $(\ad_c x_2)x_i=0$, $i=1,\fudos$, $U$ is a braided vector subspace of type \emph{1 block and 2 points with mild interaction}, so
$\GK\cB(U)=\infty$ by Lemma \ref{lemma:superJordan-2-mild}. Thus $\GK \cB(V)=\infty$.
\epf

By the previous Lemma, we may assume that $\widetilde{q}_{k \, \theta} = 1$ for all $k\in\I_t$, i.e. all blocks have weak interaction with the point $\theta$. Once again, we consider a suitable decomposition of $V$, namely $V = W \oplus V_{\theta}$
where $W= \oplus_{1 \le j\leq t} V_{j}$. Then $K=\NA (V)^{\mathrm{co}\,\NA (W)}$ satisfies
\begin{align*}
\NA(V) &\simeq K \# \NA (W);& K &\simeq \NA(K^1)& &\text{and}& K^1 &= \ad_c\NA (W) (V_{\theta})
\in {}^{\NA (W)\# \ku \Gamma}_{\NA (W)\# \ku \Gamma}\mathcal{YD}.
\end{align*}
In order to analyze the braided vector space $K^1$, we introduce the elements
\begin{align}\label{eq:defn-sch-several-blocks}
\sch_{\bn} &:= (\ad_c x_{\fudos})^{n_1} \dots (\ad_c x_{t+\frac{1}{2}})^{n_{t}} x_{\theta}, & & \bn = (n_1,\dots,n_{t})\in\N^t_0.
\end{align}
Below, $(e_j)_{j\in \I_t}$ is the canonical basis of $\Z^t$.
Let $j\in\I_t$. As in \eqref{eq:x12}, we denote $x_{j+\frac{1}{2} \, j} = x_{j+\frac{1}{2}} x_j - \epsilon_j x_j x_{j+\frac{1}{2}}$. We define recursively a family $(\mu_n^{(j)})_{n\in\N_0}$ as in \eqref{eq:def-mu-n}
for $a=a_j$, $\epsilon=\epsilon_j$, and the following elements, see Remark \ref{rem:xk qcommutes with z_k}:
\begin{align}\label{eq:several-blocks-yjn}
y_j^{\langle n\rangle}:=\left\{ \begin{array}{ll} x_j x_{j+\frac{1}{2} \, j}^m, & n=2m+1 \text{ odd};
\\ x_{j+\frac{1}{2} \, j}^m, & n=2m \text{ even}. \end{array} \right.
\end{align}
We claim that $\{y_j^{\langle n\rangle} x_{j+\frac{1}{2}}^m: m,n\in\N_0\}$ is a basis of $\cB(V_j)$. This is clear when $\epsilon_j=-1$;
for $\epsilon_j=1$, it follows since $y_j^{\langle n\rangle}=\frac{(-1)^k}{2^k}x_j^n$, $k=\lfloor\frac{n}{2}\rfloor$.
This basis allows a unified notation in Proposition \ref{pr:poseidon} below.

\begin{lemma}\label{lemma:several-blocks-1pt-auxiliar}
\begin{enumerate}\renewcommand{\theenumi}{\alph{enumi}}\renewcommand{\labelenumi}{(\theenumi)}
	\item\label{item:several-blocks-1pt-auxiliar-1}
	Let $j\in\I_t$ and $\bn = (n_1,\dots,n_{t})\in\N^t_0$. Then
	\begin{align}\label{eq:several-blocks-1pt-auxiliar-11}
	\ad_c x_j(\sch_{\bn}) &= \ad_c x_{j+\frac{1}{2} \, j}(\sch_{\bn}) = 0, &&
	\\ \label{eq:several-blocks-1pt-auxiliar-12}
	\ad_c x_{j+\frac{1}{2}}&(\sch_{\bn}) = \prod_{i<j} q_{ji}^{n_i} \, \sch_{\bn + e_j}; &&
	\\\label{eq:several-blocks-1pt-auxiliar-2}
	g_j \cdot \sch_{\bn} &= q_{j\, \theta} \prod_{i=1}^t q_{ji}^{n_i} \, \sch_{\bn},
	&
	g_{\theta} \cdot \sch_{\bn} &= \epsilon_{\theta} \prod_{i=1}^t q_{\theta \, i}^{n_i} \sch_{\bn};
	\\	\label{eq:several-blocks-1pt-auxiliar-4}
	\partial_j(\sch_{\bn}) &= \partial_{j+\frac{1}{2}} (\sch_{\bn}) = 0;
	&
	\partial_{\theta} (\sch_{\bn}) &= \prod_{j=1}^{t} \mu_{n_j}^{(j)} \, y_1^{\langle n_1\rangle} \dots y_{t}^{\langle n_{t}\rangle}.
	\end{align}
	
\item\label{item:several-blocks-1pt-auxiliar-2}
Let $\mathcal A := \{\bn \in\N^t_0: 0 \le \bn \le \ba\}$, where $\ba = (2|a_1|, \dots 2|a_t|)$.
The elements $(\sch_{\bn})_{\bn \in \mathcal A}$ form a basis of $K^1$.
\end{enumerate}
	\end{lemma}

If $\bk, \bn \in \N_0^t$, then $\bk \leq \bn$ means $0\le k_j\le n_j$ for all $j\in \I_t$.

\pf
For \eqref{eq:several-blocks-1pt-auxiliar-11}, we use \eqref{eq:several-blocks-qcommute} and \eqref{eq:-1block+point}:
\begin{multline}
\ad_c x_j(\sch_{\bn}) = \prod_{i<j} q_{ji}^{n_i} \prod_{i>j} q_{ji}^{n_i(n_j+1)}
\\
\left( \ad_c x_{\fudos}^{n_1} \dots x_{j-\frac{1}{2}}^{n_{j-1}} x_{j+\frac{3}{2}}^{n_{j+1}} \dots x_{t+\frac{1}{2}}^{n_{t}} \right)
(\ad_c x_j) (\ad_c x_{j+\frac{1}{2}})^{n_j} x_{\theta} =0;
\end{multline}
the other equality follows analogously, and \eqref{eq:several-blocks-1pt-auxiliar-12} is a consequence of \eqref{eq:several-blocks-qcommute}.
For the first equality in \eqref{eq:several-blocks-1pt-auxiliar-2} we use
the first one of \eqref{eq:-1block+point} and the assumption
$a_{ij}=0$ for $i,j\in\I_t$. For the second equality in \eqref{eq:several-blocks-1pt-auxiliar-2} we use the second one of
\eqref{eq:-1block+point-bis} and \eqref{eq:several-blocks-1pt-auxiliar-11}.
For  \eqref{eq:several-blocks-1pt-auxiliar-4} we argue recursively on $N=\sum n_i$.
The first equality follows using Remark \ref{rem:der-often}.
The second for $N=0$ follows at once.
Let $\sch=\sch_{\bn}$, $j:=\min \{ i: n_i\neq 0 \}$, so we can write
\begin{align*}
\sch&=x_{j+\frac{1}{2}}\widetilde{\sch}- q \, \widetilde{\sch}x_{j+\frac{1}{2}},
& q &= q_{j \,\theta}\epsilon_j \prod_{i=j}^t q_{ji}^{n_i},
& \widetilde{\sch} &= \sch_{0,\dots,n_j-1,n_{j+1}, \dots,n_{t}}.
\end{align*}
Then we make a computation similar to that in Lemma \ref{lemma:derivations-zn}:
\begin{align*}
\partial_{\theta}(\sch)&=x_{j+\frac{1}{2}} \partial_{\theta}(\widetilde{\sch})- q q_{\theta j} \partial_{\theta}(\widetilde{\sch})(x_{j+\frac{1}{2}}+a_j \, x_j) \\
&= \prod_{i=j}^{t} \mu_{n_i}^{(i)} \left(x_{j+\frac{1}{2}}  y_j^{\langle n_j\rangle}- \epsilon_j^{n_j-1} y_j^{\langle n_j\rangle}  x_{j+\frac{1}{2}} \right)
y_{j+1}^{\langle n_{j+1}\rangle}\dots y_{t}^{\langle n_{t}\rangle} \\
&= \mu_{n_j+1}^{(j)} \prod_{i=j+1}^{t} \mu_{n_i}^{(i)}  y_j^{\langle n_j+1\rangle} y_{j+1}^{\langle n_{j+1}\rangle}\dots y_{t}^{\langle  n_{t}\rangle}.
\end{align*}
\eqref{item:several-blocks-1pt-auxiliar-2}:
By \eqref{item:several-blocks-1pt-auxiliar-1} the elements
$\sch_{\bn}$, $\bn\in\N_0^t$, generate $K_1$. By \eqref{eq:several-blocks-1pt-auxiliar-4},
$\sch_{\bn}\neq 0$ iff $\bn \in \cA$; and
$(\sch_{\bn})_{\bn \in \mathcal A}$ is linearly independent since
its image by $\partial_{\theta}$ is so.
\epf

\begin{lemma}\label{lemma:several-blocks-1pt-coaction}
	The coaction \eqref{eq:coaction-K^1} on $\sch_{\bn}$, $\bn\in \N^t_0$, is given by
	\begin{align}\label{eq:several-blocks-1pt-coaction}
	\delta(\sch_{\bn}) &=  \sum_{0\le \bk\le \bn} \nu_{\bk}^{\bn} \,
	y_1^{\langle n_1-k_1\rangle}\dots y_{t}^{\langle n_t-k_t\rangle} g_1^{k_1} \dots g_{t}^{k_t}g_{\theta} \otimes \sch_{k_1,\dots,k_{t}}
	\end{align}
	for some scalars $\nu_{\bk}^{\bn}$, $0\le \bk\le \bn$, where
	$\nu_{\bn}^{\bn}=1$.
\end{lemma}
\pf
We argue recursively on $N=\sum n_i$ as in the proof of \eqref{eq:several-blocks-1pt-auxiliar-4}.
The statement for $N=0$ follows at once. Let $\sch=\sch_{0,\dots,0, n_j+1, n_{j+1},\dots,n_{t}}$, so
\begin{align*}
\sch&=x_{j+\frac{1}{2}}\widetilde{\sch}- q \, \widetilde{\sch}x_{j+\frac{1}{2}},
& \mbox{for }q &= q_{j \,\theta}\prod_{i=j}^t q_{ji}^{n_i}, \  \widetilde{\sch} = \sch_{0,\dots,0,n_j,n_{j+1}, \dots,n_{t}}.
\end{align*}
We may assume $j=1$. Then we compute
\begin{multline*}
\delta(\sch) = (x_{\fudos}\otimes 1+g_1\otimes x_{\fudos}) \delta(\widetilde{\sch}) - q \, \delta(\widetilde{\sch})
(x_{\fudos}\otimes 1+g_1\otimes x_{\fudos}) \\
= \sum_{0\le \bk\le \bn} \nu_{\bk}^{\bn}  \,
\left( x_{\fudos} y_1^{\langle n_1-k_1\rangle}- y_1^{\langle n_1-k_1\rangle}(\epsilon_1 x_{\fudos}+(a_1+\epsilon_1k_1)x_1 \right) \times \\
\times y_2^{\langle n_2-k_2\rangle} \dots y_{t}^{\langle n_t-k_t)\rangle} g_1^{k_1} \dots g_{t}^{k_t}g_{\theta} \otimes \sch_{k_1,\dots,k_{t}} \\
+ \sum_{0\le \bk\le \bn} \nu_{\bk}^{\bn}  \prod_{i=1}^t q_{1i}^{n_i-k_i} \,
y_1^{\langle n_1-k_1\rangle}\dots y_{t}^{\langle n_t-k_t\rangle} g_1^{k_1+1} \dots g_{t}^{k_t}g_{\theta} \\
\otimes \left( x_{\fudos} \sch_{k_1,\dots,k_{t}} - q_{1 \,\theta}\prod_{i=1}^t q_{1i}^{k_i}\sch_{k_1,\dots,k_{t}} x_{\fudos} \right) \\
\end{multline*}
\begin{multline*}= \sum_{0\le \bk\le \bn} \nu_{\bk}^{\bn}  \,
y_1^{\langle n_1+1-k_1\rangle} y_2^{\langle n_2-k_2\rangle} \dots y_{t}^{\langle n_t-k_t\rangle} g_1^{k_1} \dots g_{t}^{k_t}g_{\theta}
\otimes \sch_{k_1,\dots,k_{t}} \\
+ \nu_{\bk}^{\bn}  \prod_{i=1}^t q_{1i}^{n_i-k_i} \,
y_1^{\langle n_1-k_1\rangle}\dots y_{t}^{\langle n_t-k_t\rangle} g_1^{k_1+1} \dots g_{t}^{k_t}g_{\theta}
\otimes \sch_{k_1+1,\dots,k_{t}}.
\end{multline*}
Thus we obtain the recursive formula for
\begin{align*}
\nu_{n_1+1,k_2\dots,k_t}^{n_1+1,\dots,n_t} &= \nu_{n_1,k_2\dots,k_t}^{n_1,\dots,n_t} \prod_{i=2}^t q_{1i}^{n_i-k_i} \qquad\qquad
\nu_{0,k_2,\dots,k_t}^{n_1+1,\dots,n_t} = \nu_{0,k_2,\dots,k_t}^{n_1,\dots,n_t},
\\
\nu_{k_1,\dots,k_t}^{n_1+1,\dots,n_t} &= \nu_{k_1,\dots,k_t}^{n_1,\dots,n_t} + \nu_{k_1-1,\dots,k_t}^{n_1,\dots,n_t} \epsilon_1^{n_1-k_1+1}\prod_{i=2}^t q_{1i}^{n_i-k_i}, \qquad 1\le k_i\le n_i.
\end{align*}
In particular, $\nu_{n_1+1,\dots,n_t}^{n_1+1,\dots,n_t}=\nu_{n_1,\dots,n_t}^{n_1,\dots,n_t}=1$.
\epf

\begin{lemma}\label{lemma:braiding-K-several-blocks-point}
	The braided vector space $K^1$ is of diagonal type with respect to the basis
$(\sch_{\bn})_{\bn \in \cA}$
with braiding matrix $(p_{\bm,\bn})_{\bm, \bn\in \cA}$, where
	\begin{align}\label{eq:several-blocks-braiding-matrix}
	p_{\bm,\bn} &=\epsilon_{\theta}
	\prod_{i,j=1}^{t} q_{ij}^{m_i n_j}q_{i \, \theta}^{m_i} q_{\theta \, j}^{n_j}.
	\end{align}
	Hence, the corresponding generalized  Dynkin diagram  has labels
	\begin{align*}
	p_{\bm,\bm}&= \epsilon_1^{m_1}\cdots \epsilon_t^{m_t}\epsilon_{\theta}, &
	p_{\bm,\bn}p_{\bn,\bm} &= \epsilon_{\theta}^2, & \bm &\neq \bn.
	\end{align*}
\end{lemma}
\pf
We compute
\begin{align*}
c(\sch_{\bm} \otimes \sch_{\bn}) &=
g_1^{m_1} \dots g_{t}^{m_t}g_{\theta} \cdot \sch_{\bn} \otimes \sch_{\bm} \\
&= \epsilon_{\theta} \prod_{i,j=1}^{t} q_{ij}^{m_i n_j}q_{i \, \theta}^{m_i} q_{\theta \, j}^{n_j} \sch_{\bn} \otimes \sch_{\bm}
\end{align*}
by Lemmas \ref{lemma:several-blocks-1pt-auxiliar} and \ref{lemma:several-blocks-1pt-coaction}. Now we compute
\begin{align*}
p_{\bm,\bm}&= \epsilon_{\theta}\prod_{i=1}^{t} q_{ii}^{m_i^2} (q_{i \, \theta}q_{\theta \,i})^{m_i} =
\epsilon_1^{m_1}\cdots \epsilon_t^{m_t}\epsilon_{\theta}, \\
p_{\bm,\bn}p_{\bn,\bm} &= \epsilon_{\theta}^2
\left( \prod_{i=1}^{t} q_{ii}^{2m_in_i} (q_{i \, \theta}q_{\theta \,i})^{m_i+n_i} \right)
\left( \prod_{i=1}^{t} (q_{ij}q_{ji})^{m_in_j} \right) =\epsilon_{\theta}^2,
\end{align*}
since the interaction is weak and $\epsilon_k^2=1$ if $k\in \I_t$.
\epf

\emph{End of the proof of Theorem \ref{thm:variosblocks-unpto}. \ }
If $\epsilon_{\theta}\in\G'_3$, then we may assume $\epsilon_k=1$ for all $k\in \I_t$.
In this case the Dynkin diagram of $K^1$ contains a subdiagram of affine Cartan type $A_2^{(1)}$,
so $\GK K=\infty$.

If $\epsilon_{\theta}^2=1$, then the braiding  of $K^1$ corresponds to a quantum linear space
where the labels of the vertices are given by $p_{\bm,\bm}\in\{\pm 1\}$, $\bm\in \cA$.
Thus
\begin{align}\label{eq:variosblocks-unpto-gkd}
\qquad \GK K= \vert \{\bm\in \cA:\, p_{\bm,\bm}=1\} \vert.  \qquad\qquad \qquad\qquad \quad\qed
\end{align}

\subsection{The Nichols algebras $\pos(\bq,\ghost)$}\label{subsection:YD>3-severalblocks-1pt-poseidon}

Here we present the Nichols algebras of finite $\GK$ from the previous Subsection.
Let $t\geq 2$ and $\theta = t +1$; we keep the notation from \S \ref{subsubsection:YD>3-severalblocks-pts}. We fix
\begin{itemize}
	\item $\bq\in\ku^{\theta \times \theta}$ such that $\epsilon_i:=q_{ii}=\pm 1$, $q_{ij}q_{ji}=1$ for all $i\neq j \in \I_{\theta}$.
	\item $\ghost=(\ghost_j)\in\N^{t}$, $(a_j)$ such that $\ghost_j = \begin{cases} -2a_j, &\epsilon_j = 1, \\
	a_j, &\epsilon_j = -1;\end{cases}$ cf. \eqref{eq:discrete-ghost}.
\end{itemize}

Let $\pos(\bq,\ghost)$ be the braided vector space with basis $(x_i)_{i\in\Idd}$ and braiding
\begin{align}\label{eq:braiding-several-blocks-1pt}
c(x_i\ot x_j) &= \left\{ \begin{array}{ll}
q_{ij} \, x_j\otimes x_i, & \lfloor i\rfloor \leq t, \, \lfloor i\rfloor \neq \lfloor j\rfloor, \\
\epsilon_j \, x_j\otimes x_i, & \lfloor i\rfloor =j \leq t, \\
(\epsilon_j \, x_j+x_{\lfloor j\rfloor})\otimes x_i, & \lfloor i\rfloor \leq t, \, j=\lfloor i\rfloor +\frac{1}{2}, \\
q_{\theta j} \, x_j\otimes x_{\theta}, & i = \theta, \, j\in\I_{\theta}, \\
q_{\theta j} \, (x_j+a_j x_{\lfloor j\rfloor})\otimes x_{\theta}, & i = \theta, \, j\notin\I_{\theta}.
\end{array} \right.
\end{align}
The braided vector spaces in Theorem \ref{thm:variosblocks-unpto} with finite $\GK$ have this shape.

\smallbreak
The following relations hold in $\cB(\pos(\bq,\ghost))$, see e.g. \S \ref{subsection:point-block-presentation}:
\begin{align}\label{eq:poseidon-defrels-Jordan}
x_{i+\frac{1}{2} }&x_i -x_ix_{i+\frac{1}{2} }+\frac{1}{2}x_i^2 =0, & &i\in\I_t, \, \epsilon_{i}=1;
\\ \label{eq:poseidon-defrels-super-Jordan}
x_i^2 = 0, \, x_{i+\frac{1}{2} }&x_{i+\frac{1}{2} \, i}- x_{i+\frac{1}{2} \, i}x_{i+\frac{1}{2} } - x_ix_{i+\frac{1}{2} \, i}=0, & i&\in\I_t, \, \epsilon_{i}=-1;
\\\label{eq:poseidon-defrels-blocks-commute}
x_ix_j &= q_{ij} \, x_jx_i, & \lfloor i\rfloor &\neq \lfloor j\rfloor  \in\I_t;
\\ \label{eq:poseidon-defrels-q-commute}
x_ix_{\theta} &= q_{i \theta} \, x_{\theta} x_i, & i& \in\I_t;
\\ \label{eq:poseidon-defrels-q-Serre}
(\ad_c &x_{i+\frac{1}{2}})^{1+|2a_i|}(x_{\theta})=0, & i& \in\I_t.
\end{align}

We denote by $\sch_{\bn}$, $\bn\in\N_0^t$, the element defined by \eqref{eq:defn-sch-several-blocks}
in $T(V)$ or some quotient of $T(V)$.
Let $\mathcal A := \{\bn \in\N^t_0: 0 \le \bn \le \ba= (2|a_1|, \dots 2|a_t|)\}$.
Let $p_{\bm,\bn}$ be defined by \eqref{eq:several-blocks-braiding-matrix}
and $\epsilon_{\bn}:= \epsilon_{\theta} \prod_{i=1}^t \epsilon_i^{n_i}$.

Let $W$ be a braided vector space of diagonal type with braiding matrix
$(p_{\bm,\bn})_{\bm,\bn \in \cA}$.
Then $\cB(W)$ is presented by generators $w_{\bn}$, $\bn\in \cA$ and relations
\begin{align*}
w_{\bm}w_{\bn} &= p_{\bm,\bn} \, w_{\bn}w_{\bm}  &
\bm & \neq \bn \in \cA;& w_{\bn}^2&=0, & \bn \in \cA, \epsilon_{\bn} &= -1.
\end{align*}
We order $\cA$ lexicographically.  Then  a basis of $\cB(W)$ is:
\begin{align*}
B_W=\Big\{ \prod_{\bn \in \cA} w_{\bn}^{b_{\bn}}: \,
0\le b_{\bn} < 2 \mbox{ if }\epsilon_{\bn}=-1 \Big\}.
\end{align*}

\begin{remark}\label{rem:relations-poseidon}
	By Lemma \ref{lemma:braiding-K-several-blocks-point}, $K^1$ is isomorphic to $W$ as braided vector spaces.
	Hence there is an isomorphism of braided Hopf algebras
	$\psi:\cB(W)\to K$ such that $\psi(w_{\bn})=\sch_{\bn}$,  $\bn \in \cA$; and the following identities hold in $\cB(\pos(\bq,\ghost))$:
	\begin{align}\label{eq:poseidon-rels-K-1}
	\sch_{\bm}\sch_{\bn} &= p_{\bm,\bn} \, \sch_{\bn}\sch_{\bm}  &
	\bm & \neq \bn \in \cA;
	\\ \label{eq:poseidon-rels-K-2}
	\sch_{\bn}^2&=0, & \bn \in \cA, \, \epsilon_{\bn} &= -1;
	\end{align}
	and the following set is a basis of $K$:
	\begin{align*}
	B_K= \Big\{ \prod_{\bn \in \cA} \sch_{\bn}^{b_{\bn}}: \,
	0\le b_{\bn}< 2 \mbox{ if }\epsilon_{\bn}=-1 \Big\}.
	\end{align*}
\end{remark}

Recall the definition of $y_j^{\langle n\rangle}$ in \eqref{eq:several-blocks-yjn} for $j\in\I_t$, $n\in\N_0$.


\begin{prop} \label{pr:poseidon}
	The algebra $\cB(\pos(\bq,\ghost))$ is presented by generators $x_i$, $i\in \Idd$, and relations
	\eqref{eq:poseidon-defrels-Jordan}, \eqref{eq:poseidon-defrels-super-Jordan},
	\eqref{eq:poseidon-defrels-blocks-commute}, \eqref{eq:poseidon-defrels-q-commute},
	\eqref{eq:poseidon-defrels-q-Serre}, \eqref{eq:poseidon-rels-K-1}, \eqref{eq:poseidon-rels-K-2}.
	A basis of $\cB(\pos(\bq,\ghost))$ and the $\GK$ are given by
	\begin{align*}
	B =\big\{ y_1^{\langle m_1\rangle} x_{\fudos}^{m_2} \dots y_t^{\langle m_{2t-1} \rangle} x_{t+\frac{1}{2}}^{m_{2t}} \prod_{\bn \in \cA} \sch_{\bn}^{b_{\bn}}: \, &0\le b_{\bn}< 2 \mbox{ if }\epsilon_{\bn}=-1, \\
& b_{\bn}, m_i \in \N_0 \mbox{ if }\epsilon_{\bn}=1, \, i\in \I_t \big\}, \\
	\GK \cB(\pos(\bq,\ghost))= 2t+ \vert \{\bm \in \cA:  p_{\bm,\bm}&=1\}\vert.
	\end{align*}
\end{prop}
\pf
We first prove that $B$ is a basis of $\cB := \cB(\pos(\bq,\ghost))$: since $\cB\simeq K\#\cB(V_1)$ and
$\cB(V_1)\simeq \otimes_{i=1}^{t} \cB(\cV_{g_i}(\chi_i, \eta_i))$, the claim follows from
Remark \ref{rem:relations lstr-a(10)1} and Propositions \ref{pr:-1block}, \ref{pr:1block}. The formula for the
$\GK \cB$ follows by computing the Hilbert series.

Relations \eqref{eq:poseidon-defrels-Jordan}, \eqref{eq:poseidon-defrels-super-Jordan}, \eqref{eq:poseidon-defrels-blocks-commute},
\eqref{eq:poseidon-defrels-q-commute}, \eqref{eq:poseidon-defrels-q-Serre}, \eqref{eq:poseidon-rels-K-1},
\eqref{eq:poseidon-rels-K-2} hold in $\cB$ as we have discussed already.
Hence the quotient $\cBt$ of $T(V)$ by these relations projects onto $\cB$.

We claim that the subspace $I$ spanned by $B$ is a left ideal of $\cBt$. Indeed, $x_j I\subseteq I$ and $x_{j+\frac{1}{2}} I\subseteq I$
for all $j\in\I_t$ by \eqref{eq:poseidon-defrels-Jordan}, \eqref{eq:poseidon-defrels-super-Jordan}, \eqref{eq:poseidon-defrels-blocks-commute}.
It remains to prove that $x_{\theta} I\subseteq I$.
By \eqref{eq:poseidon-defrels-blocks-commute}, \eqref{eq:poseidon-defrels-q-commute}, \eqref{eq:poseidon-defrels-q-Serre},
\begin{align*}
y_j^{\langle m_j\rangle}\sch_{\bn} &=  q_{j \theta}^{m_j} \prod_{i=1}^{t} q_{ji}^{m_jn_i} \, \sch_{\bn} y_j^{\langle m_j\rangle},
\\
x_{j+\frac{1}{2}}\sch_{\bn} &=  q_{j \theta}^{m_j} \prod_{i=1}^{t} q_{ji}^{m_jn_i} \, \sch_{\bn}x_{j+\frac{1}{2}}
+ \prod_{i<j} q_{ji}^{n_i} \, \sch_{n_1,\dots, n_j+1,\dots ,n_{t}},
\\
\sch_{n_1,\dots, |2a_j|+1,\dots ,n_{t}} &=0.
\end{align*}
As $\sch_{0,\dots,0}=x_{\theta}$, from the previous equations we have that
\begin{multline*}
x_{\theta}y_1^{\langle m_1\rangle} x_{\fudos}^{m_2} \dots y_t^{\langle m_{2t-1}\rangle} x_{t+\frac{1}{2}}^{m_{2t}} \sch_{\bn}^{b_{\bn}} \\
\in \sum_{t_i=0}^{\min \{ m_{2i}, |2a_i|\}} y_1^{\langle m_1\rangle} x_{\fudos}^{m_2-t_1} \dots y_t^{\langle m_{2t-1}\rangle}
x_{t+\frac{1}{2}}^{m_{2t}-t_t} \sch_{t_1,\dots,t_t} \sch_{\bn}^{b_{\bn}}.
\end{multline*}
Using this fact, \eqref{eq:poseidon-rels-K-1} and  \eqref{eq:poseidon-rels-K-2} we conclude that
$$ x_{\theta} y_1^{\langle m_1\rangle} x_{\fudos}^{m_2} \dots y_t^{\langle m_{2t-1}\rangle} x_{t+\frac{1}{2}}^{m_{2t}}
\sch_{\bn}^{b_{\bn}}\in I, $$
and the claim follows. Since $1\in I$, $\cBt$ is spanned by $B$. Thus $\cBt \simeq \cB$ since $B$ is a basis of $\cB$.
\epf

If $\epsilon_i = -1$ at least one $i$, then  $\toba(\pos(\bq,\ghost))$ is not a domain. Conversely,

\begin{prop} \label{pr:poseidon-domain}If  $\bq_{ii} = \epsilon_i =1$ for all $i\in \I_{\theta}$, then
$\toba(\pos(\bq,\ghost))$ is a domain.
\end{prop}

\pf Similar to the proof of Proposition \ref{pr:lstr-11disc-domain}:
Consider the algebra $\Bg$ generated by $X_i$, $i\in \Idd$ and $(\Sch_{\bn})_{\bn \in \cA}$
with defining relations
\eqref{eq:poseidon-defrels-Jordan},
\eqref{eq:poseidon-defrels-blocks-commute}, \eqref{eq:poseidon-defrels-q-commute},
\eqref{eq:poseidon-defrels-q-Serre}, \eqref{eq:poseidon-rels-K-1} and \eqref{eq:several-blocks-1pt-auxiliar-12}  (with $X_i$ in the place of $x_i$ and $\Sch_{\bn}$ in the place of $\sch_{\bn}$).
Clearly the assignments $X_i \leftrightarrow x_i$ and $\Sch_{\bn} \leftrightarrow \sch_{\bn}$ provide an algebra isomorphism
$\Bg \simeq \cB(\pos(\bq,\ghost))$.
Consider the filtration of $\Bg$ where $\deg X_i =0$, $i \in \I_t$, and all the other defining generators having degree 1. We claim that
$\gr \Bg$ is presented by $(\overline{X}_i)_{i\in \Idd}$, $(\overline{\Sch}_{\bn})_{\bn \in \cA}$ with the relations
\begin{align}
	\label{eq:grBg-11}
	\overline{X}_{i+\frac{1}{2} }\overline{X}_i &=\overline{X}_i\overline{X}_{i+\frac{1}{2} }, & i&\in\I_t, \\
	\label{eq:grBg-12}
	\overline{X}_i\overline{X}_j &= q_{ij} \, \overline{X}_j\overline{X}_i, & \lfloor i\rfloor &\neq \lfloor j\rfloor  \in\I_t;
	\\
	\label{eq:grBg-15}
	\overline{X}_{j}\overline{\Sch}_{\bn} &= q_{k\, \theta} \prod_{i=1}^t q_{ki}^{n_i} \,\overline{\Sch}_{\bn}  \overline{X}_{j}, & j& \in\Idd_k, \quad k \in\I_t;
	\\
	\label{eq:grBg-16}
		\overline{\Sch}_{\bm}\overline{\Sch}_{\bn} &= p_{\bm,\bn} \, \overline{\Sch}_{\bn}\overline{\Sch}_{\bm}  &
		\bm & \neq \bn \in \cA.
		\end{align}
Indeed, the algebra $\widetilde \Bg$ with the mentioned presentation admits a surjective algebra homomorphism onto $\gr \Bg$.
But $\widetilde \Bg$ is a quantum polynomial ring, hence it has a PBW-basis analogous to $B$ above and the claim follows.
Now $\widetilde \Bg$ is a domain, hence so is $\Bg \simeq \cB(\pos(\bq,\ghost))$.
\epf

\subsection{Several blocks, several points}\label{subsubsection:YD>3-severalblocks-severalpoints}
	Let $V$ be a braided vector space as in  \eqref{eq:braiding-blocks-points} with $\theta = t+N$, $N \ge 2$; and $t \ge 2$ as always in this Section. Let $\X$ be the set of connected components of $V_{\diag}$.
	We keep all the assumptions in \S \ref{subsubsection:YD>3-severalblocks-pts}.
Here is our first reduction.

\begin{lemma}\label{lemma:several-blocks-connected-components}
Assume that there are $k \neq \ell\in\I_t$ (two  blocks), $J\in \X$ with $\vert J \vert \ge 2$ and  $i \in J$
such that $c_{ik}c_{ki} \neq \id$, $c_{i\ell}c_{\ell i} \neq \id$.
Then $\GK\cB(V)=\infty$.
\end{lemma}

\pf By assumption, there is $j \in \I_{t+1, \theta}$, $i\neq j$ with  $q_{ij}q_{ji}\neq 1$.
By Theorem \ref{thm:pm1bp}, we may assume that $a_{ik},a_{i\ell}\neq 0$.
It is enough to consider the case $t = N =2$: we may set $k=1$, $\ell=2$, $i=3$, $j=4$. Let $q=q_{34}q_{43}\neq 1$.
By Lemma \ref{lemma;several-blocks-1pt-weak}, $q_{13}q_{31}=q_{23}q_{32}=1$; that is, the interaction is weak. We may
assume $q_{14}q_{41}=q_{24}q_{42}=1$, with $\ghost_{41}, \ghost_{42}\in\N_0$.

We fix the decomposition $V = W \oplus W'$, where $W= V_{1}\oplus V_{2}$, $W'= V_{3}\oplus V_{4}$.
Then $K=\NA (V)^{\mathrm{co}\,\NA (W)}$ satisfies
\begin{align*}
\NA(V) &\simeq K \# \NA (W);& K &\simeq \NA(K^1)& &\text{and}& K^1 &= \ad_c\NA (W) (W')
\in {}^{\NA (W)\# \ku \Gamma}_{\NA (W)\# \ku \Gamma}\mathcal{YD}.
\end{align*}
By Lemma \ref{lemma:braiding-K-several-blocks-point}, $\sch_{0,0}:=x_3$, $\sch_{1,0}:=(\ad_c x_{\fudos})x_3$, $\sch_{0,1}:=(\ad_c x_{\futres})x_3$ and $\sch_{1,1}:=(\ad_c x_{\fudos})(\ad_c x_{\futres})x_3$ are linearly independent elements of $K^1$; the coaction and the braiding of these elements
are given by \eqref{eq:several-blocks-1pt-coaction} and \eqref{eq:several-blocks-braiding-matrix}.
Let $U$ be the subspace of $K^1$ spanned by $\sch_{0,0}$, $\sch_{1,0}$, $\sch_{0,1}$, $\sch_{1,1}$ and $x_4$.
As $y_i^{\langle n\rangle}\cdot x_4=0$ for $i=1,2$ and all $n\in\N_0$, and $\delta(x_4)=g_4\otimes x_4$, we have that
\begin{align*}
c(\sch_{m,n}\otimes x_4) &= g_1^m g_2^n g_3\cdot x_4 \otimes \sch_{m,n} = q_{14}^m q_{24}^n q_{34} \, x_4 \otimes \sch_{m,n}, \\
c(x_4\otimes \sch_{m,n}) &= g_4\cdot \sch_{m,n} \otimes x_4 = q_{41}^m q_{42}^n q_{43} \, \sch_{m,n} \otimes x_4,
\end{align*}
so $U$ is a braided vector subspace of $K^1$, whose braiding is of diagonal type with Dynkin diagram
$$ 
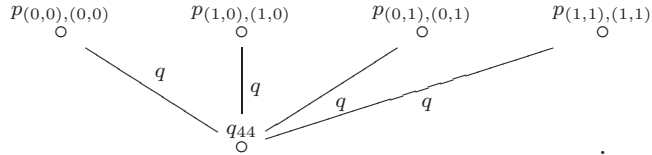
\xymatrix{ \overset{p_{(0,0),(0,0)}}{\circ} \ar  @{-}[rd]^{q}  & \overset{p_{(1,0),(1,0)}}{\circ} \ar  @{-}[d]^{q} &
	\overset{p_{(0,1),(0,1)}}{\circ} \ar  @{-}[ld]^{q} & \overset{p_{(1,1),(1,1)}}{\circ} \ar  @{-}[lld]^{q}
	\\ & \overset{q_{44}}{\circ} & & .} $$
Then $\GK\cB(U)=\infty$ either by Lemma \ref{lemma:points-trivial-braiding} if some $p_{(i,j),(i,j)}=1$, or else by Hypothesis \ref{hyp:nichols-diagonal-finite-gkd}; thus $\GK\cB(V)=\infty$.
\epf

\begin{lemma}\label{lemma:several-blocks-connected-components-2}
	Assume that there are $k \neq \ell\in\I_t$ (two  blocks), $J\in \X$ with $\vert J \vert \ge 2$ and  $i \neq j \in J$
	such that $c_{ik}c_{ki} \neq \id$, $c_{j\ell}c_{\ell j} \neq \id$.
	Then $\GK\cB(V)=\infty$.
\end{lemma}
\pf
For simplicity, we may assume that $t = 2$; that $V_{\diag}$ is connected, i.e.
$J = \I_{3,\theta}$, $\theta \geq 4$; and that
$k = 1$, $\ell = 2$, $i = 3$, $j = \theta$. Finally we assume $c_{1\theta}c_{\theta1}=\id$, $c_{23}c_{32}=\id$, c.f. Lemma \ref{lemma:several-blocks-connected-components}.

\smallbreak
\emph{Step 1}. $q_{13}q_{31}=q_{2 \theta}q_{\theta 2}=-1$.

As $V_{1}\oplus\big( \oplus_{3\le h \le\theta} V_{h} \big)$ is a braided vector subspace of $V$,
we may assume $\theta=4$,
$q_{33}=q_{34}q_{43}=q_{44}=-1$, cf. Theorem \ref{thm:points-block-eps-1}.
We consider the flag of braided subspaces: $0 =\cV_0 \subsetneq \cV_1  \subsetneq \cV_2 = V$, where $\cV_1$ is spanned by $(x_i)_{i\in \I_4}$.
Let  $\Bdiag := \gr \cB (V)$, a pre-Nichols algebra of  $\cV^{\textrm{diag}}$, see \S \ref{subsection:filtr-nichols}.
The Dynkin diagram of $\cV^{\textrm{diag}}$ has vertices $\{1, \fudos, 2, \futres ,3, 4 \}$
and edges as follows:
\begin{align*}
\xymatrix{ & \overset{-1}{\underset{1}{\circ}} \ar  @{-}[d]^{-1} & \overset{-1}{\underset{4}{\circ}} \ar  @{-}[d]^{-1} & \\
	\overset{-1}{\underset{\fudos}{\circ}}\ar  @{-}[r]_{-1}  & \overset{-1}{\underset{2}{\circ}} \ar  @{-}[r]_{-1}  & \overset{-1}{\underset{3}{\circ}}
\ar  @{-}[r]_{-1}  & \overset{-1}{\underset{\futres}{\circ}}.}
\end{align*}
Hence $\cV^{\textrm{diag}}$ is of Cartan type $D_5^{(1)}$, so $\GK \cB(\cV^{\textrm{diag}})=\infty$ by
Theorem \ref{thm:nichols-diagonal-finite-gkd}, and consequently $\GK\cB(V)=\infty$.

\medbreak
\emph{Step 2}.
$q_{13}q_{31}=-1$ and $q_{2 \theta}q_{\theta 2}=1$.

As $V_{1}\oplus\big( \oplus_{3\le h \le\theta} V_{h}\big)$ is a braided vector subspace of $V$, again we may assume $\theta=4$,
$q_{33}=q_{34}q_{43}=q_{44}=-1$. We fix the decomposition $V = W\oplus W'$, where $W= V_{2}$, $W'= V_{1}\oplus V_{3}\oplus V_{4}$.
Then $K=\NA (V)^{\mathrm{co}\,\NA (W)}$ satisfies
\begin{align*}
\NA(V) &\simeq K \# \NA (W);& K &\simeq \NA(K^1)& &\text{and}& K^1 &= \ad_c\NA (W) (W')
\in {}^{\NA (W)\# \ku \Gamma}_{\NA (W)\# \ku \Gamma}\mathcal{YD}.
\end{align*}
Set $x_5=(\ad_c x_{\futres})x_4\neq 0$. Then the subspace $U$ spanned by $V_1=\langle x_1,x_{\fudos}\rangle$, $x_3$, $x_4$, $x_5$ is a braided
vector subspace of $K^1$, with a decomposition \emph{one block plus three points}, where the block is $V_1$, with mild interaction with $x_3$,
and $x_3$, $x_4$, $x_5$ is a connected component of the diagonal part, with diagram $\xymatrix{\overset{-1}{\underset{4}{\circ}} \ar  @{-}[r]^{-1} & \overset{-1}{\underset{3}{\circ}} \ar  @{-}[r]^{-1} & \overset{-1}{\underset{5}{\circ}}}$.
Hence $\GK\cB(U)=\infty$ by Theorem \ref{thm:points-block-eps-1}, so $\GK\cB(V)=\infty$.

\medbreak

\emph{Step 3}.  $q_{13}q_{31}=q_{2 \theta}q_{\theta 2}=1$, $a_{13} \neq 0$, $a_{2\theta}\neq 0$.

We fix the decomposition $V = W \oplus W'$, where $W= V_{1}\oplus V_{2}$, $W'=  V_{\diag}$.
Then $K=\NA (V)^{\mathrm{co}\,\NA (W)}$ satisfies
\begin{align*}
\NA(V) &\simeq K \# \NA (W);& K &\simeq \NA(K^1)& &\text{and}& K^1 &= \ad_c\NA (W) (W')
\in {}^{\NA (W)\# \ku \Gamma}_{\NA (W)\# \ku \Gamma}\mathcal{YD}.
\end{align*}
The subspace $U$ of $K^1$ spanned by $z=(\ad_c x_{\fudos}) x_3$, $z'=(\ad_c x_{\futres})x_{\theta}$ and
$x_h$, $h\in \I_{3, \theta}$, is a braided vector subspace of diagonal type of dimension $N$
whose Dynkin diagram has either two ramifications for $N>2$, or is a 4-cycle for $N=2$.
Then $\GK\cB(U)=\infty$ by Hypothesis \ref{hyp:nichols-diagonal-finite-gkd}, and consequently $\GK\cB(V)=\infty$.
\epf

We are now ready to complete the proof of the main result of this monograph.
Let $V$ be a braided vector space as in\eqref{eq:braiding-blocks-points}; since $2 \le t$, it is not of diagonal type
(the case $t=1$ is settled in \S \ref{sec:yd-dim>3}). Let $\D$ be the flourished graph   of $V$.

\begin{theorem}\label{theorem:final}
	If  $\D$ is admissible, then $\GK \NA(V) < \infty$.
\end{theorem}

\pf Since $\D$ is admissible, all assumptions in \S \ref{subsubsection:YD>3-severalblocks-pts} are valid. Let
\begin{align}\label{eq:final-blocks-points}
\cV_1 &= V_1 \oplus  \dots \oplus V_t,&  \cV_2 &=  V_{t+1} \oplus \dots \oplus V_{\theta} = V_{\diag}.
\end{align}
By \ref{item:B}, the blocks braided commute, hence
\begin{align*}
\NA(\cV_1) &= \NA(V_1) \underline{\otimes} \dots \underline{\otimes} \NA(V_t),
\end{align*}
where $\underline{\otimes}$ means the braided tensor product of algebras, cf. \cite{Grana}.
In particular, one gets a PBW-basis of $\NA(\cV_1)$ by ordered juxtaposition of the bases of
$\NA(V_1)$, \dots, $\NA(V_t)$; hence $\GK \NA(\cV_1) = 2t$.

\medbreak
We consider  $K=\NA (V)^{\mathrm{co}\,\NA (\cV_1)}$. Then
\begin{align*}
\NA(V) &\simeq K \# \NA (\cV_1);& K &\simeq \NA(K^1)& &\text{and}& K^1 &= \ad_c\NA (\cV_1) (\cV_2)
\in {}^{\NA (\cV_1)\# \ku \Gamma}_{\NA (\cV_1)\# \ku \Gamma}\mathcal{YD}.
\end{align*}

As usual let $\X$ be the set of connected components
of the Dynkin diagram of $\cV_2 = V_{\diag}$. If $J \in \X$, then $V_J := \oplus_{j\in J} V_j$ and $K^1_J = \ad_c\NA (\cV_1) (V_J)$.

\begin{example}\label{example:admissible-mild}
	If $\D$ contains a mild interaction, i.e. an edge $\xymatrix{\ar  @{-}[r]^{(-1, 1)}  & }$, then it is either $\cyc_1$, that is $\xymatrix{\boxminus \ar  @{-}[r]^{(-1, 1)}  &\overset{-1}{\bullet}}$ or else $\cyc_2$, that is $\xymatrix{\boxminus \ar  @{-}[r]^{(-1, 1)}  &\overset{-1}{\bullet} \ar  @{-}[r]^{-1}  & \overset{-1}{\circ}}$.
	This follows from \ref{item:local} and \ref{item:block}. Thus, the Theorem follows under this assumption.
\end{example}

From now on, we assume that the interactions between a block $V_i$, $i\in \I_t$, and a component $V_J$, $J \in \X$, are all weak.
We claim:

\begin{enumerate}[leftmargin=*,label=\rm{(\alph*)}]
	\item\label{item:final-a} $K^1 = \oplus_{J \in \X} K^1_J$.
	
	\item\label{item:final-b} $K^1_J$ is a braided subspace of diagonal type of $K^1$.
	
	\item\label{item:final-c}	$c(K^1_J \otimes K^1_L) = K^1_L \otimes K^1_J$ and $c^2_{K^1_J \otimes K^1_L} = \id_{K^1_J \otimes K^1_L}$ for all $J \neq L\in \X$.
\end{enumerate}

\ref{item:final-a}: Clearly  $K^1 = \sum_{J \in \X} K^1_J$. Consider the $\zt$-grading of $V$ given by $\deg V_j = e_j$, where $(e_j)_{j\in \I_\theta}$ is the canonical basis of $\zt$; it extends to gradings in $T(V)$ and $\NA(V)$. Then $ K^1_J$ is a graded subspace whose homogeneous components have degree modulo $\Z^t$ in $\langle e_j: j \in J\rangle$. Hence the sum is direct.

\smallbreak
\ref{item:final-b}: Suppose first that $J = \{j\}$ has one element. Let $h \in \I_t$.
If $\ghost_{hj} = 0$, i.e. $j$ and $h$ are disconnected, then
$\ad_c (V_h)(V_j) =0$. Thus $K^1_J = \ad_c\NA \left(\oplus_{h\in \I_t: \ghost_{hj} \neq 0} V_h \right) (V_j)$
and this is of diagonal type by Lemma \ref{lemma:braiding-K-several-blocks-point}.
Similarly, if $\vert J \vert >1$, then there is exactly one $i \in \I_t$ by Definition \ref{def:flourished-graph-admissible} \ref{item:concom}, hence $K^1_J = \ad_c\NA \left(V_i \right) (V_j)$
is of diagonal type by Lemma \ref{lemma:braiding-K-weak-block-points}.

\smallbreak
\ref{item:final-c}: Let $J \neq L\in \X$. We analize three possibilities:

$(\alpha)$: \emph{$J$ is connected only to the block $i$, $L$ is connected only to the block $h$.}

Then $K_J^1$, respectively $K_L^1$, is of diagonal type with respect to the basis
$(z_{j,n})_{j\in J, 0\le n\le \vert 2a_{ij}\vert}$, respectively $(z_{\ell,m})_{\ell\in L, 0\le m\le \vert 2a_{h\ell}\vert}$, see Lemma \ref{lemma:braiding-K-weak-block-points}. Now, if $\epsilon_i  = 1$, then the coaction is given by \eqref{eq:coact-zjn}
is given by
\begin{align*}
\delta (z_{j,n}) &= \sum _{k=0}^n \nu_{k,n}\, x_i^{n-k}g_i^{k} g_j \otimes z_{j,k},&
\delta (z_{\ell,m}) &= \sum _{t=0}^m \nu_{t,m}\, x_h^{m - t}g_h^{t} g_\ell \otimes z_{\ell,t}.
\end{align*}
If $i=h$, then
$c(z_{j,n} \otimes z_{\ell,m}) =  q_{ji}^m q_{i\ell}^n  q_{j\ell} \, z_{\ell,m} \otimes z_{j,n},
$ as shown in the proof of Lemma \ref{lemma:braiding-K-weak-block-points}.
Thus
\begin{align*}
c^2 (z_{j,n} \otimes z_{\ell,m}) &= q_{ji}^m q_{i\ell}^n  q_{j\ell} q_{\ell i}^n q_{ij}^m  q_{\ell j}
\, z_{j,n} \otimes z_{\ell,m} = z_{j,n} \otimes z_{\ell,m}.
\end{align*}
Indeed, $q_{ji}q_{ij} =1 = q_{i\ell}   q_{\ell i}$ because the interactions are weak, while $ q_{j\ell}q_{\ell j} = 1$ because $j$ and $\ell$ live in
different connected components.
If $i\neq h$, then $\ad_c x_i (z_{\ell,m}) = \ad_c x_i (\ad_c x_{h + \hspace{-1pt}\frac{1}{2}})^m (x_{\ell})  = 0$, hence
\begin{align*}
&& c (z_{j,n} \otimes z_{\ell,m}) &= g_i^{n} g_j \cdot z_{\ell,m} \otimes z_{j,n} =  q_{jh}^m q_{i\ell}^n  q_{j\ell} \, z_{\ell,m} \otimes z_{j,n} \\
&\implies&
c^2 (z_{j,n} \otimes z_{\ell,m}) &= q_{jh}^m q_{i\ell}^n  q_{j\ell} q_{\ell i}^n q_{h j}^m  q_{\ell j}
\, z_{j,n} \otimes z_{\ell,m} = z_{j,n} \otimes z_{\ell,m}.
\end{align*}

If $\epsilon_i  = -1$ , then the argument is similar, using \eqref{eq:coact-zjn-even}, \eqref{eq:coact-zjn-odd} and \eqref{eq:-1block+point}.

\smallbreak
$(\beta)$: \emph{$J$ is connected is connected only to the block $i$ and $L =\{\ell\}$ is connected to the blocks $h_1, \dots, h_s \in \I_t$.}

Here $K_J^1$ is as before, and $K_L^1$, is of diagonal type with respect to the basis  $(\sch_{\bm})_{\bm \in \cA}$
see Lemma \ref{lemma:braiding-K-several-blocks-point}. Applying this Lemma, $\ell$ plays the role of $\theta$ and the index set is
$\{h_1, \dots, h_s\}$ instead of $\I_t$, e.g.
$\mathcal A = \{\bn \in\N^s_0: 0 \le \bn \le \ba\}$.
Now
\begin{align*}
\ad_c x_i(\sch_{\bm}) &=0 \begin{cases}
\text{by \eqref {eq:several-blocks-1pt-auxiliar-11} if } i\in \{h_1, \dots, h_s\}, \\
\text{because $\ell$ and $i$ are disconnected, if } i\notin \{h_1, \dots, h_s\}.
\end{cases}
\end{align*}
Hence
\begin{align*}
c(z_{j,n} \otimes \sch_{\bm}) &= g_i^{n} g_j \cdot \sch_{\bm} \otimes z_{j,n} = q_{j\ell}  q_{i\ell}^n \prod_{1\le a \le s} (q_{j h_a} q_{i h_a}^n)^{m_a}
\, \sch_{\bm} \otimes  z_{j,n}.
\end{align*}
On the other hand, by \eqref{eq:several-blocks-1pt-coaction}, taking into account \eqref{eq:-1block+point} when $i\in \{h_1, \dots, h_s\}$
or that all interactions are weak, concluding that
\begin{align*}
c(\sch_{\bm} \otimes  z_{j,n}) &= \prod_{1\le a \le s} g_{h_a}^{m_a} g_{\ell}\cdot z_{j,n} \otimes \sch_{\bm} = q_{\ell j}  q_{\ell i}^n 
\prod_{1\le a \le s} (q_{ h_a j} q_{h_a i}^n)^{m_a}
\, z_{j,n} \otimes \sch_{\bm}.
\end{align*}
Thus
\begin{align*}
c(z_{j,n} \otimes \sch_{\bm}) &= q_{j\ell}q_{\ell j}  q_{i\ell}^n q_{\ell i}^n    \prod_{1\le a \le s} (q_{h_a j}q_{j h_a} q_{i h_a}^n q_{h_a i}^n)^{m_a}
\, z_{j,n} \otimes \sch_{\bm} = z_{j,n} \otimes \sch_{\bm}.
\end{align*}
Indeed, $q_{j\ell}q_{\ell j}=1$ because $j$ and $\ell$ live in different connected components, 
$q_{i\ell}^n q_{\ell i}=1=q_{h_a j}q_{j h_a}$ and $q_{i h_a}q_{h_a i}=1$ if $h_a\neq i$ because the interactions are weak,
while $q_{ih_a}q_{h_a i}= \epsilon_i^2=1$ if $h_a=i$.

\smallbreak
$(\gamma)$: \emph{$J=\{j\}$ is connected to to the blocks $i_1, \dots, i_r \in \I_t$ and $L =\{\ell\}$ is connected to to the blocks $h_1, \dots, h_s \in \I_t$.}

Here $K_J^1$ and $K_L^1$ are of diagonal type with respect to bases $(\sch_{\bm})_{\bm \in \cA_J}$, $(\sch_{\bn})_{\bn \in \cA_K}$,
see Lemma \ref{lemma:braiding-K-several-blocks-point}. The proof of this case follows as the previous one 
by \eqref{eq:several-blocks-1pt-coaction}, taking into account \eqref{eq:several-blocks-1pt-auxiliar-11} when some $i_b\in \{h_1, \dots, h_s\}$,
or that all interactions are weak.

\smallbreak
Therefore $K^1$ is of diagonal type and
\begin{align*}
K = \NA(K^1) &= \NA(K^1_{J_1}) \underline{\otimes} \dots \underline{\otimes} \NA(K^1_{J_r}),
\end{align*}
where $J_1, \dots J_r$ is a numbering of $\X$.
In particular, one gets a PBW-basis of $K$ by ordered juxtaposition of the bases of
$\NA(K^1_{J_1})$, \dots, $\NA(K^1_{J_r})$; hence $\GK K = \sum_{J\in \X} \GK \NA (K^1_J)$
and
\begin{align*}
\GK \NA(V) &=  2t + \sum_{J\in \X} \GK \NA (K^1_J).
\end{align*}
\epf

\begin{theorem}\label{theorem:final-domain}
	If  $\D$ is admissible, then $\NA(V)$ is a domain if and only if all blocks are Jordan, i.e. $\boxplus$, and all $J\in \X$ are points with label 1, i.e. $\overset{1}{\bullet}$.
\end{theorem}

\pf
If $\NA(V)$ is a domain, then clearly all blocks are Jordan and all $J\in \X$ are points with label 1. Conversely,
if this holds, then the proof goes as the proof of
Proposition \ref{pr:poseidon-domain}, using the preceding arguments: consider an appropriate filtration of $\NA(V)$ and show that $\gr \NA(V)$ is a domain, having a PBW-basis with all elements of infinite height.
\epf

\section{Appendix}\label{section:bvs-abgps}

\subsection{Nichols algebras over abelian groups}\label{subsection:ydz}

\subsubsection{The context}\label{subsect:pale-context}
Let $\Gamma$ be an abelian group and $V = \oplus_{g\in \Gamma} V_g \in \ydG$, $\dim V < \infty$.
The following discussion aims to describe the general form of $V$.
Let $V_g^{(\lambda)}$ be the $(T-\lambda)$-primary component with respect to the action of $g$
on the  homogeneous component $V_g$, and
let $V_g^{\lambda} = \ker (g - \lambda \id)_{\vert V_g} \subseteq V_g^{(\lambda)}$. Then
$V_g = \oplus_{\lambda \in \ku^{\times}} V_g^{(\lambda)}$,
and
\begin{align}\label{eq:braiding-ydz}
c(V_g^{(\lambda)} \ot V_h^{(\mu)}) &= V_h^{(\mu)} \ot V_g^{(\lambda)},& g, h &\in \Gamma, \lambda,\mu  \in \ku^{\times}.
\end{align}

\begin{lemma}\label{lemma:primary-1}  Assume that $\GK \NA(V_g)< \infty$. Then
	\begin{enumerate}[leftmargin=*,label=\rm{(\alph*)}]
		\item\label{item:primary-1} If $\lambda \in \ku^{\times}$, $\lambda \notin \G_2 \cup \G_3$, then
		$V_g^{\lambda} =  V_g^{(\lambda)}$ has dimension $\leq 1$.
		
		\smallbreak
		\item\label{item:primary-2} If $\lambda \in \G'_3$, then
		$V_g^{\lambda} =  V_g^{(\lambda)}$ has dimension $\leq 2$.
		
		\smallbreak
		\item\label{item:primary-3} If $V_g^{1} \neq 0$, then either
		$V_g  =  V_g^{1}$ (i.e. $g$ acts trivially on $V_g$) or else $V_g$ has dimension 2 and $g$ acts by a Jordan block.
		
		\smallbreak
		\item\label{item:primary-4} If $V_g^{-1} \neq 0$, then either
		$V_g^{(-1)}  =  V_g^{-1}$ or else $V_g^{(-1)}$ has dimension 2 and $g$ acts by a Jordan block.
			\end{enumerate}

\end{lemma}

\pf By Theorem \ref{theorem:blocks}, $ V_g^{(\lambda)}$ could contain a block only if $\lambda^2 = 1$.
If $v_1 \in V_g^{\lambda}$ and $v_2 \in V_g^{\mu}$ are linearly independent, then the span of $v_1, v_2$ is a braided subspace of diagonal type with braiding matrix $\begin{pmatrix}
\lambda & \mu \\ \lambda & \mu
\end{pmatrix}$. Thus if $\lambda = \mu$, this is of Cartan type $A_1^{(1)}$ unless $\lambda \in \G_2 \cup \G_3$.
This shows \ref{item:primary-1}, and \ref{item:primary-2} follows similarly.
If $V_g^{(1)} \neq 0$, then $V_g  =  V_g^{(1)}$ by Lemma \ref{lemma:points-trivial-braiding} (otherwise there are $v_1$, $v_2$ as above with $\mu =1 \neq \lambda$). If $W_1$, $W_2$ are two different blocks inside $V_g$, then $c^2_{W_1 \ot W_2} \neq \id$, thus $\GK \NA(V_g) = \infty$ by the results in \S \ref{section:YD>3-2blocks}.
If $\lambda^2 = 1$ and $V_g^{(\lambda)} \supseteq W \oplus \ku x$, where $W$ has dimension 2 and $g$ acts on it by a Jordan block, then the ghost of $W \oplus \ku x$ is negative, hence $\GK \NA(V_g^{(\lambda)}) = \infty$.
So, \ref{item:primary-3} and \ref{item:primary-4} are proved.
\epf

Let us say that the braiding \eqref{eq:braiding-ydz} is \emph{pale} when it is not of the form described in \S \ref{subsubsec:intro-class}.
Therefore, if the braiding is pale, then either $\lambda  \in \G_2 \cup \G_3$ and	 $V_g^{\lambda} =  V_g^{(\lambda)}$ has dimension $>1$,
or vice versa for $\mu$, or both.
In the next \S \ we consider the smallest possible case, namely a pale block and a point, cf. \S \ref{subsubsection:YD3-3points-notdiag}.
Clearly any braided vector space with pale braiding contains \emph{a pale block and a point}, so this is a necessary first step towards the general case.
We classify all braidings of type \emph{a pale block and a point} in Theorem \ref{th:paleblock-point-resumen}. As an outcome, the pale braidings to be considered next are of the form
\begin{align}\label{eq:pale-braiding-open}
V &= V_g^{-1} \oplus V_h^{(\mu)}, & \dim V_g^{-1} \ge 3.
\end{align}

\begin{question}\label{question:paleblock}
Determine when $\GK \NA(V) < \infty$, if $V$ has the shape \eqref{eq:pale-braiding-open}.
\end{question}

\subsubsection{A pale block and a point}\label{subsubsec:pale-block}

Let
$V$ be a braided vector space of dimension 3
with braiding given in the  basis $(x_i)_{i\in\I_3}$ by
\begin{align}\label{eq:braiding-paleblock-point-bis}
(c(x_i \otimes x_j))_{i,j\in \I_3} &=
\begin{pmatrix}
\epsilon x_1 \otimes x_1&  \epsilon x_2  \otimes x_1& q_{12} x_3  \otimes x_1
\\
\epsilon x_1 \otimes x_2 & \epsilon x_2  \otimes x_2& q_{12} x_3  \otimes x_2
\\
q_{21} x_1 \otimes x_3 &  q_{21}(x_2 +  x_1) \otimes x_3& q_{22} x_3  \otimes x_3
\end{pmatrix}.
\end{align}
Let $V_1 = \langle x_1, x_2\rangle$, $V_2 = \langle x_3\rangle$. Let $\Gamma = \Z^2$ with a basis $g_1, g_2$.
We realize $V$ in $\ydG$ by  $V_1 = V_{g_1}$, $V_2 = V_{g_2}$, $g_1\cdot x_1 = \epsilon x_1$, $g_2\cdot x_1 = q_{21} x_1$,
$g_1\cdot x_2 = \epsilon x_2$, $g_2\cdot x_2 = q_{21} (x_2 +  x_1)$, $g_i\cdot x_3 = q_{i2} x_3$.
Our goal is to compute $\GK \NA(V)$. Since $V_1$ is of diagonal Cartan type, we may assume that $\epsilon \in \G_2 \cup \G_3$.

\smallbreak
As usual, let $\widetilde{q}_{12} = q_{12}q_{21}$; in particular the Dynkin diagram of the braided subspace $\langle x_1, x_3\rangle$ is  $\xymatrix{\overset{\epsilon}  {\circ} \ar  @{-}[r]^{\widetilde{q}_{12}}  & \overset{q_{22}}  {\circ} }$.
Below we consider the filtration
\begin{align}\label{eq:pale-block-filtration}
\cV_1 = \langle x_1, x_3\rangle \subset \cV_2 = V
\end{align} of $V$ and the corresponding $V^{\diag}$, cf. \S  \ref{subsection:filtr-nichols}.
We summarize the results of this Subsection in the next Theorem.

\begin{theorem}\label{th:paleblock-point-resumen}
Let $V$ be as above.
	\begin{enumerate}[leftmargin=*,label=\rm{(\alph*)}]
		\item\label{item:paleblock-point13} If $\epsilon \in \G_3$, then  $\GK \NA(V) = \infty$.
		
		\item\label{item:paleblock-point-1} If $\epsilon =-1$ and $\GK \NA(V) < \infty$, either of the following holds:
		
\begin{enumerate}[leftmargin=*,label=\rm{(\roman*)}]		
	\item\label{item:paleblock-endymion}  $\widetilde{q}_{12} = 1$ and $q_{22} = \pm 1$;  in this case $\GK \NA(V) = 1$.
	
	\item\label{item:paleblock-open} $q_{22} = -1 = \widetilde{q}_{12}$;  in this case $\GK \NA(V) = 2$.
\end{enumerate}
\end{enumerate}
\end{theorem}

\pf \ref{item:paleblock-point13}: If $\epsilon = 1$, then this is Proposition \ref{prop:paleblock1}.
If $\epsilon \in \G'_3$, then this is proved in \S \ref{subsubsec:pale-block-caseomega}.

\ref{item:paleblock-point-1}: Assume that $\widetilde{q}_{12} = 1$. If $q_{22} = \pm 1$, then the claim follows from Propositions \ref{prop:paleblock2} and \ref{prop:paleblock3}.  If $q_{22} \neq  \pm 1$, then this is proved in \S \ref{subsubsec:pale-block-case-1}.
 Assume finally that $\widetilde{q}_{12} \neq 1$. Then we reduce to four possible cases described in \eqref{eq:paleblock-4possible}, by means of Hypothesis \ref{hyp:nichols-diagonal-finite-gkd}.
The analysis of these cases is performed in Proposition \ref{prop:pale-block-case2a}--that takes care of \ref{item:paleblock-open}--and Lemmas \ref{lemma:pale-block-case2b}, \ref{lemma:pale-block-case2c} and \ref{lemma:pale-block-case2d}, discarding the remaining cases.
\epf

As  in many other places, we consider
$K=\NA (V)^{\mathrm{co}\,\NA (V_1)}$;
then $K = \oplus_{n\ge 0} K^n$ inherits the grading of $\NA (V)$;
$\NA(V) \simeq K \# \NA (V_1)$ and $K$ is the Nichols algebra of
$K^1= \ad_c\NA (V_1) (V_2)$.
Now $K^1\in {}^{\NA (V_1)\# \ku \Gamma}_{\NA (V_1)\# \ku \Gamma}\mathcal{YD}$ with the adjoint action
and the coaction given by \eqref{eq:coaction-K^1}, i.e.
$\delta =(\pi _{\NA (V_1)\#  \ku \Gamma}\otimes \id)\Delta _{\NA (V)\#  \ku \Gamma}$.
For later use, we introduce $\sh_{m, n} = (\ad_{c} x_{1})^m(\ad_{c} x_{2})^n x_{3}$, and particularly
\begin{align*}
w_m &= (\ad_{c} x_{1})^m x_{3} = \sh_{m,0},&
z_n &= (\ad_{c} x_{2})^n x_3 = \sh_{0, n}.
\end{align*}
 Then
\begin{align}\label{eq:paleblock-1}
g_1\cdot &\sh_{m, n}  =  q_{12} \epsilon^{m+n} \sh_{m, n}, & g_2\cdot w_m  &=  q_{21}^{m}q_{22} w_m,
\\ \label{eq:paleblock-2}
z_{n+1} &= x_2z_n - q_{12}\epsilon^{n} z_n x_2,&
\sh_{m + 1, n} &= x_1\sh_{m, n} - q_{12}\epsilon^{m + n}\sh_{m, n} x_1,
\\ \label{eq:paleblock-3}
\partial_1(&\sh_{m, n}) =0, & \partial_2(\sh_{m, n}) &=0,
\\ \label{eq:paleblock-3.5}
&&\partial_3(w_{m}) &= \prod_{0\le j \le m-1}(1 - \epsilon^j \widetilde{q}_{12}) x_1^m.
\end{align}

\noindent\emph{Proof of \eqref{eq:paleblock-3.5}.} For $m=0$, it is clear. Recursively, $\partial_3(w_{m + 1})=$
\begin{align*}
&= \partial_3(x_1w_m - q_{12}\epsilon^{m} w_m x_1) =
\prod_{0\le j \le m-1}(1 - \epsilon^j \widetilde{q}_{12}) (1 - q_{12}\epsilon^{m}q_{21}) x_1^{m+1}. \qed
\end{align*}

\subsubsection{The block has $\epsilon = 1$}\label{subsubsec:pale-block-case1}
Here $\NA (V_1) \simeq S(V_1)$ is a polynomial algebra, so that $x_1$ and $x_2$ commute, and
\begin{align}\label{eq:paleblock-4}
(\ad_{c} x_{2})^p \sh_{m, n} = \sh_{m, n + p}.
\end{align}
Consequently $\sh_{m, n}$, $m, n \in \N_0$ generate $K^1$. We claim that
\begin{align}\label{eq:paleblock-5}
g_2\cdot \sh_{m, n} &= q_{21}^{m+n}q_{22} \sum_{0\le j \le n} \binom{n}{j}  \sh_{m+j, n-j},
\end{align}
\pf By induction on $n$. If $n =0$, then \eqref{eq:paleblock-5} is the second part of \eqref{eq:paleblock-1}. Assume that \eqref{eq:paleblock-5} holds for  $n$; we prove it for $n+1$ by induction on $m$. First,
\begin{align*}
g_2\cdot \sh_{0, n + 1} &= g_2\cdot (z_{n+1}) = g_2\cdot (x_2z_n - q_{12}\epsilon^{n} z_n x_2)
\\
&= q_{21}^{n+1}q_{22} \sum_{0\le j \le n} \binom{n}{j} \Big((x_2 +  x_1)     \sh_{j, n-j}
- q_{12}\epsilon^{n}    \sh_{j, n-j}      (x_2 +  x_1)\Big)
\\
& \overset{\eqref{eq:paleblock-4}}{=} q_{21}^{n+1}q_{22} \Big(\sum_{0\le j \le n} \binom{n}{j} (\sh_{j, n+1-j} + \sh_{j+1, n-j}) \Big)
\end{align*}
and the claim follows. Next, assuming that \eqref{eq:paleblock-5} holds for $m$ and $n$, we have
\begin{align*}
g_2\cdot \sh_{m+1, n} &=  g_2\cdot (x_1 \sh_{m, n} - q_{12}\epsilon^{m + n} \sh_{m, n} x_1)
\\
&= q_{21}^{m+n+1}q_{22} \sum_{0\le j \le n} \binom{n}{j} \big(   x_1  \sh_{m+j, n-j}
- q_{12}\epsilon^{m+n}    \sh_{m+j, n-j}      x_1\big)
\end{align*}
and the claim follows by \eqref{eq:paleblock-2}. \epf

Now we may assume that  $\widetilde{q}_{12} = 1$ by Lemma \ref{lemma:points-trivial-braiding}. Then
$w_m =0$ for all $m >0$, by \eqref{eq:paleblock-3} and \eqref{eq:paleblock-3.5}; thus $\sh_{m, n} =0$ for all $m >0$ by \eqref{eq:paleblock-4}. Hence
\begin{align}\label{eq:paleblock-5bis}
g_2\cdot z_{n} &= q_{21}^{n}q_{22}   z_{n}, & n&\in \N_0.
\end{align}

\begin{lemma}\label{le:paleblock1} The following are equivalent:
	\begin{enumerate}[leftmargin=*,label=\rm{(\roman*)}]
		\item $\GK K < \infty$.
		\item $q_{22} =-1$.
	\end{enumerate}
	If this happens, then  $K \simeq \Lambda (K^1)$ has $\GK = 0$.
\end{lemma}

\pf 	To start with, we claim  that  for all $n\in\N_0$:
\begin{align}
\label{eq:paleblock-6}
\partial_3(z_n) &= (-1)^n x_1^n,
\\\label{eq:paleblock-7}
\delta(z_n) &= \sum_{0\le j \le n} (-1)^{n+j} \binom{n}{j} x_1^{n-j} g_1^jg_2 \otimes z_j.
\end{align}
First, \eqref{eq:paleblock-6} is clear for $n=0$. If it holds for $n$, then $\partial_3(z_{n+1})=$
\begin{align*}
\partial_3(x_2z_n - q_{12} z_n x_2) =
(-1)^n x_2 x_1^n - q_{12} (-1)^n x_1^n q_{21} (x_2 + x_1)
= (-1)^n x_1^{n+1}.
\end{align*}
By the previous remarks, \eqref{eq:paleblock-6} implies that the family $(z_n)_{n \in \N_0}$ is a basis of $K^1$. Next,  \eqref{eq:paleblock-7} for $n=0$ is just $\delta(x_{3}) = g_2 \otimes x_3$.
If \eqref{eq:paleblock-7} holds for $n$, then

\begin{align*}
\delta(z_{n+1}) &= \sum_{0\le j \le n} (-1)^{n+j} \binom{n}{j}
\big((x_2 \otimes 1 + g_1 \ot x_2)  x_1^{n-j} g_1^jg_2 \otimes z_j
\\ &\qquad - q_{12}x_1^{n-j} g_1^jg_2 \otimes z_j (x_2 \otimes 1 + g_1 \ot x_2)\big) \\
&= \sum_{0\le j \le n} (-1)^{n+j} \binom{n}{j}
\big(\underbrace{x_2 x_1^{n-j} g_1^jg_2 \otimes z_j}_{(\alpha)} + \underbrace{g_1x_1^{n-j} g_1^jg_2 \otimes x_2z_j}_{(\beta)}
\\ &\qquad - \underbrace{q_{12}x_1^{n-j} g_1^jg_2x_2 \otimes z_j}_{(\gamma)}
- \underbrace{q_{12}x_1^{n-j} g_1^jg_2g_1 \otimes z_jx_2}_{(\eta)}\big)
\end{align*}
Now
\begin{align*}
(\alpha) + (\gamma) &= \big( x_1^{n-j}x_2 g_1^jg_2 - \widetilde{q}_{12}x_1^{n-j}(x_2 + x_1) g_1^jg_2\big) \otimes z_j = -x_1^{n + 1- j} g_1^jg_2 \otimes z_j; \\
(\beta) + (\eta) &= x_1^{n-j} g_1^jg_2g_1 \otimes ( x_2z_j - q_{12} z_jx_2) = x_1^{n-j} g_1^{j+1}g_2\otimes z_{j+1}
\end{align*}
and the claim follows. Therefore, by \eqref{eq:paleblock-5bis} and \eqref{eq:paleblock-7}, we have for $p, n\in \N_0$
\begin{align*}
c(z_n \otimes z_p) &=  \sum_{0\le j \le n} (-1)^{n+j} \binom{n}{j} \ad_c (x_1^{n-j} g_1^jg_2) z_p \otimes z_j
= q_{12}^{n}q_{21}^p q_{22}  z_p \otimes z_{n}.
\end{align*}
In conclusion, $K^1$ is of diagonal type with braiding matrix $(q_{21}^{n}q_{22})_{n,p \in \N_0}$ with respect to the basis
$(z_n)_{n \in \N_0}$; thus locally $\xymatrix{\overset{q_{22}}  {\circ} \ar  @{-}[rr]^{q_{22}^2} & & \overset{q_{22}}  {\circ} }$.
Then  $\GK \NA(K^1) = \infty$ unless $q_{22} = -1$; in this last case, $\NA(K^1) = \Lambda(K^1)$.
\epf

\begin{prop}\label{prop:paleblock1}
	$\GK \NA(V) = \infty$.
\end{prop}

\pf By Lemma \ref{le:paleblock1}, it remains to consider the case  $q_{22} =-1$. Now $\NA(V)$ is a graded connected finitely generated algebra
isomorphic to $\NA(V_1) \# \Lambda(W)$, where $W$ has the basis $(z_n)_{n\in \N_0}$. Since the elements $z_{n_1} \dots z_{n_\ell}$
are linearly independent when $n_1 < \dots < n_\ell$, $\GK \NA(V) = \infty$ by Lemma \ref{lemma:rosso-lemma19-gral}.
\epf

\subsubsection{The block has $\epsilon = -1$}\label{subsubsec:pale-block-case-1}
Here $\NA (V_1) \simeq \Lambda(V_1)$ is an exterior algebra and consequently $\sh_{m, n}$, $m, n \in \{0,1\}$ generates $K^1$.
It is easy to check that
\begin{align}
\label{eq:paleblock-8}
g_2 \cdot z_1 &=  q_{21}q_{22}(z_1 + w_1), \qquad \partial_3(z_1) = (1 - \widetilde{q}_{12}) x_2 - \widetilde{q}_{12} x_1,
\\ \label{eq:paleblock-9}
\delta(z_1) &= g_1g_2 \otimes z_1 + \big((1 - \widetilde{q}_{12}) x_2 - \widetilde{q}_{12} x_1\big)g_2 \otimes x_3.
\end{align}

\medbreak
\emph{\textbf{Case 1: $\widetilde{q}_{12} = 1$.}} Then
$w_1 =0$  by \eqref{eq:paleblock-3} and \eqref{eq:paleblock-3.5}; hence $\sh_{1, 1} = -(\ad_{c} x_{2}) w_1  =0$.
In consequence, $z_0 = x_3$ and $z_1$ form a basis of $K^1$, as they are linearly independent by \eqref{eq:paleblock-8}. Then
\begin{align}
\label{eq:paleblock-10}
\begin{aligned}
c(x_3 \ot x_3) &= q_{22}x_3 \ot x_3, &  c(x_3 \ot z_1) &= q_{21}q_{22} z_1 \otimes x_3,\\
c(z_1\ot x_3) &= q_{12}q_{22} x_3 \otimes z_1, &  c(z_1 \ot z_1) &=  -q_{22} z_1 \ot z_1.
\end{aligned}
\end{align}
Thus $K^1$ is of diagonal type with Dynkin diagram $\xymatrix{ \overset{q_{22}}{\circ} \ar  @{-}[r]^{q_{22}^2} & \overset{-q_{22}}{\circ} }$.
If $q_{22}\neq \pm 1$, then this diagram does not appear in the Table 1 in \cite{H-classif} and we conclude by our Hypothesis \ref{hyp:nichols-diagonal-finite-gkd} that $\GK K = \GK \NA(V) = \infty$.

\smallbreak
For $q\in \ku^{\times}$, let $\eny_{\pm}(q) = V$ be the braided vector space as in \eqref{eq:braiding-paleblock-point-bis} under the assumptions that $\epsilon = -1$,
$q_{12} = q = q_{21}^{-1}$, $q_{22} = \pm 1$. We call $\NA(\eny_{\pm}(q))$ and the Nichols algebra $\NA(\eny_{\star}(q))$ studied in Proposition \ref{prop:pale-block-case2a} the \emph{Endymion algebras}.

\begin{prop}\label{prop:paleblock2}
	The algebra $\cB(\eny_+(q))$ is presented by generators $x_1,x_2, x_3$ and relations
	\begin{align}\label{eq:endymion-1}
	x_1^2&=0, \quad x_2^2=0, \quad x_1x_2 =- x_2x_1,
	\\ \label{eq:endymion-2}
	(x_2x_3 - q x_3x_2)^2&=0,
	\\ \label{eq:endymion-3}  x_3(x_2x_3 - q x_3x_2) &=q^{-1} (x_2x_3 - q x_3x_2)x_3,
	\\ \label{eq:endymion-4}
	x_1x_3 &= q x_3x_1.
	\end{align}
	Let $z_1 = x_2x_3 - q x_3x_2$. Then $\cB(\eny_+(q))$ has a PBW-basis
	\begin{align*}
	B=\{ x_1^{m_1} x_2^{m_2} x_3^{m_3}z_{1}^{m_{4}}: m_1, m_2, m_4 \in \{0,1\}, \, m_3 \in\N_0\};
	\end{align*}
	hence $\GK \cB(\eny_+(q)) = 1$.
\end{prop}

\pf Relations \eqref{eq:endymion-1} are 0 in $\cB(\eny_+(q))$ because $\NA (V_1) \simeq \Lambda(V_1)$; \eqref{eq:paleblock-10} implies  \eqref{eq:endymion-2}
and \eqref{eq:endymion-3}. In turn \eqref{eq:endymion-4}, i.e. $w_1 =0$, holds  by \eqref{eq:paleblock-3} and \eqref{eq:paleblock-3.5}.
Hence the quotient $\cBt$ of $T(\eny_+(q))$
by \eqref{eq:endymion-1}, \eqref{eq:endymion-2}, \eqref{eq:endymion-3} and \eqref{eq:endymion-4}
projects onto $\cB(\eny_+(q))$.
Also the following relations hold in $\cBt$ by the last identity in \eqref{eq:endymion-1} and \eqref{eq:endymion-4} via the $q$-Jacobi identity, resp. the second identity in \eqref{eq:endymion-1}:
\begin{align} \label{eq:endymion-5}
x_1 z_1 &= -q z_1x_1, & x_2z_1 &= -q z_1x_2.
\end{align}
Now the subspace $I$ spanned by $B$ is a right ideal of $\cBt$; this follows without troubles from \eqref{eq:endymion-1}, \eqref{eq:endymion-2}, \eqref{eq:endymion-3}, \eqref{eq:endymion-4} and \eqref{eq:endymion-5}.
To prove that $\cBt \simeq \cB(\eny_+(q))$, it remains to show that
$B$ is linearly independent in $\cB(\eny_+(q))$. But $\NA(V) \simeq \NA(K^1) \# \NA (V_1)$, so that $\NA(V) \simeq \NA(K^1) \ot \NA (V_1)$ as vector spaces and the claim follows by the commutation relations already exposed.
Then $B$ is a basis of $\cB(\eny_+(q))$ and $\cBt=\cB(\eny_+(q))$.
The computation of $\GK$ follows from the Hilbert series at once.
\epf

\begin{prop}\label{prop:paleblock3}
	The algebra $\cB(\eny_-(q))$ is presented by generators $x_1,x_2, x_3$ and relations \eqref{eq:endymion-1}, \eqref{eq:endymion-4},
	\begin{align}
	\label{eq:endymion-2b}
	x_3 ^2&=0,
	\\ \label{eq:endymion-3b}  x_3(x_2x_3 - q x_3x_2) &= -q^{-1} (x_2x_3 - q x_3x_2)x_3.
	\end{align}
	Let $z_1 = x_2x_3 - q x_3x_2$. Then $\cB(\eny_-(q))$ has a PBW-basis
	\begin{align*}
	B=\{ x_1^{m_1} x_2^{m_2} x_3^{m_3}z_{1}^{m_{4}}: m_1, m_2, m_3 \in \{0,1\}, \, m_4 \in\N_0\};
	\end{align*}
	hence $\GK \cB(\eny_-(q)) = 1$.
\end{prop}

\pf Adapt mutatis mutandis the proof of Proposition \ref{prop:paleblock2}.
\epf

\medbreak
\emph{\textbf{Case 2: $\widetilde{q}_{12} \neq  1$.}}
The braided vector space $V^{\diag}$ has Dynkin diagram
$$\xymatrix{ \overset{-1}{\circ} \ar  @{-}[r]^{\widetilde{q}_{12}} & \overset{q_{22}}{\circ}
	\ar	@{-}[r]^{\widetilde{q}_{12}}  & \overset{-1}{\circ} }.$$
By Hypothesis  \ref{hyp:nichols-diagonal-finite-gkd} and after inspection of the Table 2 in \cite{H-classif},
we conclude that there are 4 possible diagrams, listed next:
\begin{align}\label{eq:paleblock-4possible}
\begin{aligned}
q_{22} &=-1 = \widetilde{q}_{12}, & q_{22} &=q, \widetilde{q}_{12} = q^{-1}, q^2 \neq 1, \\
q_{22} &=-1, \widetilde{q}_{12} = \omega \in \G_3',& q_{22} &=-\omega, \widetilde{q}_{12} = \omega \in \G_3'.
\end{aligned}
\end{align}

We deal below with the different possibilities.

\medbreak

For $q\in \ku^{\times}$, let $\eny_{\star}(q) = V$ be the braided vector space as in \eqref{eq:braiding-paleblock-point-bis} under the assumptions that
$\epsilon = -1$, $q_{22}=-1$, $q_{12}=q$, $q_{21}=-q^{-1}$.  Recall \eqref{eq:xijk}.

\begin{prop}\label{prop:pale-block-case2a}
The algebra $\cB(\eny_{\star}(q))$ is presented by generators $x_1,x_2, x_3$ and relations \eqref{eq:endymion-1},
\begin{align}\label{eq:endymion-6}
x_3^2=0, \quad x_{31}^2&=0,
\\ \label{eq:endymion-7}
x_2 [x_{23},x_{13}]_c- q^2 [x_{23},x_{13}]_cx_2 &=q \, x_{13}x_{213},
\\ \label{eq:endymion-8}
x_{213}^2 &= 0.
\end{align}
Moreover $\cB(\eny_{\star}(q))$ has a PBW-basis
\begin{multline*}
B=\{ x_2^{m_1}x_{23}^{m_2} x_{213}^{m_3}[x_{23},x_{13}]_c^{m_4} x_1^{m_5}x_{13}^{m_6} x_3^{m_7} : \\
m_1, m_3, m_5, m_6, m_7 \in \{0,1\}, \, m_2, m_4 \in\N_0\};
\end{multline*}
hence $\GK \cB(\eny_{\star}(q)) = 2$.
\end{prop}

\pf Relations \eqref{eq:endymion-1} are 0 in $\cB(\eny_{\star}(q))$ because $\NA (V_1) \simeq \Lambda(V_1)$; \eqref{eq:endymion-6}
are 0 since $x_1$, $x_3$ generate a Nichols algebra of Cartan type $A_2$ at $-1$.

Notice that $[x_{23},x_{13}]_c=x_{23}x_{13}+x_{13}x_{23}$. By \eqref{eq:endymion-1} and \eqref{eq:endymion-6},
\begin{align}\label{eq:endymion-aux-1}
x_2x_{23}&=-q \, x_{23}x_2, & x_2x_{213}&=q \, x_{213}x_2,
\\ \label{eq:endymion-aux-2}
x_{23}x_3&=-q \, x_3x_{23}, & x_{23}x_1&=-q^{-1}x_1x_{23}-q^{-1}x_{213},
\\ \label{eq:endymion-aux-3}
x_{213}x_1&=q^{-1}x_1x_{213}, & [x_{23},x_{13}]_cx_1&=-q^{-2} x_1[x_{23},x_{13}]_c,
\\ \label{eq:endymion-aux-4}
x_1x_{13}&=-q \, x_{13}x_1, & x_{213}x_3&=q[x_{23},x_{13}]_c-q^2x_3x_{213},
\\ \label{eq:endymion-aux-5}
x_{13}x_3&=-q \, x_3x_{13}, & x_{13}[x_{23},x_{13}]_c &=[x_{23},x_{13}]_c x_{13},
\\ \label{eq:endymion-aux-6}
x_{213}x_{13}&=q \, x_{13}x_{213}, & [x_{23},x_{13}]_c x_3 &=q^2 x_3 [x_{23},x_{13}]_c.
\end{align}
By direct computation, $\partial_3(x_{213})=4x_2x_1 \neq 0$ (being in $\Lambda(V_1)$), and
\begin{multline*}
\partial_3 ([x_{23},x_{13}]_c)= 2x_{23}x_1+2q^{-1} x_2x_{13}+2x_{13}x_2+2q^{-1}x_1(x_{23}+x_{13}) \\
= -2q^{-1}(x_{213}+x_{23}x_1)+2q^{-1}(x_{213}-q \, x_{13}x_2)+2x_{13}x_2+2q^{-1}x_1(x_{23}+x_{13}) \\
=2q^{-1}x_1x_{13} \neq 0, \quad \text{being in } \langle x_1, x_3\rangle \text{ as discussed above}.
\end{multline*}

Now we prove that \eqref{eq:endymion-7} holds in $\cB(V)$. Notice that $\partial_1$ annihilates all factors in each term; also,
$\partial_2$ annihilates the right-hand side, and
\begin{align*}
\partial_2(x_2 [x_{23},x_{13}]_c- q^2 [x_{23},x_{13}]_cx_2) &= g_2 \cdot [x_{23},x_{13}]_c- q^2 [x_{23},x_{13}]_c =0.
\end{align*}
Now we compute $\partial_3$. For the left-hand side,
\begin{multline*}
\partial_3 (x_2 [x_{23},x_{13}]_c- q^2 [x_{23},x_{13}]_cx_2)= 2q^{-1}x_2x_1x_{13}+2x_1x_{13}(x_2+x_1) \\
=-2q^{-1} \, x_1(x_{213}-q \, x_{13}x_2)+2 x_1x_{13}x_2 =-2q^{-1}\, x_1x_{213}+ 4 x_1x_{13}x_2,
\end{multline*}
while for the right-hand side,
\begin{multline*}
\partial_3 (x_{13}x_{213})= 2 x_1 (g_2\cdot x_{213})+4x_{13}x_1x_2= q^{-1} \big(-2q^{-1}\, x_1x_{213}+ 4 x_1x_{13}x_2).
\end{multline*}
Hence \eqref{eq:endymion-7} holds in $\cB(V)$. For \eqref{eq:endymion-8}, we also check that $\partial_i(x_{213}^2)=0$
for $i=1,2,3$. As $\partial_1$, $\partial_2$ annihilate $x_{213}$ it remains the case $i=3$:
\begin{align*}
\partial_3(x_{213}^2) &= 4x_2x_1(g_3\cdot x_{213})+4 x_{213}x_2x_1=4 \big(q^{-2} x_2x_1 x_{213}+x_{213}x_2x_1 \big)=0.
\end{align*}
Hence the quotient $\cBt$ of $T(\eny_{\star}(q))$ by \eqref{eq:endymion-1}, \eqref{eq:endymion-6},
\eqref{eq:endymion-7} and \eqref{eq:endymion-8} projects onto $\cB(\eny_{\star}(q))$.

We notice that \eqref{eq:endymion-aux-1}, \dots, \eqref{eq:endymion-aux-6} hold in $\cBt$ since they are derived from
\eqref{eq:endymion-1}, \eqref{eq:endymion-6}, \eqref{eq:endymion-7} and \eqref{eq:endymion-8}. The following relations also hold in $\cBt$:
\begin{align*}
x_{213}x_{23} &= q (x_{23}+x_{13})x_{213}, &
x_{213} [x_{23},x_{13}]_c &= q^2 [x_{23},x_{13}]_c x_{213}, \\
[x_{23},x_{13}]_c x_{23} &= (x_{23}+x_{13})[x_{23},x_{13}]_c.
\end{align*}
Thus the subspace $I$ spanned by $B$ is a right ideal of $\cBt$.  To prove that $\cBt \simeq \cB(\eny_{\star}(q))$, it remains to show that
$B$ is linearly independent in $\cB(\eny_{\star}(q))$.

Recall the filtration \eqref{eq:pale-block-filtration}.
We claim that the classes of $[x_{23},x_{13}]_c$ and $x_{23}^2$ in $\Bdiag$ are non-zero primitive elements of $\Bdiag$. By direct computation,
\begin{align*}
\Delta([x_{23},x_{13}]_c) &= [x_{23},x_{13}]_c\otimes 1 - 2x_{13}x_1\otimes x_3 - x_{13}\otimes x_{13} \\
& \qquad - 2x_1\otimes x_3x_{13} + 1 \otimes [x_{23},x_{13}]_c; \\
\Delta(x_{23}^2) &=x_{23}^2\otimes 1 - x_{13} \otimes x_{23}+(2q^{-1}x_2x_{13}-q \, x_1x_{13})\otimes x_3 \\
& \qquad -2 x_1\otimes x_{23}x_3 + 1\otimes x_{23}^2.
\end{align*}
Hence it suffices to prove that
$[x_{23},x_{13}]_c\in\cB^4_5-\cB^4_4$, $x_{23}^2\in\cB^4_6-\cB^4_5$.

Suppose that $[x_{23},x_{13}]_c\in\cB^4_4$. Then $[x_{23},x_{13}]_c$ is written as a linear combination of words in letters $x_1$, $x_3$.
As $\cB(V)$ is spanned by $B$, there exists $a\in\ku$ such that
$[x_{23},x_{13}]_c= a \, x_1x_{13}x_3$. Applying $\partial_1$ and $\partial_3$ to this equality,
\begin{align*}
0&=-q^2a \, x_{13}x_3, & 2q^{-1}\, x_1x_{13} &= a \, x_1x_{13}
\end{align*}
Hence we have a contradiction, so $[x_{23},x_{13}]_c\in\cB^4_5-\cB^4_4$.
Now suppose that $x_{23}^2\in\cB^4_5$. As $\cB(V)$ is spanned by $B$, there exist $a_i\in\ku$ such that
\begin{align*}
x_{23}^2 &=a_1\,[x_{23},x_{13}]_c+a_2\,x_{23}x_{13}+a_3\,x_{23}x_1x_3+a_4\,x_{213}x_3 \\
& \qquad +a_5\,x_{213}x_1 + a_6 \, x_1x_{13}x_3  + a_7\, x_2x_1x_{13} + a_8\, x_2 x_{13}x_3.
\end{align*}
Thus $0 = \partial_2(x_{23}^2)= qa_7\, x_1x_{13}-q^2a_8\, x_{13}x_3$, so $a_7=a_8=0$.
Applying $\partial_1$,
\begin{align*}
0 &= qa_3\,x_{23}x_3+a_5\,x_{213}-q^2 a_6 \, x_{13}x_3,
\end{align*}
so $a_3=a_5=a_6=0$. Applying $\partial_3$,
\begin{align*}
2q^{-1}x_2x_{13}-q \, x_1x_{13} &=
2q^{-1}a_1\, x_1x_{13} +q^{-1}a_2\,(2x_2+x_1)x_{13} +2a_2\,x_{23}x_1 \\
& \qquad +4a_4\,x_2x_1x_3 +a_4\,x_{213}.
\end{align*}
Using the derivations $\partial_i$ we prove that $x_2x_{13}$, $x_1x_{13}$, $x_2x_{13}$,  $x_{23}x_1$, $x_2x_1x_3$, $x_{213}$
are linearly independent. Hence $2a_2=0$ and $2q^{-1}a_2=2q^{-1}$, so we have a contradiction. Thus $x_{23}^2\in\cB^4_6-\cB^4_5$.

As $V^{\diag}$ is of type $A_3$ at $q=-1$, $\Bdiag$ has a PBW basis whose set of generators
contains $x_2$, $x_{23}$, $x_{213}$, $x_1$, $x_{13}$, $x_3$; notice that $[x_{23},x_{13}]_c$ is another PBW generator disconnected from the previous ones.
As $x_{23}^2, [x_{23},x_{13}]_c\neq 0$ in
$\Bdiag$ are primitive (the corresponding scalars are  both $1$), $x_{23}$ and $[x_{23},x_{13}]_c$ they have infinity height in $\Bdiag$;
hence the classes of the elements of $B$ are linearly independent in $\Bdiag$, so $B$ is linearly independent in $\cB(V)$.
Then $B$ is a basis of $\cB(\eny_{\star}(q))$ and $\cBt \simeq \cB(\eny_{\star}(q))$.
The computation of $\GK$ follows from the Hilbert series at once.
\epf

\begin{lemma}\label{lemma:pale-block-case2b}
If $\epsilon = -1$, $q_{22}=-1$, $\widetilde{q}_{12} = \omega \in \G_3'$,
then $\GK \NA(V) = \infty$.
\end{lemma}

\pf We consider the decomposition $V = V_1\oplus V_2$ but with different projection and section, and correspondingly a different algebra of coinvariants as in \S \ref{subsubsec:pale-block}: now we have
$\cK =\NA (V)^{\mathrm{co}\,\NA (V_2)}$.
Then $\cK = \oplus_{n\ge 0} \cK^n$ inherits the grading of $\NA (V)$;
$\NA(V) \simeq \cK \# \NA (V_2)$ and $\cK$ is the Nichols algebra of $\cK^1= \ad_c\NA (V_2) (V_1)$.
As $x_3^2=0$, $\cK^1$ is spanned by $x_1$, $x_{31}$, $x_2$, $x_{32}$, recall
the notation \eqref{eq:xijk}.
Now $\cK^1\in {}^{\NA (V_2)\# \ku \Gamma}_{\NA (V_2)\# \ku \Gamma}\mathcal{YD}$ with the adjoint action
and the coaction \eqref{eq:coaction-K^1}; in this case, this is explicitly
\begin{align*}
\delta(x_1)&= g_1\otimes x_1, & \delta(x_{31})&= g_1g_2\otimes x_{31}+(1-\omega)x_3g_1\otimes x_1, \\
\delta(x_2)&= g_1\otimes x_2, & \delta(x_{32})&= g_1g_2\otimes x_{32}+(1-\omega)x_3g_1\otimes (x_2+x_1).
\end{align*}
As $x_3\cdot x_{31}=x_3\cdot x_{32}=0$, the braiding of $K^1$ satisfies:
\begin{align*}
c(x_{31}\otimes x_{31}) &= g_1g_2 \cdot x_{31} \otimes x_{31} = \omega x_{31} \otimes x_{31}, \\
c(x_{32}\otimes x_{31}) &= g_1g_2 \cdot x_{31} \otimes x_{32} = \omega x_{31} \otimes x_{32}, \\
c(x_{31}\otimes x_{32}) &= g_1g_2 \cdot x_{32} \otimes x_{31} = \omega (x_{32}+x_{31}) \otimes x_{31}, \\
c(x_{32}\otimes x_{32}) &= g_1g_2 \cdot x_{32} \otimes x_{32} = \omega (x_{32}+x_{31}) \otimes x_{32}.
\end{align*}
Hence $W=\langle x_{31},x_{32} \rangle$ is a braided vector subspace of $K^1$. Now $W$ is isomorphic to the block $\cV(\omega,2)$ (consider the basis $\omega x_{31}$, $x_{32}$), so $\GK\NA(W)=\infty$ by Theorem \ref{theorem:blocks}. Therefore $\GK \NA(V)=\infty$.
\epf

\begin{lemma}\label{lemma:pale-block-case2c}
	If $\epsilon = -1$, $q_{22}=q$, $\widetilde{q}_{12} = q^{-1}$, $q^2 \neq 1$,  then $\GK \NA(V) = \infty$.
\end{lemma}

\pf
Let  $\cB := \cB (V)$ and $\Bdiag := \gr \cB (V)$ with respect to \eqref{eq:pale-block-filtration}, a pre-Nichols algebra of  $V^{\textrm{diag}}$, see \S \ref{subsection:filtr-nichols}.
Thus $\cB^n_m$ is the $m$-th term of the filtration in $\cB^n = \cB^n(V)$.
Recall $x_{13}$, $x_{23}$ as in \eqref{eq:xijk} and set $u=[x_{23},x_3]_c$. By direct computation,
\begin{align*}
\Delta(u) &= u\otimes 1 - (1+q) \, x_{13} \otimes x_3-(1+q^{-1}) \, x_1\otimes x_3^2 + 1\otimes u.
\end{align*}
Suppose that $u\in\cB^3_3$. Notice that $\cB^3_3$ is spanned by the monomials in letters $x_1$ and $x_3$, since $T(V)^3_3$ is spanned by these monomials, cf. Remark \ref{obs:filtered} and Lemma \ref{lemma:filtered}. As $x_1$, $x_3$ generate a Nichols algebra of type super $A_2$,
these monomials are written as a linear combination of $x_1 x_{13}$, $x_1x_3^2$, $x_{13}x_3$, $x_3^3$; the last term is zero if $q^3=1$. Hence
\begin{align*}
u&=a_1\,x_1 x_{13} + a_2\, x_1x_3^2 + a_3\, x_{13}x_3 + a_4 \, x_3^3 & \mbox{for some }&a_i\in\ku.
\end{align*}
Thus $0= \partial_1(u)= -a_1q_{12} \,x_{13} + a_2q_{12}^2\, x_3^2$, which says that $a_1=a_2=0$, and
\begin{align*}
- (1+q) \, x_{13} &= \partial_3(u) = a_3(q-1)\, x_1x_3 + a_3\, x_{13}+ a_4(3)_q \, x_3^2.
\end{align*}
As $x_1x_3$, $x_{13}$, $x_3^2$ are linearly independent, we have $1+q =0$, a contradiction. Thus $u\in\cB^3_4-\cB^3_3$.
Let $x_4$ be the class of $u$ in $\Bdiag$. Then $x_4$ is a non-zero primitive element in $\Bdiag$.
Let $\widetilde{\cB}_1$ be the subalgebra of $\Bdiag$ generated by $Z = \langle x_1,x_2,x_3,x_4 \rangle$,
and consider the algebra filtration of $\widetilde{\cB }_1$, such that the generators have degree one.
This is a Hopf algebra filtration, hence the associated graded algebra  $\widetilde{\cB}_2$ is a braided Hopf algebra, and
$\NA(Z)$ is a subquotient of $\widetilde{\cB}_2$.
Then $\GK \NA(Z)\le \GK \widetilde{\cB}_2\le \GK \widetilde{\cB}_1\le \GK\Bdiag =\GK \cB (V)$.
The braided vector space $Z$ has diagram
\begin{align*}
\xymatrix{ \overset{q}{\circ} \ar  @{-}[r]^{q^{-1}}\ar  @{-}[rd]^{q^3} \ar@{-}[d]_{q^{-1}}
	& \overset{-1}{\circ} \ar@{-}[d]^{q^{-2}}
	\\  \overset{-1}{\circ} \ar@{-}[r]_{q^{-2}} & \overset{-q^2}{\circ}. }
\end{align*}
Hence the diagram has a $4$-cycle and $\GK\NA(Z)=\infty$ by Hypothesis \ref{hyp:nichols-diagonal-finite-gkd}, so $\GK \cB (V)=\infty $.
\epf

In the previous proof, Theorem \ref{thm:nichols-diagonal-finite-gkd} works for many values of $q$, without appealing to Hypothesis \ref{hyp:nichols-diagonal-finite-gkd}.

\begin{lemma}\label{lemma:pale-block-case2d}
	If $\epsilon = -1$, $q_{22}=-\omega$, $\widetilde{q}_{12}=\omega\in \G_3'$,  then $\GK \NA(V) = \infty$.
\end{lemma}

\pf Let  $\cB := \cB (V)$, $\Bdiag := \gr \cB (V)$ with respect to \eqref{eq:pale-block-filtration}, see \S \ref{subsection:filtr-nichols},
and $\cB^n_m$  the $m$-th term of the filtration in $\cB^n = \cB^n(V)$.
Recall $x_{13}$, $x_{23}$ \eqref{eq:xijk}. Set
\begin{align*}
& x_{13,3}=[x_{13},x_3]_c,  && x_{23,3}=[x_{23},x_3]_c, && u=[x_{23,3},x_3]_c.
\end{align*}
By direct computation,
\begin{align*}
\Delta(u) &= u\otimes 1 +2\omega\, x_{13,3} \otimes x_3 +\omega \, x_{13}\otimes x_3^2 +(\omega^2-\omega) \, x_1\otimes x_3^3 + 1\otimes u.
\end{align*}
Suppose that $u\in\cB^4_4$. Notice that $\cB^4_4$ is spanned by the monomials in letters $x_1$ and $x_3$, since $T(V)^4_4$ is spanned by these monomials, cf. Remark \ref{obs:filtered} and Lemma \ref{lemma:filtered}. As $x_1$, $x_3$ generate a Nichols algebra of type super $B_2$,
these monomials are written as a linear combination of $x_1 x_{13,3}$, $x_{13,3}x_3$, $x_{13}^2$, $x_1x_{13}x_3$,
$x_{13} x_3^2$, $x_1x_3^3$, $x_3^4$. As $\partial_1(u)=0$, all the terms with a factor $x_1$ have coefficient 0. Thus
\begin{align*}
u&=a_1\,x_{13,3}x_3 + a_2\, x_{13}^2 + a_3\, x_{13} x_3^2 + a_4 \, x_3^4 & \mbox{for some }&a_i\in\ku.
\end{align*}
Hence,
\begin{align*}
2\omega\, x_{13,3} &= \partial_3(u) = a_1\,x_{13,3} + a_1(\omega^2-1)\,x_{13}x_3 + a_2q_{21}(1-\omega)\, x_1x_{13} \\
& \qquad + a_3(1-\omega)\, x_{13} x_3 + a_3(\omega^2-\omega)\, x_1 x_3^2 + a_4 (4)_{-\omega} \, x_3^3.
\end{align*}
We have a contradiction since $x_{13,3}$, $x_{13}x_3$, $x_1x_{13}$, $x_1 x_3^2$, $x_3^3$ are linearly independent.
Thus $u\in\cB^4_5-\cB^4_4$. Let $x_4$ be class of $u$ in $\Bdiag$. Then $x_4$ is a non-zero primitive element in $\Bdiag$.
Let $\widetilde{\cB}_1$ be the subalgebra of $\Bdiag$ generated by $Z = \langle x_1,x_2,x_3,x_4 \rangle$,
and consider the algebra filtration of $\widetilde{\cB }_1$, such that the generators have degree one.
This is a Hopf algebra filtration, hence the associated graded algebra  $\widetilde{\cB}_2$ is a braided Hopf algebra.
The Nichols algebra $\NA(Z)$ is a subquotient of $\widetilde{\cB}_2$.
Then $\GK \NA(Z)\le \GK \widetilde{\cB}_2\le \GK \widetilde{\cB}_1\le \GK\Bdiag =\GK \cB (V)$.
The braided vector space $Z$ is of diagonal type with diagram
\begin{align*}
\xymatrix{ \overset{-1}{\circ} \ar@{-}[r]^{\omega} & \overset{-\omega}{\circ} \ar@{-}[r]^{\omega} \ar@{-}[d]^{\omega} & \overset{-1}{\circ}
	\\  & \overset{1}{\circ}. & }
\end{align*}
Hence $\GK\NA(Z)=\infty$ by Lemma \ref{lemma:points-trivial-braiding}, so $\GK \cB (V)=\infty $.
\epf

\begin{remark}\label{rem:paleblock-conjecture-on-K}
Let us come back to the Nichols algebra $K$ described in \S \ref{subsubsec:pale-block}.
Collecting previous information or by direct computation, we have
\begin{align}
\label{eq:paleblock-11}
g_2 \cdot w_1 &=  q_{21}q_{22}  w_1, \qquad \partial_3(w_1) = (1 - \widetilde{q}_{12})  x_1,
\\ \label{eq:paleblock-12}
\delta(w_1) &= g_1g_2 \otimes w_1 + (1 - \widetilde{q}_{12}) x_1g_2 \otimes x_3,
\\\label{eq:paleblock-13}
g_2 \cdot \sh_{1, 1} &=  q_{21}^2q_{22}\sh_{1, 1}, \qquad \partial_3(\sh_{1, 1}) = (1 - \widetilde{q}_{12})^2 x_2x_1,
\\ \label{eq:paleblock-14}
&\begin{aligned}
\delta(\sh_{1, 1}) &= g_1^2g_2 \otimes \sh_{1, 1}
+ \big((1 - \widetilde{q}_{12}) x_2 - \widetilde{q}_{12}x_1\big) g_1g_2 \otimes w_1
\\ &-  (1 - \widetilde{q}_{12}) x_1 g_1g_2 \otimes z_1
+(1 - \widetilde{q}_{12})^2 x_2  x_1 g_2 \otimes x_3
\end{aligned}
\end{align}
This shows that $x_3, z_1, w_1, \sh_{1,1}$ form a basis of $K^1$. Also,
using this information, \eqref{eq:paleblock-8} and \eqref{eq:paleblock-9}, we compute
\begin{align*}
&c(x_3  \otimes x_3) = q_{22} x_3  \otimes x_3, & &c(x_3  \otimes w_1) = q_{21}q_{22} w_1 \otimes x_3, \\
&c(x_3  \otimes z_1) = q_{21}q_{22}(z_1+ w_1) \otimes x_3,& &c(x_3  \otimes \sh_{1,1})
=  q_{21}^2q_{22} \sh_{1,1} \otimes x_3,
\\
&\begin{aligned}
c(w_1  \otimes x_3) &= q_{12}q_{22} x_3  \otimes w_1 \\ &+ q_{22}(1 - \widetilde{q}_{12}) w_1 \ot x_3,
\end{aligned} & &c(w_1  \otimes w_1) = - \widetilde{q}_{12} q_{22} w_1 \otimes w_1, \\
&
\begin{aligned}
c(w_1  \otimes z_1) &= - \widetilde{q}_{12}q_{22}(z_1+ w_1) \otimes w_1 \\ + q_{21}q_{22}(1 &- \widetilde{q}_{12}) \sh_{1,1} \ot x_3,
\end{aligned}
& &c(w_1  \otimes \sh_{1,1}) =  q_{12}q_{21}^2q_{22} \sh_{1,1} \otimes w_1,
\\
&
\begin{aligned}
c(z_1  \otimes x_3) &= q_{12}q_{22} x_3 \otimes z_1 \\ + \big((1 &- \widetilde{q}_{12}) z_1 - \widetilde{q}_{12} w_1\big) \ot x_3,
\end{aligned}
& & \begin{aligned}
c(z_1  \otimes w_1) &= -\widetilde{q}_{12}q_{22} w_1 \otimes z_1
\\ - (1 &- \widetilde{q}_{12}) q_{21}q_{22} \, \sh_{1,1}\ot x_3,
\end{aligned}
\\
&
\begin{aligned}
c(z_1  \otimes z_1) &= - \widetilde{q}_{12}q_{22}(z_1+ w_1) \otimes z_1 \\ - q_{21}q_{22}(1 &- 2\widetilde{q}_{12})\, \sh_{1,1}  \ot x_3,
\end{aligned}
& &c(z_1  \otimes \sh_{1,1}) = q_{12}q_{21}^2q_{22}\, \sh_{1,1} \otimes z_1,
\\
&
\begin{aligned}
c(\sh_{1,1}  \otimes x_3) &= q_{12}^2q_{22} x_3 \otimes \sh_{1,1} \\
+ q_{12}q_{22}\big((1 &- \widetilde{q}_{12}) z_1 - \widetilde{q}_{12} w_1\big) \ot w_1 \\
& - q_{22} (1 - \widetilde{q}_{12})^2 \sh_{1,1} \otimes x_3,
\end{aligned}
& & \begin{aligned}
c(\sh_{1,1}  \otimes w_1) &= q_{12}^2q_{21}q_{22} w_1 \otimes \sh_{1,1}
\\ + \widetilde{q}_{12}q_{22} (1 &- \widetilde{q}_{12})  \, \sh_{1,1}\ot w_1,
\end{aligned}
\\
&
\begin{aligned}
c(\sh_{1,1}  \otimes z_1) &= q_{12}^2q_{21}q_{22} (z_1 + w_1) \otimes \sh_{1,1} \\ &+\widetilde{q}_{12}^2q_{22}\, \sh_{1,1}  \ot w_1,
\end{aligned}
& &c(\sh_{1,1}  \otimes \sh_{1,1}) = \widetilde{q}_{12}^2q_{22}\, \sh_{1,1} \otimes \sh_{1,1}
\end{align*}

Now Theorem \ref{th:paleblock-point-resumen} says that $\GK K = \infty$ unless $q_{22} =-1 = \widetilde{q}_{12}$, where $\GK K = 2$. We ask:

\begin{question}\label{question:paleblock-weird-bvs}
Let $K^1$ be a braided vector space with basis $x_3, z_1, w_1, \sh_{1,1}$ and braiding given just above. Prove directly the statement about the $\GK$, and in the case $q_{22} =-1 = \widetilde{q}_{12}$ find the defining relations.
Is there a class of braided vector spaces including this example?
\end{question}

\end{remark}

\subsubsection{The block has $\epsilon = \omega\in \G_3'$}\label{subsubsec:pale-block-caseomega}

\

\emph{Case 1: $\widetilde{q}_{12} = 1$.}
Let  $\Bdiag := \gr \cB (V)$, a pre-Nichols algebra of  $V^{\textrm{diag}}$, see \S \ref{subsection:filtr-nichols}; we use below the notation from that \S.
By direct computation,
\begin{align*}
\Delta(z_1) &= z_1\otimes 1 -x_1\otimes x_3+1\otimes z_1.
\end{align*}
Suppose that $z_1\in\cB^2_2$. Notice that $\cB^2_2$ is spanned by the monomials in letters $x_1$ and $x_3$, since $T(V)^2_2$ is spanned by these monomials.
As $x_3x_1= q_{21}x_1x_3$, $ z_1= a_1\,x_1^2 + a_2\, x_1 x_3 + a_3\, x_3^2$ for some $a_i\in\ku$, so
\begin{align*}
0 &= \partial_1(z_1)=  a_1(1+\omega)\, x_1 + q_{12}a_2 \, x_3,\\
-x_1 &= \partial_3(z_1)=  a_2\, x_1 + a_3(1+\omega) \, x_3,
\end{align*}
Hence we have a contradiction, and $z_1\in\cB^2_3-\cB^2_2$.

Let $x_4$ be class of $z_1$ in $\Bdiag$. Then $x_4$ is a non-zero primitive element in $\Bdiag$.
Let $\widetilde{\cB}_1$ be the subalgebra of $\Bdiag$ generated by $Z = \langle x_1,x_2,x_3,x_4 \rangle$,
and consider the algebra filtration of $\widetilde{\cB }_1$, such that the generators have degree one.
This is a Hopf algebra filtration, hence the associated graded algebra  $\widetilde{\cB}_2$ is a braided Hopf algebra.
The Nichols algebra $\NA(Z)$ is a subquotient of $\widetilde{\cB}_2$.
Then $\GK \NA(Z)\le \GK \widetilde{\cB}_2\le \GK \widetilde{\cB}_1\le \GK\Bdiag =\GK \cB (V)$.
The braided vector space $Z$ has diagram
\begin{align*}
\xymatrix{ & \overset{\omega}{\circ} \ar@{-}[ld]_{\omega^2} \ar@{-}[d]^{\omega^2} & \\
	\overset{\omega}{\circ} \ar@{-}[r]_{\omega^2} & \overset{\omega q_{22}}{\circ} \ar  @{-}[r]_{q_{22}^2}
	& \overset{q_{22}}{\circ}, }
\end{align*}
so $\GK\NA(Z)=\infty $ by Hypothesis \ref{hyp:nichols-diagonal-finite-gkd}, and then $\GK \cB (V)=\infty $.

\medbreak

\emph{Case 2: $\widetilde{q}_{12} \neq  1$.}
Here the braided vector space $V^{\diag}$ has Dynkin diagram
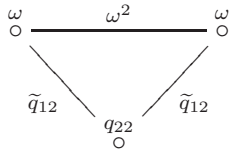
\begin{align*}
\xymatrix{ \overset{\omega}{\circ} \ar  @{-}[rr]^{\omega^2}\ar  @{-}[rd]_{\widetilde{q}_{12}} & & \overset{\omega}{\circ}
	\ar	@{-}[ld]^{\widetilde{q}_{12}}
	\\  & \overset{q_{22}}{\circ} &
}
\end{align*}
By Hypothesis  \ref{hyp:nichols-diagonal-finite-gkd} and after inspection of the Table 2 in \cite{H-classif}, we conclude that $q_{22} =-1$, $\widetilde{q}_{12} = \omega^2$.
Let  $\Bdiag := \gr \cB (V)$, a pre-Nichols algebra of  $V^{\textrm{diag}}$, see \S \ref{subsection:filtr-nichols}.
Set $x_{12}=(\ad_c x_1)x_2$, $x_{123}=(\ad_c x_1)z_1$. By direct computation,
\begin{align*}
\Delta(z_2) &= z_2\otimes 1 +\omega^2 (x_1^2+x_{12})\otimes x_3+\omega^2 x_1\otimes z_1 + 1\otimes z_2.
\end{align*}
Suppose that $z_2\in\cB^3_4$. Notice that $\cB^3_4$ is spanned by the monomials in letters $x_i$, $i\in \I_3$
with at most one $x_2$, since $T(V)^3_4$ is spanned by these monomials, cf. Remark \ref{obs:filtered} and Lemma \ref{lemma:filtered}.
We check that all these monomials are written as a linear combination of $x_1 x_{12}$, $x_1^2 x_2$, $x_{123}$, $x_1x_2x_3$, $x_1w_1$,
$w_1x_2$, $w_1 x_3$, $x_1z_1$, $z_1x_3$. As $\partial_1$, $\partial_2$ annihilate $z_2$, all the terms with a factor $x_1$ or $x_2$ do not appear. Thus $ z_2= a_1\,x_{123} + a_2\, w_1 x_3 + a_3\, z_1x_3$ for some $a_i\in\ku$, so
\begin{align*}
\omega^2 (x_1^2+x_{12}) &= \partial_3(z_2)=  a_1(1-\omega^2)(x_{12}+\omega x_1x_2) + a_2(w_1+(1-\omega^2)x_1x_3)\\
& \qquad + a_3 (z_1+(1-\omega^2)x_2x_3-\omega^2x_1x_3).
\end{align*}
As $x_1^2$, $x_{12}$, $x_1x_2$, $w_1$, $x_1x_3$, $z_1$, $x_2x_3$ are linearly independent, we have a contradiction. Thus $z_2\in\cB^3_5-\cB^3_4$.

Let $x_4$ be class of $z_2$ in $\Bdiag$. Then $x_4$ is a non-zero primitive element in $\Bdiag$.
Let $\widetilde{\cB}_1$ be the subalgebra of $\Bdiag$ generated by $Z = \langle x_1,x_2,x_3,x_4 \rangle$,
and consider the algebra filtration of $\widetilde{\cB }_1$, such that the generators have degree one.
This is a Hopf algebra filtration, hence the associated graded algebra  $\widetilde{\cB}_2$ is a braided Hopf algebra.
The Nichols algebra $\NA(Z)$ is a subquotient of $\widetilde{\cB}_2$.
Then $\GK \NA(Z)\le \GK \widetilde{\cB}_2\le \GK \widetilde{\cB}_1\le \GK\Bdiag =\GK \cB (V)$.
The braided vector space $Z$ has diagram
\begin{align*}
\xymatrix{ \overset{\omega}{\circ} \ar  @{-}[rr]^{\omega^2}\ar  @{-}[rd]_{\omega^2} & & \overset{\omega}{\circ} \ar	@{-}[ld]^{\omega^2} &
	\\  & \overset{-1}{\circ} \ar  @{-}[rr]^{\omega
	} & & \overset{-\omega^2}{\circ} }
\end{align*}
so $\GK\NA(Z)=\infty $ by Hypothesis \ref{hyp:nichols-diagonal-finite-gkd}, and then $\GK \cB (V)=\infty $.

\subsection{Admissible flourished diagrams}\label{subsection:examples-adm-fl-diag}
We comment on the general shape of these diagrams.
Let $\D$ be an admissible flourished graph with $t= t_+ + t_-$ blocks and $\theta$ vertices (i.e. $\theta - t$ points), cf. Definition \ref{def:flourished-graph-admissible}.
By Example \ref{example:admissible-mild}, we may assume that all interactions between points and blocks are weak, what we do in the following discussion.
As usual, let $\X$ be the set of connected components of $\D_{\diag}$ and let
\begin{align*}
\X_{\pm} &= \{J\in \X: J =\{j\}, q_{jj} = \pm 1 \},&
\X_{\text{gen}} &= \X - (\X_+ \cup \X_-).
\end{align*}
Then the axioms \ref{item:local}, \ref{item:concom} and \ref{item:concom-1} imply that
every $J \in \X_{\text{gen}}$ is connected to one and only block $\boxplus$.

Also, if $t =1$, then the classification follows from \ref{item:local}: either a block $\boxplus$ connected to various connected components as in Table \ref{tab:conn-comp} or else a block $\boxminus$ connected to various
points with label $\pm 1$.
Assume that $t > 1$. By the preceding remarks and connectedness of $\D$, the general shape of $D$ is as follows:
\begin{itemize}[leftmargin=*]
	\item For $h \in \I_{t_+}$, the $h$-th block  $\boxplus$ is linked to various connected components $J \in \X_{\text{gen}}$ as in Table \ref{tab:conn-comp}, and also to some connected components in $\X_{\pm}$; at least one of these last is linked to another block.
	
		\item For $h \in \I_{t_+ + 1, t}$, the $h$-th block  $\boxminus$ is linked to various connected components in $\X_{\pm}$; at least one of these last is linked to another block.
\end{itemize}
Here is an example:
\begin{align*}
\xymatrix{ & \overset{-1}{\bullet} \ar  @{-}[r]^{-1}  & \overset{-1}{\circ} \dots
	\overset{-1}{\circ} \ar  @{-}[r]^{-1}  & \overset{-1}{\circ} & \overset{-1}{\bullet} \ar  @{-}[r]^{-1}  & \overset{-1}{\circ}
\\ \boxplus \ar  @{-}[r]^{1}\ar  @{-}[ru]^{1} \ar @{-}[rrd]^{5}  & \overset{\omega}{\bullet} & \boxplus \ar  @{-}[r]^{60} \ar  @{-}[rru]^{2}  & \overset{-1}{\bullet}
\\ \overset{-1}{\bullet} \ar  @{-}[r]^{23} & \boxminus  \ar  @{-}[r]^{8}
&\overset{1}{\bullet} \ar  @{-}[r]^{3} \ar  @{-}[u]^{18} &
 \boxplus \ar  @{-}[r]^{1} &\overset{-1}{\bullet} \ar  @{-}[r]^{\omega^2 }  & \overset{\omega}{\circ}
}
\end{align*}

\end{document}